\documentclass{amsart}%
\usepackage{amsfonts, amsmath, amssymb}%
\usepackage{graphicx, epic,eepic, hyphenat, epsf}%
\usepackage{pstricks}
\usepackage{pst-grad}

\input xy
\xyoption{all}

\newtheorem{theorem}{Theorem}
\theoremstyle{definition}

\newtheorem{corollary}{Corollary}
\newtheorem{definition}{Definition}
\newtheorem{example}{Example}
\newtheorem{lemma}{Lemma}
\newtheorem{proposition}{Proposition}
\newtheorem{remark}{Remark}
\numberwithin{equation}{section}
\newcommand{\brak}[1]{\langle #1\rangle}

\DeclareMathOperator{\id}{id}
\def\co{\colon\thinspace} 
\def\mf{\mathfrak}

\newskip\stdskip                      % standard vertical space
\begin{document}

\title{Twin TQFT$s$ and Frobenius algebras}
\author{Carmen Caprau}

\address{Department of Mathematics, California State University, Fresno\\
5245 North Backer Avenue M/S PB 108, Fresno, CA 93740-8001, USA}
\email{ccaprau@csufresno.edu}

\date{}
\subjclass[2000]{57R56, 57M99; 81T40; 19D23}
\keywords{Cobordisms, Frobenius algebras, symmetric monoidal categories, TQFTs}
\thanks{The author was supported in part by NSF grant DMS 0906401}

\begin{abstract}

We introduce the category of singular 2-dimensional cobordisms and show that it admits a completely algebraic description as the free symmetric monoidal category on a \textit{twin Frobenius algebra}, by providing a description of this category in terms of generators and relations. A twin Frobenius algebra $(C, W, z, z^*)$ consists of a commutative Frobenius algebra $C,$ a symmetric Frobenius algebra $W,$ and an algebra homomorphism $z \co C \to W$ with dual $z^* \co W \to C$, satisfying some extra conditions. We also introduce a generalized 2-dimensional Topological Quantum Field Theory defined on singular 2-dimensional cobordisms and show that it is equivalent to a twin Frobenius algebra in a symmetric monoidal category.
\end{abstract} 
\maketitle

%%%%%%%%%%%%%%%%%%%%%%%%%%%%%%%%%%%%%%%%%%%%%%%%%%%
\section{Introduction}
 %%%%%%%%%%%%%%%%%%%%%%%%%%%%%%%%%%%%%%%%%%%%%%%%%%%
 A $2$-dimensional Topological Quantum Field Theory (TQFT) is a symmetric monoidal functor from the category \textbf{2Cob} of $2$-dimensional cobordisms to the category $\textbf{Vect}_k$ of vector spaces over a field $k.$ The objects in \textbf{2Cob} are smooth compact $1$-manifolds without boundary, and the morphisms are the equivalence classes of smooth compact oriented cobordisms between them, modulo diffeomorphisms that restrict to the identity on the boundary. The category \textbf{2Cob} of $2$-cobordisms and that of 2D TQFTs are well understood, and it is known that 2D TQFTs are characterized by commutative Frobenius algebras, in the sense that the category of 2D TQFTs is equivalent as a symmetric monoidal category  to the category of commutative Frobenius algebras. For the classic results involving these concepts, we refer to \cite{A, D, S} and the book~\cite{K}.
 
A. Lauda and H. Pfeiffer studied in~\cite{LP} a special type of extended TQFTs defined on \textit{open-closed cobordisms}. These cobordisms are certain smooth oriented $2$-manifolds with corners that can be viewed as cobordisms between compact $1$-manifolds with boundary, that is, between disjoint unions of circles $S^1$ and unit intervals $I = [0, 1].$ An \textit{open-closed TQFT} is a symmetric monoidal functor $Z \co \textbf{2Cob}^{ext} \to \textbf{Vect}_k,$ where $\textbf{2Cob}^{ext}$ denotes the category of open-closed cobordisms.  Lauda and Pfeiffer showed that open-closed TQFTs are characterized by what they call \textit{knowledgeable Frobenius algebras} $(A, C, \iota, \iota^*),$ where the vector space $C : = Z(S^1)$ associated with the circle has the structure of a commutative Frobenius algebra, the vector space $A : = Z(I)$ associated with the interval has the structure of a symmetric Frobenius algebra, and there are linear maps $\iota \co C \to A$ and $\iota^* \co A \to C$ satisfying certain conditions. This result was obtained by providing a description of the category of open-closed cobordisms in terms of generators and the Moore-Segal relations. They defined a normal form for such cobordisms, characterized by topological invariants, and then proved the sufficiency of the relations by constructing a sequence of \textit{moves} which transforms the given cobordism into the normal form. They also showed that the category $\textbf{2Cob}^{ext}$ of open-closed cobordisms is equivalent to the symmetric monoidal category freely generated by a knowledgeable Frobenius algebra. We remark that the entire construction in~\cite{LP} was given for an arbitrary symmetric monoidal category, and not only for $\textbf{Vect}_k.$
 
In \cite{CC0}, the author constructed a bigraded tangle cohomology theory which depends on one parameter, via a setup with webs and dotted foams (singular cobordisms) modulo a finite set of local relations. This work is a blend of Bar-Natan's \cite{BN2} approach to `local' Khovanov homology and Khovanov's framework \cite{Kh3} using webs and foams. In \cite{CC1}, the author generalized the construction given in \cite{CC0} to a two-parameter theory for tangles, which we call the \textit{universal $\mf{sl}(2)$ foam cohomology} (it is `universal' in the sense of \cite{Kh2}). This two-parameter theory corresponds to a certain Frobenius algebra structure defined on $\mathbb{Z}[i, a, h, X]/(X^2 -hX -a),$ and which, for the case of of links, is a categorification of the quantum $\mf{sl}(2)$-link invariant. Adding the relation $a = h = 0$ yields an isomorphic version of the $\mf{sl}(2)$ Khovanov homology~\cite{BN1, Kh1}, while imposing $a = 1, h = 0$ recovers Lee's theory \cite{Lee}. The advantage of working with foams instead of classical 2-cobordisms is that the construction in \cite{CC1} (as well as its particular case introduced in ~\cite{CC0}) yields a theory that satisfies a honest functoriality property with respect to tangle or link cobordisms, rel. boundary, that is, with no sign indeterminacy. In particular, it resolves the sign ambiguity residing in the functoriality property of the Khovanov homology.  We note that there is also the work by Clark, Morrison and Walker \cite{CMW} that fixes the functoriality property of Khovanov's invariant through the use of singular cobordisms (they called these `disoriented cobordisms').

In \cite{CC2}, the author described a method that computes fast and efficient the $\mf{sl}(2)$ foam cohomology groups, and also provided a purely topological version of the $\mf{sl}(2)$ foam theory in which no dots are required on cobordisms. 

 We briefly review below the gadgets used in \cite{CC0, CC1}. A \textit{web} is a planar graph with bivalent vertices near which the two incident edges are oriented either towards the vertex or away from it. Webs without vertices, thus oriented circles, are also allowed. Examples of webs are depicted in~\eqref{eq:examples_webs}. A \textit{foam} (also called \textit{singular cobordism}) is an abstract cobordism between webs---regarded up to boundary-preserving isotopies---and has \textit{singular arcs} (and/or \textit{singular circles}) where orientations disagree, and near which the facets incident with a certain singular arc are compatibly oriented, inducing an orientation on that arc. Examples of such cobordisms are given in~\eqref{eq:examples_singcobordisms} (the red curves in a singular cobordism diagram are singular arcs/circles).

The author arrives at webs and foams by considering a link diagram $L$ (we talk here about links instead of tangles just for simplicity) and resolving each crossing in $L$ in one of the following ways:
 \[ \raisebox{-13pt}{\includegraphics[height=0.4in]{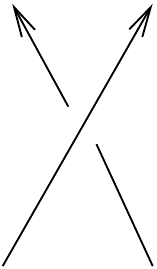}} \longrightarrow \, \raisebox{-13pt} {\includegraphics[height=0.4in]{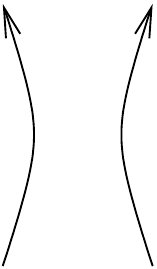}} \quad \text{and} \quad \raisebox{-13pt} {\includegraphics[height=0.4in]{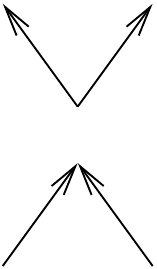}} \]
Then she associates to $L$ a `formal' chain complex [L], whose objects are formally graded webs, and whose morphisms are formal linear combinations of foams. When considered as an object in the category of complexes of foams modulo a certain set of local relations, $[L]$ is an up-to-homotopy invariant of $L$. The jump from the geometric  setup to an algebraic one, allowing to obtain a cohomology theory, is done via a `tautological' functor' similar to the one used in ~\cite{BN2}.

There are many similarities between the algebraic structure of the category of open-closed cobordisms and certain relations satisfied by the singular cobordisms, although the two types of cobordisms are topologically different. The original Khovanov homology relies on a 2D TQFT, and it would be quite desirable and refreshing to have some kind of TQFT defined on foams/singular cobordisms, and use it to obtain another method for defining the universal $\mf{sl}(2)$ foam cohomology theory (specifically, the purely topological one---with no dots on cobordisms---discussed in \cite{CC2}, Section 4), and a generalization of the Khovanov homology. In particular, this would provide us with knowledge of the algebraic structure that governs this cohomology theory that provides a properly functorial Khovanov homology theory.

In this paper we make the first step in achieving this goal. The singular cobordisms considered here are a particular case of those used in~\cite{CC0, CC1, CC2}, in the sense that the 1-manifolds are disjoint unions of oriented circles and bi-webs (webs with \textit{exactly} two bivalent vertices). The second step in reaching our goal will be treated in a subsequent paper, where we also show that it suffices to work with bi-webs, as opposed to arbitrary webs (webs with an even number of bivalent vertices).

We introduce the category $\textbf{Sing-2Cob}$ of singular $2$-cobordisms and show that it is equivalent as a symmetric monoidal category to the category freely generated by, what we call, a \textit{twin Frobenius algebra}. A twin Frobenius algebra is almost the same as a knowledgeable Frobenius algebra; specifically, all properties of the latter one are satisfied by the first one, except for the ``Cardy relation" which is replaced by what we call the ``genus-one relation". The definition of twin Frobenius algebras and their category is given in Section~\ref{sec:twin FA}. 

We present in Section~\ref{sec:singcobordisms} a normal form for an arbitrary singular $2$-cobordism and characterize the category $\textbf{Sing-2Cob}$ in terms of generators and relations. In Section~\ref{sec:TQFTs} we define \textit{twin TQFTs}  as symmetric monoidal functors $\textbf{Sing-2Cob} \to \mathcal{C},$ where $\mathcal{C}$ is an arbitrary symmetric monoidal category, and prove that  the category of twin TQFTs in $\mathcal{C}$ is equivalent, as a symmetric monoidal category, to the category of twin Frobenius algebras in $\mathcal{C}.$ 

In Section~\ref{sec:example} we provide examples of twin Frobenius algebras and thus twin TQFTs in $\mathcal{C}  = \textbf{Vect}_k$ and/or $\mathcal{C} = \textbf{R-Mod},$ where $k$ is a field and $R$ a commutative ring.

%%%%%%%%%%%%%%%%%%%%%%%%%%%%%%%%%%%%%%%%%%%%%%%%%%%
\section{Twin Frobenius Algebras} \label{sec:twin FA}
%%%%%%%%%%%%%%%%%%%%%%%%%%%%%%%%%%%%%%%%%%%%%%%%%%%

\subsection{Definitions}

Throughout the paper, we consider an arbitrary symmetric monoidal (tensor) category $(\mathcal{C}, \otimes, \textbf{1}, \alpha, \lambda, \rho, \tau)$ with unit object $\textbf{1} \in \mathcal{C},$ associativity law  $\alpha_{X,Y,Z}: (X \otimes Y) \otimes Z \to X \otimes (Y \otimes Z),$ left-unit and right-unit laws $\lambda_X \co \textbf{1} \otimes X \to X$ and  $\rho_X \co X \otimes \textbf{1} \to X,$ and with symmetric braiding $\tau_{X,Y} \co X \otimes Y\to Y \otimes X,$ for $X, Y$ and $Z$ objects in $\mathcal{C}.$

For reader's convenience, we recall a few definitions.

An \textit{algebra object} $(C, m, \iota)$ in $\mathcal{C}$ consists of an object $C$ and morphisms $m \co C \otimes C \to C$ and $\iota \co \textbf{1} \to C$ in $\mathcal{C}$ such that:
\[ m \circ (\id_C \otimes m) \circ \alpha_{C,C,C} = m \circ (m \otimes \id_C) \]
\[ m \circ (\id_C \otimes \iota) = \rho_C \quad \text{and} \quad m \circ (\iota \otimes \id_C) = \lambda_C . \]

A \textit{coalgebra object} $(C, \Delta, \epsilon)$ in $\mathcal{C}$ is an object $C$ and morphisms $\Delta \co C \to C\otimes C$ and $\epsilon \co C \to \textbf{1}$ such that:
\[(\id_C \otimes \Delta) \circ \Delta = \alpha_{C,C,C} \circ (\Delta \otimes \id_C) \circ \Delta  \]
\[ (\id_C \otimes \epsilon) \circ \Delta = \rho_C^{-1} \quad \text{and} \quad (\epsilon \otimes \id_C)\circ \Delta = \lambda_C^{-1} .\]

A \textit{homomorphism of algebras} $f \co C \to C'$ between two algebra objects $(C, m, \iota)$ and $(C', m', \iota')$ in $\mathcal{C}$ is a morphism $f$ of $\mathcal{C}$ such that:
\[ f \circ m = m' \circ (f \otimes f) \quad \text{and} \quad f \circ \iota = \iota'. \]

A \textit{homomorphism of coalgebras} $f \co C \to C'$ between two coalgebra objects $(C, \Delta, \epsilon)$ and $(C', \Delta', \epsilon')$ in $\mathcal{C}$ is a morphism $f$ of $\mathcal{C}$ such that:
\[(f \otimes f) \circ \Delta = \Delta' \circ f \quad \text{and} \quad \epsilon' \circ f = \epsilon.  \]

A \textit{Frobenius algebra object} $(C, m, \iota, \Delta, \epsilon)$ in $\mathcal{C}$ consists of an object $C$ together with morphisms $m, \iota, \Delta, \epsilon$ such that:
\begin{itemize}
\item $(C, m, \iota)$ is an algebra object and $(C, \Delta, \epsilon)$ is a coalgebra object in $\mathcal{C},$
\item $(m \otimes \id_C) \circ \alpha_{C, C, C}^{-1} \circ (\id_C \otimes \Delta) = \Delta \circ m = (\id_C \otimes m) \circ \alpha_{C, C, C} \circ (\Delta \otimes \id_C).$
\end{itemize}

A Frobenius object $(C, m, \iota, \Delta, \epsilon)$ in $\mathcal{C}$ is called \textit{commutative} if $m \circ \tau = m,$ and it is called \textit{symmetric} if $\epsilon \circ m = \epsilon \circ m \circ \tau.$

Given two Frobenius algebra objects $(C, m, \iota, \Delta, \epsilon)$ and $(C', m', \iota', \Delta', \epsilon'),$ a \textit{homomorhism of Frobenius algebras} $f \co C \to C'$ is a morphism $f$ in $\mathcal{C}$ which is both a homomorphism of algebra and coalgebra objects.

\begin{definition}
 A \textit{twin Frobenius algebra} $\mathbf{T}: = (C, W, z, z^*)$ in $\mathcal{C}$ consists of
\begin{itemize}
\item  a commutative Frobenius algebra $C = (C, m_C, \iota_C, \Delta_C, \epsilon_C),$
\item a symmetric Frobenius algebra $W = (W, m_W, \iota_W, \Delta_W, \epsilon_W),$
\item two morphisms $z \co C \to W$ and $z^* \co W \to C$ of $\mathcal{C}$
\end{itemize}
such that $z$ is a homomorphism of algebra objects in $\mathcal{C}$ and
\begin{equation}
\epsilon_C \circ m_C \circ (\id_C \otimes z^*) = \epsilon_W \circ m_W \circ (z \otimes \id_W),  \hspace{2cm}(\text{duality})
\label{eq:z^* dual to z}
\end{equation}
\begin{equation}
m_W \circ (\id_W \otimes z) = m_W \circ \tau_{W, W} \circ (\id_W \otimes z),  \hspace{1.5cm} (\text{centrality condition})
\label{eq:knowledge_center}
\end{equation}
\begin{equation}
z \circ m_C \circ \Delta_C \circ z^* = m_W \circ \tau_{W,W} \circ \Delta_W. \hspace{2cm}(\text{genus-one condition})
\label{eq:genus1_relation}
\end{equation}

The first equality says that  $z^*$ is the morphism dual to $z$ (which implies that $z^*$ is a homomorphism of coalgebras in $\mathcal{C}$). If $\mathcal{C} = \textbf{Vect}_k,$ the second equality says that $z(C)$ is contained in the center of the algebra $W.$
\end{definition}

The reader will notice the similarities between twin and knowledgeable Frobenius algebras: their properties are almost the same, except that the \textit{Cardy condition} for a knowledgeable Frobenius algebra is replaced by the \textit{genus-one condition} in the definition of a twin Frobenius algebra.

\begin{definition}
 A \textit{homomorphism of twin Frobenius algebras}
\[ f \co (C_1, W_1, z_1, z_1^*) \to (C_2, W_2, z_2, z_2^*) \]
in a symmetric monoidal category $\mathcal{C}$ consists of a pair $f = (f_1, f_2)$ of Frobenius algebra homomorphisms $f_1 \co C_1 \to C_2 $ and $f_2 \co W_1 \to W_2$ such that $z_2 \circ f_1 = f_2 \circ z_1$ and $z_2^*\circ f_2 = f_1 \circ z_1^*.$
\end{definition}

\begin{definition}
 We denote by $\textbf{T-Frob}(\mathcal{C})$ the category whose objects are twin Frobenius algebras in $\mathcal{C}$ and whose morphisms are twin Frobenius algebra homomorphisms.
\end{definition}

\begin{proposition}
The category $\textbf{T-Frob}(\mathcal{C})$ forms a symmetric monoidal category in the following sense:
\begin{itemize}
\item The tensor product of two twin Frobenius algebra objects $\mathbf{T_1} = (C_1, W_1, z_1, z_1^*)$ and $\mathbf{T_2} = (C_2, W_2, z_2, z_2^*)$ is defined as $$\mathbf{T_1} \otimes \mathbf{T_2}: = (C_1 \otimes C_2, W_1 \otimes W_2, z_1 \otimes z_2, z_1^* \otimes z_2^*);$$
\item The unit object is given by $ \overline{\textbf{1}} : = (\textbf{1}, \textbf{1}, \id_{\textbf{1}}, \id_{\textbf{1}});$
\item The associativity and unit laws and the symmetric braiding are induced by those of $\mathcal{C};$
\item The tensor product of two homomorphisms $f = (f_1, f_2)$ and $g = (g_1, g_2)$ of twin Frobenius algebras is defined as $f \otimes g : = (f_1 \otimes f_2,\, g_1\otimes g_2).$

\end{itemize}
\end{proposition}

%%%%%%%%%%%%%%%%%%%%%%%%%%%%%%%%%%%%%%%%%%%%%%%%%%%
\subsection{The category \textbf{Th}(\textbf{T-Frob})}\label{sec:theory}
%%%%%%%%%%%%%%%%%%%%%%%%%%%%%%%%%%%%%%%%%%%%%%%%%%%

The definition of the category called the \textit{theory of twin Frobenius algebras}, denoted by \textbf{Th}(\textbf{T-Frob}), follows Laplaza's~\cite{L} construction of the `free category with group structure'.  The objects of this category are elements of the free $\{\textbf{1}, \otimes\}$-algebra over the two element set $\{C, W\}$ and is the analogue of the category \textbf{Th}(\textbf{K-Frob}) introduced in~\cite[Section 2.2]{LP}.

The objects of \textbf{Th}(\textbf{T-Frob}) are words generated by the symbols $\textbf{1}, C$ and $W,$ which are objects by themselves. If $X$ and $Y$ are objects of \textbf{Th}(\textbf{T-Frob}) then $(X \otimes Y)$ is also an object.

Consider a graph $\mathcal{G}$ whose vertices are the objects of \textbf{Th}(\textbf{T-Frob}). Then, there are the following edges: 
\[m_C \co C \otimes C \to C, \quad \iota_C \co \textbf{1} \to C, \quad \Delta_C \co C \to C \otimes C, \quad \epsilon_C \co C \to \textbf{1},  \]
\[m_W \co W \otimes W \to W, \quad \iota_W \co \textbf{1} \to W, \quad \Delta_W \co W \to W \otimes W, \quad \epsilon_W \co W \to \textbf{1},  \]
\[z \co C \to W, \quad z^* \co W \to C.  \]

For all objects $X, Y, Z$ there are the following edges:
\[ \alpha_{X, Y, Z} \co (X \otimes Y) \otimes Z \to X \otimes (Y \otimes Z), \quad \overline{\alpha}_{X, Y, Z} \co X \otimes (Y \otimes Y) \to (X \otimes Y) \otimes Z,\]
\[ \tau_{X, Y} \co X \otimes Y \to Y \otimes X, \quad \overline{\tau}_{X, Y} \co Y \otimes X \to X \otimes Y, \]
\[  \lambda_X \co \textbf{1} \otimes X \to X, \quad \overline{\lambda}_X \co X \otimes \textbf{1} \to X,  \quad \rho_X \co X \otimes \textbf{1} \to X, \quad \overline{\rho}_X \co X \to X \otimes \textbf{1}. \]

For every edge $f \co X \to Y$ and for every object $Z,$ there are edges $f \otimes Z \co X \otimes Z \to Y \otimes Z$ and $Z \otimes f \co Z \otimes X \to Z \otimes Y.$ 

We denote by $\mathcal{H}$ the category freely generated by the graph $\mathcal{G},$ and define a relation $\sim$ on $\mathcal{H}$ by requiring the following:
\begin{itemize}
\item that each pair of edges $e$ and $\overline{e}$ are inverses of each other;
\item the relations that make $(C, m_C, \iota_C, \Delta_C, \epsilon_C)$ a commutative Frobenius algebra and $(W, m_W, \iota_W, \Delta_W, \epsilon_W)$ a symmetric Frobenius algebra;
\item the relations that make $z \co C \to W$ an algebra homomorphism and those given in \eqref{eq:z^* dual to z}, \eqref{eq:knowledge_center} and \eqref{eq:genus1_relation};
\item the relations that make $\alpha_{X, Y, Z}, \,\lambda_X$ and $\rho_X$ satisfy the pentagon and triangle axioms of a monoidal category and those that make $\tau_{X,Y}$ a symmetric braiding, for all objects $X,Y, Z;$
\item the relations that state the naturality of $\alpha, \lambda, \rho, \tau, \overline{\alpha}, \overline{\lambda}, \overline{\rho}, \overline{\tau}$ and those that make $\otimes$ a functor.

\end{itemize}
We also require that for each relation $a \sim b,$ we have as well the relations $a \otimes X \sim b \otimes X$ and $X \otimes a \sim X \otimes b$ for all objects $X$ in $\mathcal{H},$ as well as the relations obtained from these by applying this process a finite number of times. 

We define the category $\textbf{Th}(\textbf{T-Frob}) : = \mathcal{H}_{/\sim},$ that is, the category $\mathcal{H}$ modulo the category congruence generated by $\sim$ defined above. The category $\textbf{Th}(\textbf{T-Frob})$ contains a twin Frobenius algebra object $\mathbf{T} = (C, W, z, z^*)$ and is the symmetric monoidal category freely generated by $\mathbf{T}.$ For each twin Frobenius algebra $\mathbf{T'} = (C', W', z', z'^*)$ in $\mathcal{C},$ there is exactly one strict symmetric monoidal functor $F_{\mathbf{T'}} \co \textbf{Th}(\textbf{T-Frob}) \to \mathcal{C}$ that maps $\mathbf{T}$ to $\mathbf{T'}$ and $\mathbf{1} \in \textbf{Th}(\textbf{T-Frob})$ to $\mathbf{1} \in \mathcal{C}.$

\begin{proposition}
The category \textbf{Th}(\textbf{T-Frob})  is equivalent as a symmetric monoidal category to the category  of symmetric monoidal functors $\textbf{Th}(\textbf{T-Frob}) \to \mathcal{C}$ and their monoidal natural transformations.
\end{proposition}

\begin{proof}
The monoidal equivalence of the two categories is constructed identical to that in the proof of~\cite[Proposition 2.8]{LP}. One has only to replace the category $\textbf{Th}(\textbf{K-Frob})$ used in~\cite{LP} with our category $\textbf{Th}(\textbf{T-Frob}).$
\end{proof}

%%%%%%%%%%%%%%%%%%%%%%%%%%%%%%%%%%%%%%%%%%%%%%%%%%% 
\section {Singular cobordisms and the category $\textbf{Sing-2Cob}$} \label{sec:singcobordisms}
%%%%%%%%%%%%%%%%%%%%%%%%%%%%%%%%%%%%%%%%%%%%%%%%%%%

In this section we define the category of singular $2$-cobordisms and give a presentation of it in terms of generators and relations. Singular $2$-cobordisms form a special type of compact, globally oriented, smooth 2-manifolds. What we call $\textbf{Sing-2Cob}$ in the following is in fact a skeleton of the category of singular 2-cobordisms; we choose particular embedded 1-manifolds as the objects of this category. 

\subsection{Description and topological invariants}

  \begin{definition} A \textit{singular 2-cobordism} (or shortly, singular cobordisms) is an abstract, piecewise oriented (but globally oriented) smooth 2-manifold $\Sigma$ with boundary $\partial \Sigma = \overline{\partial^- \Sigma} \cup \partial^+ \Sigma,$ where $\overline{\partial^ - \Sigma}$ is $\partial^- \Sigma$ with opposite orientation. Both $\partial^- \Sigma$ and $\partial^+ \Sigma$ are embedded, closed 1-manifolds, called the source and target boundary, respectively. A singular cobordism has \textit{singular arcs} and/or \textit{singular circles} where orientations disagree. There are exactly two compatibly oriented 2-cells of the underlying 2-dimensional CW-complex $\Sigma$ that meet at a singular arc/circle, and orientations of  two neighboring 2-cells induce an orientation on the singular arc/circle that they share. In our diagrams we draw the singular arcs/circles using red oriented curves.

Two singular cobordisms $\Sigma_1$ and $\Sigma_2$ are considered \textit{equivalent}, and we write $\Sigma_1\cong \Sigma_2,$  if there exists an orientation-preserving diffeomorphism $\Sigma_1 \to \Sigma_2$  which restricts to the identity on the boundary. 
\end{definition}

The boundary $\partial \Sigma$ of a singular cobordism $\Sigma$ is a disjoint union of (clockwise) oriented circles and \textit{bi-webs}. In this paper, a \textit{bi-web} is a closed oriented graph with two bivalent vertices, such that each vertex is either a source or a sink. Since we want to work with the skeleton of the category of singular cobordisms, we fix one specific oriented circle, and a specific bi-web. We denote the circle by $0$ and the bi-web by $1$:
\begin{equation}
  0 = \raisebox{-5pt}{\includegraphics[height=0.2in]{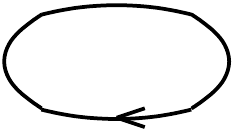}}\qquad 1= \raisebox{-5pt}{\includegraphics[height=0.2in]{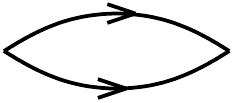}}  \label{eq:examples_webs}
  \end{equation}

\begin{definition} An object in the category $\textbf{Sing-2Cob}$ consists of a finite sequence $\textbf{n} = (n_1, n_2,\dots, n_k),$ where $n_j \in\{0,1\}.$ The length of the sequence, denoted by $\vert \textbf{n} \vert = k,$ can be any nonnegative integer, and equals the number of disjoint connected components of the corresponding object. Hence $\textbf{n}$ is the disjoint union of $s$ copies of the fixed circle and $t$ copies of the fixed bi-web, for some non-negative integers $s$ and $t$ with $s + t = \vert \textbf{n} \vert$. If $\vert \textbf{n} \vert = 0$, then $\textbf{n}$ is the empty 1-manifold.

A morphism $\Sigma\co \textbf{n} \to \textbf{m}$ in $\textbf{Sing-2Cob}$ is an equivalence class $[\Sigma]$ (induced by $\cong$) of singular 2-cobordisms with source boundary $\textbf{n}$ and target boundary $\textbf{m}.$ 
The composition of morphisms is obtained in the standard way, namely by gluing along the common boundary.
 \end{definition}
 
 Although the objects of $\textbf{Sing-2Cob}$ are embedded 1-manifolds, the morphisms are not embedded. Examples of morphisms of $\textbf{Sing-2Cob}$ are given below. The source of our cobordisms is at the top and the target at the bottom of drawings, in other words, we read morphisms as cobordisms from top to bottom, by convention.
 \begin{equation}
 \raisebox{-5pt}{\includegraphics[height=0.4in]{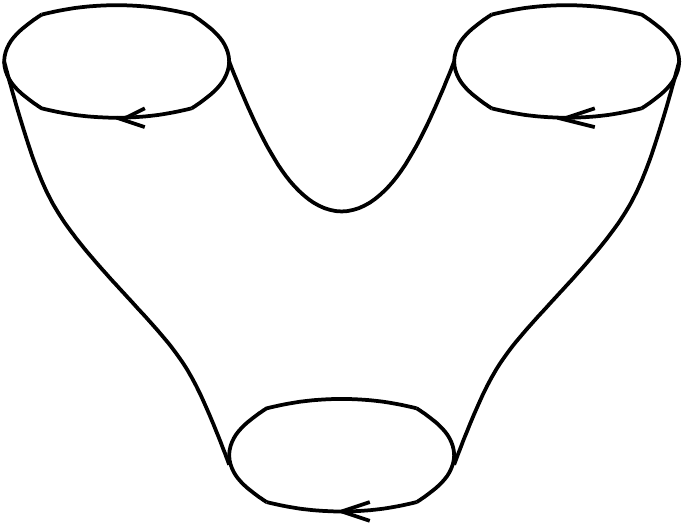}}\co (0,0) \to (0) \qquad  \raisebox{-5pt}{\includegraphics[height=0.4in]{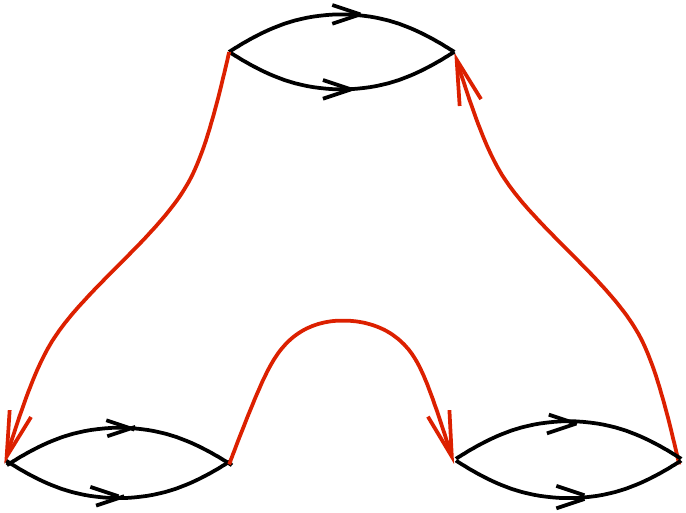}}\co (1) \to (1,1) \qquad \raisebox{-5pt}{\includegraphics[height=0.4in]{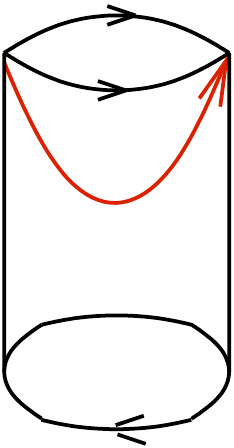}}\co (1) \to (0) \label{eq:examples_singcobordisms}
 \end{equation}

The concatenation $\textbf{n} \amalg \textbf{m} := (n_1, n_2, \dots, n_{\vert \textbf{n} \vert}, m_1, m_2, \dots, m_{\vert \textbf{m} \vert})$ of sequences together with the free union of singular cobordisms, which we also denote by $\amalg,$ endows the category $\textbf{Sing-2Cob}$ with the structure of a symmetric monoidal category. 

For each $k \in \mathbb{N},$ there is an action of the symmetric group $S_k$ on the subset of objects $\textbf{n}$ in $\textbf{Sing-2Cob}$ for which $\vert n \vert = k,$ defined by
\[ \sigma * \textbf{n}: = (n_{\sigma^{-1}(1)}, n_{\sigma^{-1}(2)}, \dots, n_{\sigma^{-1}(k)} ). \]
Given any object $\textbf{n}$ in $\textbf{Sing-2Cob}$ and any permutation $\sigma \in S_{\vert n \vert},$ there is an obvious induced  cobordism
\[ \sigma^{\textbf{n}}: \textbf{n} \to \sigma * \textbf{n}.\]
For example, if $\textbf{n} = (0, 1, 1, 0, 1)$ and $\sigma = (12) (354) \in S_5,$ the corresponding morphism $\sigma^{\textbf{n}}$ is the singular cobordism given in~(\ref{eq:permutation}).
\begin{equation}
\sigma^{\textbf{n}} = \raisebox{-13pt}{ \includegraphics[height=0.5in]{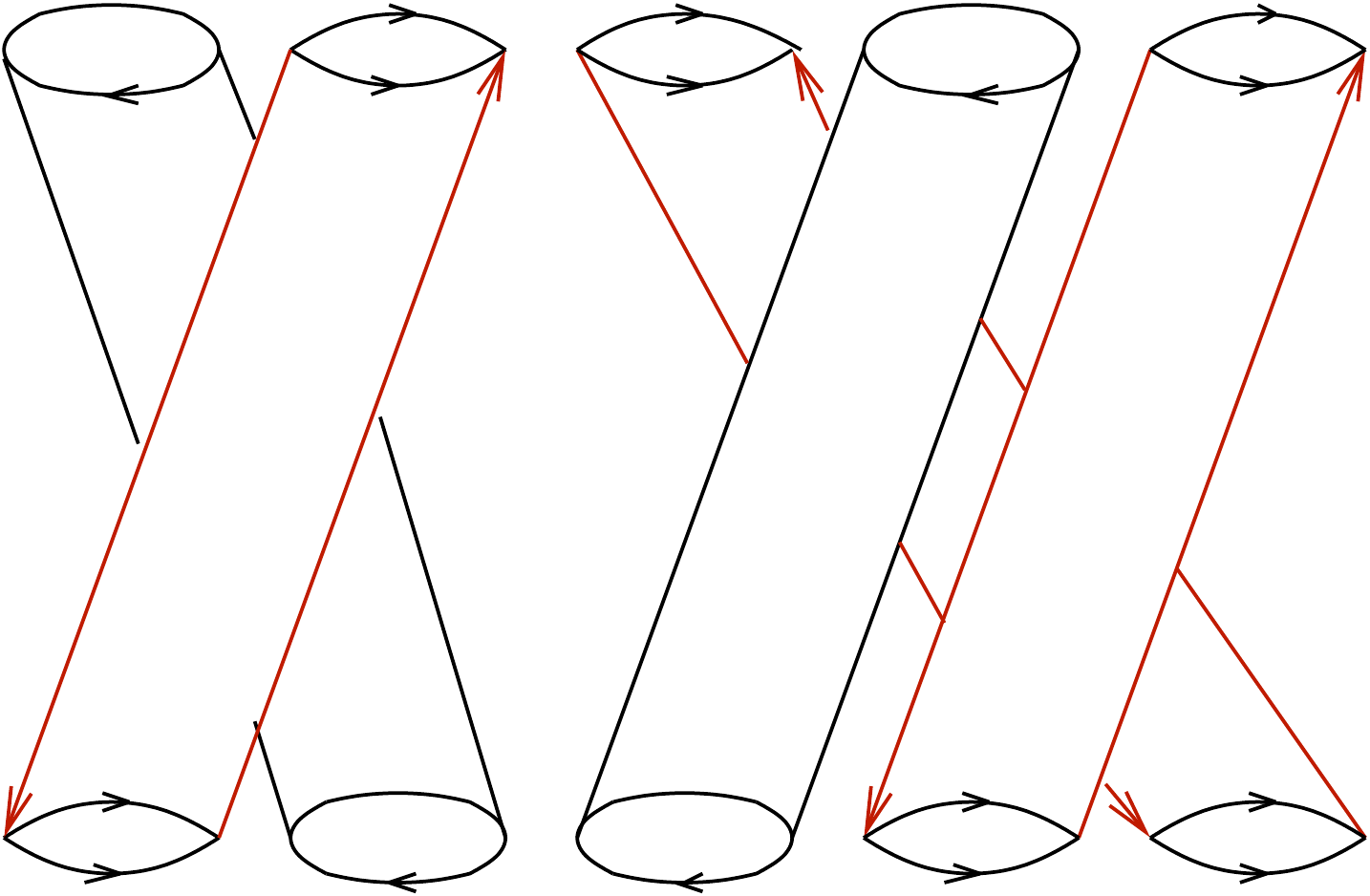}}
\label{eq:permutation}
\end{equation}

We remark that as morphisms of $\textbf{Sing-2Cob},$ these cobordisms satisfy $ \tau^{\sigma * \textbf{n}} \circ \sigma^ { \textbf{n}} = (\tau \circ \sigma)^{\textbf{n}},$ for any object $\textbf{n}$ and $\sigma, \tau \in S_{\vert n \vert}.$

\begin{definition}\label{def:boundary-permutation}
Let $\Sigma \co \textbf{n} \to \textbf{m}$ be a morphism in $\textbf{Sing-2Cob}$ and let $l$ be the number of its boundary components representing the bi-web. In other words, $l$ is the number of 1 entries of $\textbf{n} \amalg \textbf{m}.$ Number these components by $1, 2, \dots, l.$ The orientation of $\Sigma$ induces an orientation on all singular arcs of $\Sigma$ and defines a permutation $\sigma (\Sigma) \in S_l,$ called the \textit{singular boundary permutation} of $\Sigma.$  
\end{definition}\
For exemplification, we consider the morphism $\Sigma$ depicted in~(\ref{eq:boundary_permutation}) and we number the bi-webs in its boundary by $1, 2, 3$ and $4$ from left to right, starting with those in the source and followed by those in the target. The singular boundary permutation of this morphism is $\sigma(\Sigma) = (1)(234) = (234) \in S_4.$
\begin{equation}
\Sigma = \raisebox{-30pt}{ \includegraphics[height=1in]{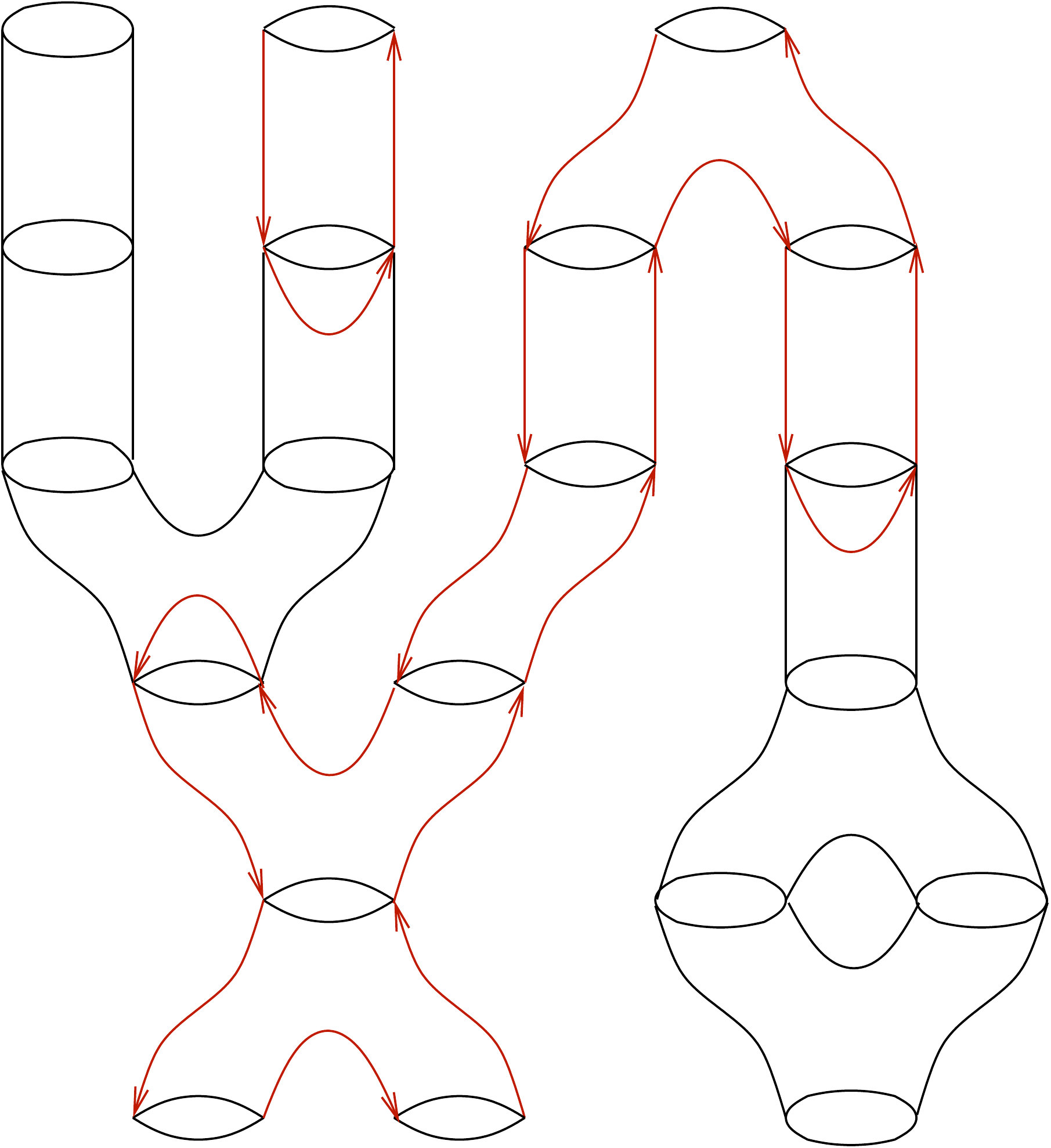}}
\label{eq:boundary_permutation}
\end{equation}

Since a singular cobordism determines a boundary permutation, we need to refine the definition of the morphisms in $\textbf{Sing-2Cob}.$

\begin{definition}
A morphism $\Sigma \co \textbf{n} \to \textbf{m}$ is a pair $\Sigma = ([\Sigma], \sigma)$ consisting of an equivalence class $[\Sigma]$ of singular cobordisms with source boundary $\textbf{n}$ and target  boundary $\textbf{m},$ and with singular boundary permutation $\sigma(\Sigma) = \sigma.$
\end{definition}

\begin{definition}\label{def:g-s}
Let $\Sigma \co \textbf{n} \to \textbf{m}$ be a morphism of $\textbf{Sing-2Cob}.$
\begin{enumerate}
\item [1.] The \textit{genus} $g(\Sigma)$ is the genus of the topological $2$-manifold underlying $\Sigma.$ 
\item [2.] The \textit{singular number} $s(\Sigma)$ is the number of singular circles that $\Sigma$ contains. (Note that each such circle is homotopic to a point within $\Sigma.$)
\end{enumerate} \end{definition}
For example, $s(\,\raisebox{-8pt}{\includegraphics[height=0.3in]{cozipper.pdf}} \circ \raisebox{-8pt}{\includegraphics[height=0.3in]{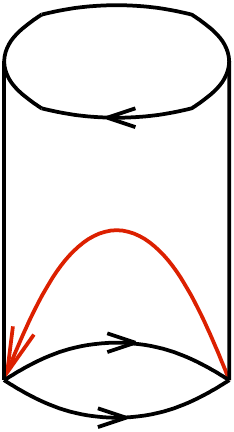}} \,) = s(\, \raisebox{-13pt}{\includegraphics[height=0.52in]{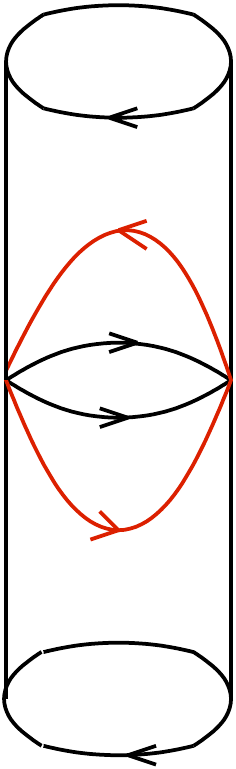}} \,) = 1$

%%%%%%%%%%%%%%%%%%%%%%%%%%%%%%%%%%%%%%%%%%%%%%%%%%%
\subsection{Structural Relations} 
%%%%%%%%%%%%%%%%%%%%%%%%%%%%%%%%%%%%%%%%%%%%%%%%%%%

In this subsection we provide a list of diffeomorphisms  which describe the algebraic structure of the category $\textbf{Sing-2Cob}$.

\begin{proposition}\label{prop:relations}
The following diffeomorphisms hold in the symmetric monoidal category $\textbf{Sing-2Cob}:$
\begin{enumerate}
\item The object $\textbf{n} = (0)$ forms a commutative Frobenius algebra object.

\begin{equation}
\raisebox{-13pt}{\includegraphics[height=0.5in]{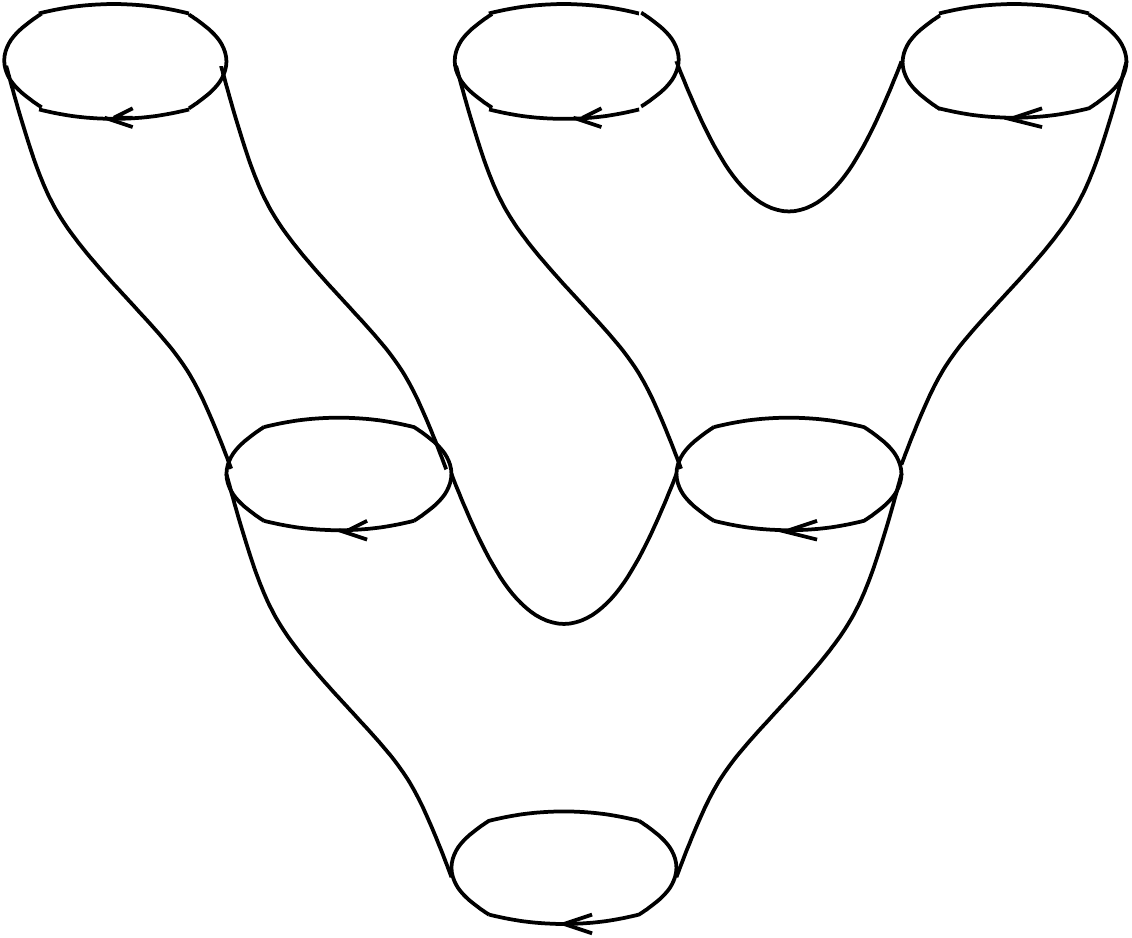}}\quad \cong \quad  \raisebox{-13pt}{\includegraphics[height=0.5in]{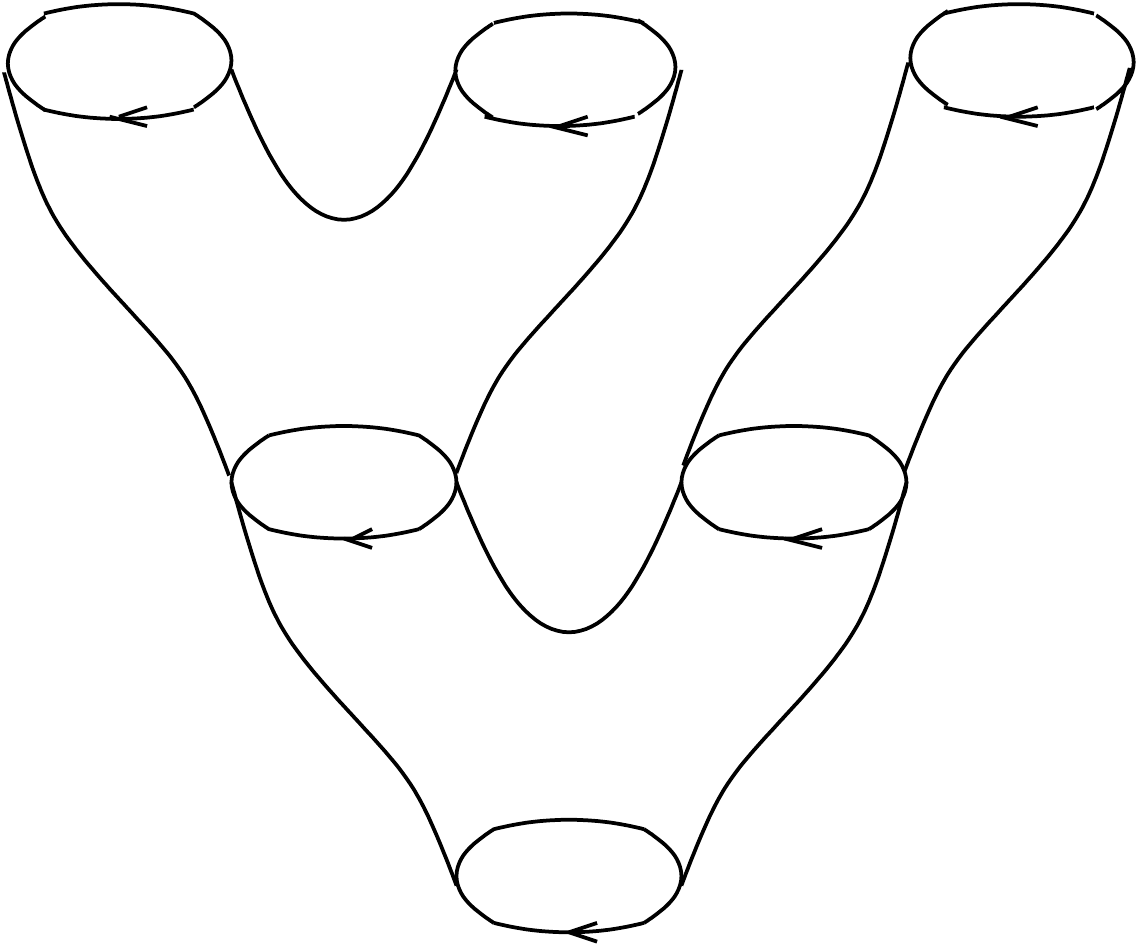}} \hspace{1cm} \raisebox{-13pt}{\includegraphics[height=0.5in]{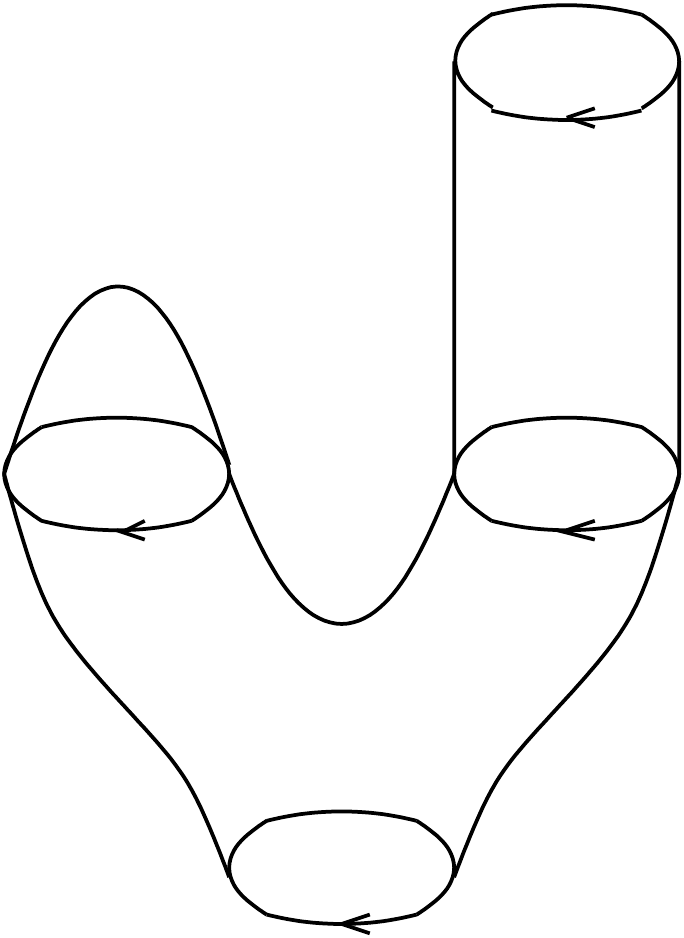}}\quad \cong \quad  \raisebox{-13pt}{\includegraphics[height=0.5in]{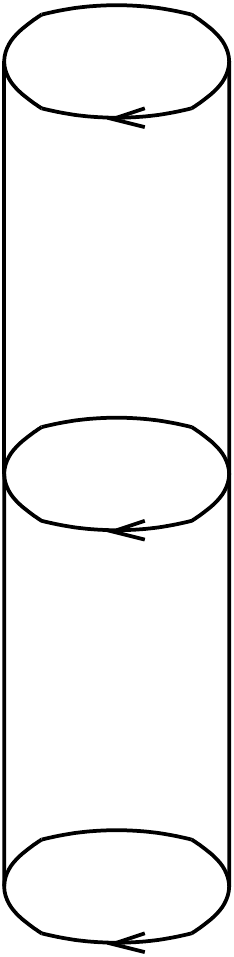}}\quad \cong \quad \raisebox{-13pt}{\includegraphics[height=0.5in]{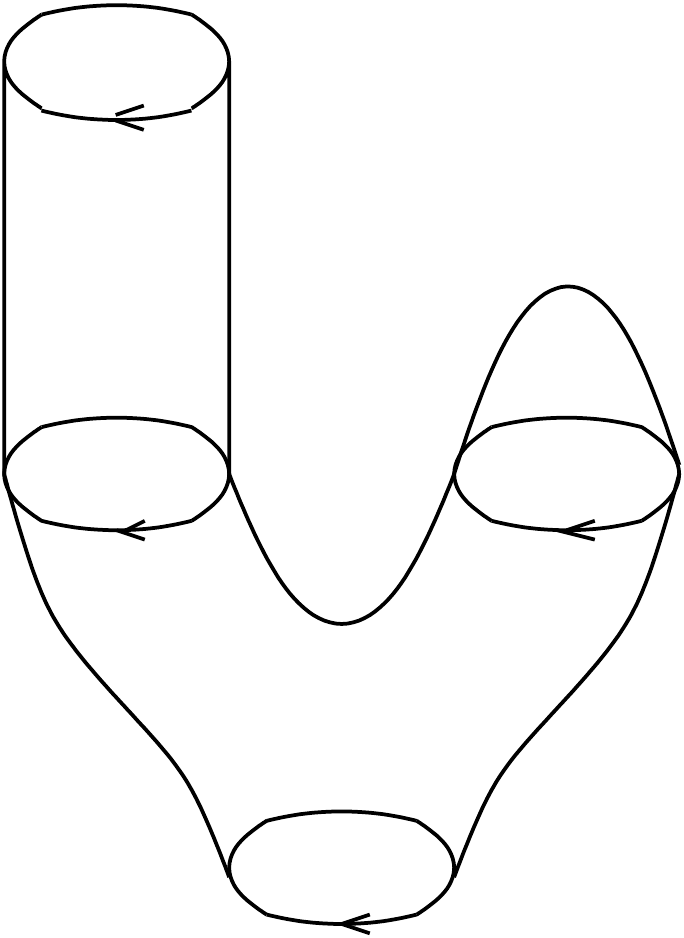}}\label{eq:circle_frob1}
\end{equation}
\begin{equation}
\raisebox{-13pt}{\includegraphics[height=0.5in]{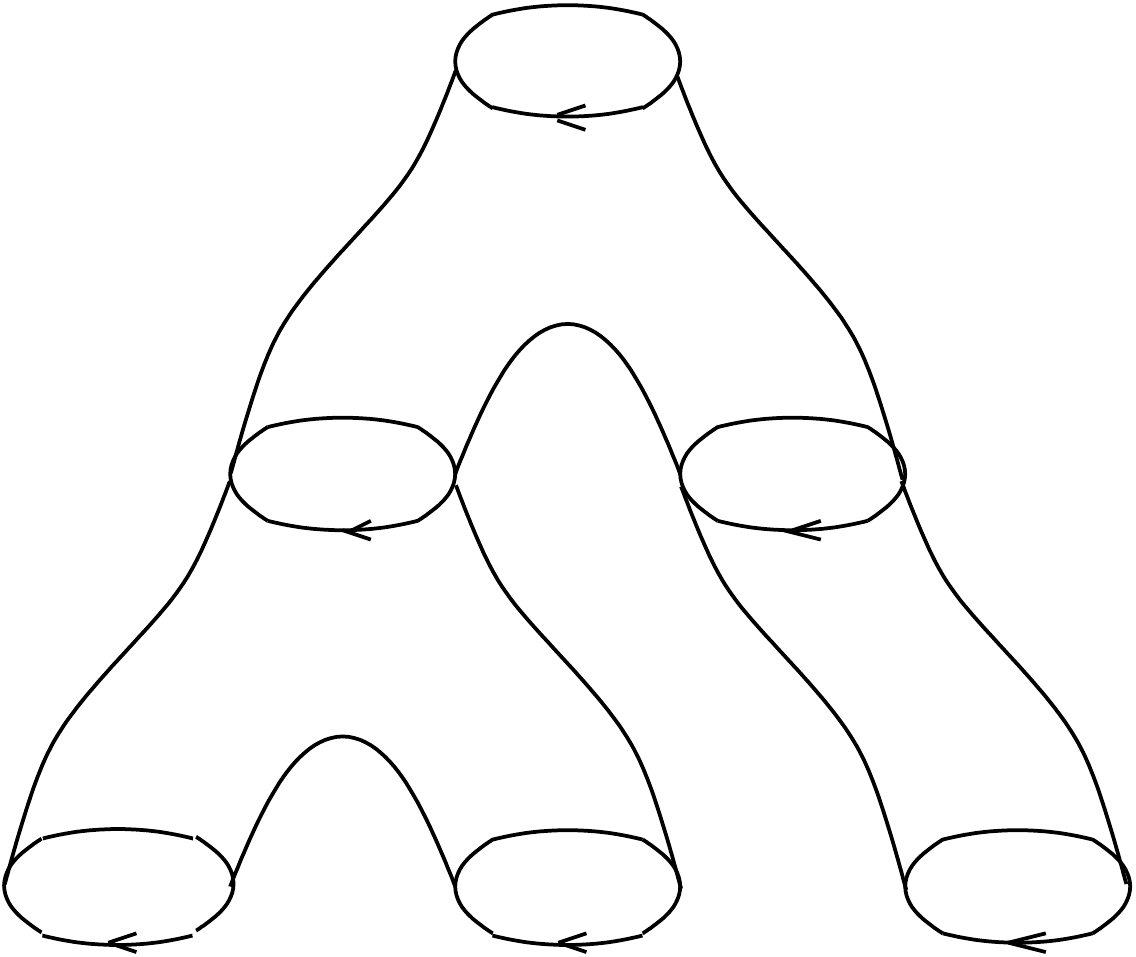}} \quad \cong \quad \raisebox{-13pt}{\includegraphics[height=0.5in]{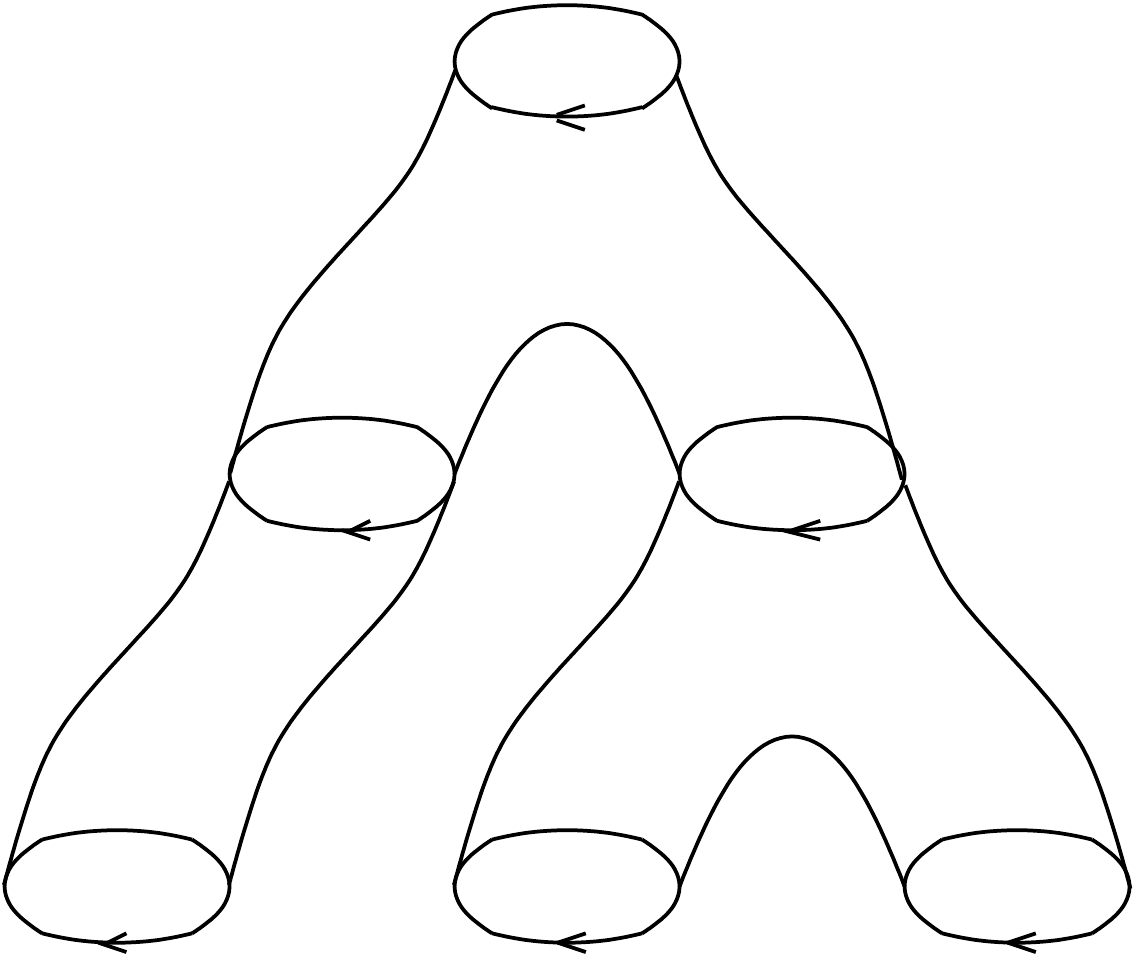}}  \hspace{1cm} \raisebox{-13pt}{\includegraphics[height=0.5in]{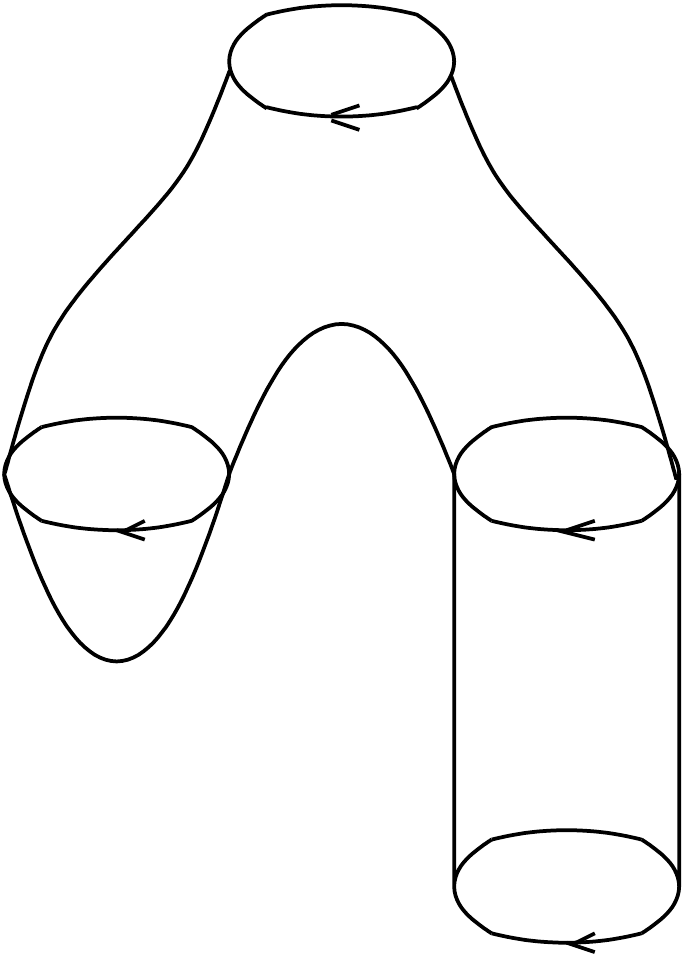}}\quad \cong \quad  \raisebox{-13pt}{\includegraphics[height=0.5in]{circle_frob4.pdf}}\quad \cong \quad \raisebox{-13pt}{\includegraphics[height=0.5in]{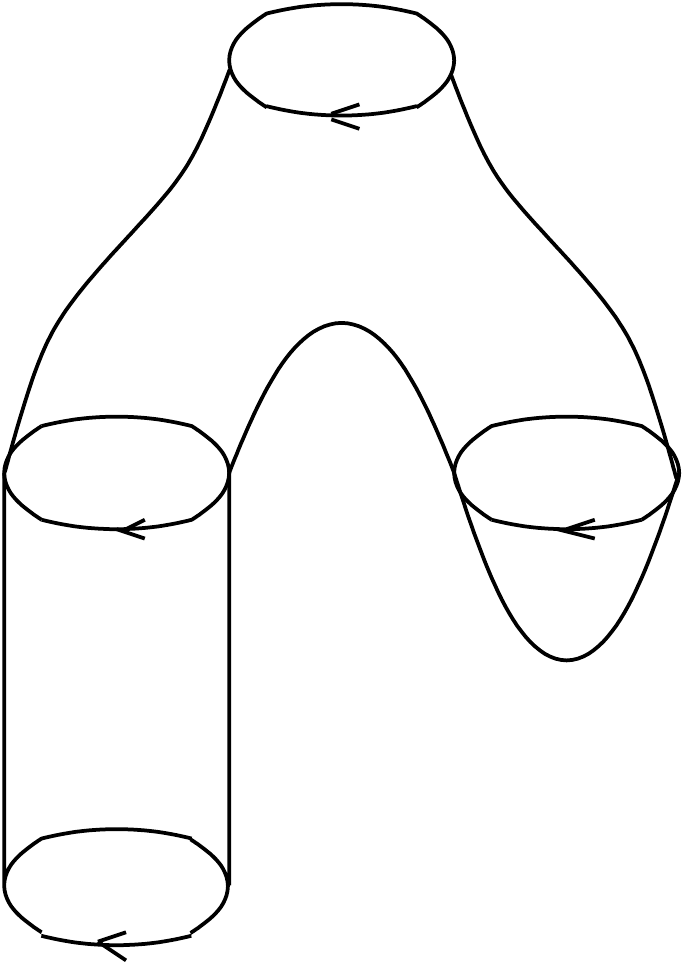}}\label{eq:circle_frob2}
\end{equation}
\begin{equation}
\raisebox{-13pt}{\includegraphics[height=0.5in]{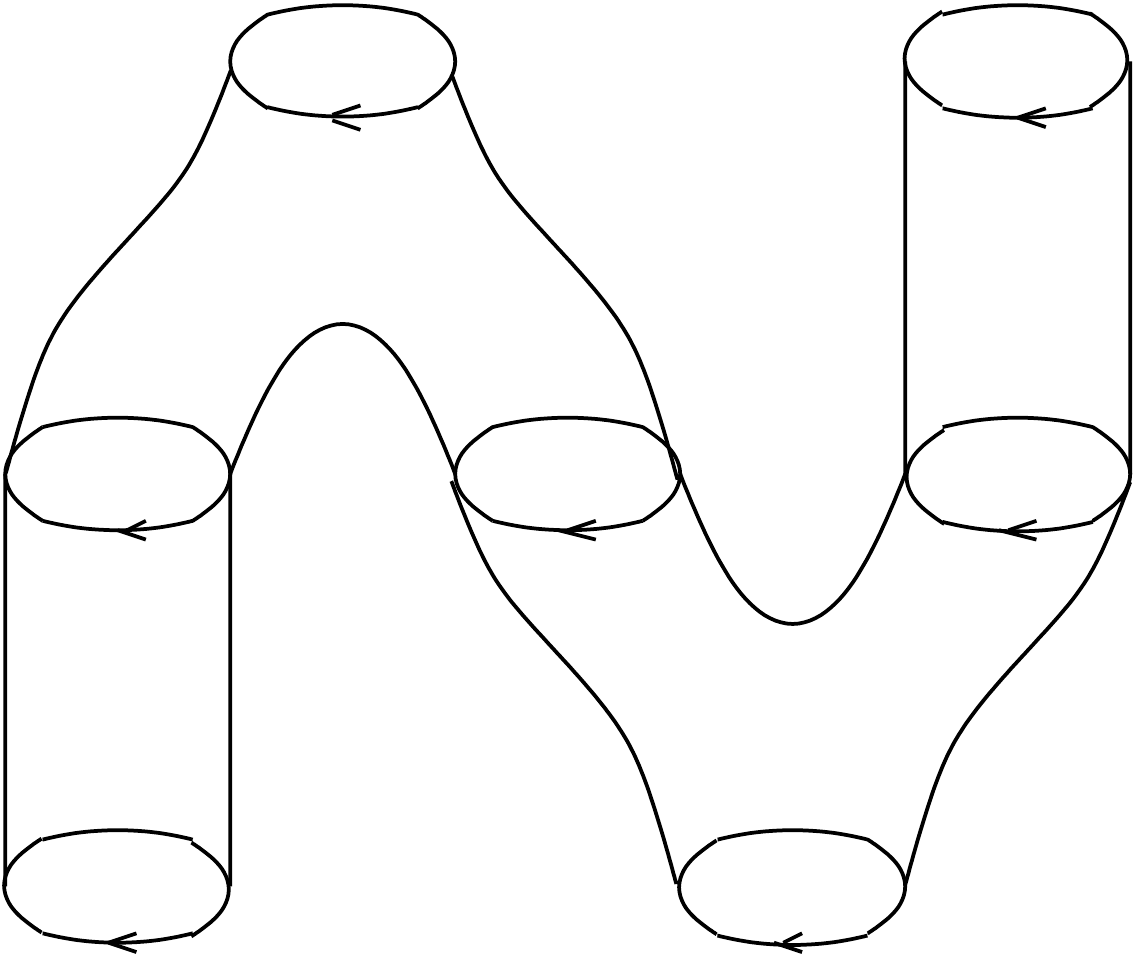}} \quad \cong \quad \raisebox{-13pt}{\includegraphics[height=0.5in]{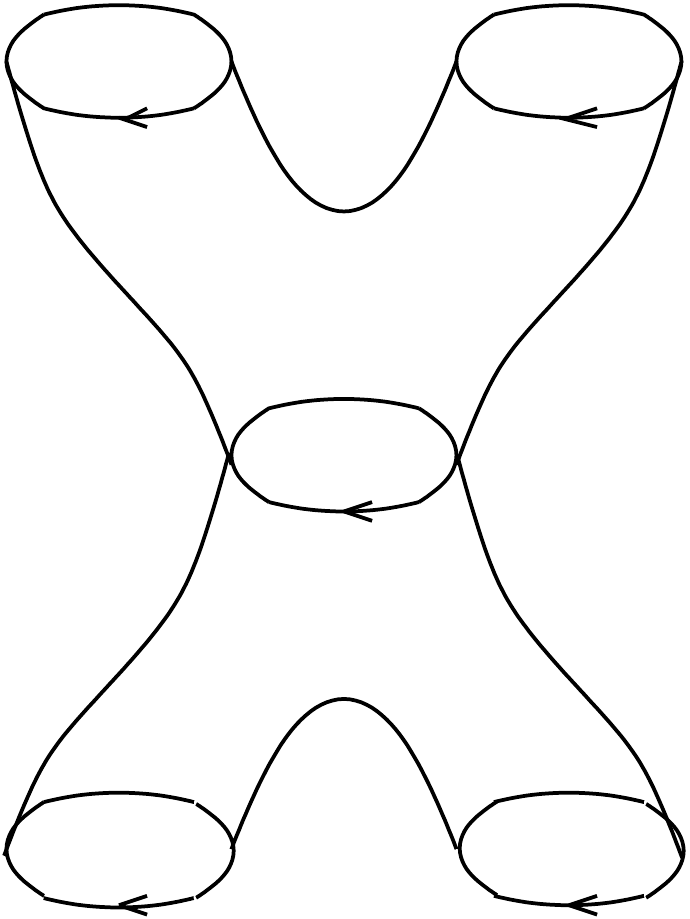}} \quad \cong \quad \raisebox{-13pt}{\includegraphics[height=0.5in]{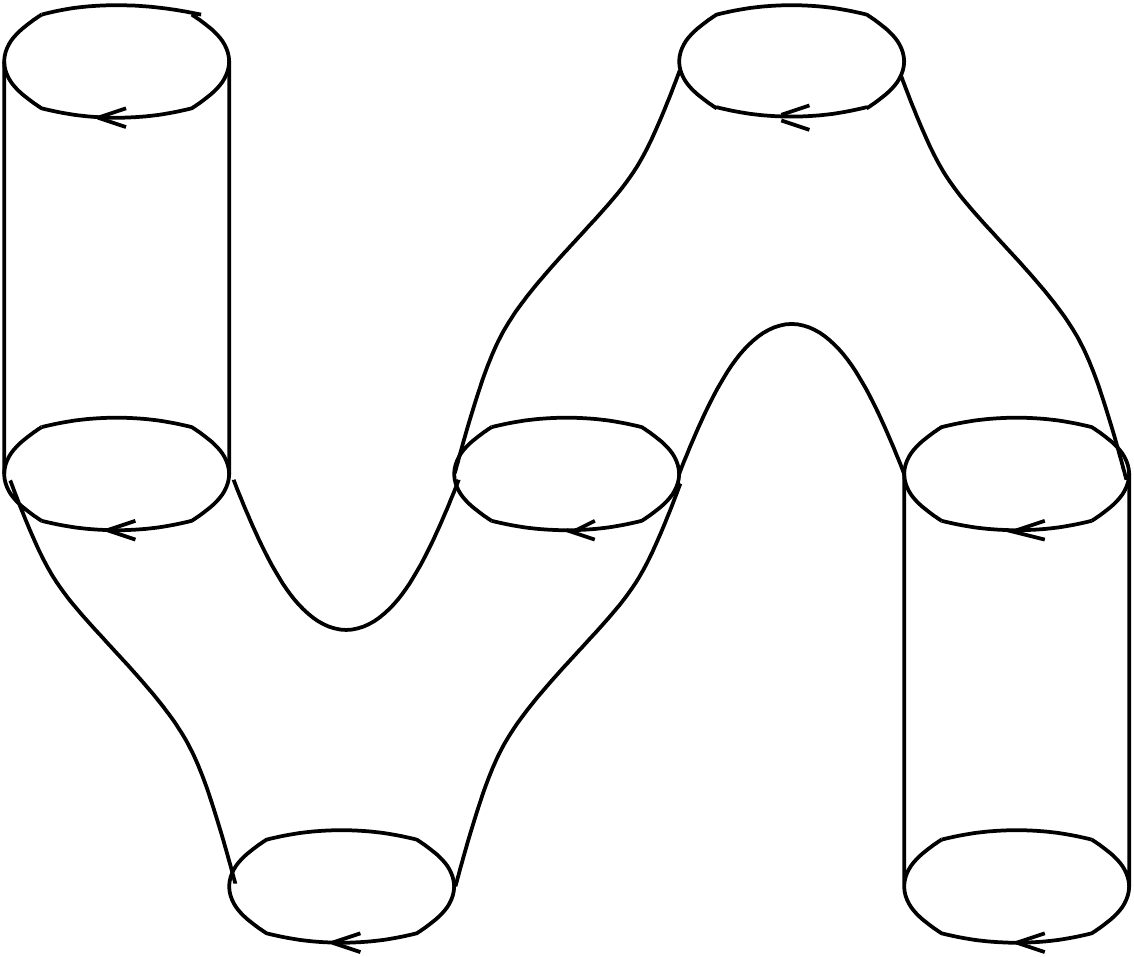}}\label{eq:circle_frob3}
\end{equation}
\begin{equation}
\raisebox{-13pt}{\includegraphics[height=0.5in]{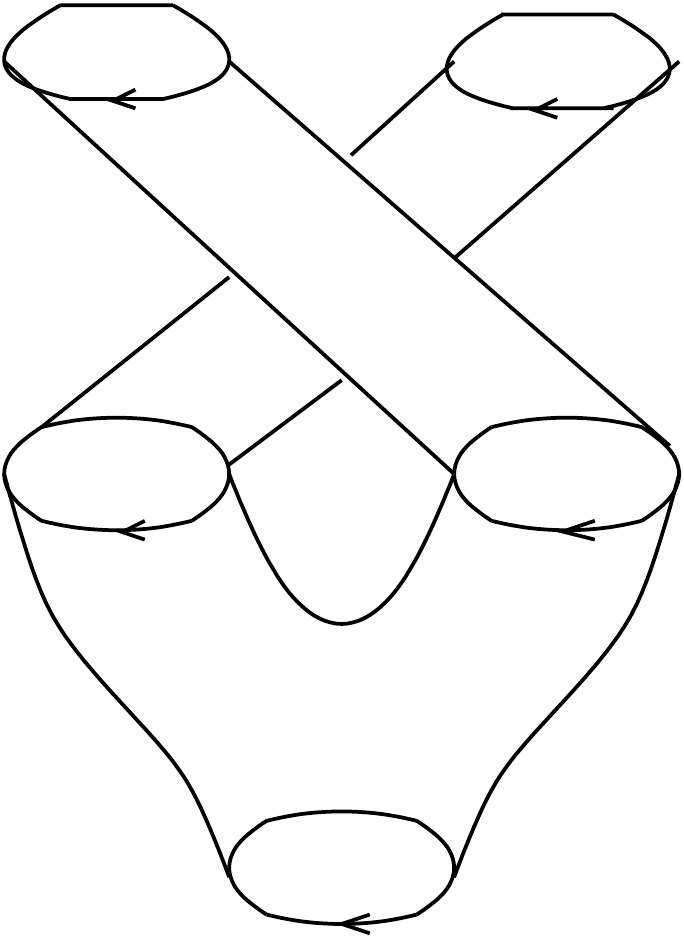}} \quad \cong \quad \raisebox{-13pt}{\includegraphics[height=0.5in]{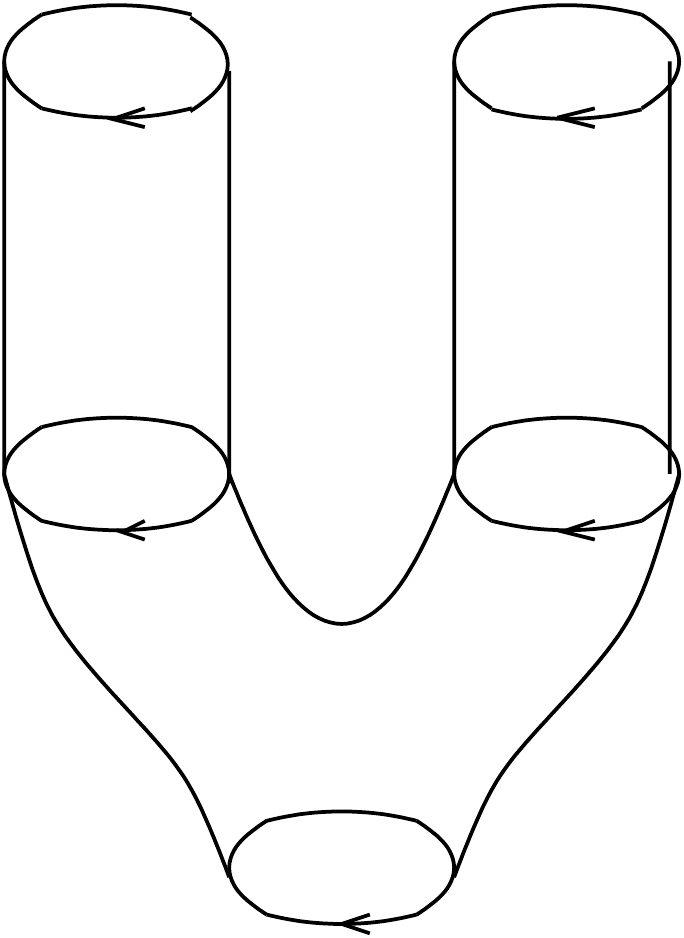}}\label{eq:circle_frob4}
\end{equation}

\item The object $\textbf{n} = (1)$ forms a symmetric Frobenius algebra object.

\begin{equation}
\raisebox{-13pt}{\includegraphics[height=0.5in]{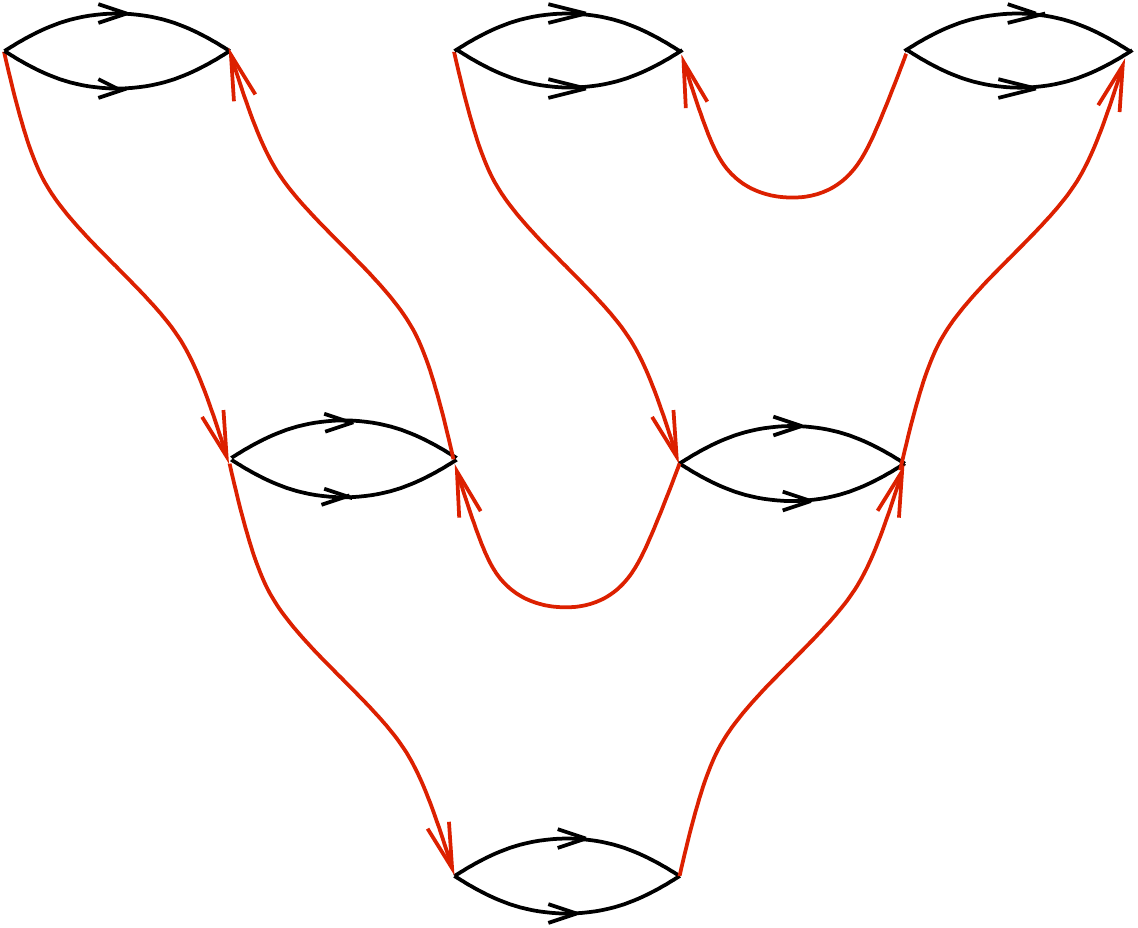}}\quad \cong \quad  \raisebox{-13pt}{\includegraphics[height=0.5in]{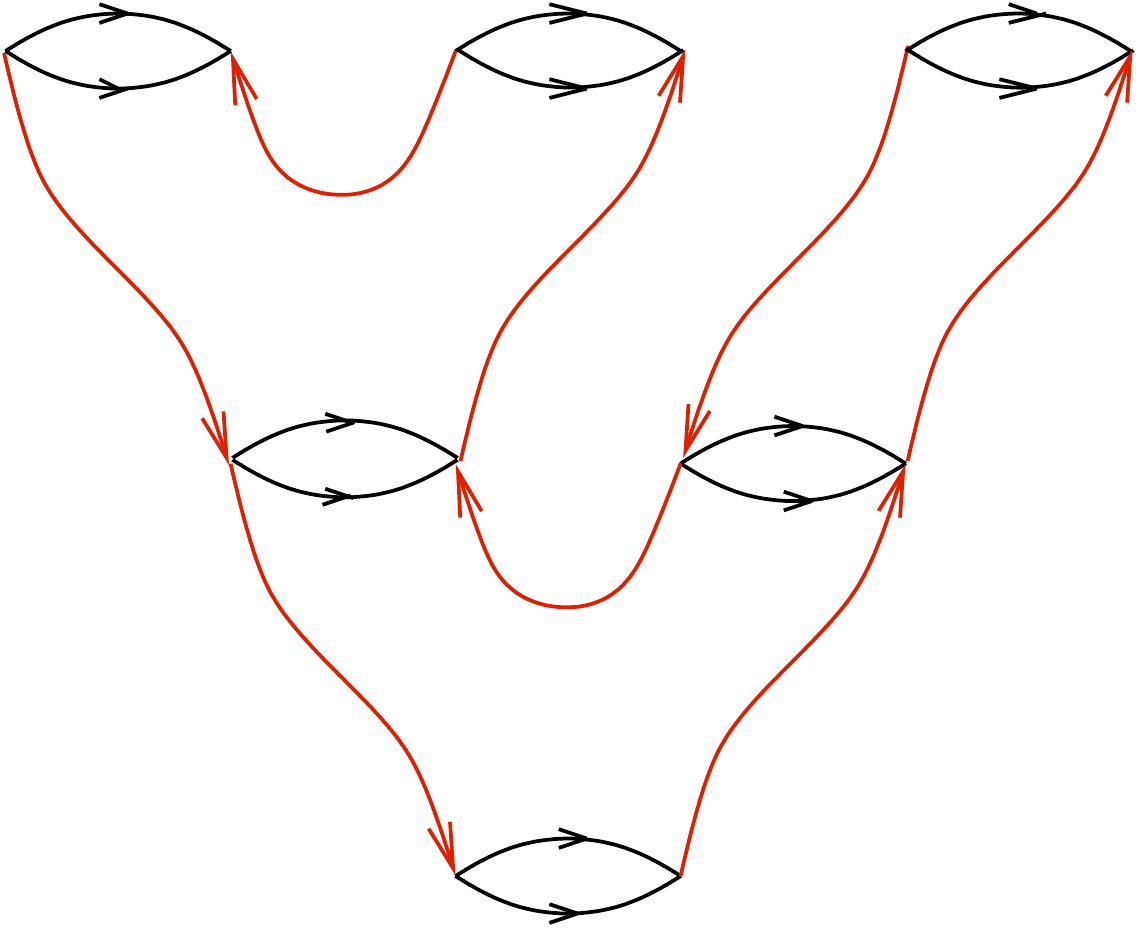}} \hspace{1cm} \raisebox{-13pt}{\includegraphics[height=0.5in]{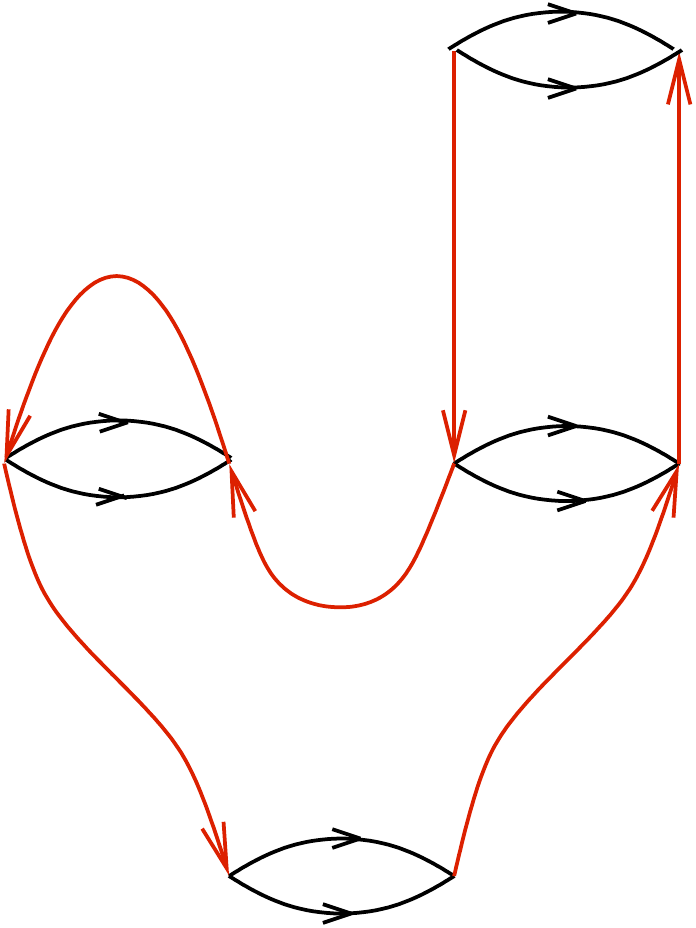}}\quad \cong \quad  \raisebox{-13pt}{\includegraphics[height=0.5in]{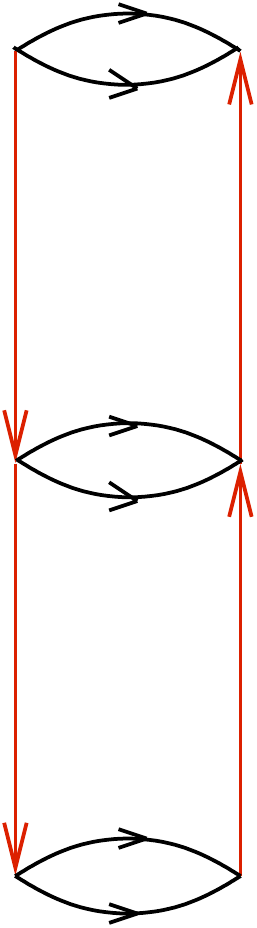}}\quad \cong \quad \raisebox{-13pt}{\includegraphics[height=0.5in]{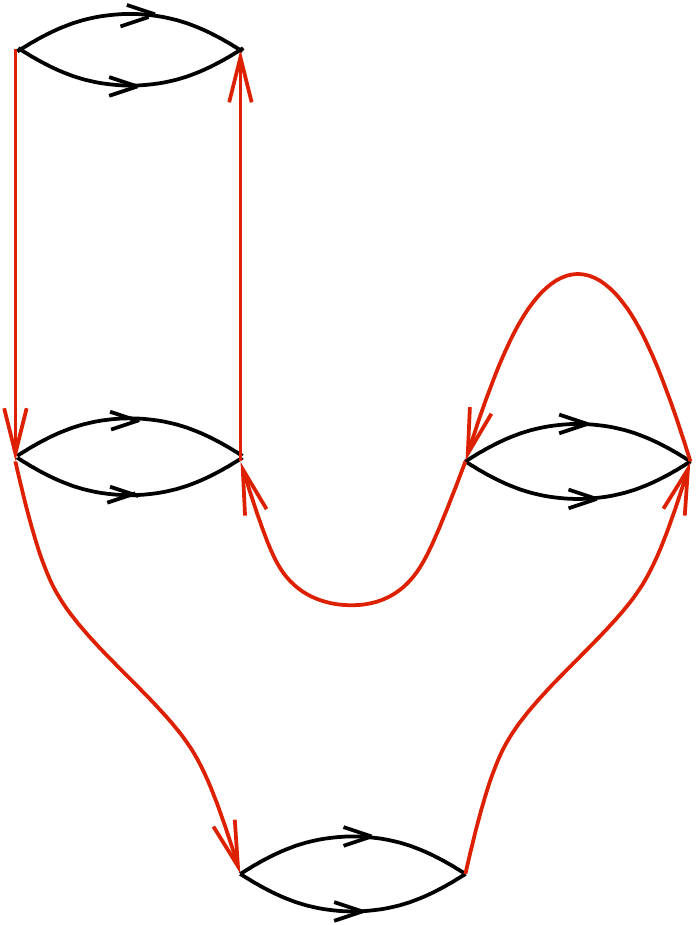}} \label{eq:web_frob1}
\end{equation}
\begin{equation}
\raisebox{-13pt}{\includegraphics[height=0.5in]{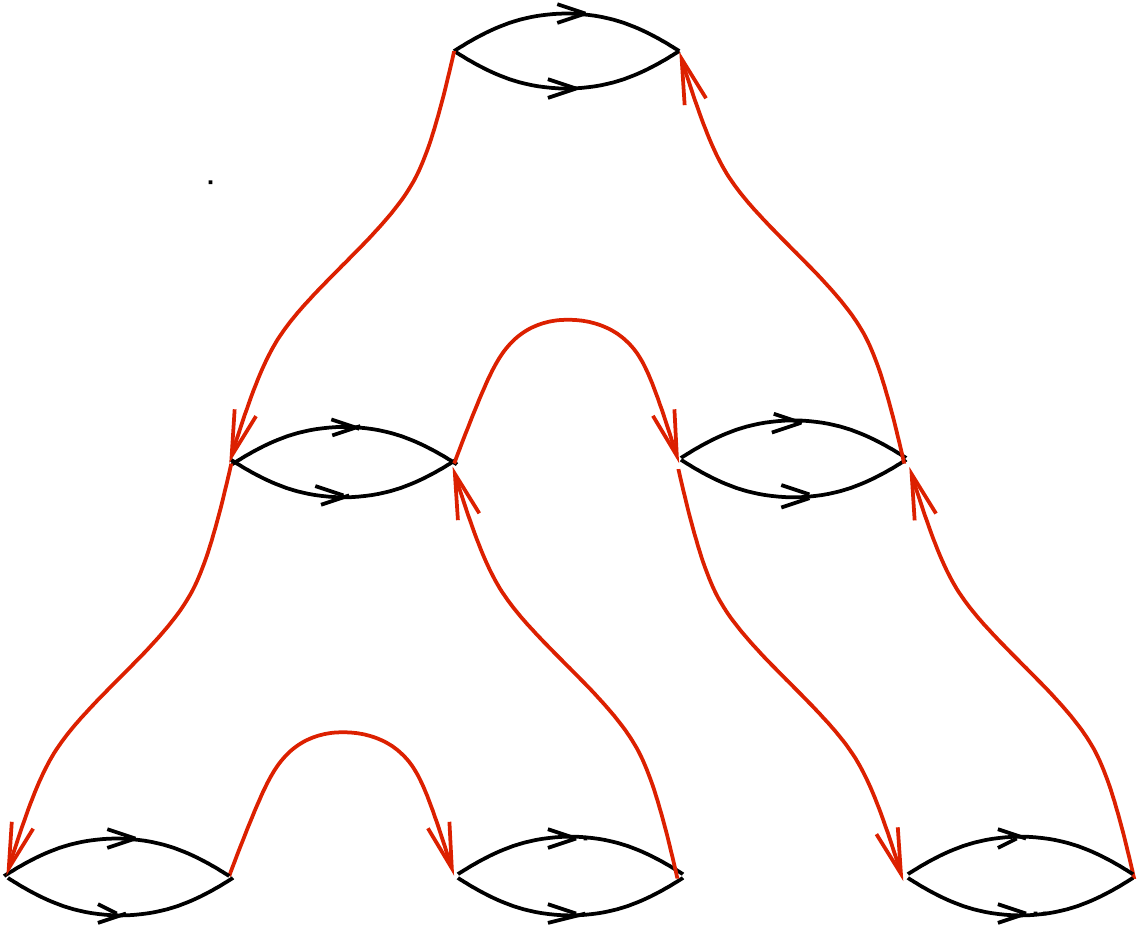}} \quad \cong \quad \raisebox{-13pt}{\includegraphics[height=0.5in]{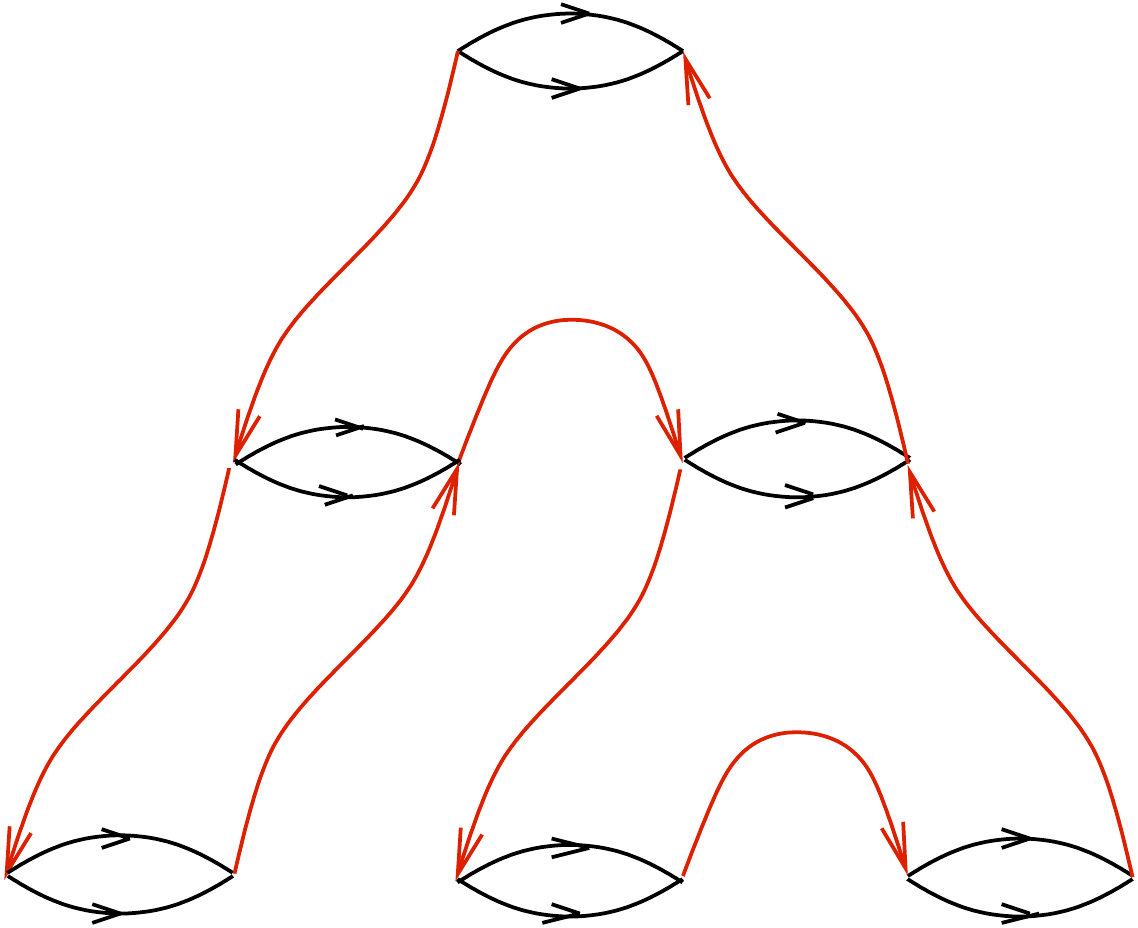}}  \hspace{1cm} \raisebox{-13pt}{\includegraphics[height=0.5in]{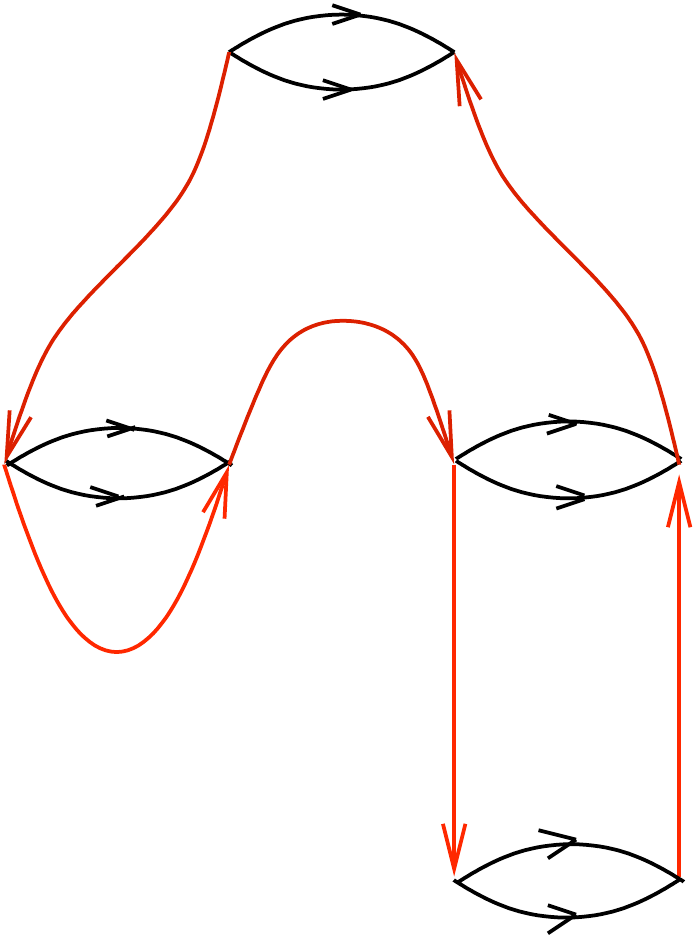}}\quad \cong \quad  \raisebox{-13pt}{\includegraphics[height=0.5in]{web_frob4.pdf}}\quad \cong \quad \raisebox{-13pt}{\includegraphics[height=0.5in]{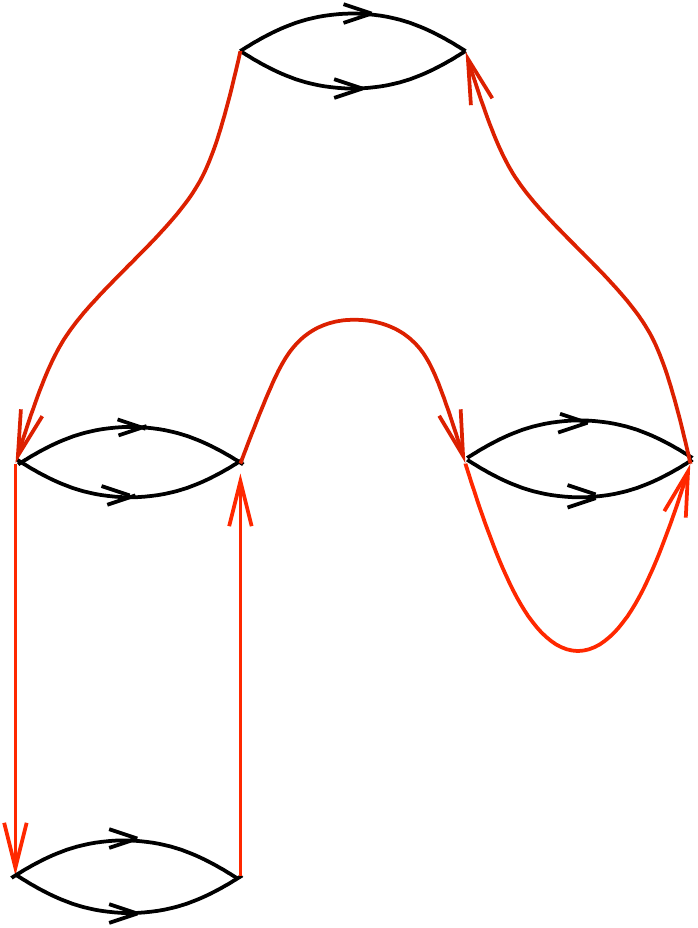}}
\label{eq:web_frob2}
\end{equation}
\begin{equation}
\raisebox{-13pt}{\includegraphics[height=0.5in]{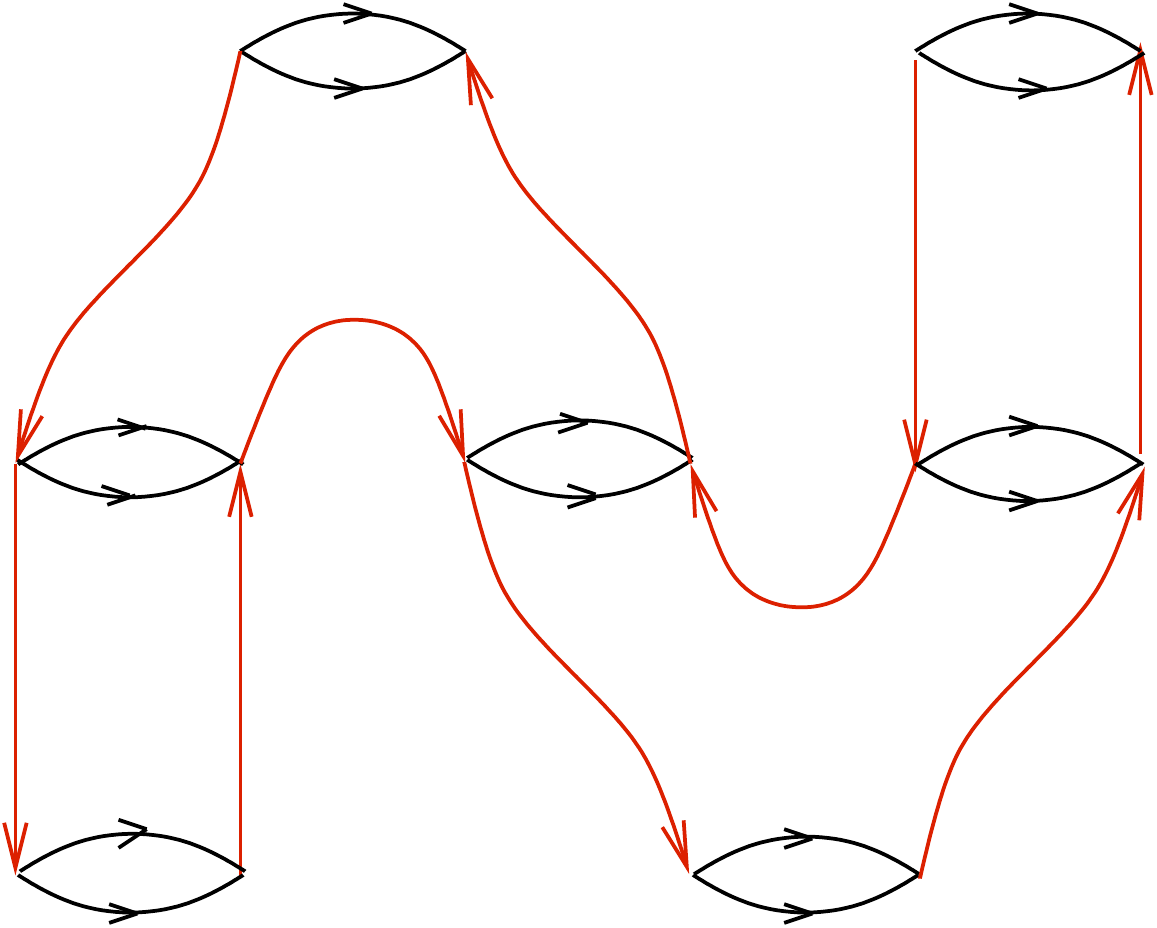}} \quad \cong \quad \raisebox{-13pt}{\includegraphics[height=0.5in]{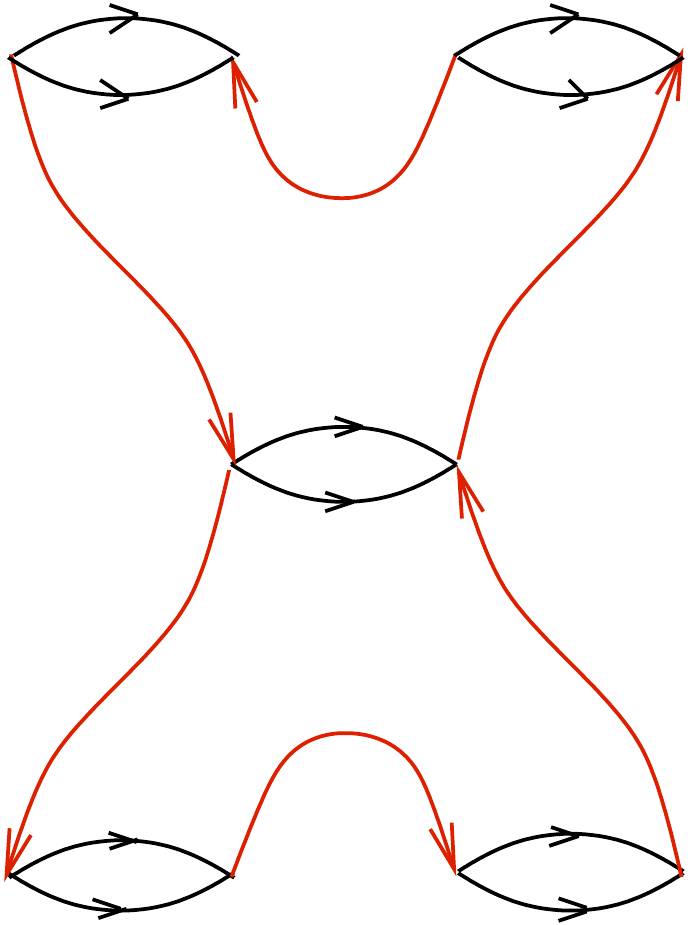}} \quad \cong \quad \raisebox{-13pt}{\includegraphics[height=0.5in]{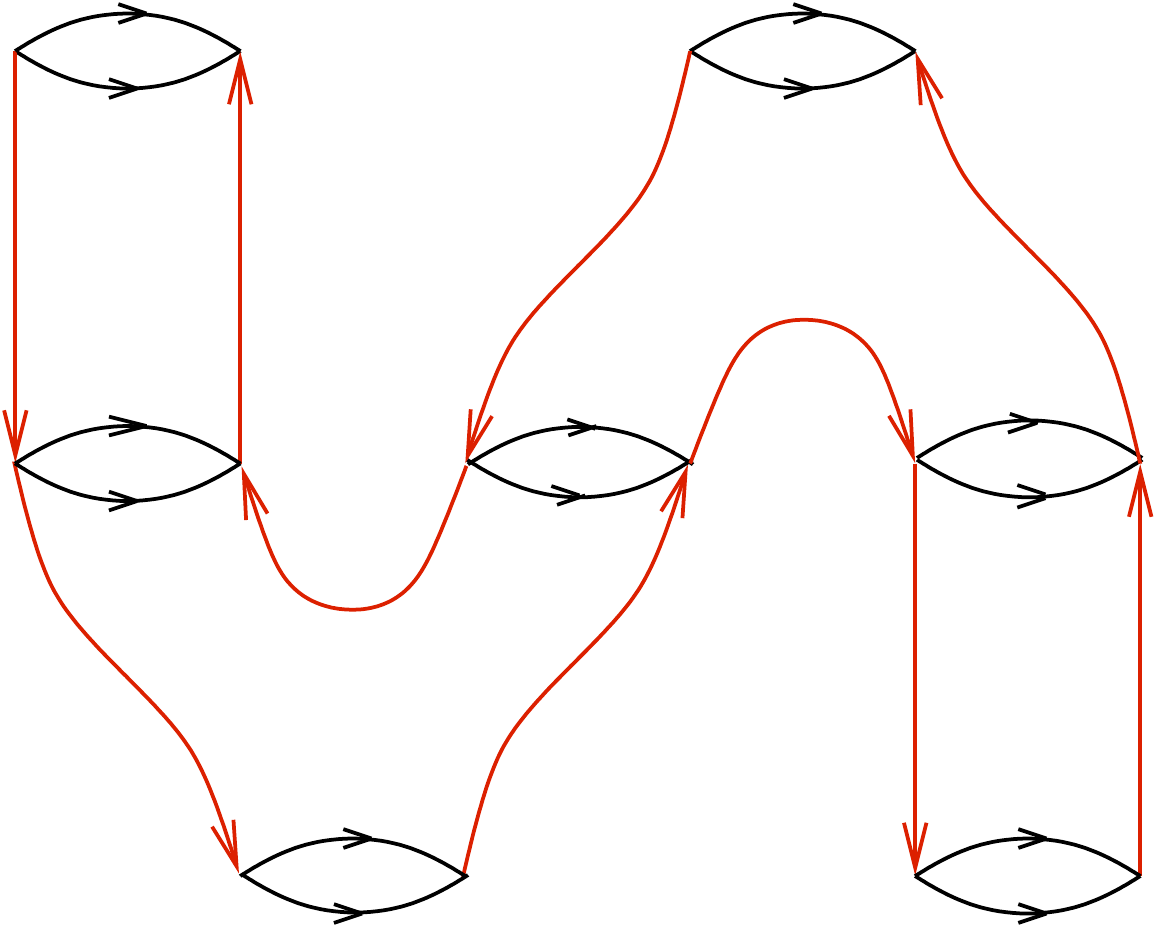}}\label{eq:web_frob3}
\end{equation}
\begin{equation}
\raisebox{-13pt}{\includegraphics[height=0.5in]{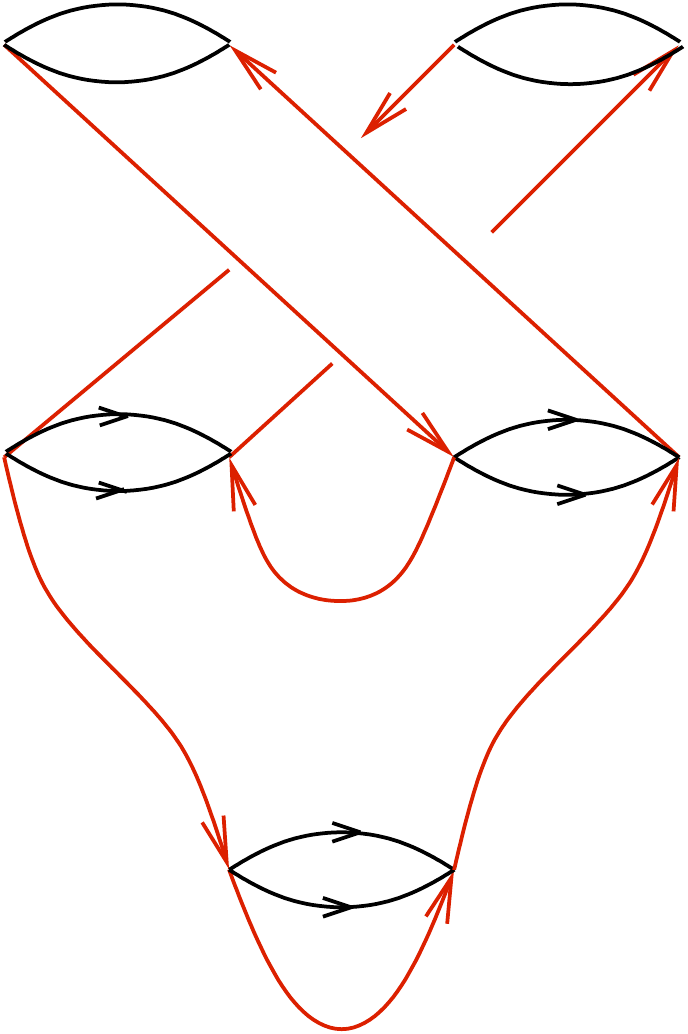}} \quad \cong \quad \raisebox{-13pt}{\includegraphics[height=0.5in]{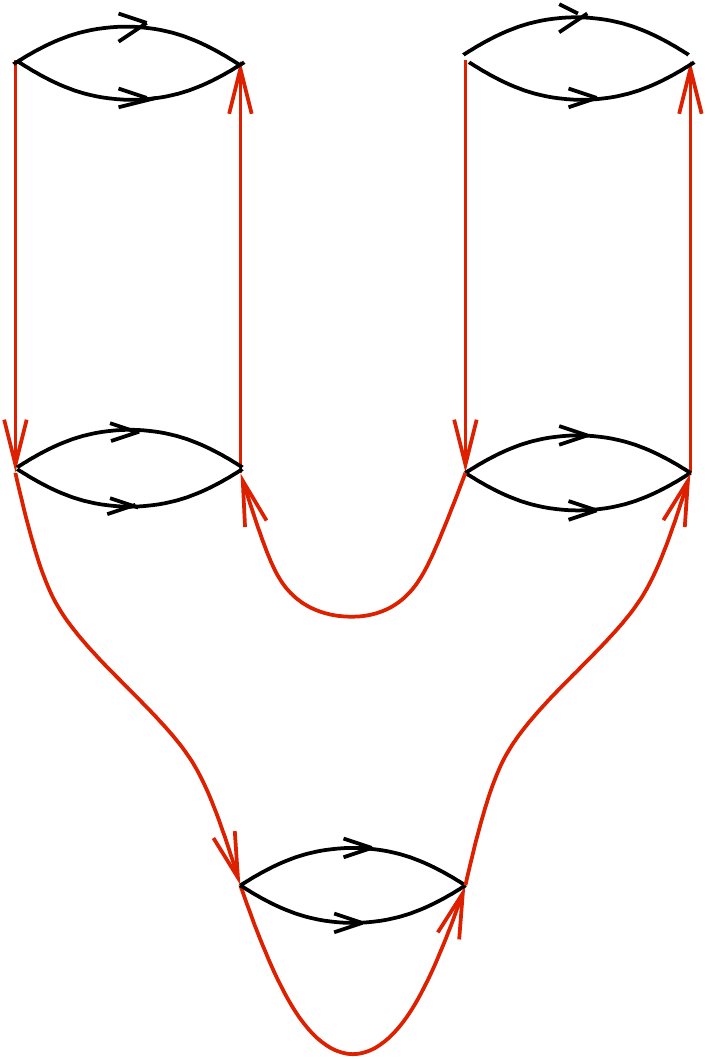}}\label{eq:web_frob4}
\end{equation}

\item The \textit{zipper} \, \raisebox{-5pt}{\includegraphics[height=0.2in]{zipper.pdf}} \, forms an algebra homomorphism. 

\begin{equation} \raisebox{-13pt}{\includegraphics[height=0.5in]{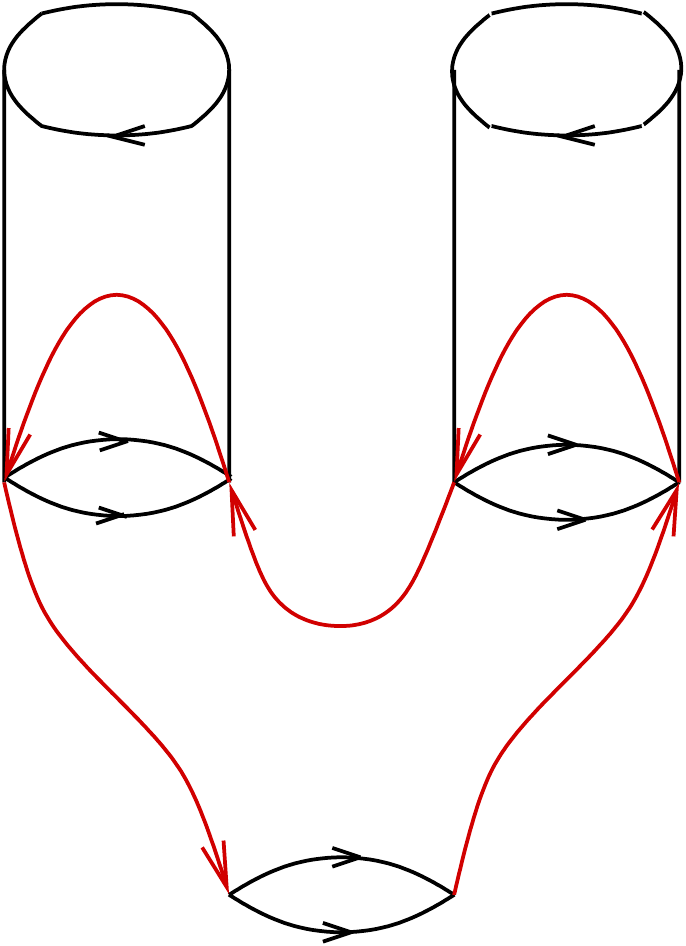}}\quad \cong \quad  \raisebox{-13pt}{\includegraphics[height=0.5in]{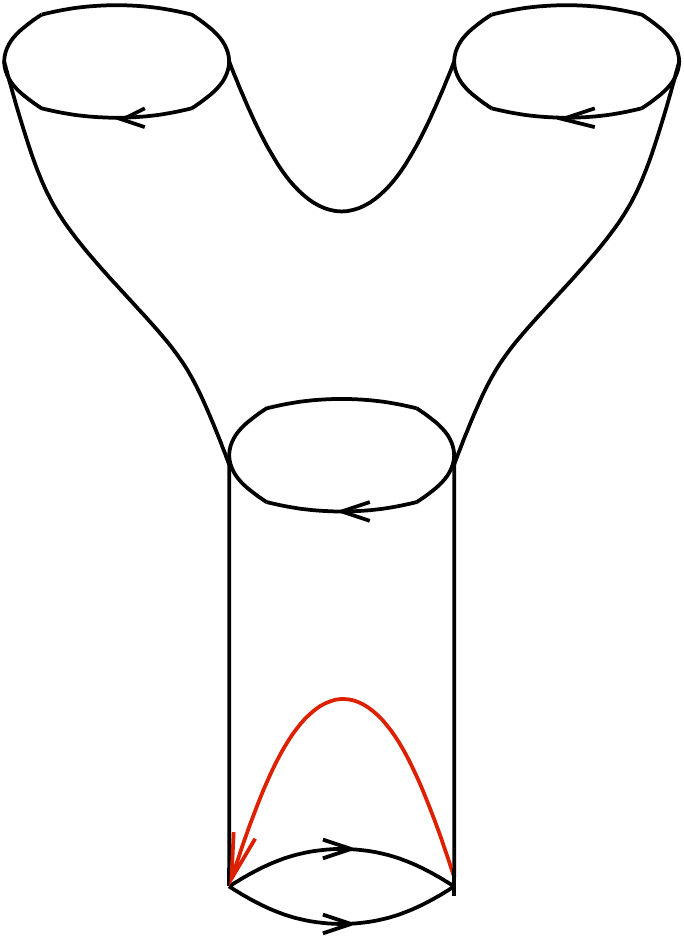}} \hspace{2cm}  \raisebox{-10pt}{\includegraphics[height=0.35in]{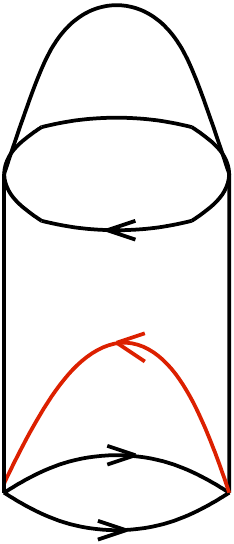}} \quad \cong \quad  \raisebox{-8pt}{\includegraphics[height=0.2in]{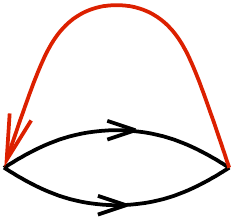}}  \label{eq:zipper_alghom}
\end{equation}

\item The \textit{cozipper} \, \raisebox{-5pt}{\includegraphics[height=0.2in]{cozipper.pdf}} \, is dual to the zipper.

\begin{equation}
\raisebox{-13pt}{\includegraphics[height=0.5in]{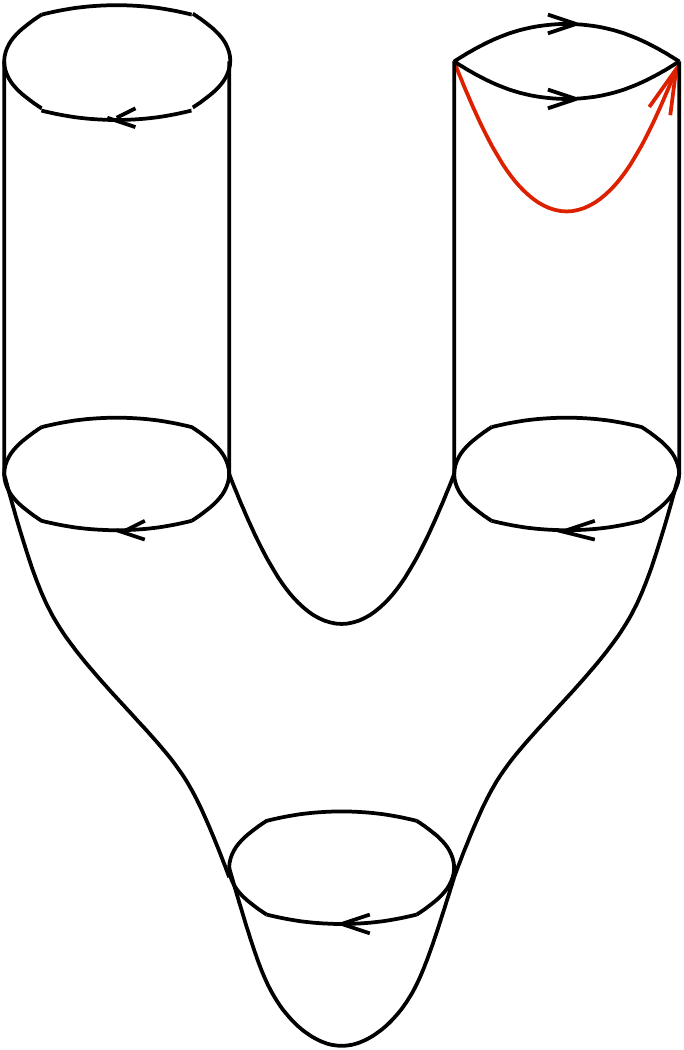}}\quad \cong \quad  \raisebox{-13pt}{\includegraphics[height=0.5in]{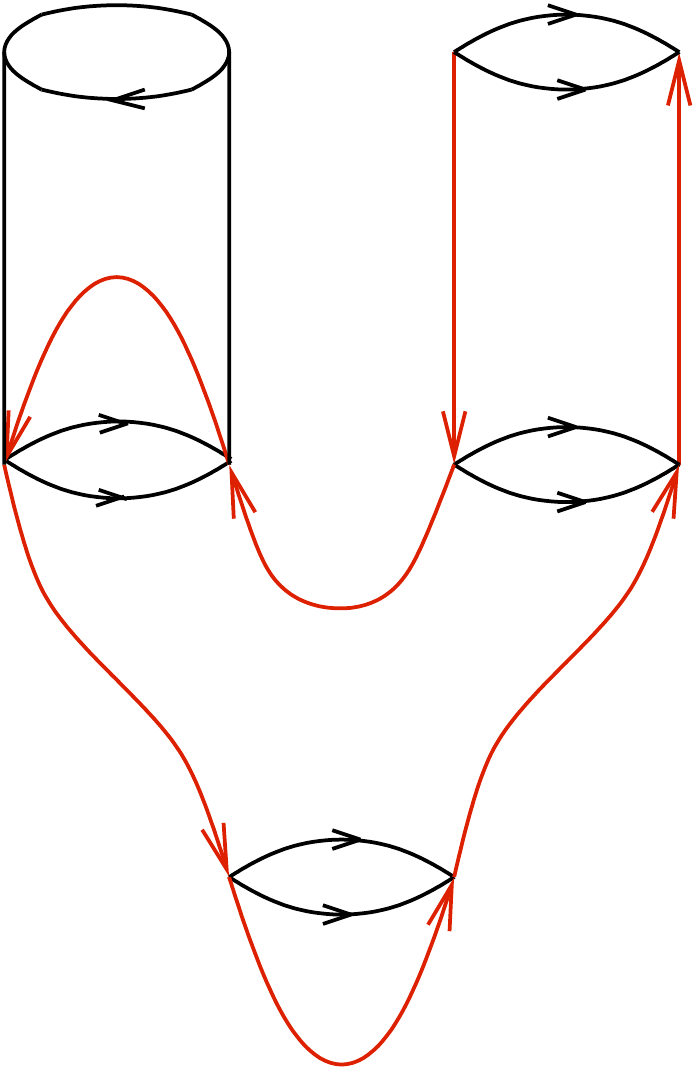}} \label{eq:cozipper_dual}
\end{equation}

\item \textit{Centrality relation}.

\begin{equation}
\raisebox{-13pt}{\includegraphics[height=0.5in]{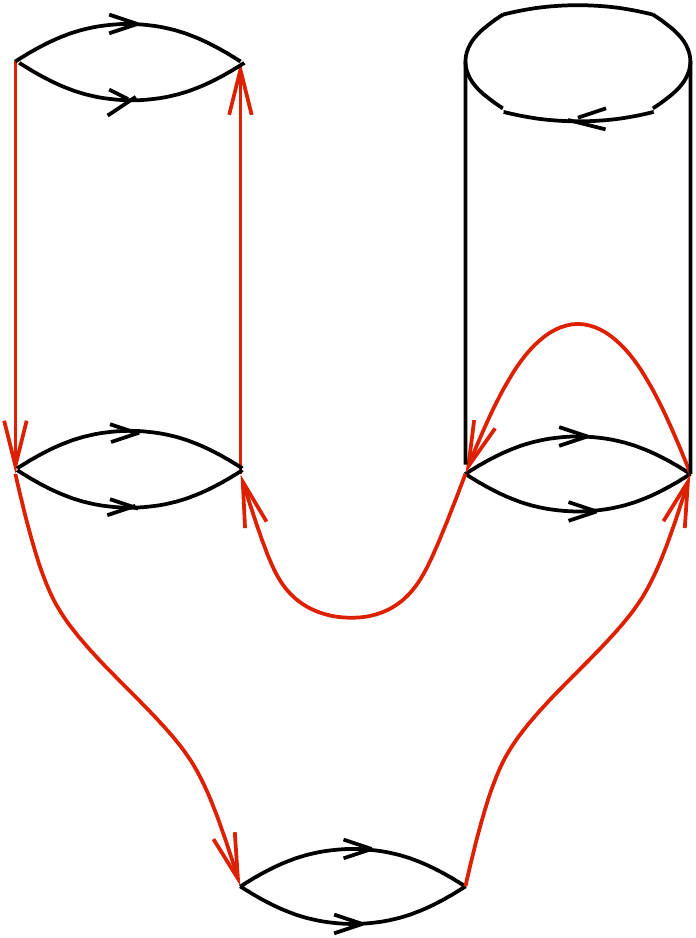}}\quad \cong \quad  \raisebox{-13pt}{\includegraphics[height=0.65in]{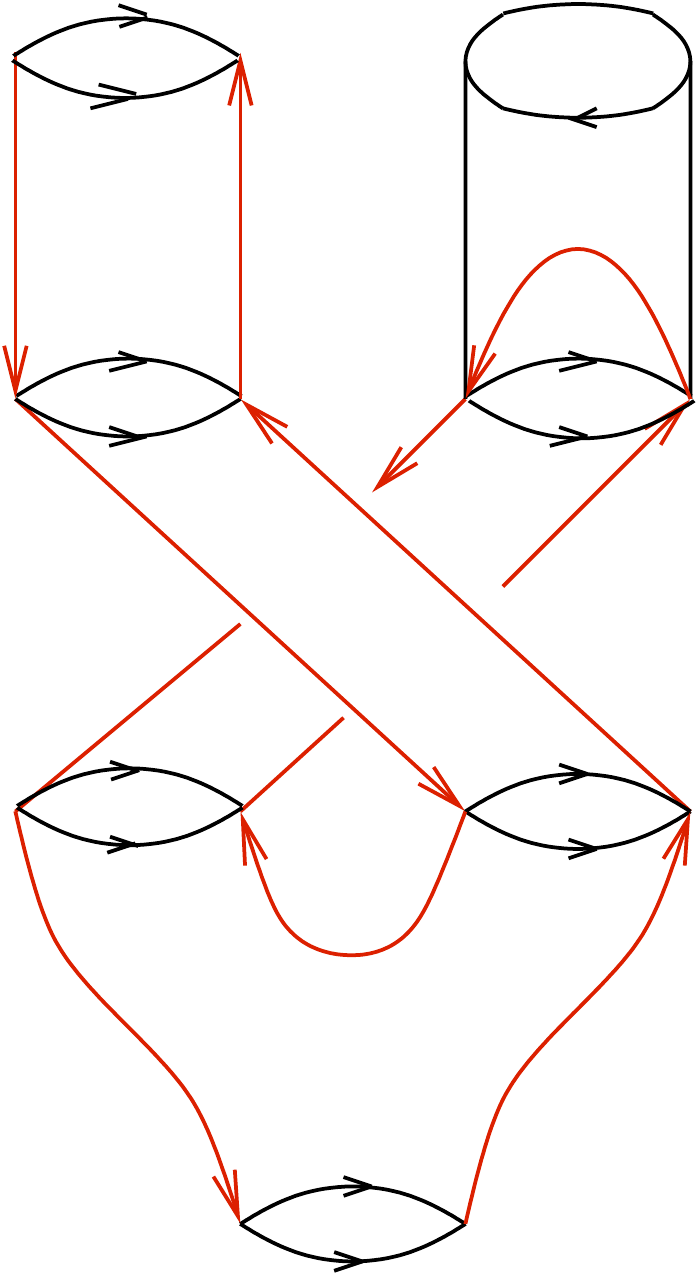}} \label{eq:center}
\end{equation}

\item The \textit{genus-one relation}. 
\begin{equation}
\raisebox{-18pt}{\includegraphics[width = 0.35in, height=0.7in]{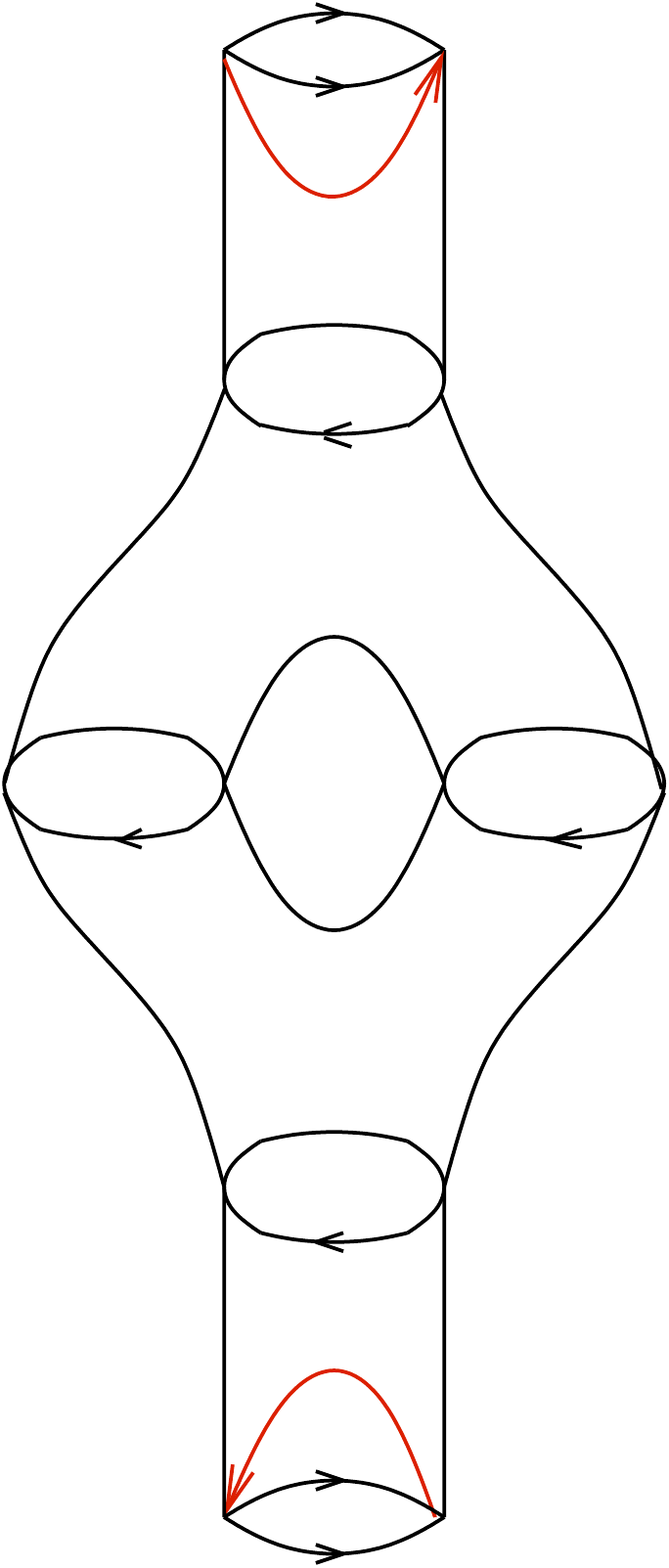}}\quad \cong \quad  \raisebox{-18pt}{\includegraphics[height=0.7in]{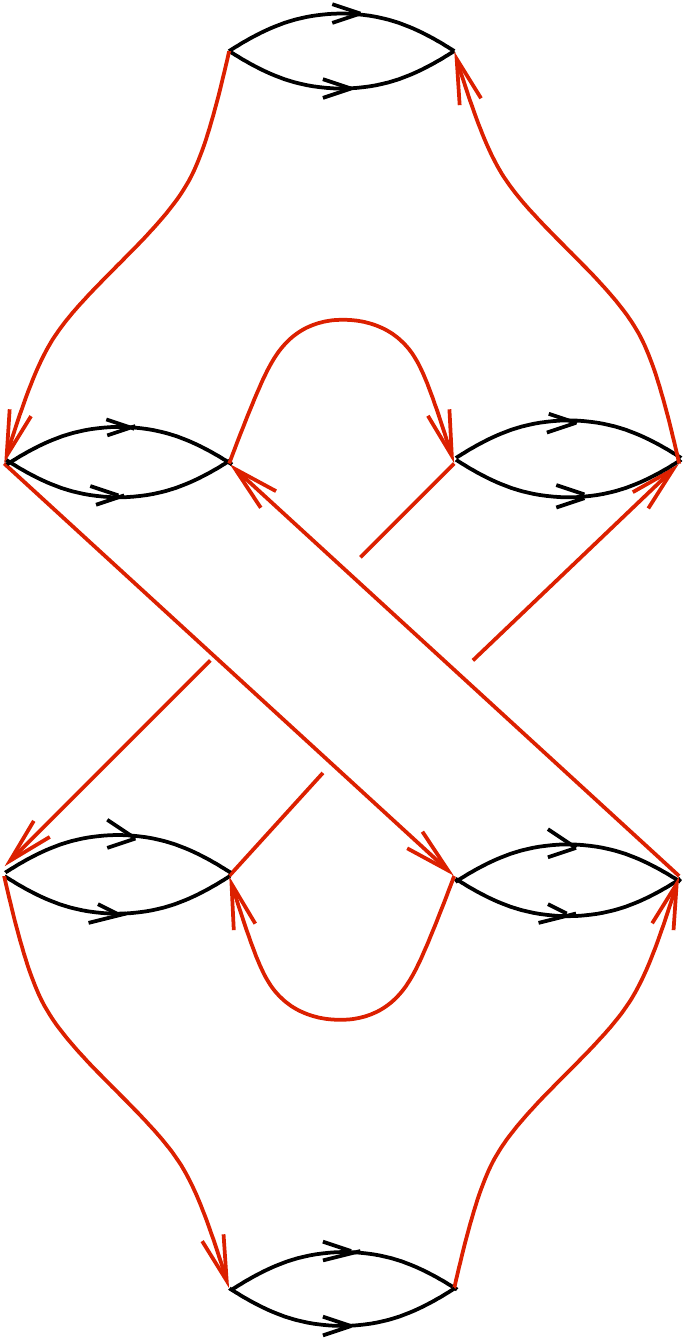}} \label{eq:genus_one}
\end{equation}

\end{enumerate}
\end{proposition}

\begin{proof} It is well-known that the first set of equivalences of cobordisms depicted in Equations~\eqref{eq:circle_frob1}-\eqref{eq:circle_frob4} hold, and it is not hard to see that also the remaining diffeomorphisms hold. We prove these using nested discs. 

We can use punctured discs to represent singular cobordisms. For example, we have the following graphical representations:
\[ \raisebox{-5pt}{\includegraphics[height=0.2in]{singunit.pdf}} \rightsquigarrow \raisebox{-15pt}{\includegraphics[height=0.5in]{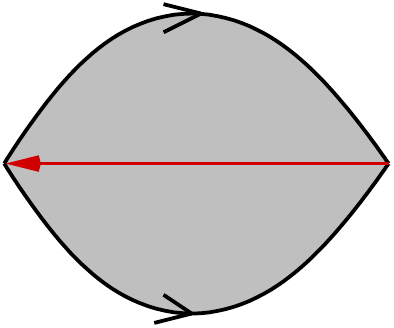}}\,\,, \qquad \raisebox{-5pt}{\includegraphics[height=0.3in]{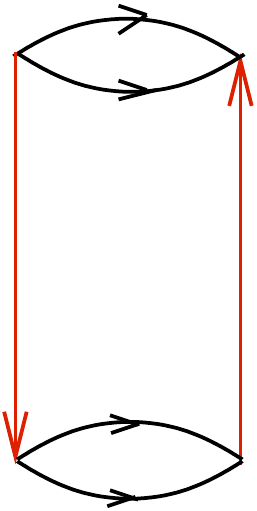}} \rightsquigarrow \raisebox{-15pt}{\includegraphics[height=0.5in]{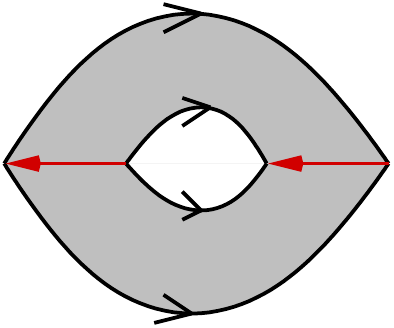}}\,\,, \qquad \raisebox{-5pt}{\includegraphics[height=0.3in]{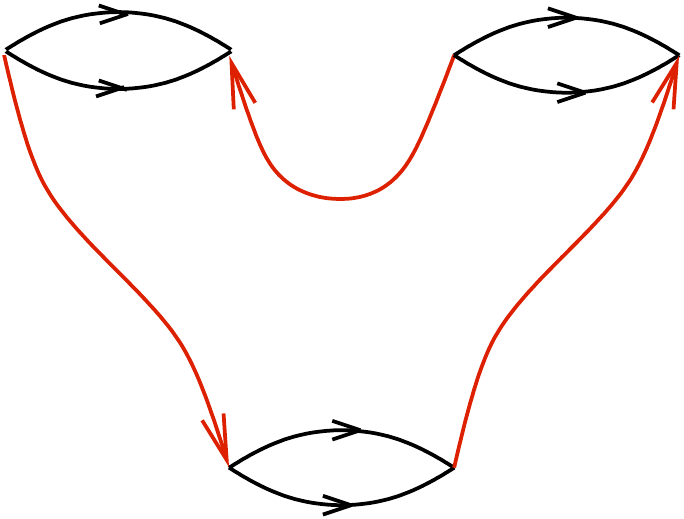}} \rightsquigarrow \raisebox{-15pt}{\includegraphics[height=0.5in]{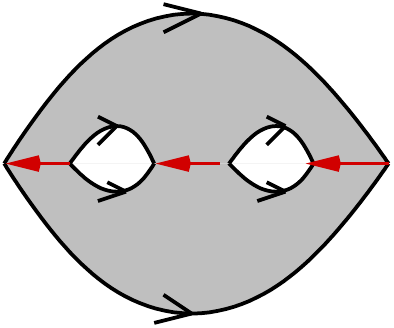}}\,\,, \qquad \raisebox{-5pt}{\includegraphics[height=0.3in]{zipper.pdf}} \rightsquigarrow \raisebox{-15pt}{\includegraphics[height=0.5in]{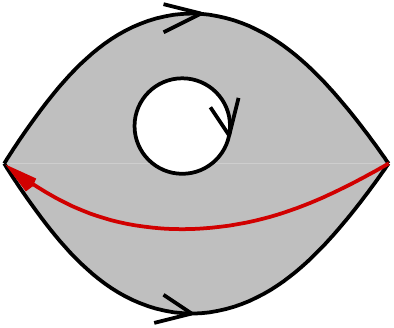}}
 \]
Then the singular cobordism \raisebox{-13pt}{\includegraphics[height=0.5in]{web_frob5.pdf}} is interpreted as  sewing in an annulus and a disc \raisebox{-6pt}{\includegraphics[height=0.25in]{identity_web_flat.pdf}} \, \raisebox{-2pt}{\includegraphics[height=0.15in]{singunit_flat.pdf}}\, in a `pair-of-pants'  \raisebox{-15pt}{\includegraphics[height=0.5in]{singmult_flat.pdf}}, which produces $\raisebox{-15pt}{\includegraphics[height=0.5in]{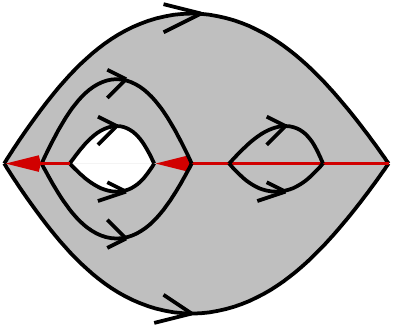}} \, \cong \,\raisebox{-15pt}{\includegraphics[height=0.5in]{identity_web_flat.pdf}}$. Thus we have $\raisebox{-13pt}{\includegraphics[height=0.5in]{web_frob5.pdf}} \,\, \cong \, \, \raisebox{-13pt}{\includegraphics[height=0.5in]{identity_web.pdf}}$. 

The associativity property depicted in Equation~\eqref{eq:web_frob1} is obtained by making the following different decompositions of a disc with three punctures:
\[\raisebox{-15pt}{\includegraphics[height=0.55in]{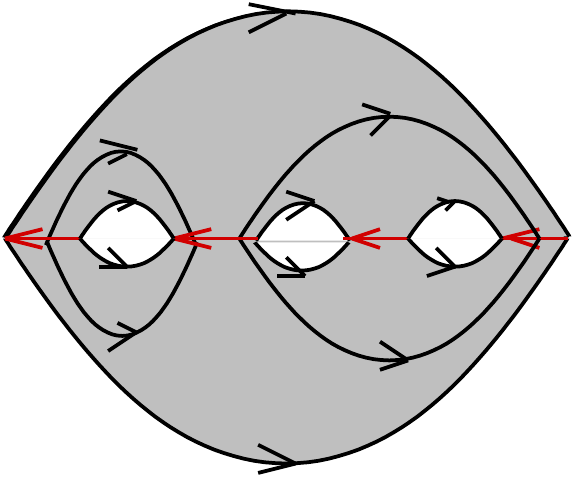}}\,\,\,  \cong \,\, \, \raisebox{-15pt}{\includegraphics[height=0.55in]{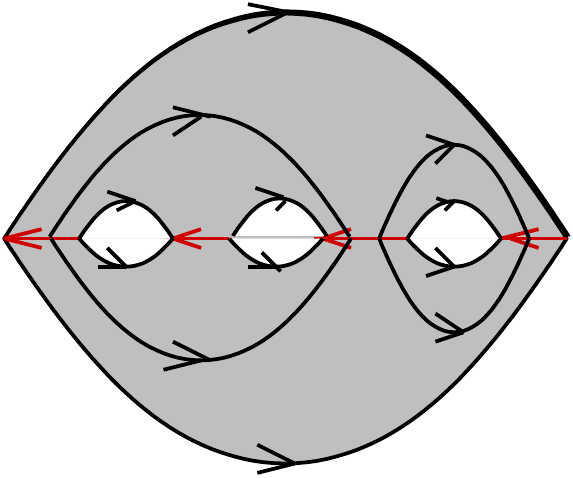}}\quad \Longrightarrow \quad \raisebox{-13pt}{\includegraphics[height=0.5in]{web_frob1.pdf}}\,\, \cong \,\,  \raisebox{-13pt}{\includegraphics[height=0.5in]{web_frob2.pdf}} \]

We explain now why the bi-web $\raisebox{-1pt}{\includegraphics[height=0.1in]{singcircle}} = (1)$ forms a \textit{symmetric} Frobenius algebra object, instead of a commutative Frobenius algebra object. For this, let's consider the following singular cobordisms and their graphical representation: 
 \[ \psset{xunit=.22cm,yunit=.22cm}
  \begin{pspicture}(3,4)
 \rput(1.5, 0){\includegraphics[height=0.23in]{singmult.pdf}}
 \rput(1.5,2.1){\includegraphics[height=0.23in]{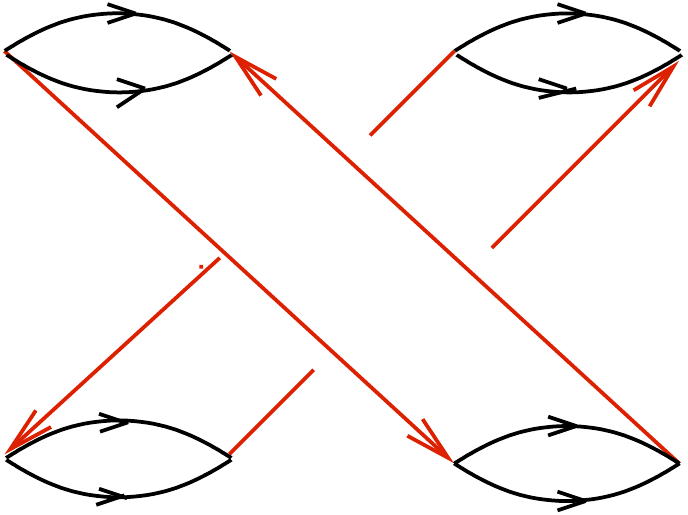}}
 \end{pspicture} 
\,\, \rightsquigarrow\,\, \raisebox{-15pt}{\includegraphics[height=0.5in]{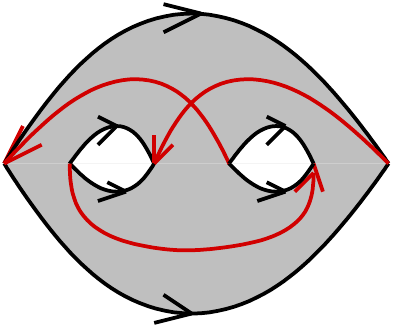}} \qquad \qquad \raisebox{-5pt}{\includegraphics[height=0.3in]{singmult.pdf}} \rightsquigarrow \raisebox{-15pt}{\includegraphics[height=0.5in]{singmult_flat.pdf}}\]
We observe that $\raisebox{-10pt}{\includegraphics[height=0.3in]{web_sym_flat}}\, \ncong \raisebox{-10pt}{\includegraphics[height=0.3in]{singmult_flat.pdf}}$, since the vertices of the bi-webs forming the boundary of the two discs have different connection types. However, connecting the vertices of the outside bi-web, thus gluing in the disc \raisebox{-9pt}{\includegraphics[height=0.3in]{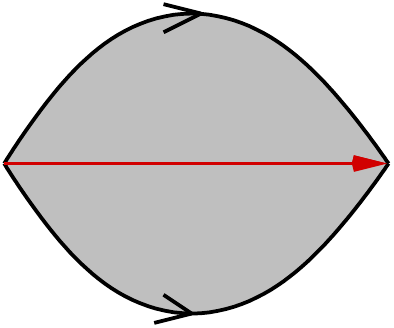}}\, (corresponding to the singular `cup' cobordism \raisebox{-5pt}{\includegraphics[height=0.2in]{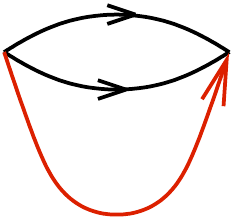}}\,) along the outside bi-web, we obtain the diffeomorphism in Equation~\eqref{eq:web_frob4}.

The first relation in Equation~\eqref{eq:zipper_alghom} is obtained from the following two different decompositions of a disc with two punctures:
\[\raisebox{-15pt}{\includegraphics[height=0.55in]{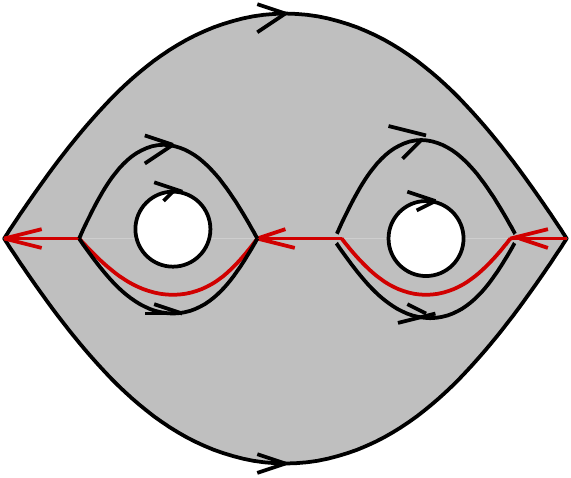}}\,\, \cong \,\,  \raisebox{-15pt}{\includegraphics[height=0.55in]{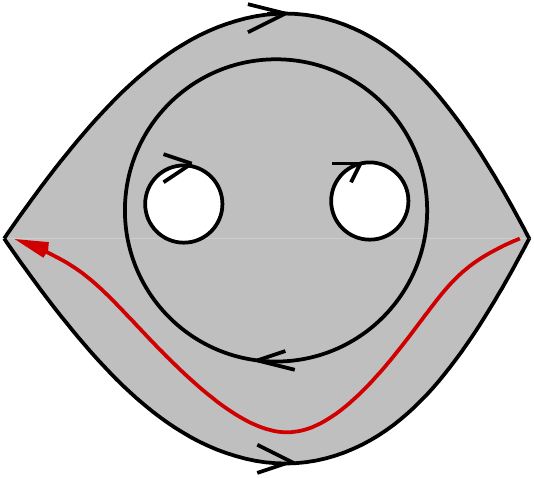}} \quad \Longrightarrow \quad \raisebox{-13pt}{\includegraphics[height=0.5in]{zipper_hom1_left.pdf}}\,\, \cong \,\,  \raisebox{-13pt}{\includegraphics[height=0.5in]{zipper_hom1_right.pdf}} \]

Finally, we verify below the centrality relation given in Equation~\eqref{eq:center}. 
\[ \raisebox{-13pt}{\includegraphics[height=0.5in]{center1.pdf}}\,  \rightsquigarrow\, \raisebox{-15pt}{\includegraphics[height=0.55in]{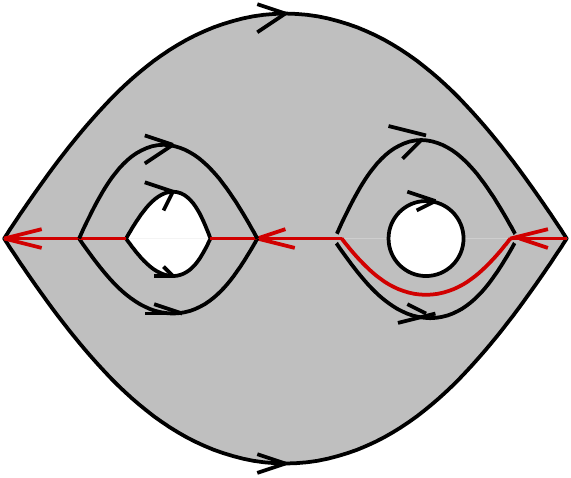}} \quad \cong \quad \raisebox{-15pt}{\includegraphics[height=0.55in]{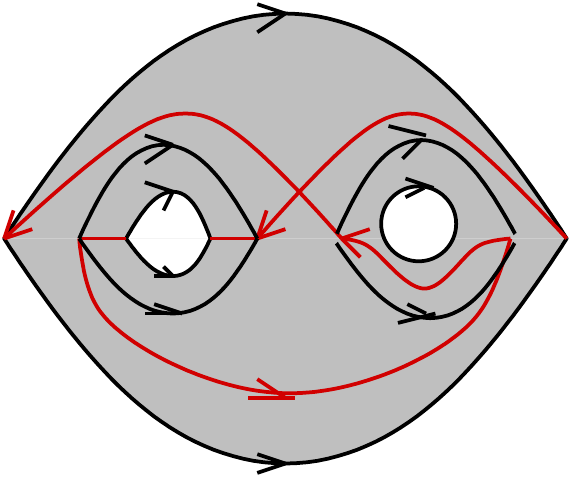}} \rightsquigarrow\, \raisebox{-13pt}{\includegraphics[height=0.65in]{center2.pdf}}\]

We leave to the interested reader the proof of the remaining diffeomorphisms.
\end{proof}

%%%%%%%%%%%%%%%%%%%%%%%%%%%%%%%%%%%%%%%%%%%%%%%%%%%%%%%%
\subsection{Consequences of relations}
We provide now additional diffeomorphisms implied by those described in Proposition~\ref{prop:relations}, and which will be useful for the remaining of the paper.

\begin{proposition}
The cozipper \, \raisebox{-8pt}{\includegraphics[height=0.28in]{cozipper.pdf}}\, is a coalgebra homomorphism, that is, the following singular cobordisms are equivalent:
\begin{equation} \raisebox{-13pt}{\includegraphics[height=0.5in]{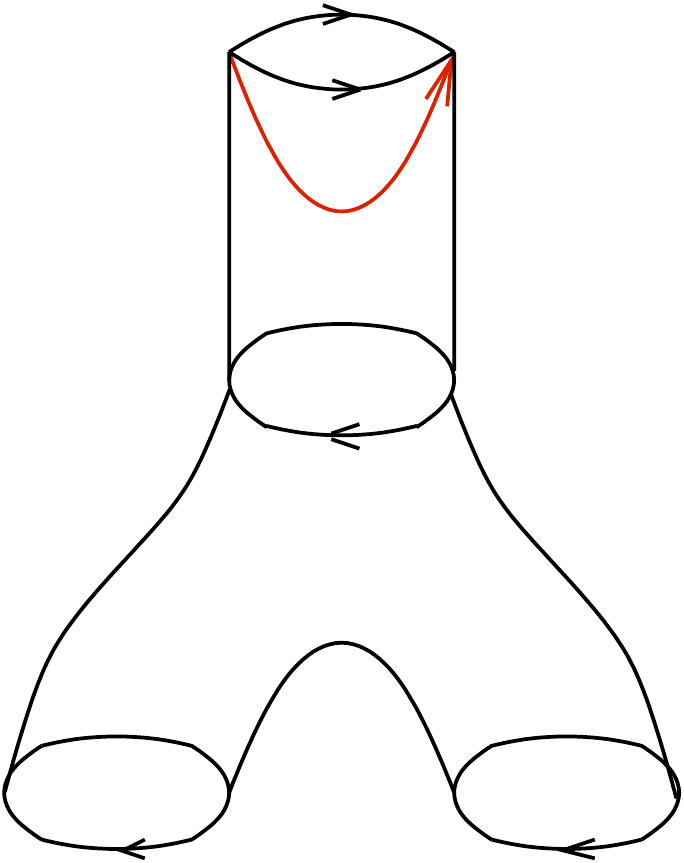}}\quad \cong \quad  \raisebox{-13pt}{\includegraphics[height=0.5in]{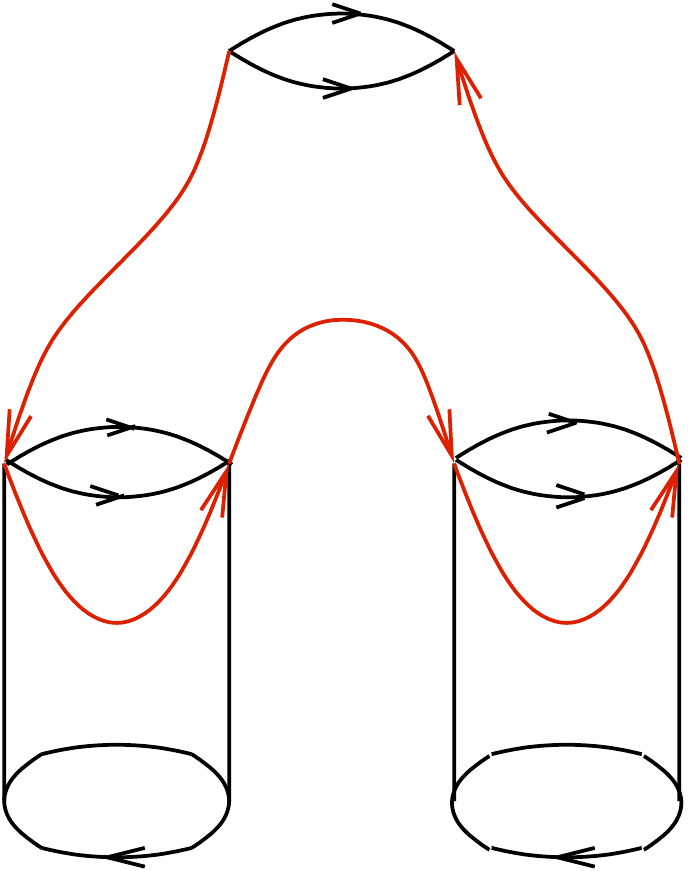}} \hspace{2cm}  \raisebox{-10pt}{\includegraphics[height=0.35in]{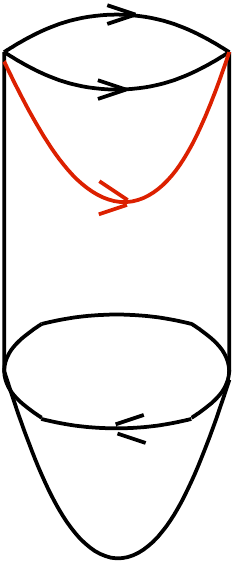}} \quad \cong \quad  \raisebox{-8pt}{\includegraphics[height=0.2in]{singcounit.pdf}} \label{eq:cozipper_coalghom}  
\end{equation}
\end{proposition}

It will be useful to define the following singular cobordisms called \textit{singular pairing} and \textit{singular copairing}:
\[\raisebox{-5pt}{\includegraphics[height=0.22in]{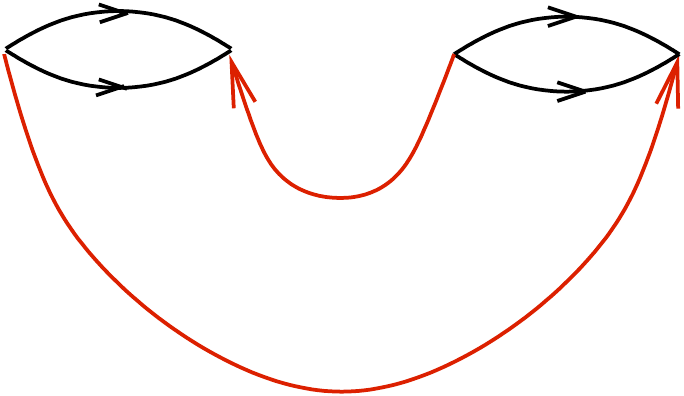}} := \, \raisebox{-8pt}{\includegraphics[height=0.35in]{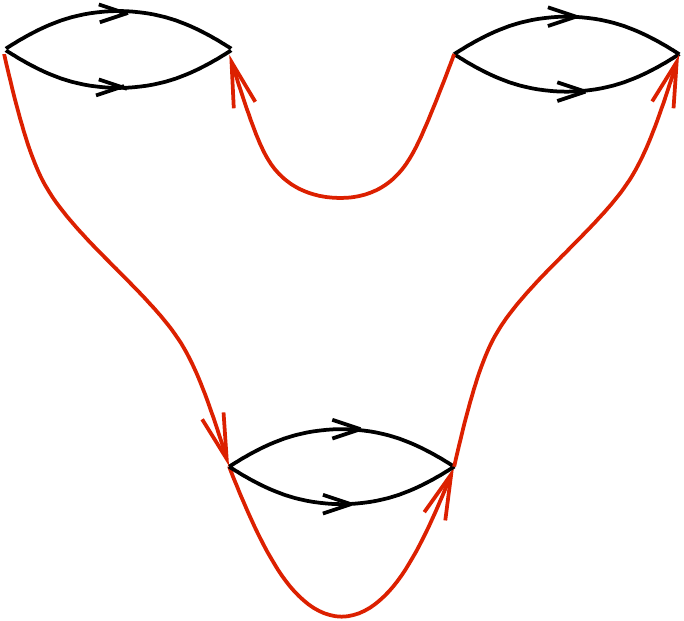}} \hspace{2cm}    \raisebox{-5pt}{\includegraphics[height=0.22in]{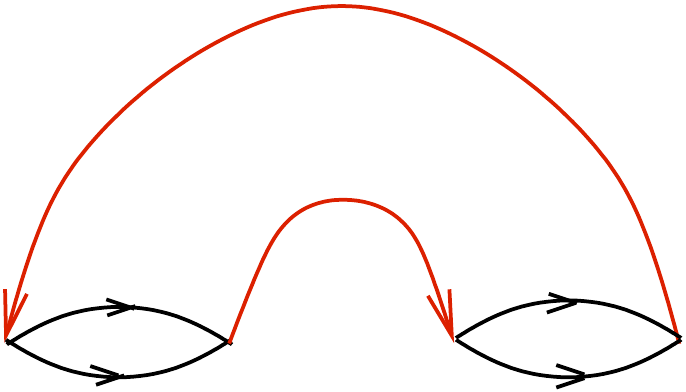}} := \, \raisebox{-8pt}{\includegraphics[height=0.35in]{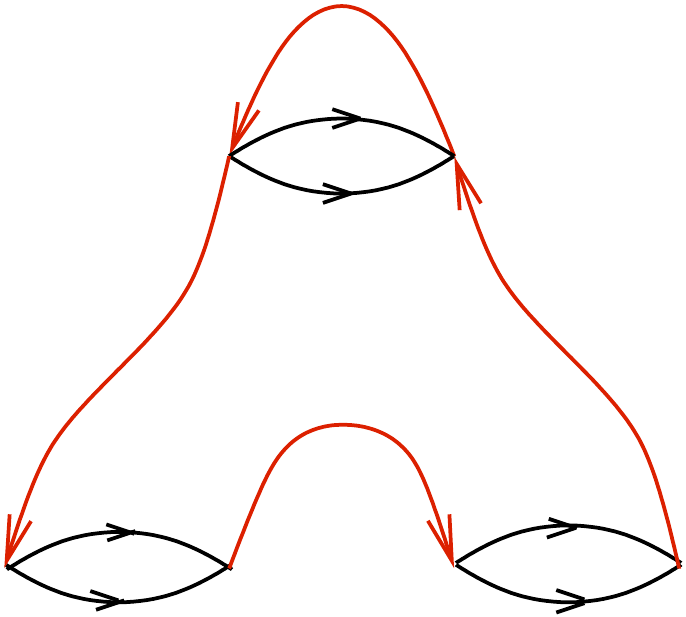}}\]

Similarly, we define the cobordisms which we call the \textit{ordinary pairing} and \textit{ordinary copairing}:
\[\raisebox{-5pt}{\includegraphics[height=0.22in]{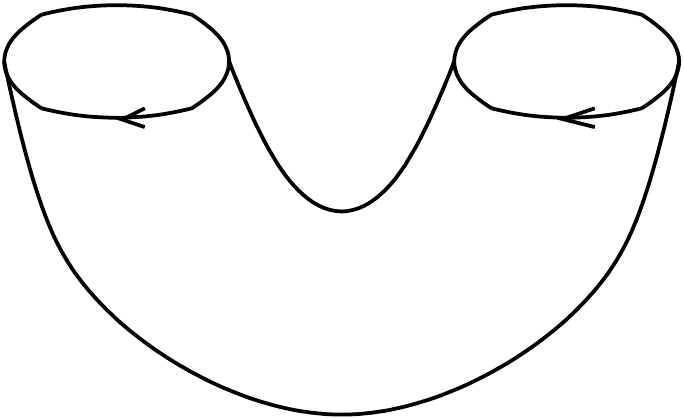}} := \, \raisebox{-8pt}{\includegraphics[height=0.35in]{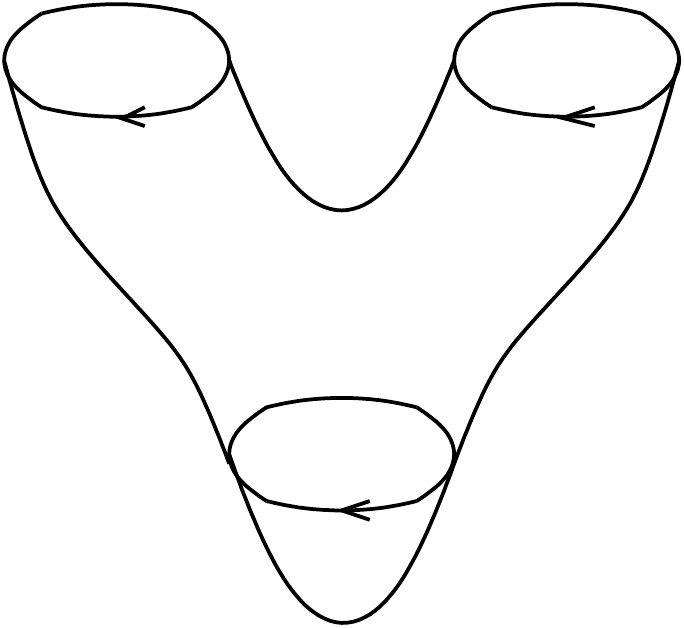}} \hspace{2cm}    \raisebox{-5pt}{\includegraphics[height=0.22in]{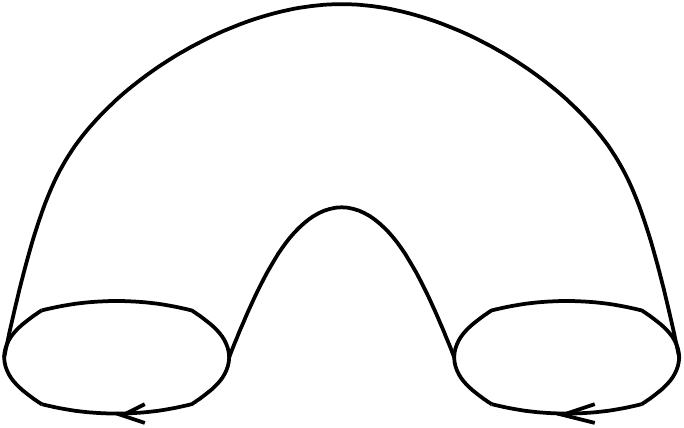}} := \, \raisebox{-8pt}{\includegraphics[height=0.35in]{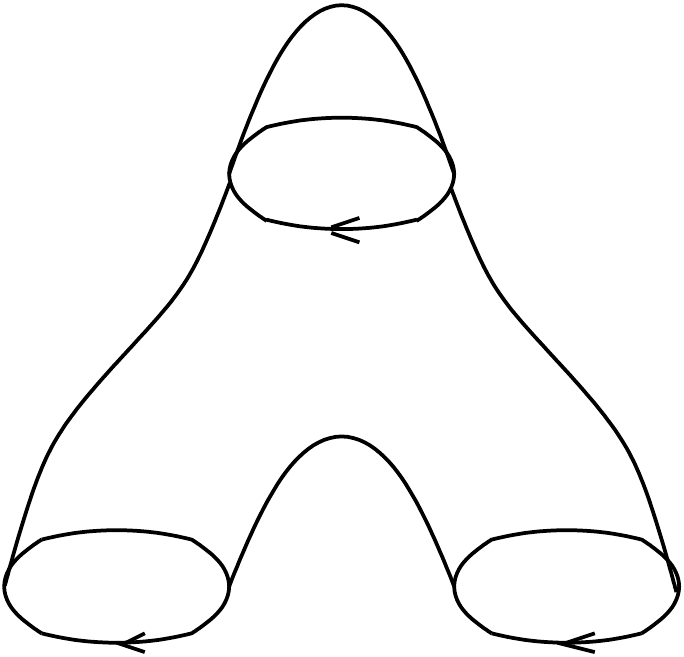}}\]
 
These cobordisms satisfy the zig-zag identities:
\begin{equation}
\raisebox{-13pt}{\includegraphics[height=0.55in]{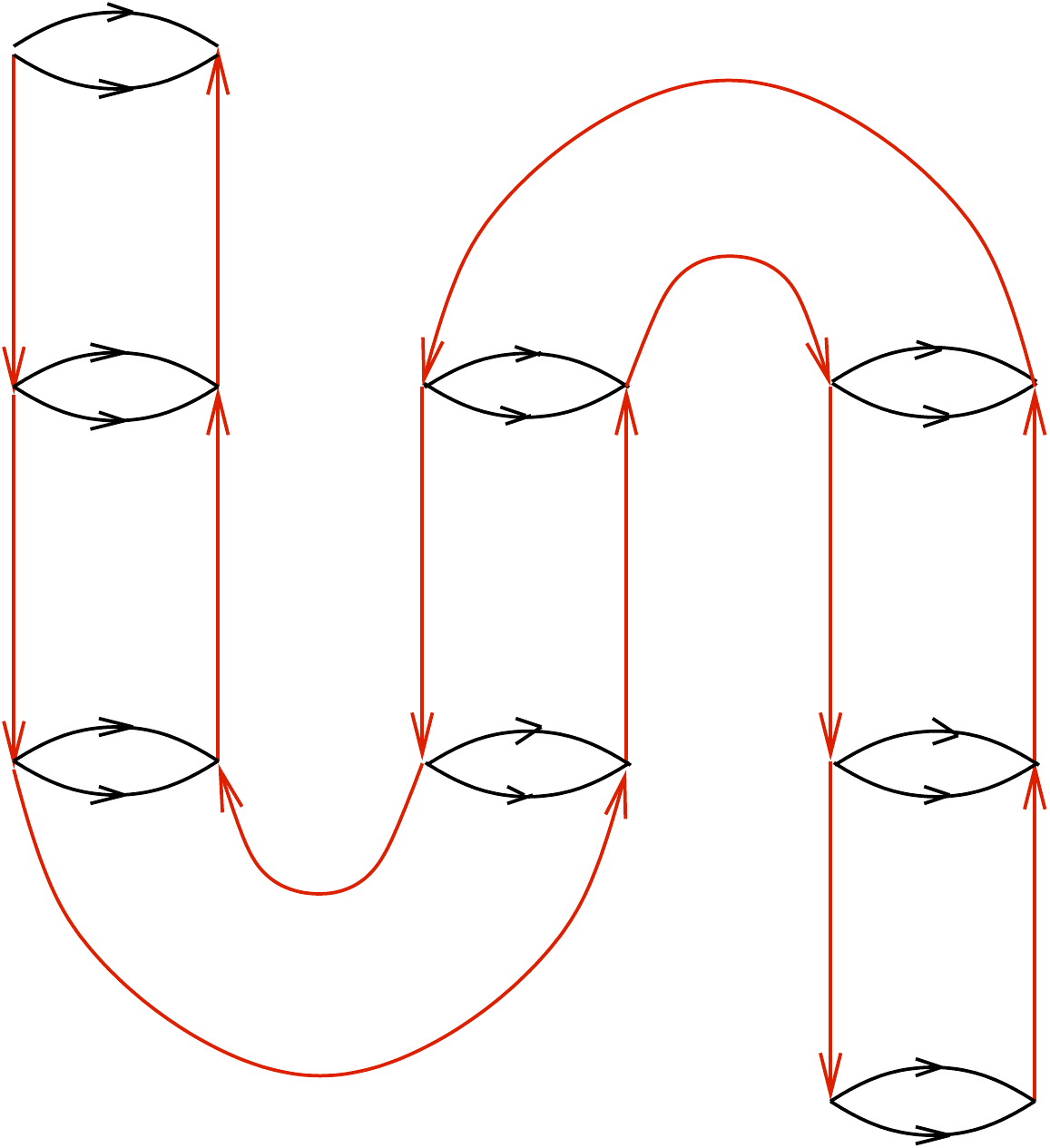}}\quad \cong \quad  \raisebox{-13pt}{\includegraphics[height=0.55in]{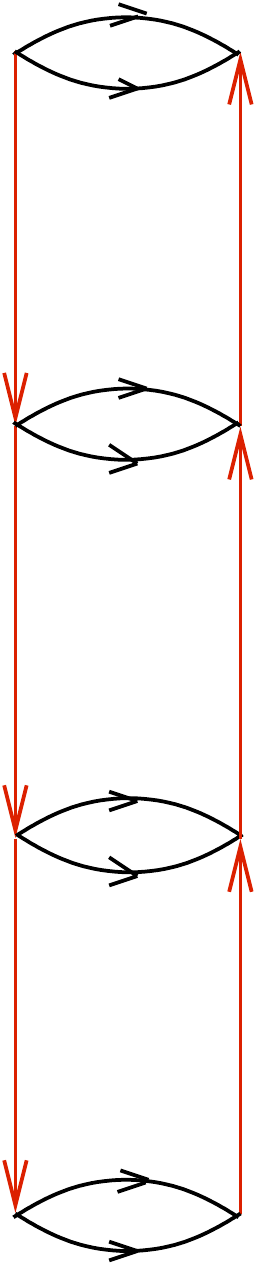}} \quad \cong \quad \raisebox{-13pt}{\includegraphics[height=0.55in]{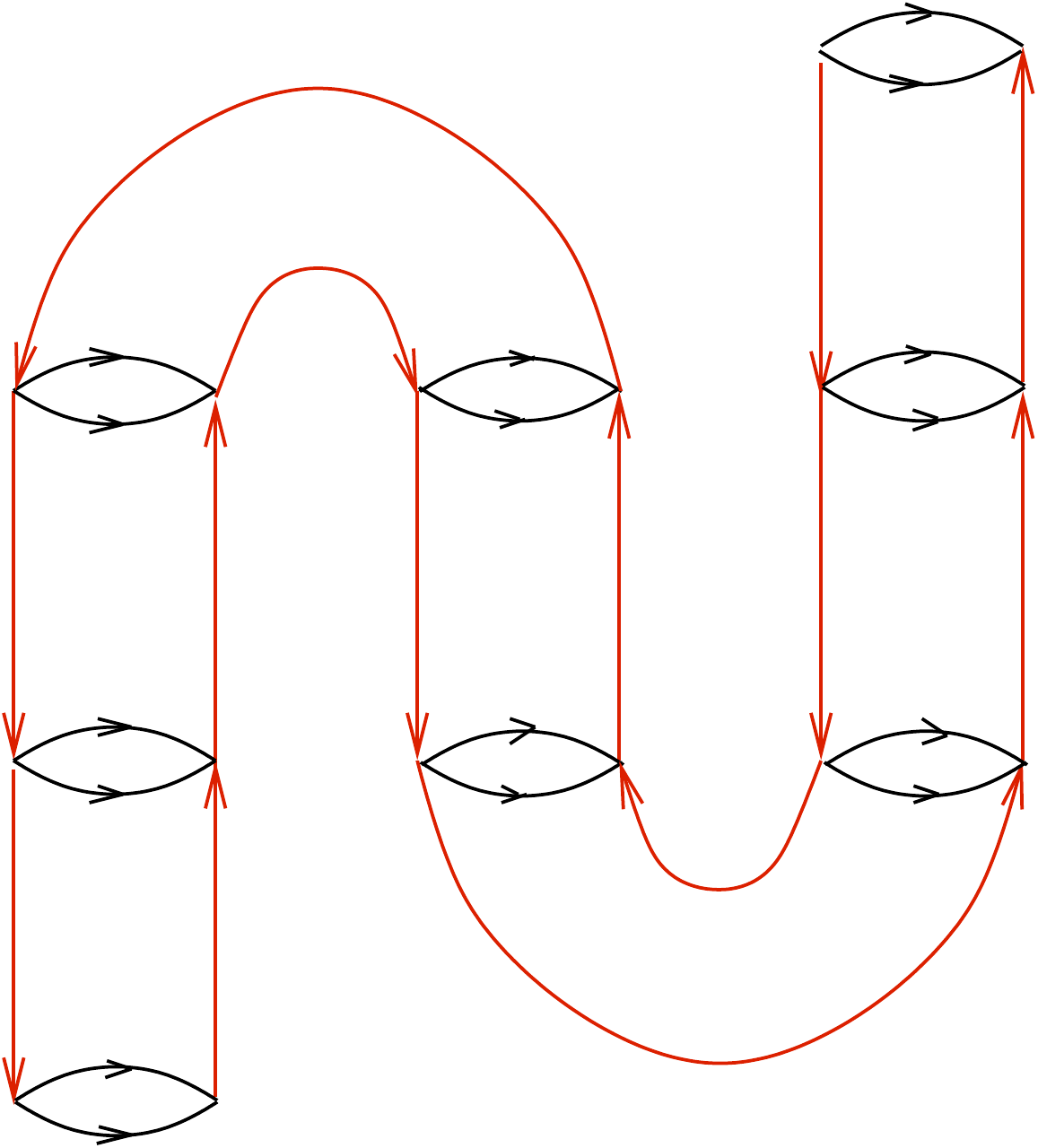}}\label{eq:sing_zig_zag}
\end{equation}

\begin{equation}
\raisebox{-13pt}{\includegraphics[height=0.55in]{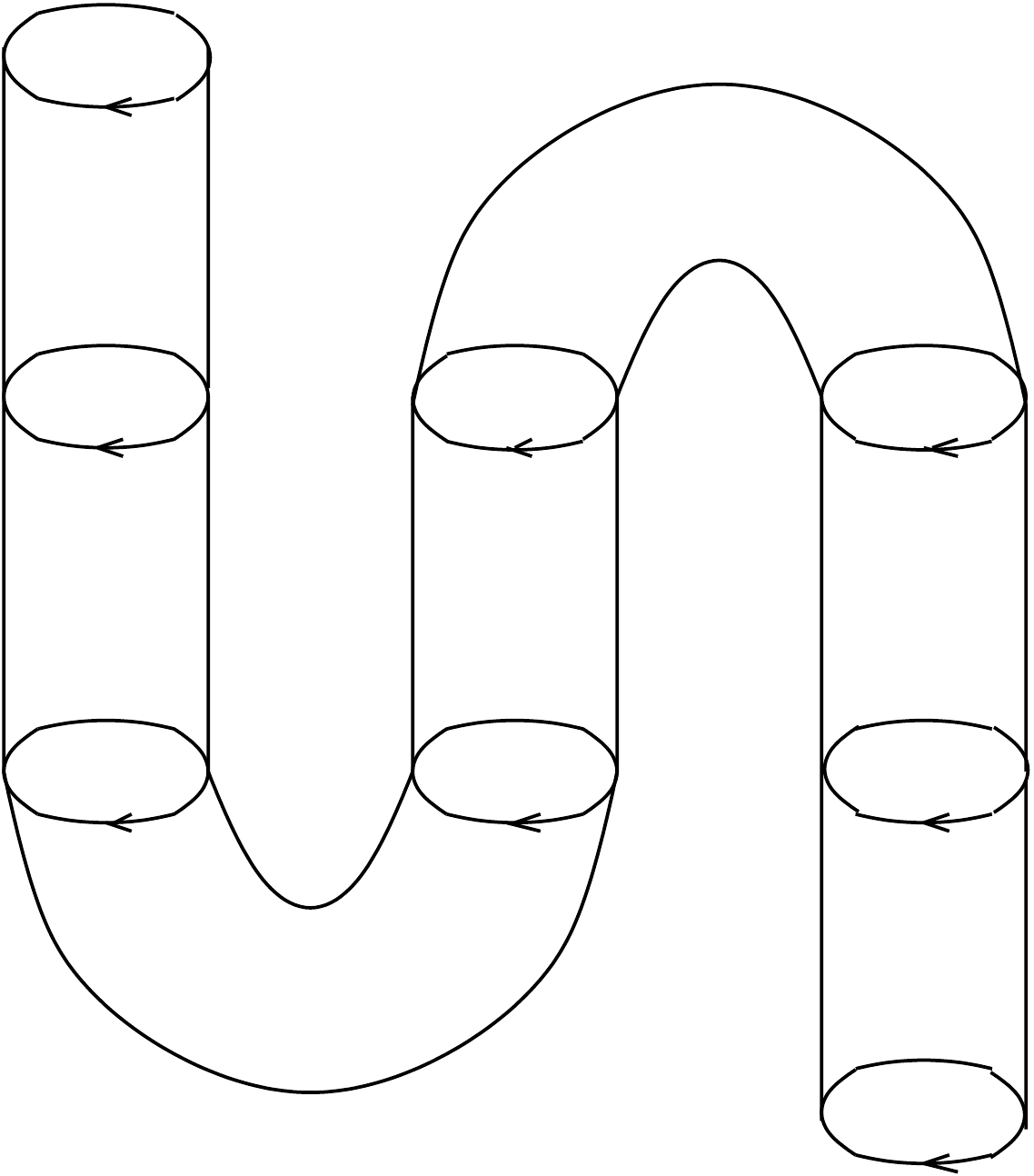}}\quad \cong \quad  \raisebox{-13pt}{\includegraphics[height=0.55in]{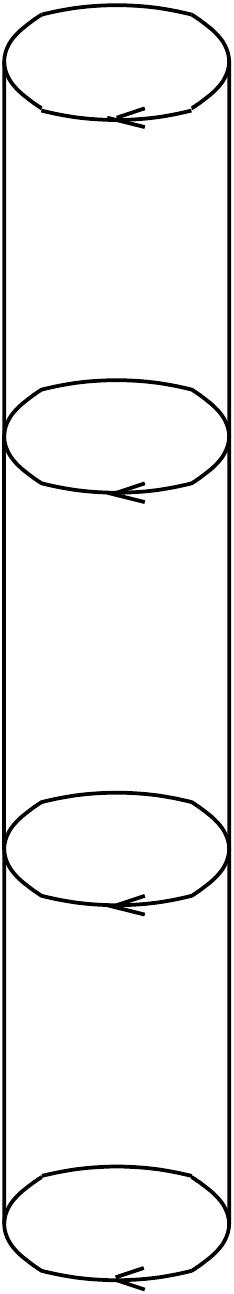}} \quad \cong \quad \raisebox{-13pt}{\includegraphics[height=0.55in]{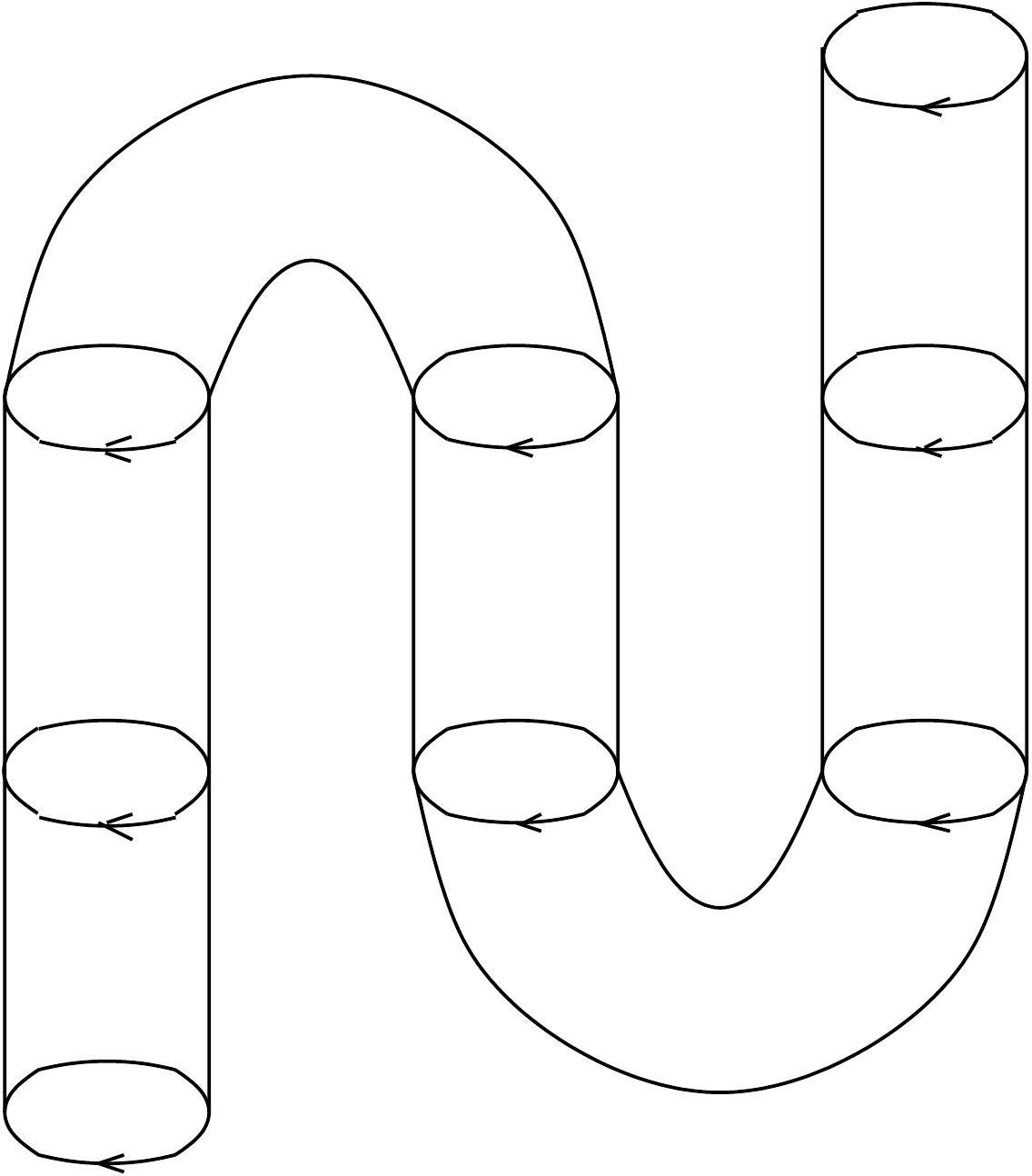}}\label{eq:zig_zag}
\end{equation}

It follows from Equations (\ref{eq:web_frob4}) and (\ref{eq:web_frob1}) that the singular pairing is invariant and symmetric:
\begin{equation}
\raisebox{-13pt}{\includegraphics[height=0.55in]{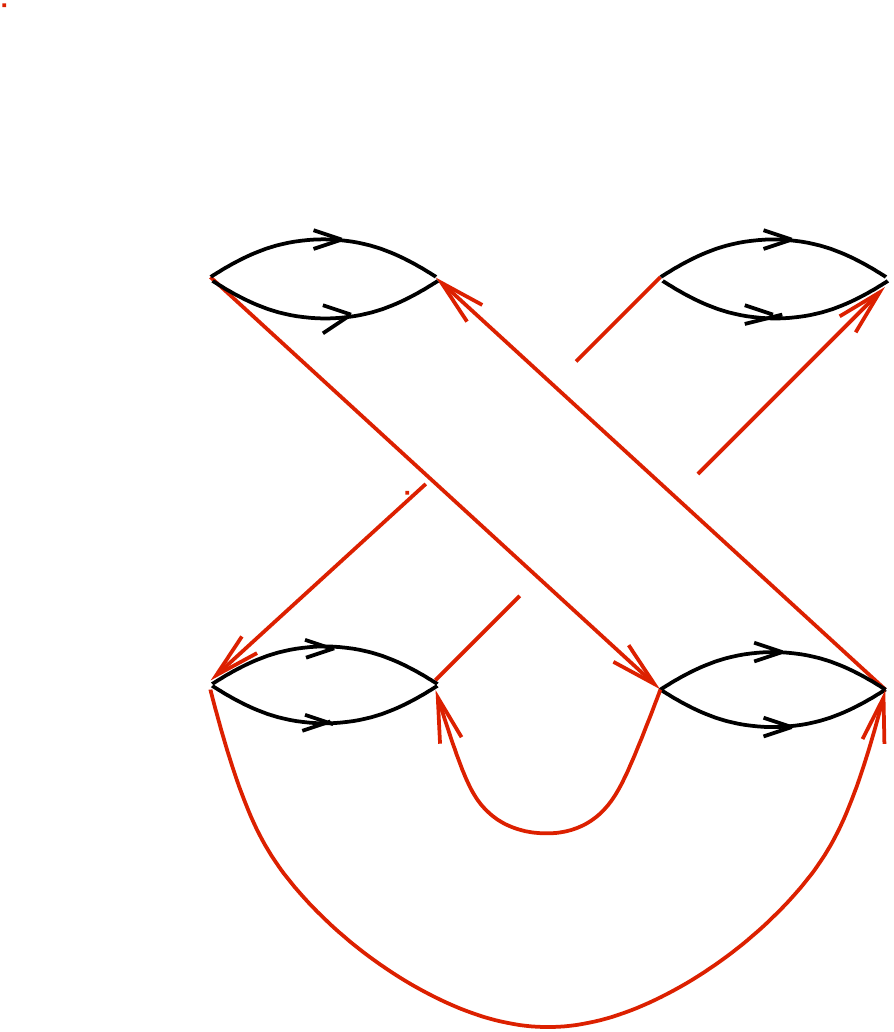}}\quad \cong \quad  \raisebox{-13pt}{\includegraphics[height=0.42in]{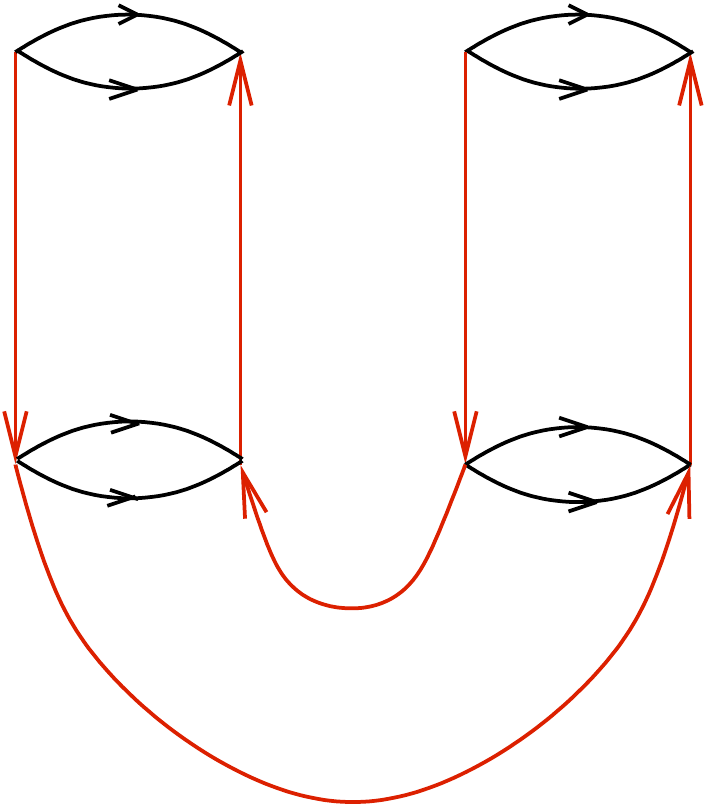}}  \hspace{1cm} \raisebox{-13pt}{\includegraphics[height=0.42in]{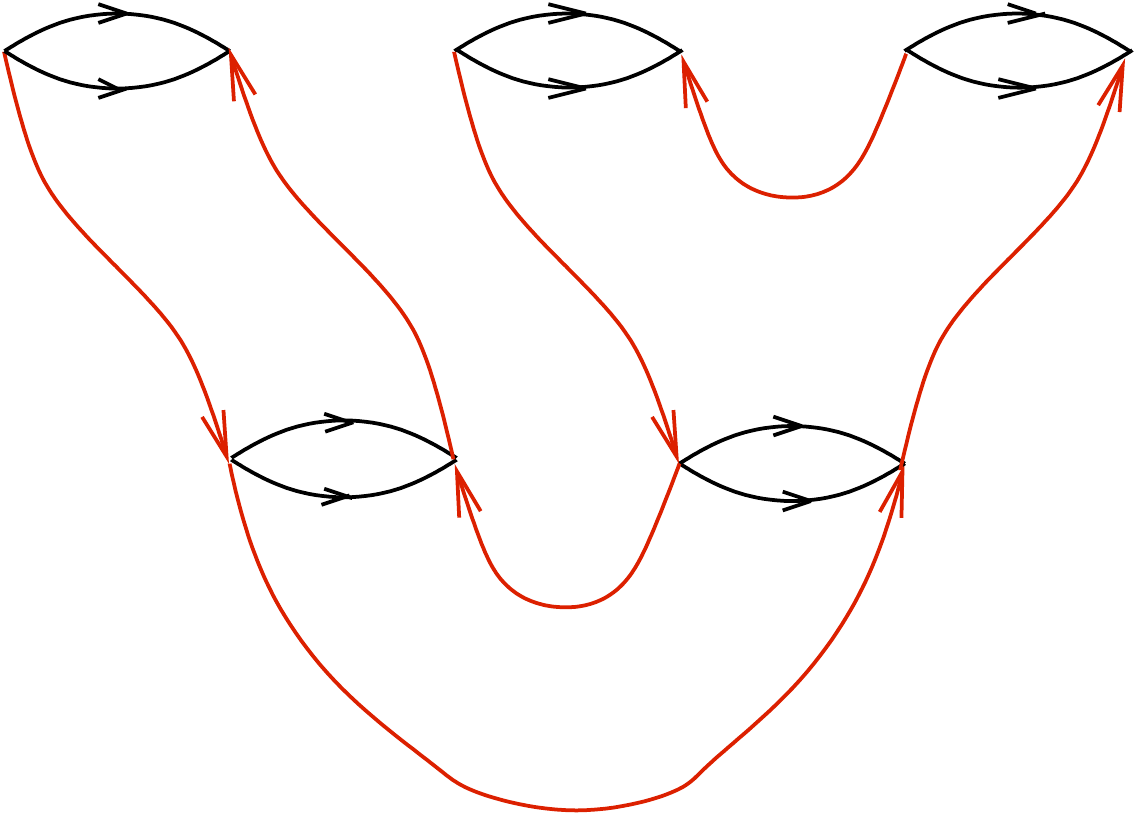}}\quad \cong \quad  \raisebox{-13pt}{\includegraphics[height=0.42in]{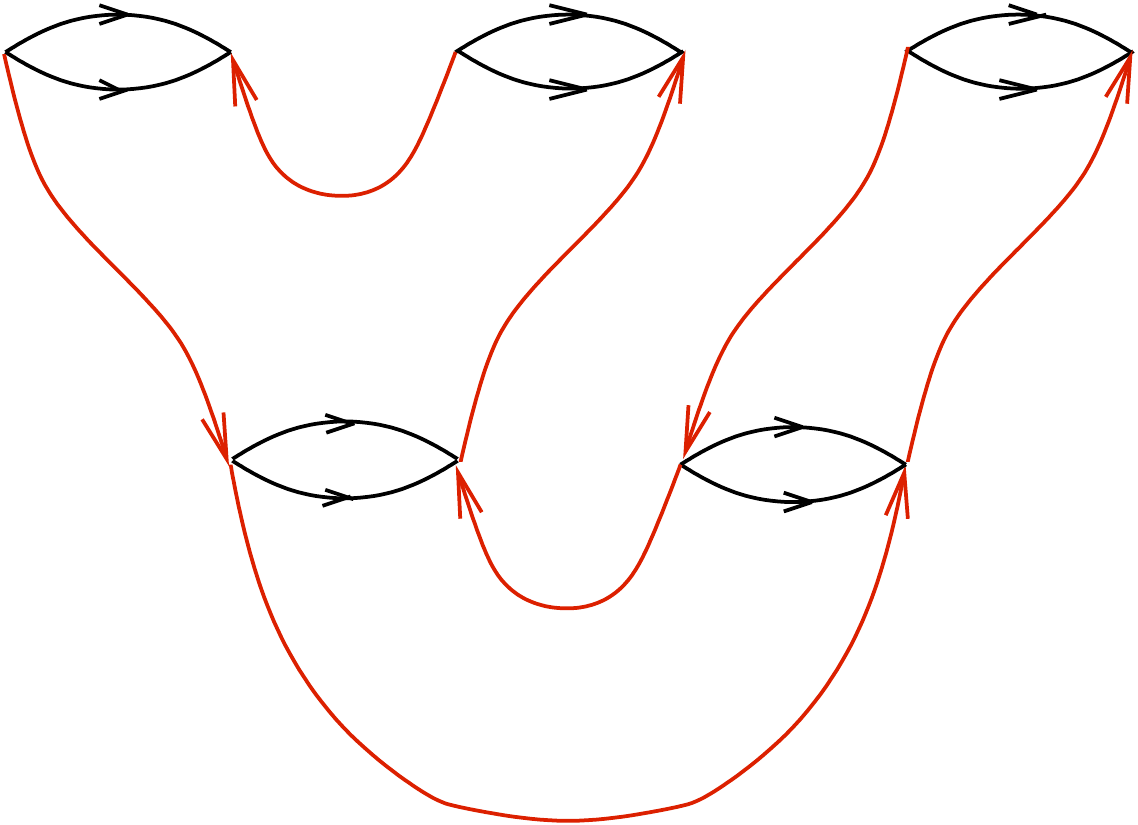}} \label{eq:singpairing_inv_sym}
\end{equation}
Similarly, it follows from Equations (\ref{eq:circle_frob4}) and (\ref{eq:circle_frob1}) that the ordinary pairing is invariant and symmetric:
\begin{equation}
\raisebox{-13pt}{\includegraphics[height=0.42in]{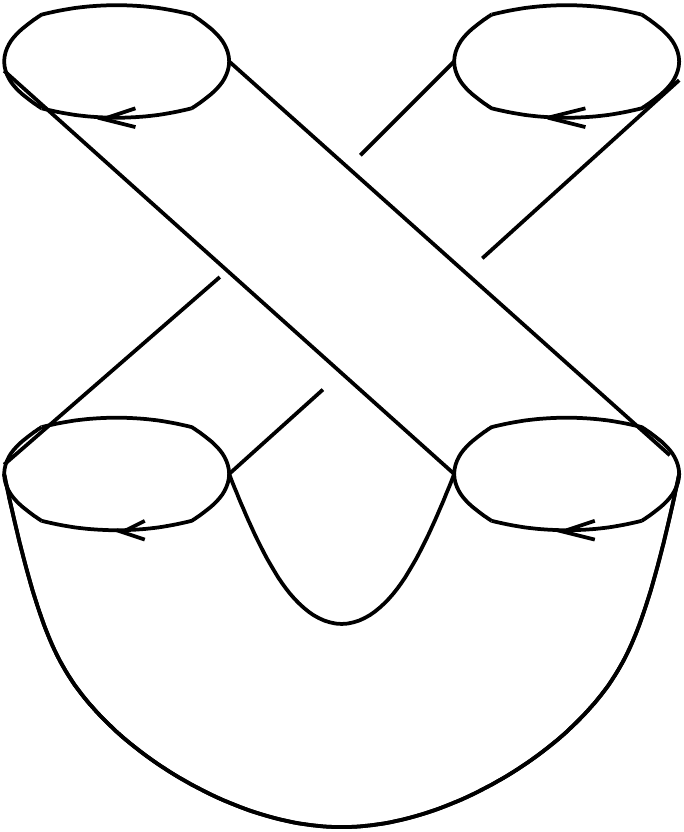}}\quad \cong \quad  \raisebox{-13pt}{\includegraphics[height=0.42in]{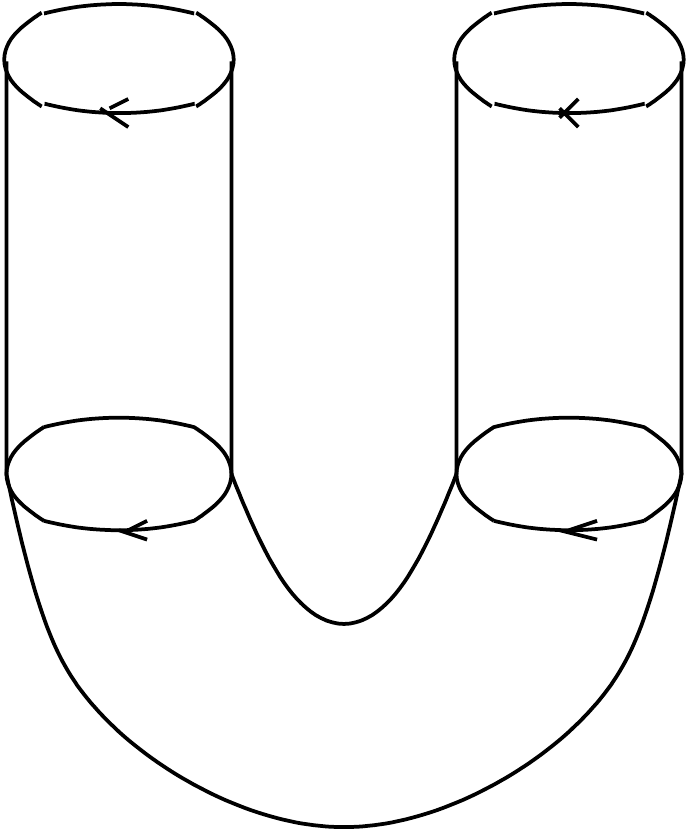}}  \hspace{1cm} \raisebox{-13pt}{\includegraphics[height=0.42in]{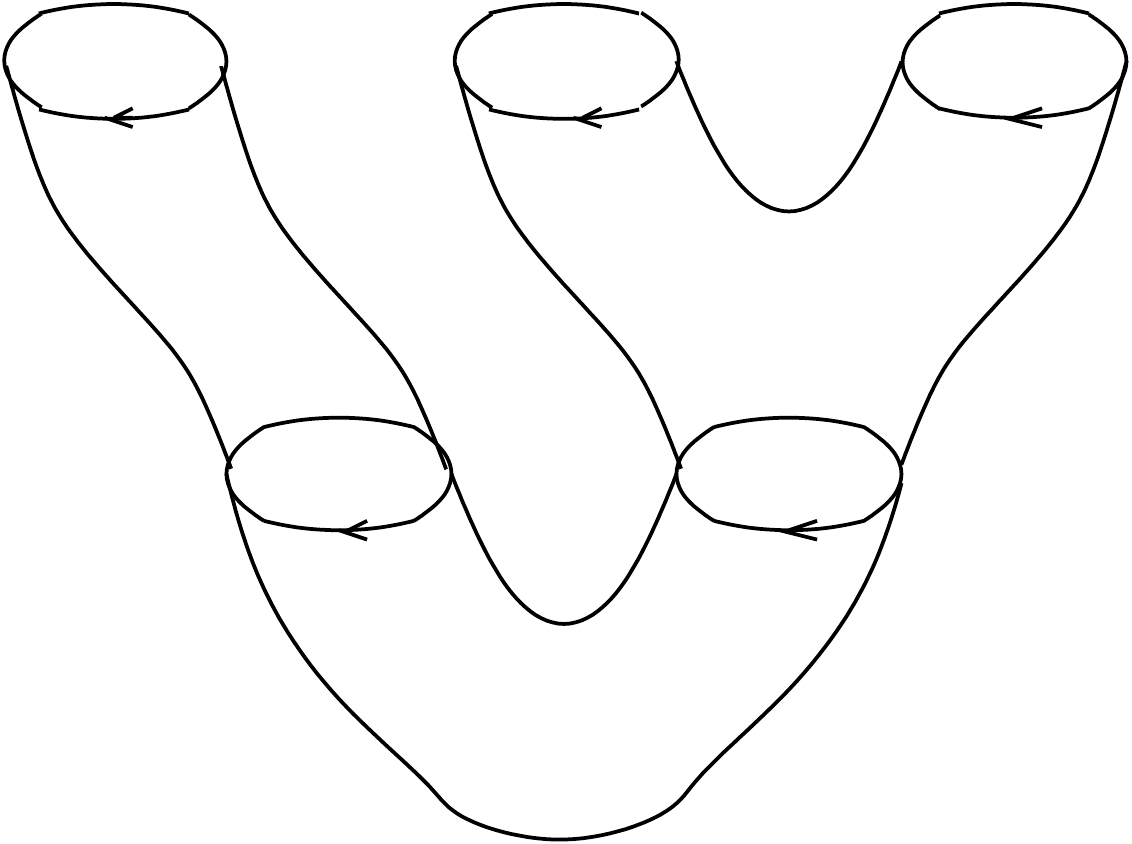}}\quad \cong \quad  \raisebox{-13pt}{\includegraphics[height=0.42in]{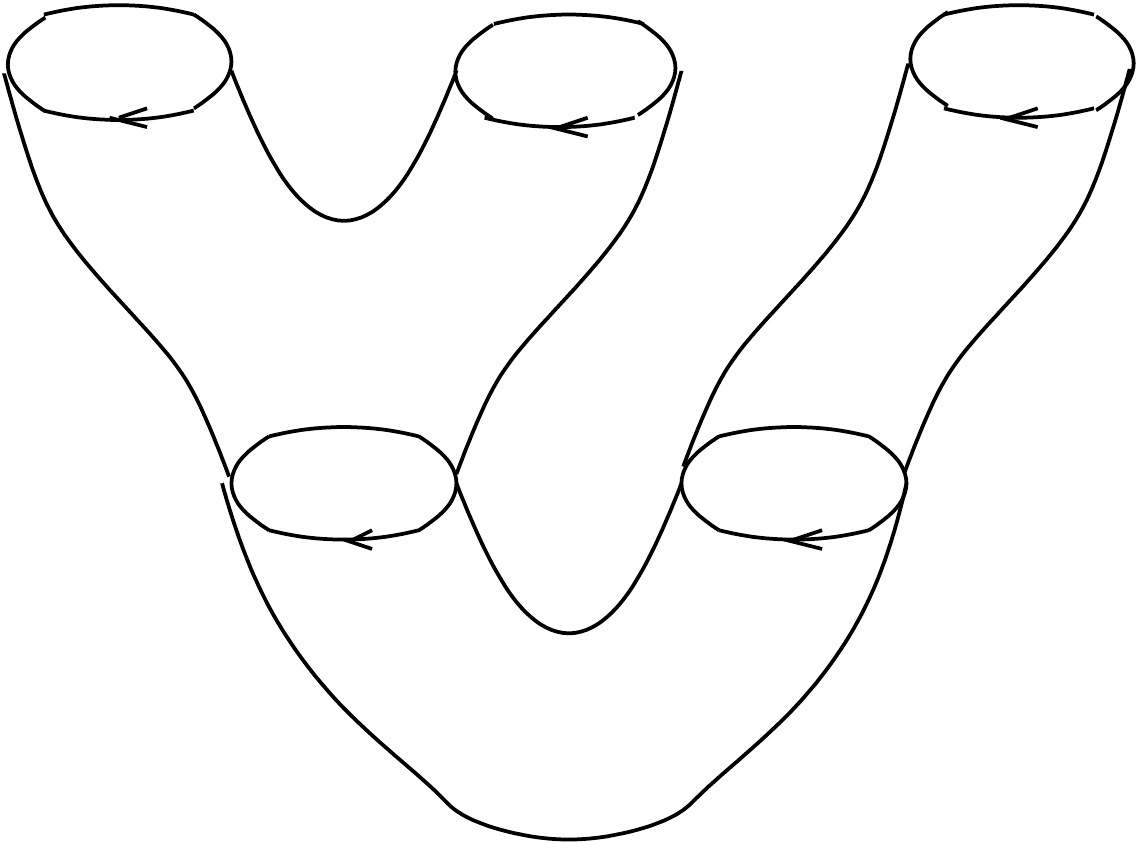}}
\end{equation}
It is easy to see that similar results hold for both singular and ordinary copairings.

\begin{proposition}
The following singular cobordisms are equivalent:

\begin{equation}
\raisebox{-13pt}{\includegraphics[height=0.42in]{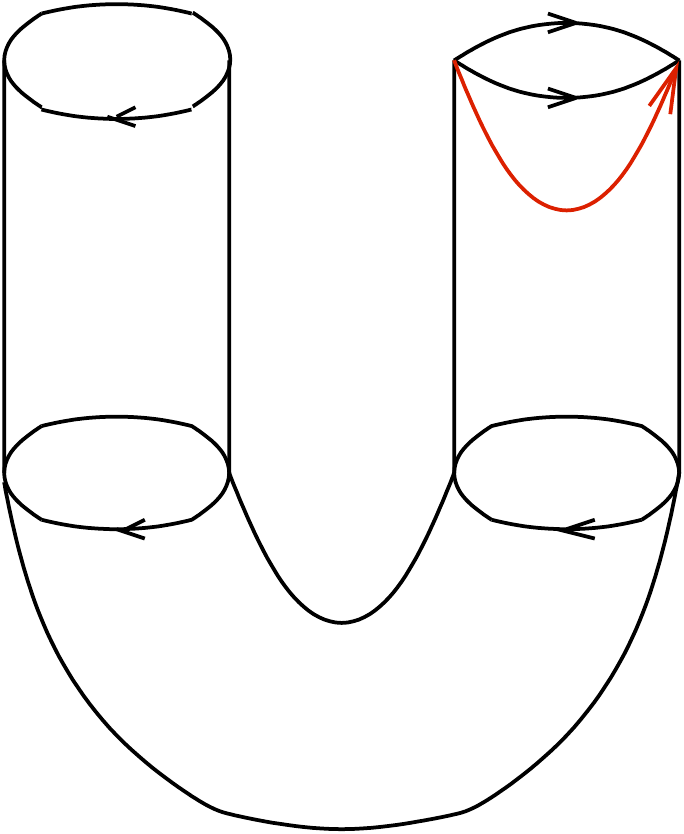}}\quad \cong \quad  \raisebox{-13pt}{\includegraphics[height=0.42in]{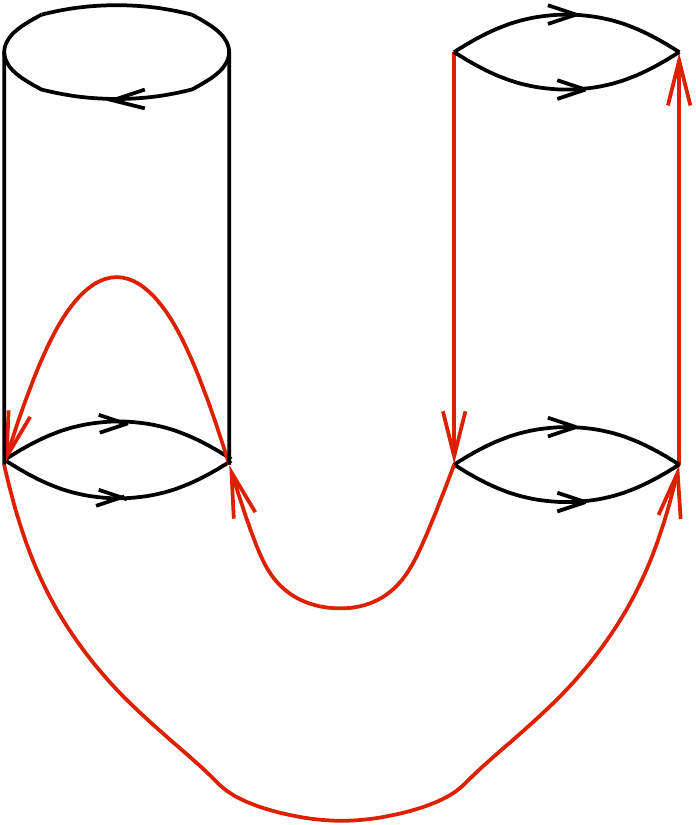}}\hspace{1cm} \raisebox{-13pt}{\includegraphics[height=0.42in]{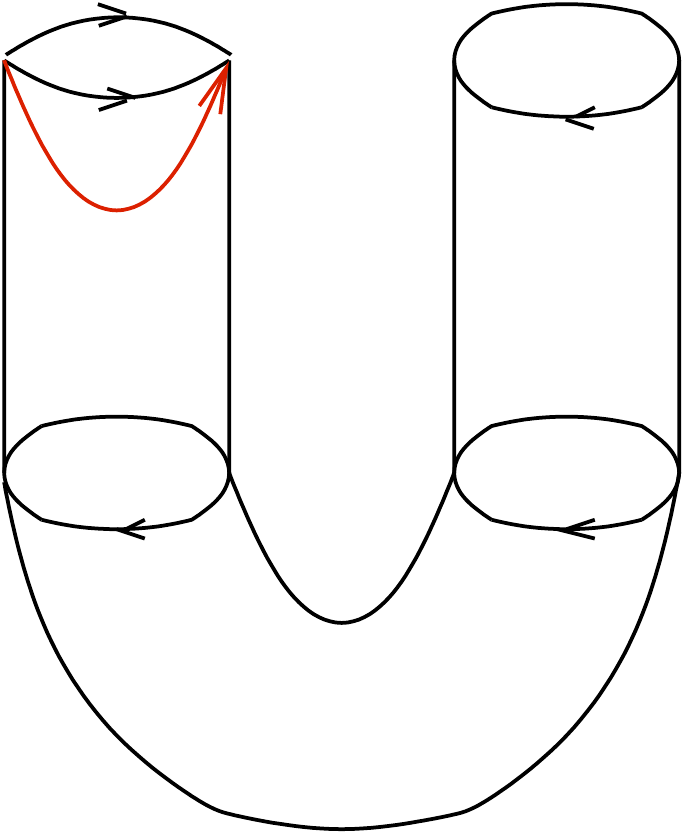}}\quad \cong \quad  \raisebox{-13pt}{\includegraphics[height=0.42in]{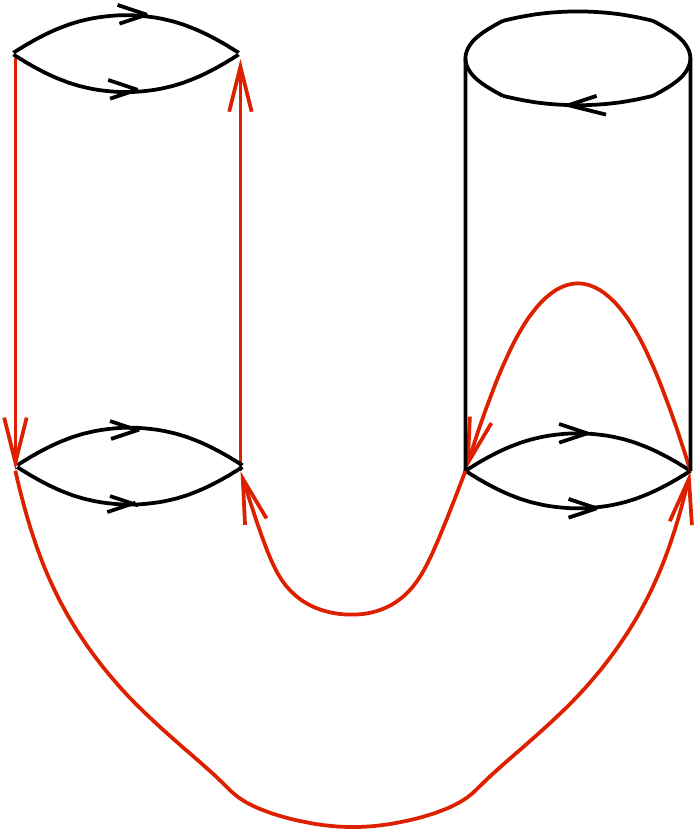}}  \label{eq:cozipper_pairing}
\end{equation}
\begin{equation}
\raisebox{-13pt}{\includegraphics[height=0.42in]{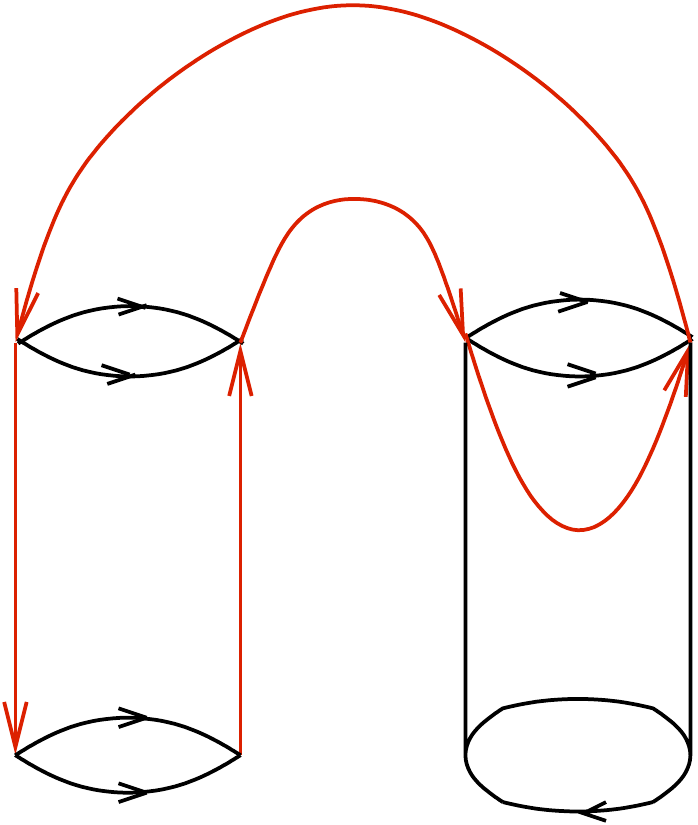}}\quad \cong \quad  \raisebox{-13pt}{\includegraphics[height=0.42in]{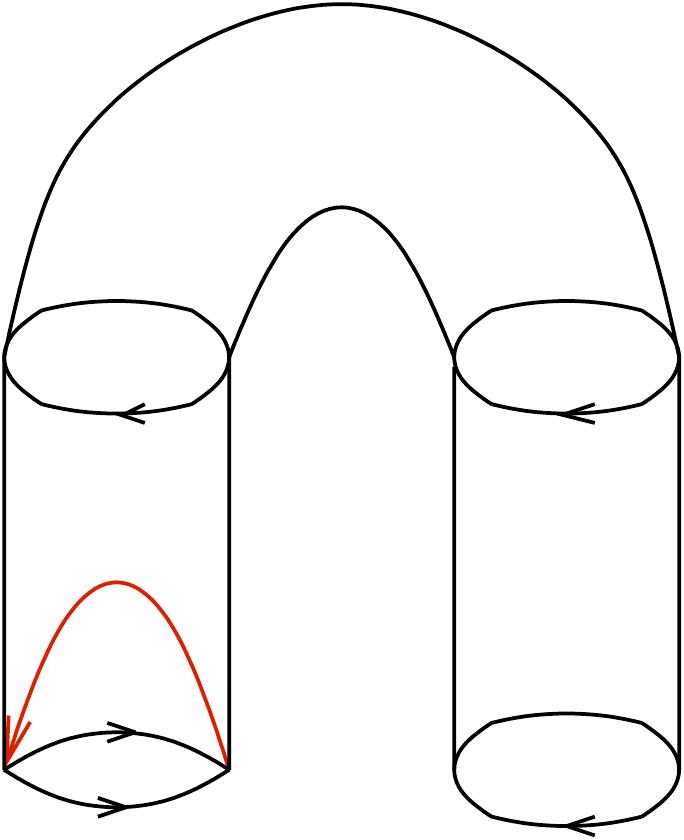}}\hspace{1cm} \raisebox{-13pt}{\includegraphics[height=0.42in]{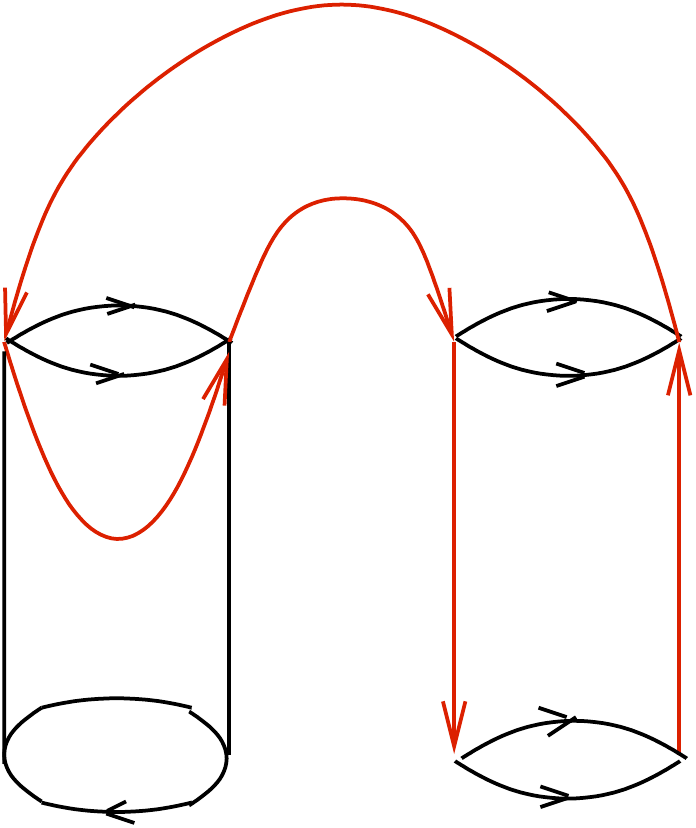}}\quad \cong \quad  \raisebox{-13pt}{\includegraphics[height=0.42in]{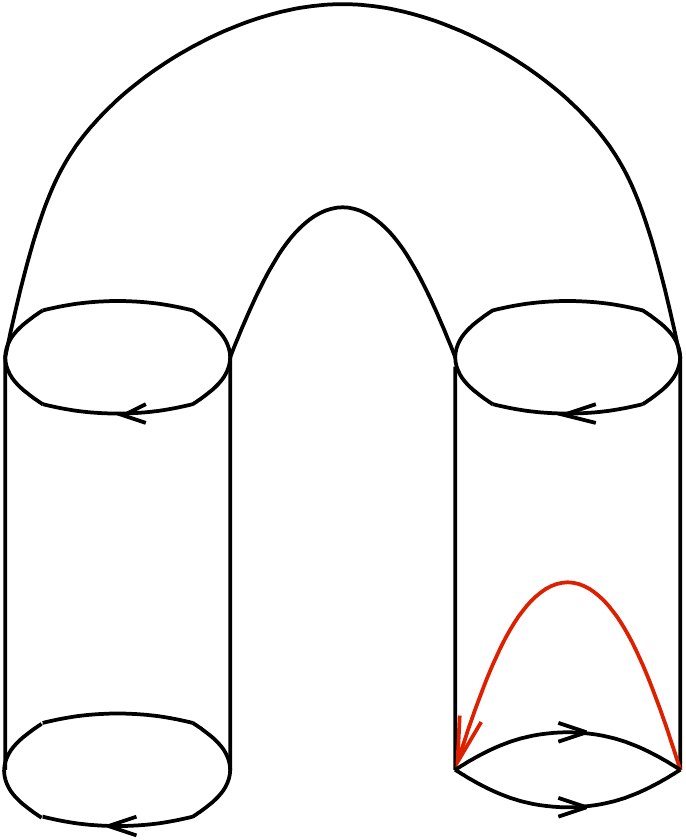}}  \label{eq:cozipper_copairing}
\end{equation}
\end{proposition}

\begin{proof}
The first equivalence of cobordisms in Equation (\ref{eq:cozipper_pairing}) is the same as the equivalence in Equation (\ref{eq:cozipper_dual}). The proof of the second diffeomorphism in Equation (\ref{eq:cozipper_pairing}) is given below, where by ``Nat'' we denote the diffeomorphisms which express the natural behavior of the symmetric twist: 

\[\raisebox{-8pt}{\includegraphics[height=0.42in]{cozipper_pairing3.pdf}}\quad \stackrel {\cong}{(\ref{eq:circle_frob4})} \quad \raisebox{-13pt} {\includegraphics[height=0.6in]{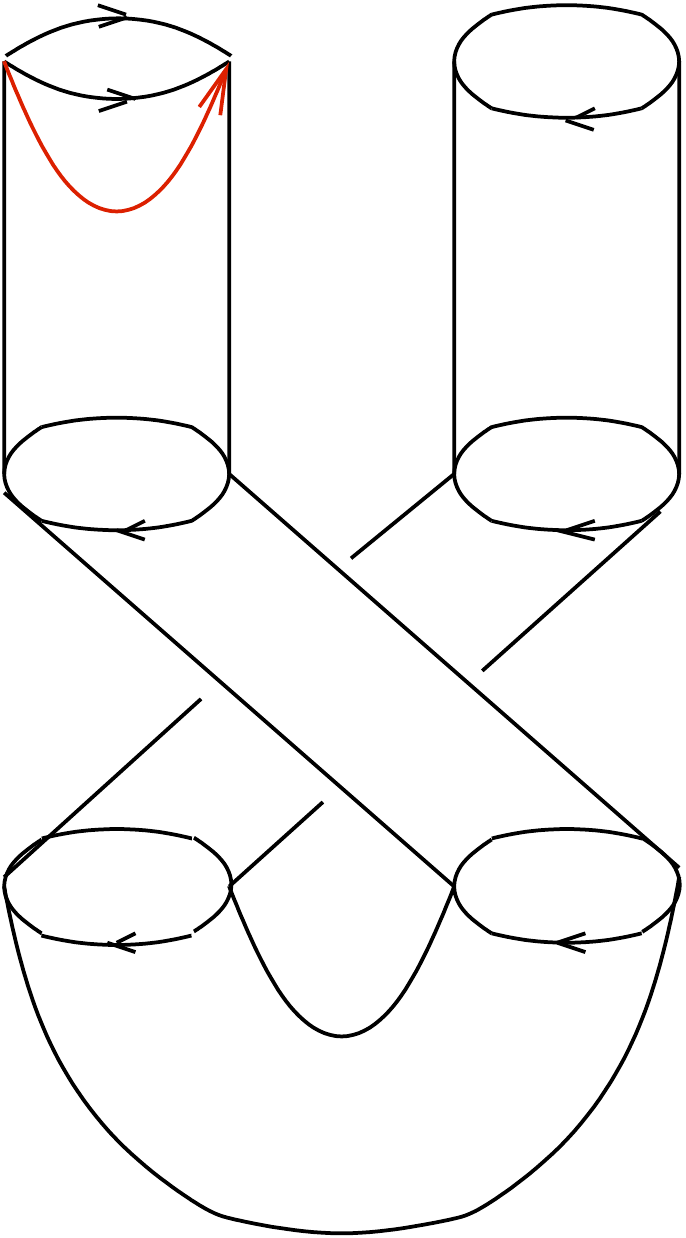}} \stackrel{\cong}{Nat}\quad \raisebox{-13pt}{\includegraphics[height=0.6in]{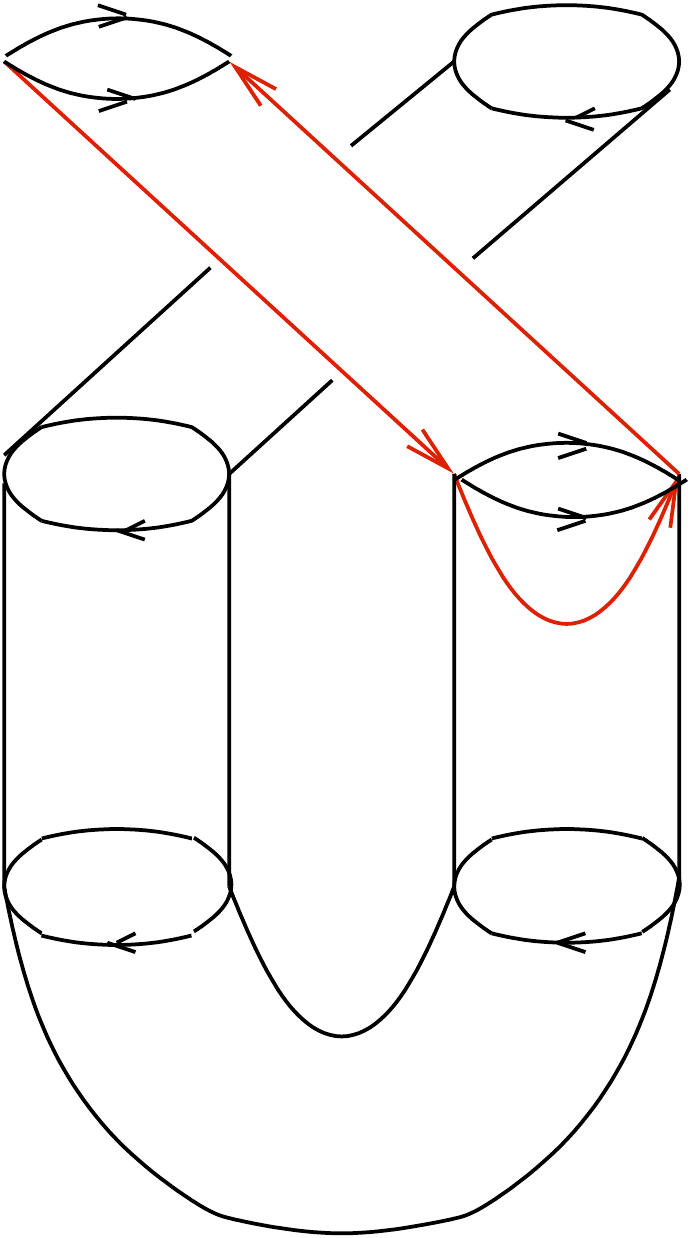}} \stackrel{\cong}{ (\ref{eq:cozipper_dual})} \quad \raisebox{-13pt} {\includegraphics[height=0.6in]{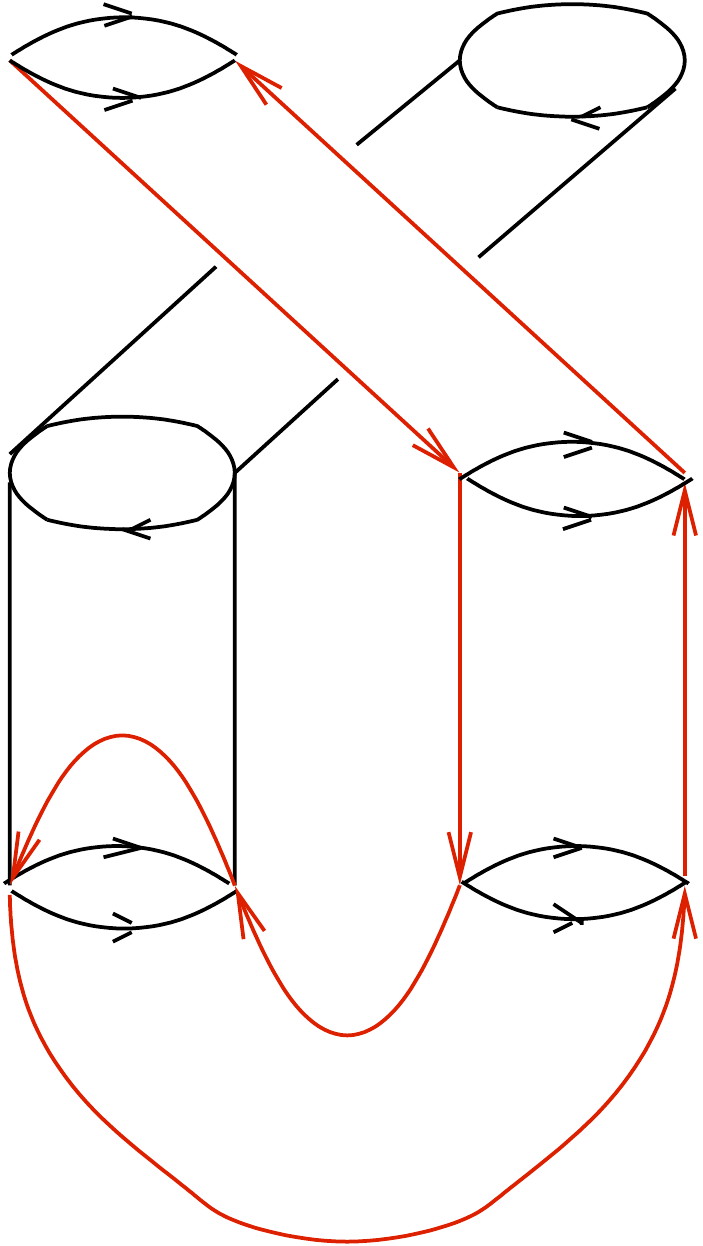}} \stackrel{\cong}{Nat}\quad \raisebox{-13pt} {\includegraphics[height=0.6in]{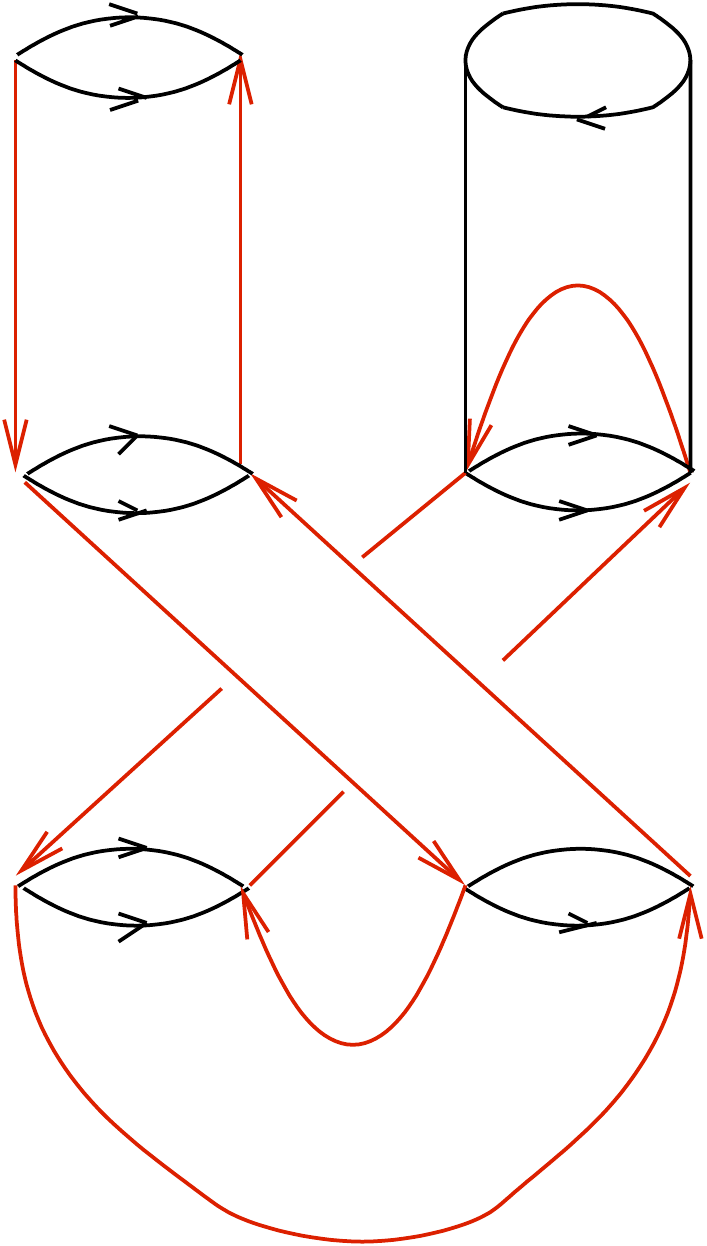}} \stackrel{\cong}{(\ref{eq:web_frob4})}\quad \raisebox{-8pt}{\includegraphics[height=0.42in]{cozipper_pairing4.pdf}}\]

The proof of Equation (\ref{eq:cozipper_copairing}) is given below:

\[ \raisebox{-8pt}{\includegraphics[height=0.42in]{cozipper_copairing1.pdf}}\quad \stackrel {\cong}{(\ref{eq:zig_zag})} \raisebox{-13pt}{\includegraphics[height=0.52in]{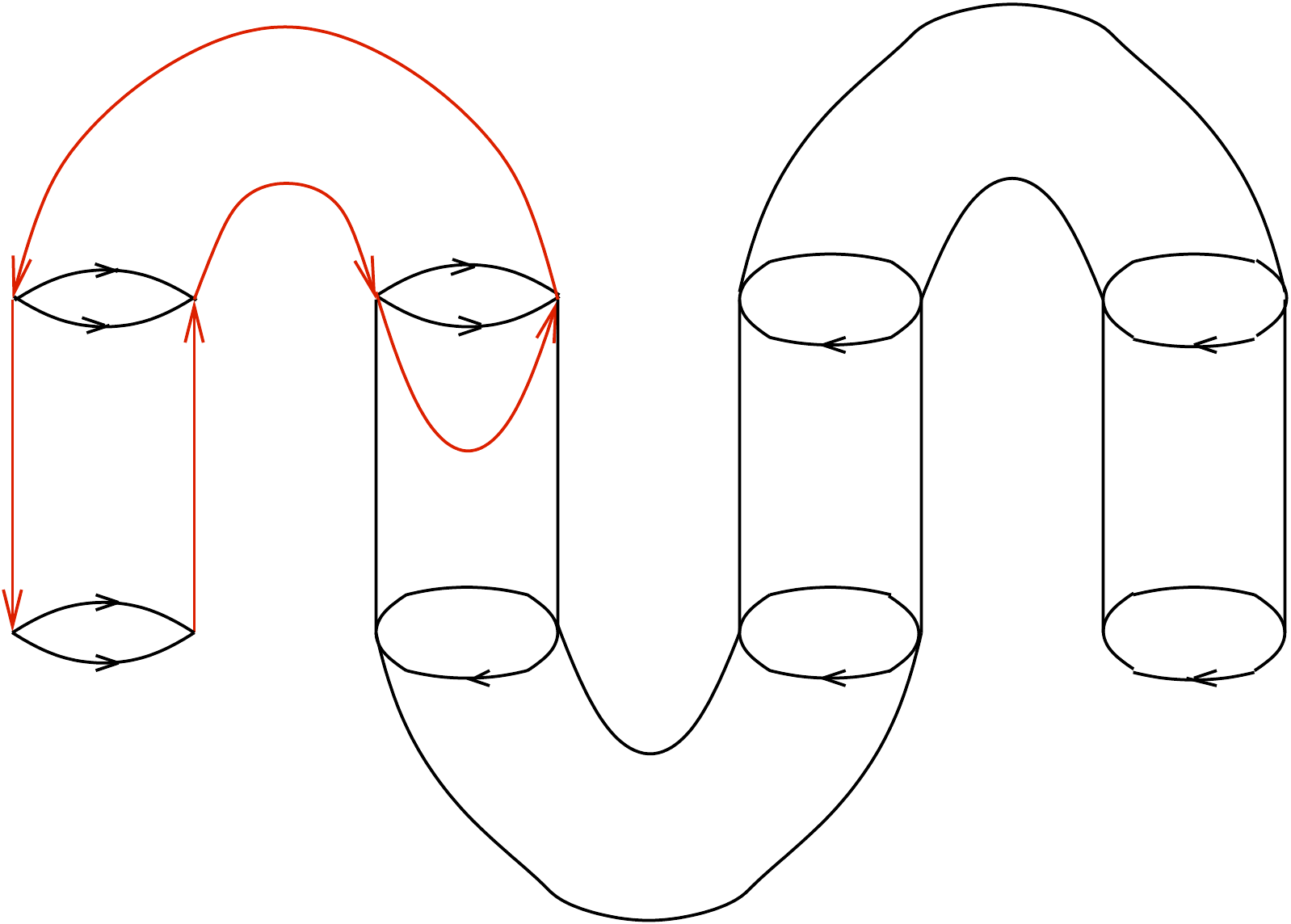}}\quad \stackrel {\cong}{(\ref{eq:cozipper_pairing})} \raisebox{-13pt}{\includegraphics[height=0.52in]{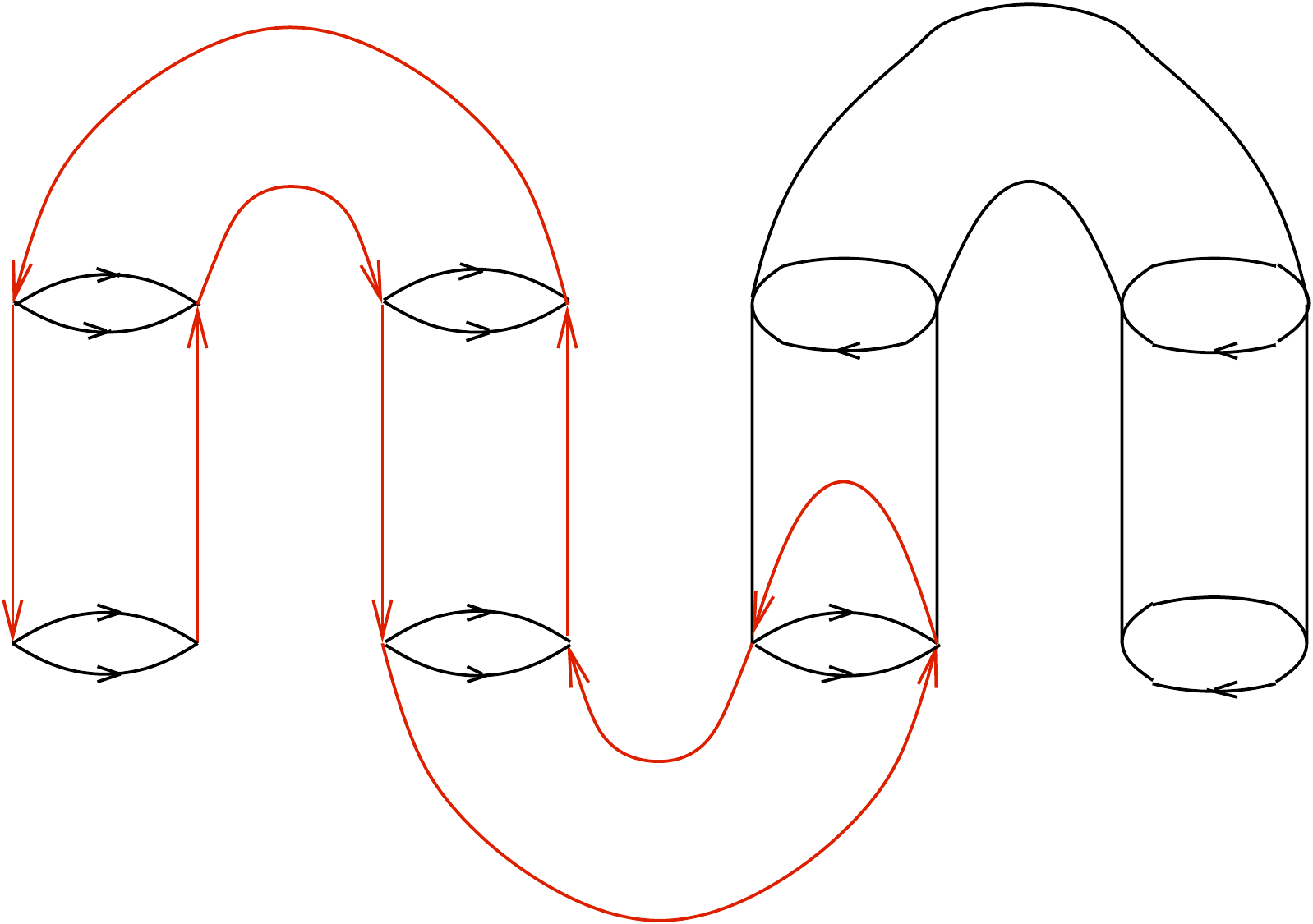}}\quad \stackrel {\cong}{(\ref{eq:sing_zig_zag})}\raisebox{-8pt}{\includegraphics[height=0.42in]{cozipper_copairing2.pdf}} \]
\[ \raisebox{-8pt}{\includegraphics[height=0.42in]{cozipper_copairing3.pdf}}\quad \stackrel {\cong}{(\ref{eq:zig_zag})} \raisebox{-13pt}{\includegraphics[height=0.52in]{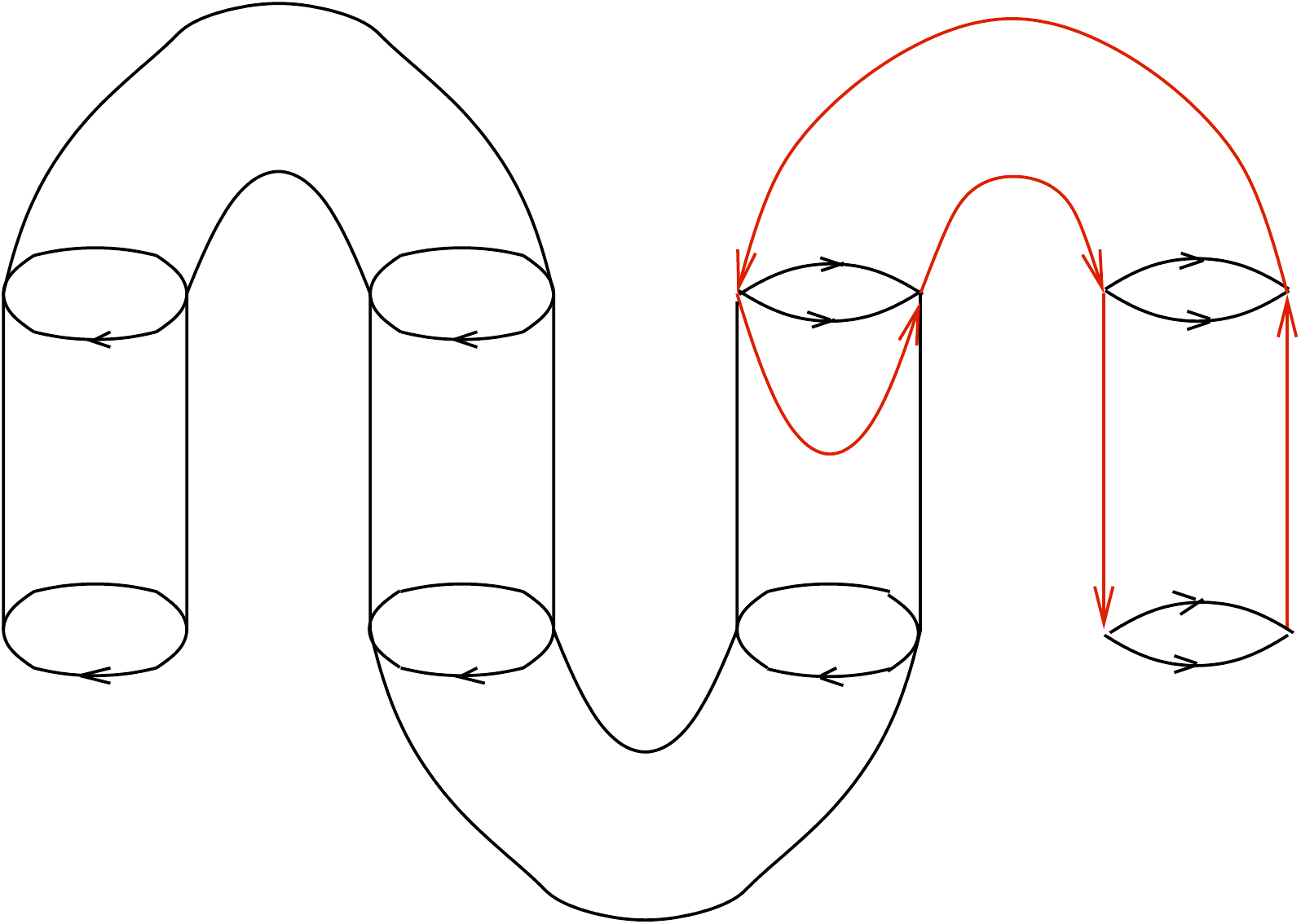}}\quad \stackrel {\cong}{(\ref{eq:cozipper_pairing})} \raisebox{-13pt}{\includegraphics[height=0.52in]{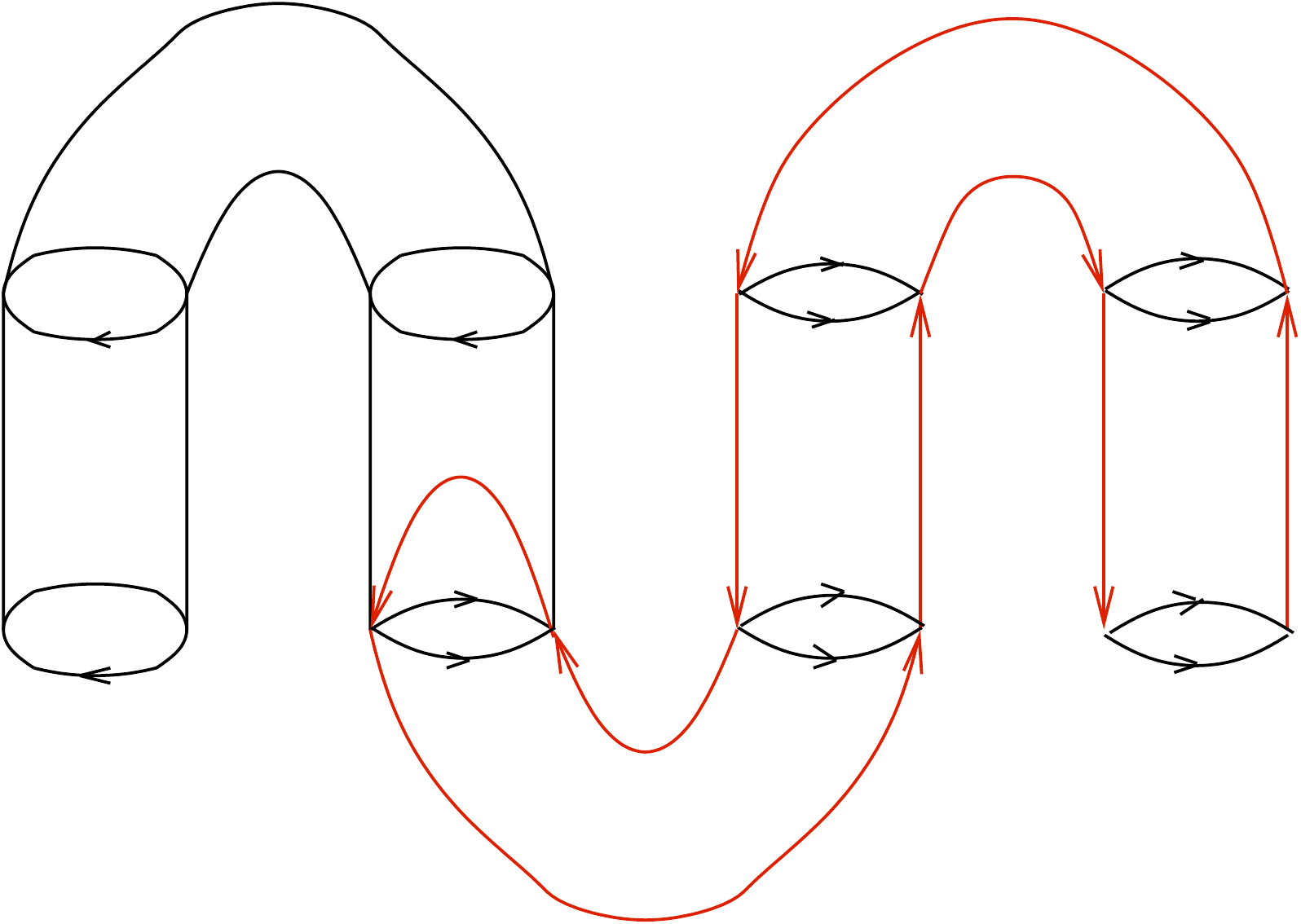}}\quad \stackrel {\cong}{(\ref{eq:sing_zig_zag})}\raisebox{-8pt}{\includegraphics[height=0.42in]{cozipper_copairing4.pdf}} \]
\end{proof}

\begin{proposition}
The following singular cobordisms are equivalent:
\begin{equation}
\raisebox{-10pt}{\includegraphics[height=0.27in]{singcomult.pdf}}\quad \cong \quad  \raisebox{-10pt}{\includegraphics[height=0.42in]{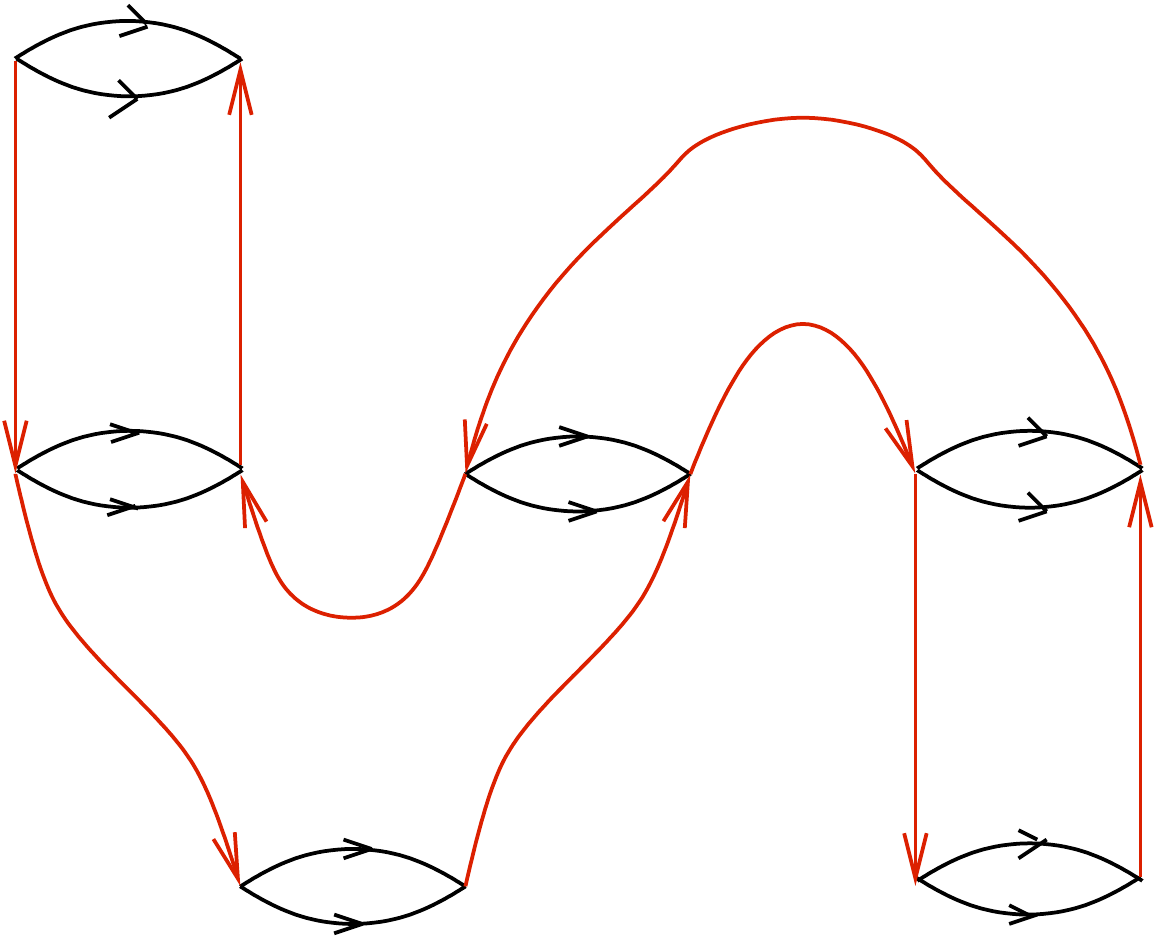}} \quad \cong \quad \raisebox{-10pt}{\includegraphics[height=0.42in]{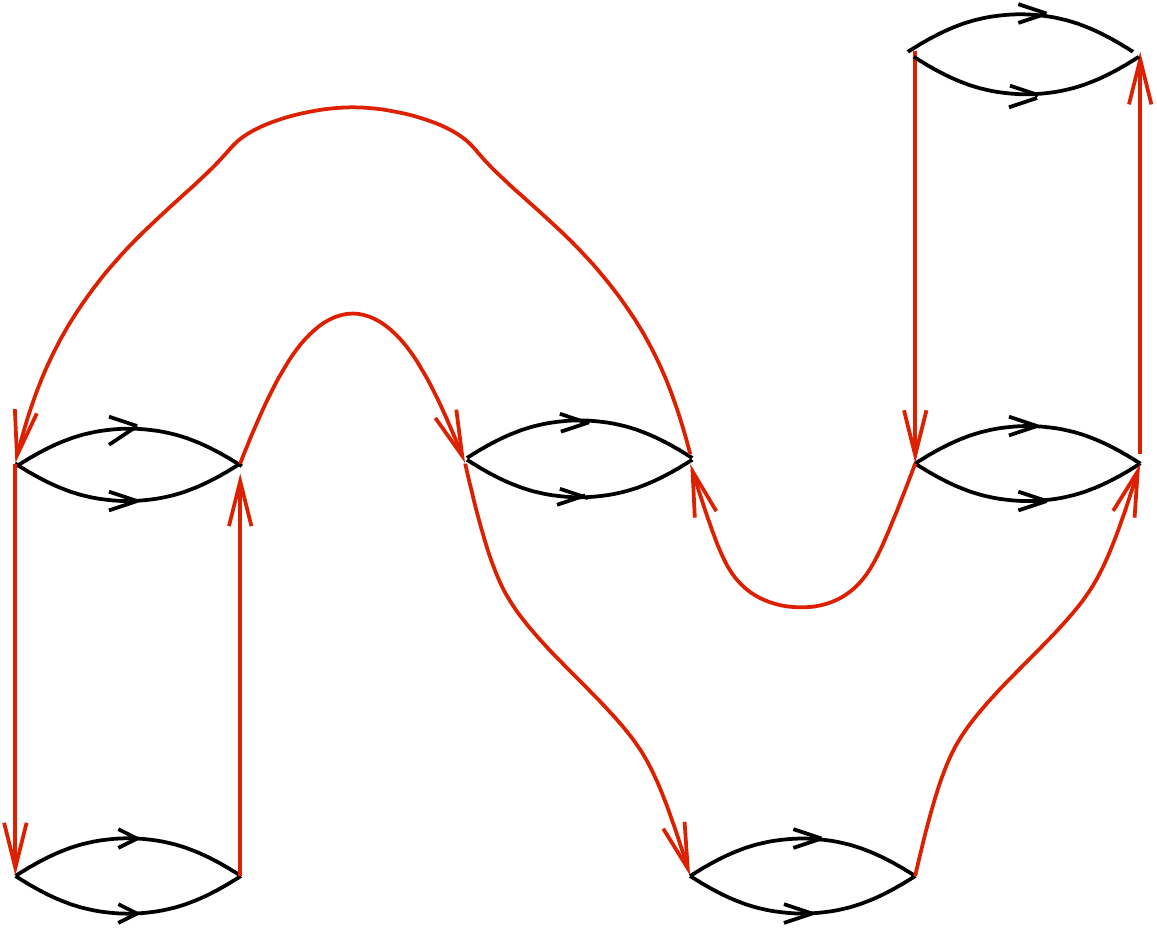}}\quad \cong \quad \raisebox{-13pt}{\includegraphics[height=0.65in]{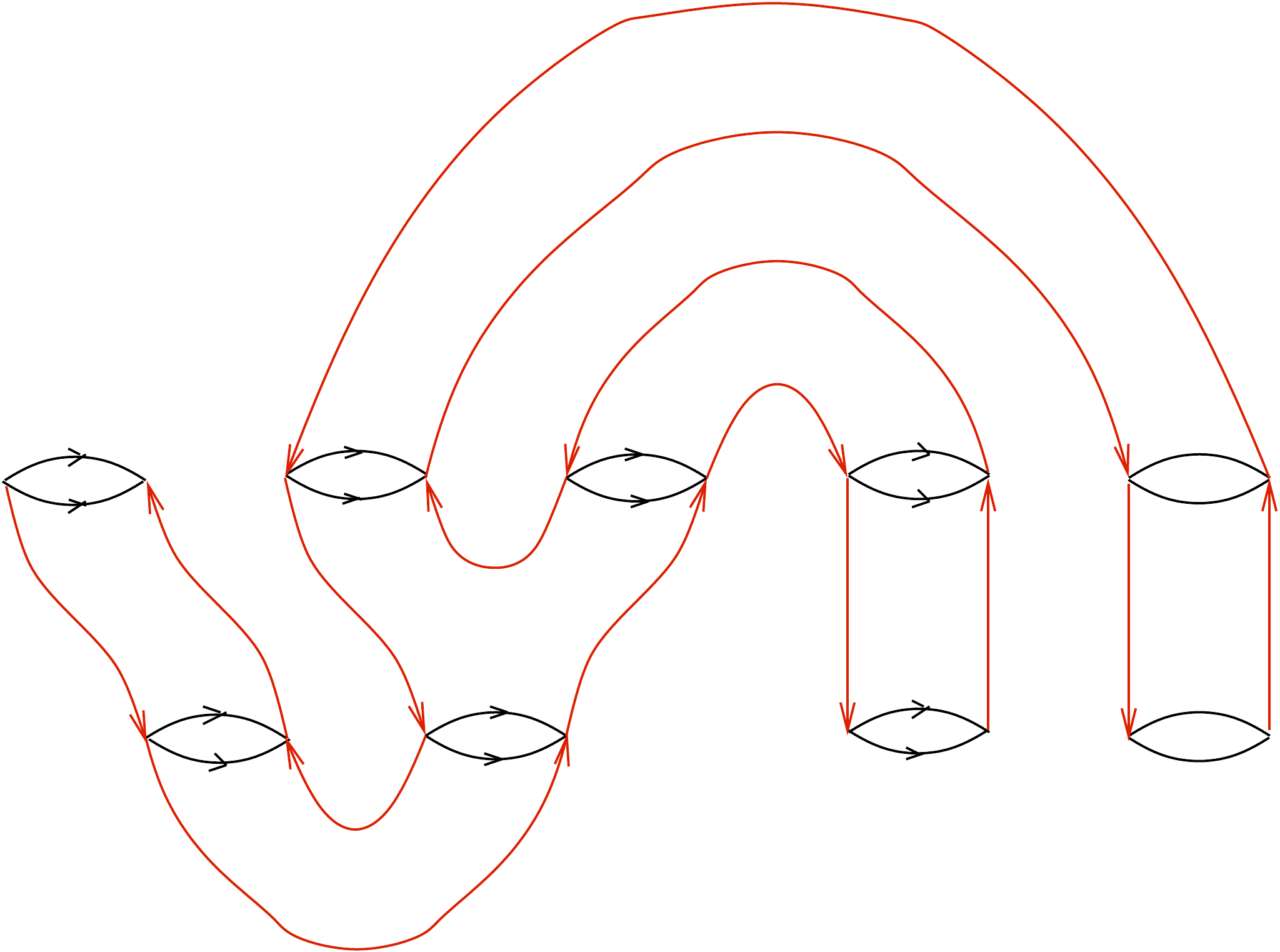}} \label{eq:singcomult_equiv}
\end{equation}
\begin{equation}
\raisebox{-10pt}{\includegraphics[height=0.28in]{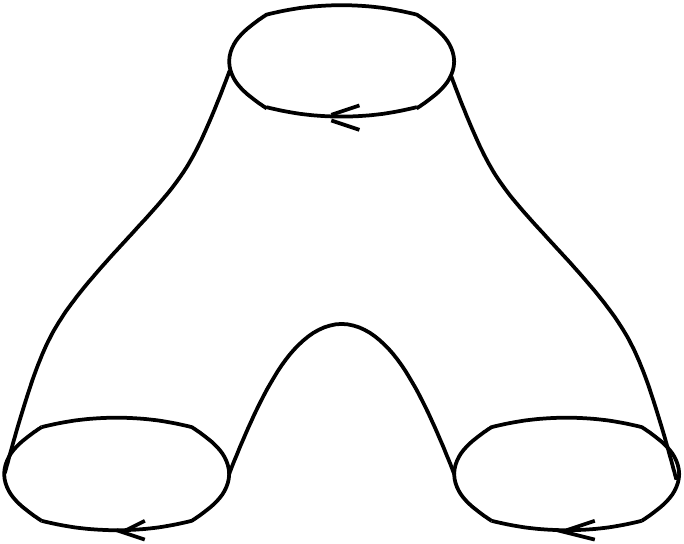}}\quad \cong \quad  \raisebox{-10pt}{\includegraphics[height=0.42in]{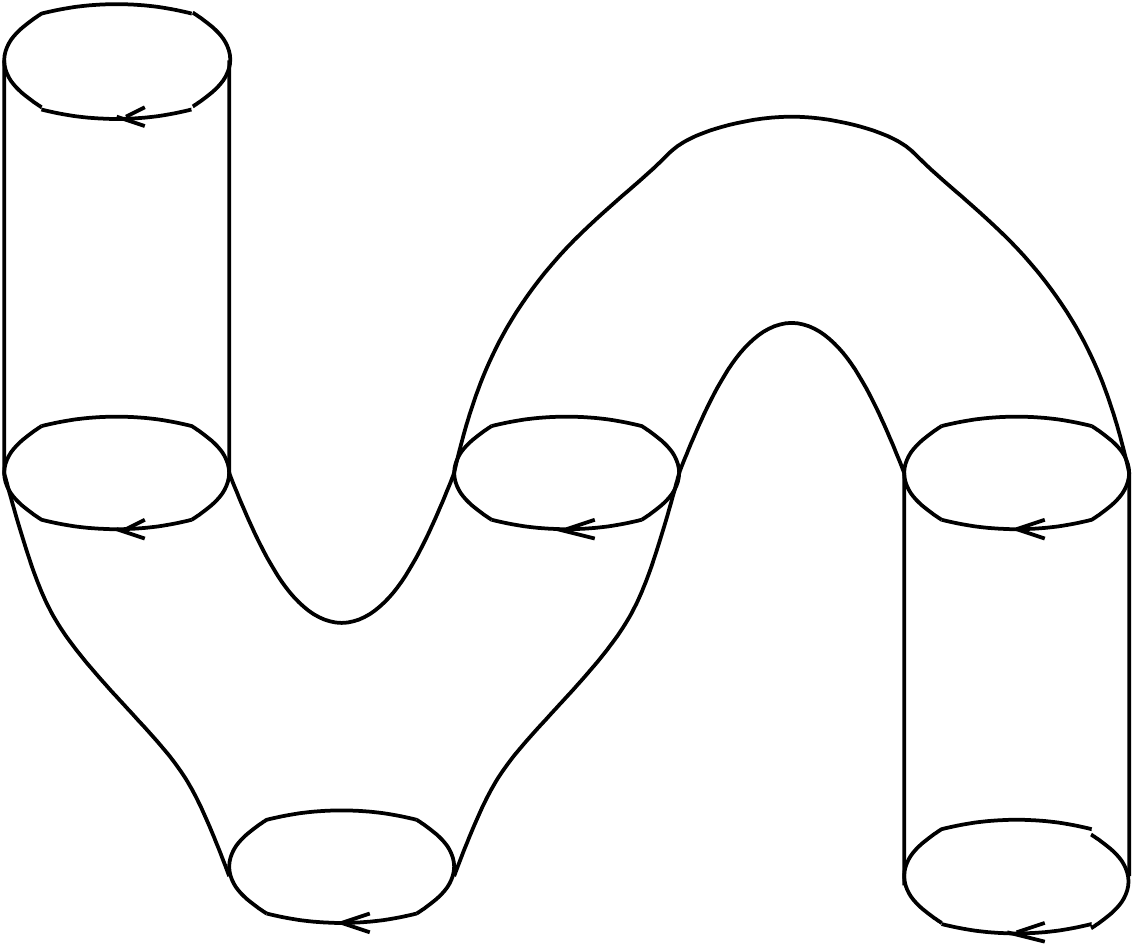}} \quad \cong \quad \raisebox{-10pt}{\includegraphics[height=0.42in]{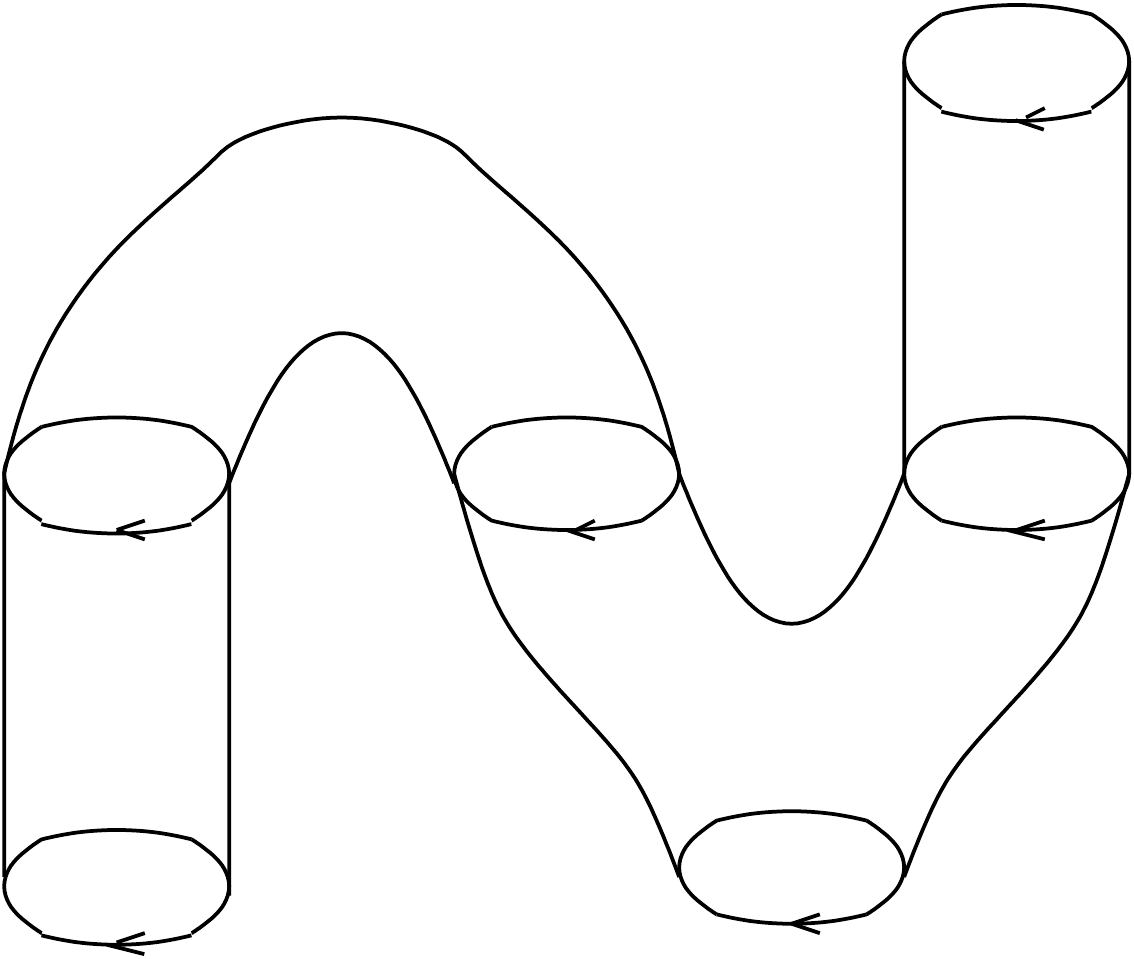}}\quad \cong \quad \raisebox{-13pt}{\includegraphics[height=0.65in]{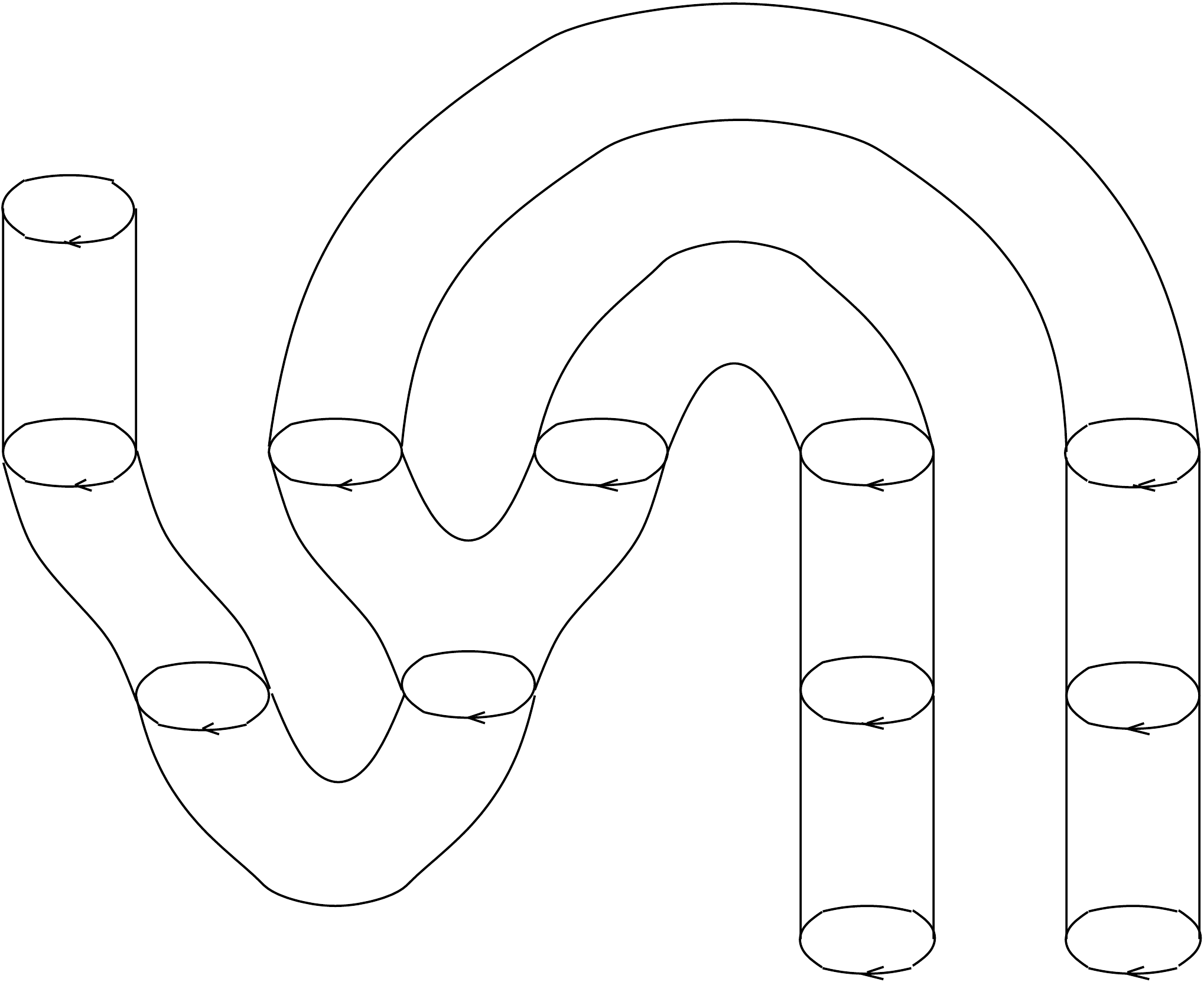}} \label{eq:comult_equiv}
\end{equation}
\end{proposition}

\begin{proof}
The  diffeomorphisms given in Equation~(\ref{eq:singcomult_equiv}) follow from the following sequences of diffeomorphisms:
\begin{equation*}
\raisebox{-5pt}{\includegraphics[height=0.27in]{singcomult.pdf}}\quad \stackrel{\cong}{(\ref{eq:web_frob1})} \quad  \raisebox{-10pt}{\includegraphics[height=0.52in]{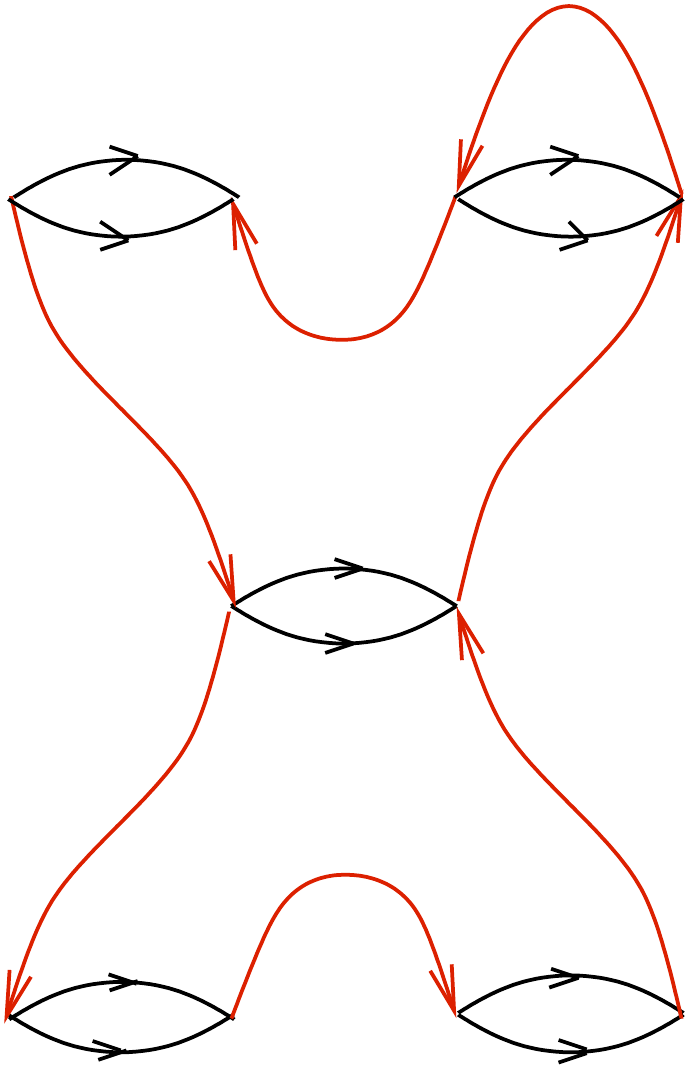}} \quad \stackrel{\cong}{(~\ref{eq:web_frob3})} \quad \raisebox{-8pt}{\includegraphics[height=0.48in]{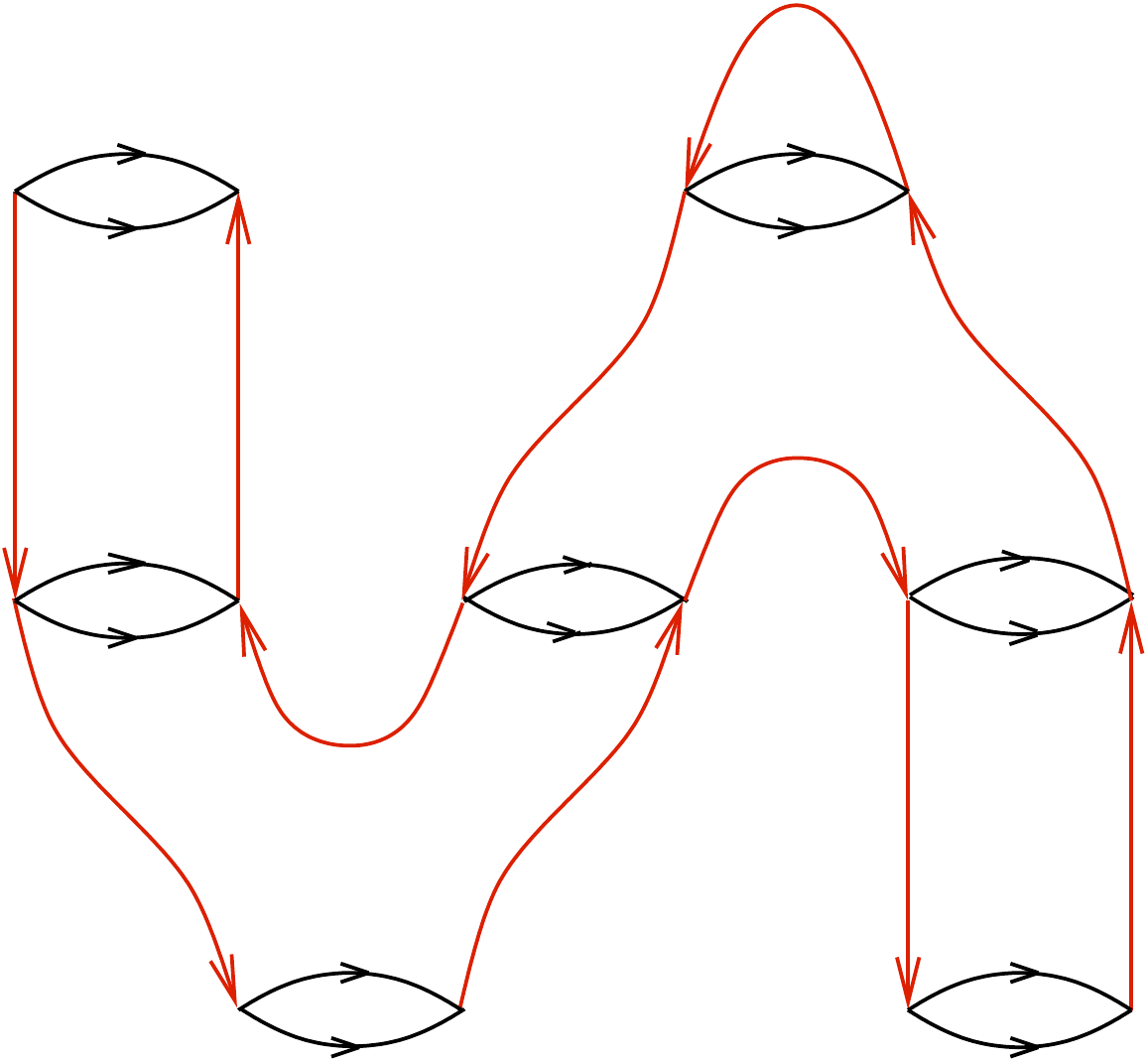}}\quad \cong \quad  \raisebox{-5pt}{\includegraphics[height=0.42in]{singcomult1.pdf}}
\end{equation*}
\begin{equation*}
\raisebox{-5pt}{\includegraphics[height=0.27in]{singcomult.pdf}}\quad \stackrel{\cong}{(\ref{eq:web_frob1})} \quad  \raisebox{-10pt}{\includegraphics[height=0.52in]{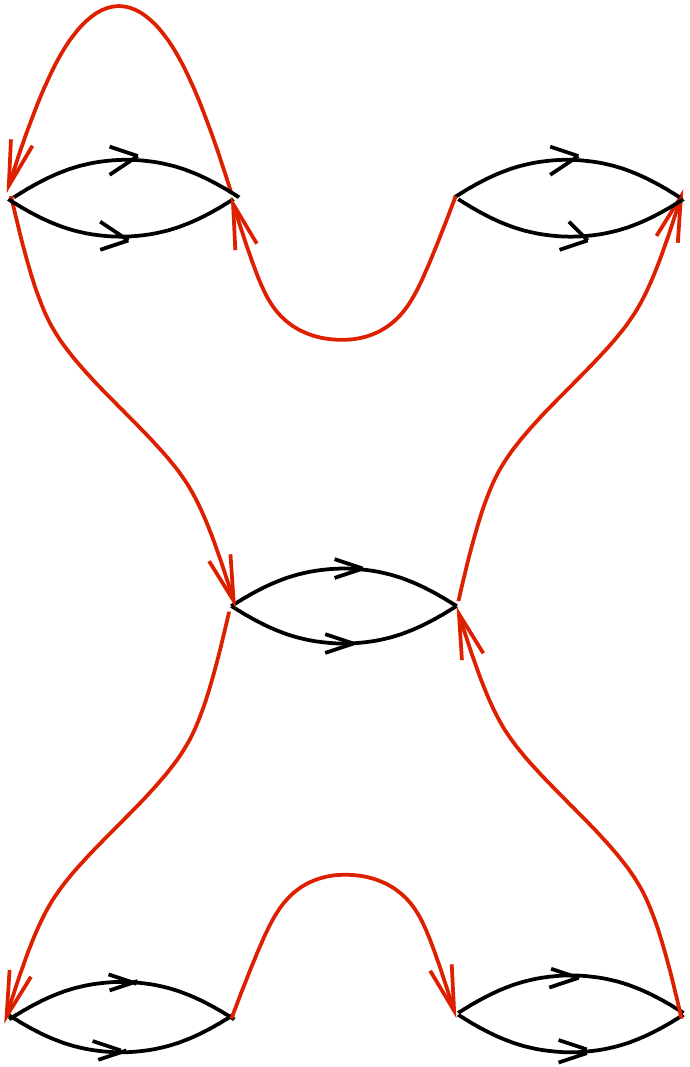}} \quad \stackrel{\cong}{(~\ref{eq:web_frob3})} \quad \raisebox{-8pt}{\includegraphics[height=0.48in]{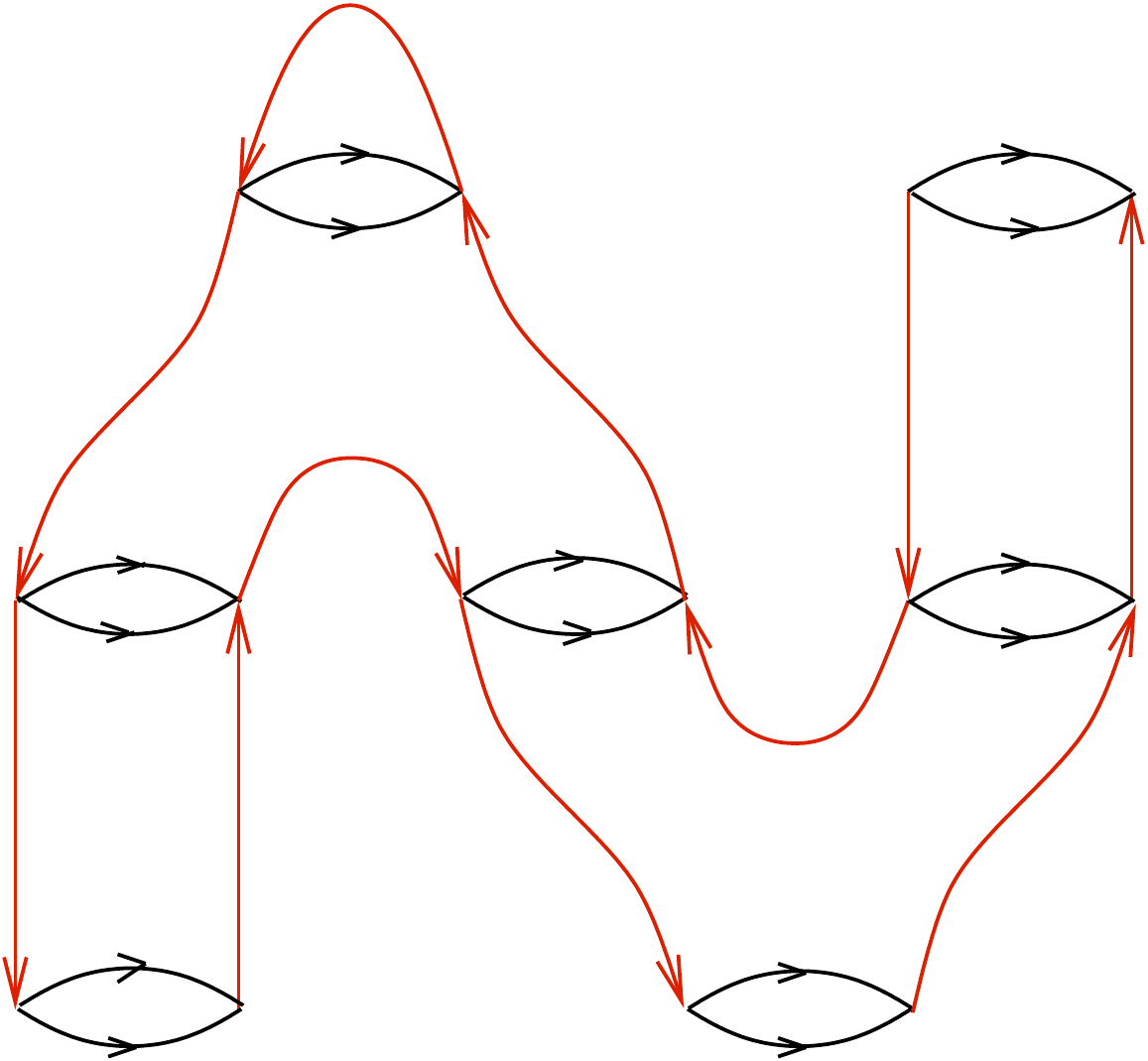}}\quad \cong \quad  \raisebox{-5pt}{\includegraphics[height=0.42in]{singcomult2.pdf}}
\end{equation*}
\begin{equation*}
\raisebox{-5pt}{\includegraphics[height=0.42in]{singcomult1.pdf}}\quad \stackrel{\cong}{(\ref{eq:sing_zig_zag})} \quad  \raisebox{-10pt}{\includegraphics[height=0.65in]{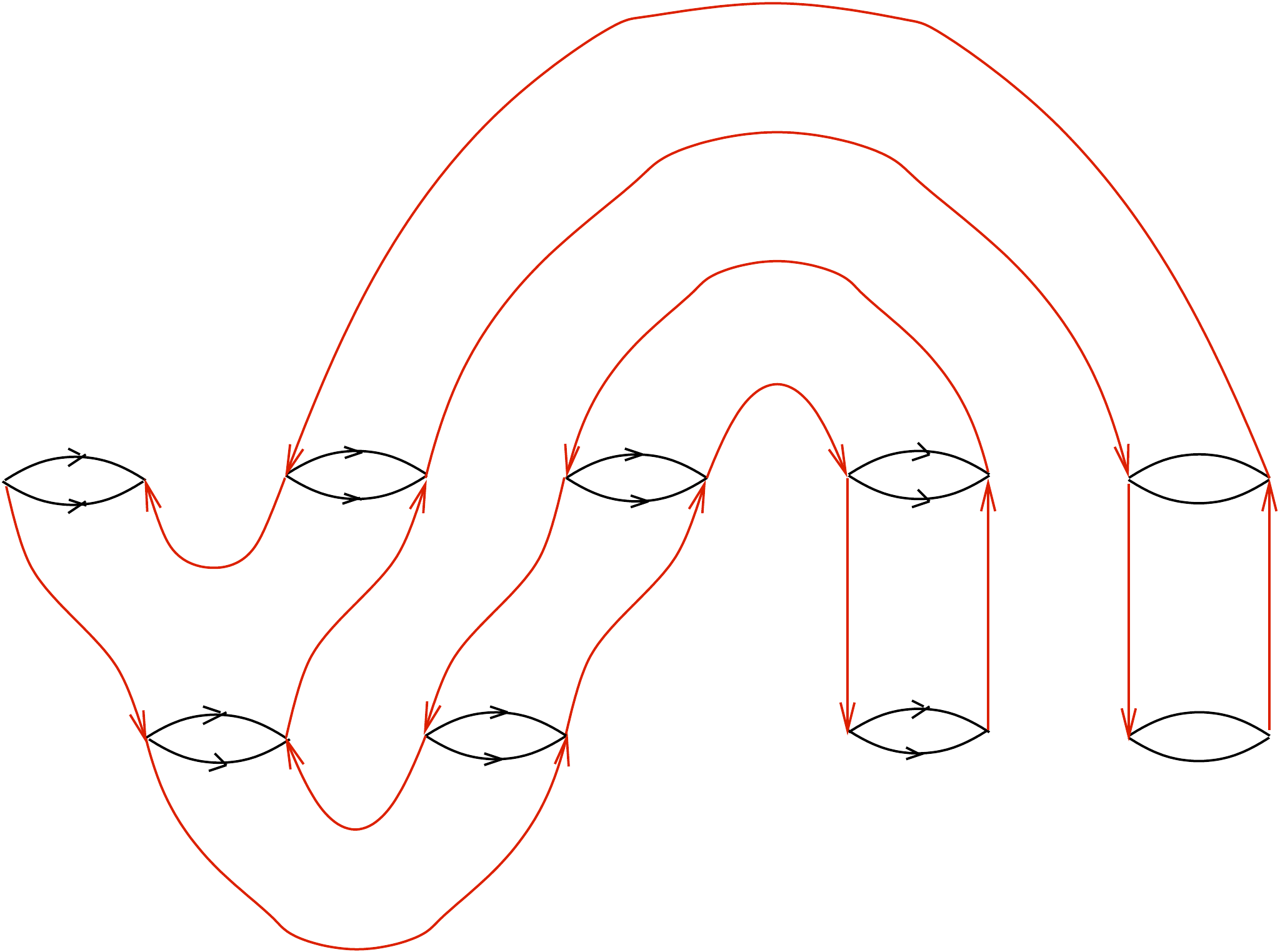}} \quad \stackrel{\cong}{(~\ref{eq:singpairing_inv_sym})} \quad \raisebox{-8pt}{\includegraphics[height=0.65in]{singcomult3.pdf}}
\end{equation*}

The proof  of Equation (\ref{eq:comult_equiv}) is done similarly by replacing the bi-web with an oriented circle, thus replacing the singular cobordisms above with their ordinary cobordisms counterparts.
\end{proof}

\begin{proposition}
The following singular cobordisms are equivalent:
\begin{equation}
\raisebox{-8pt}{\includegraphics[height=0.27in]{singmult.pdf}}\quad \cong \quad  \raisebox{-18pt}{\includegraphics[height=0.45in]{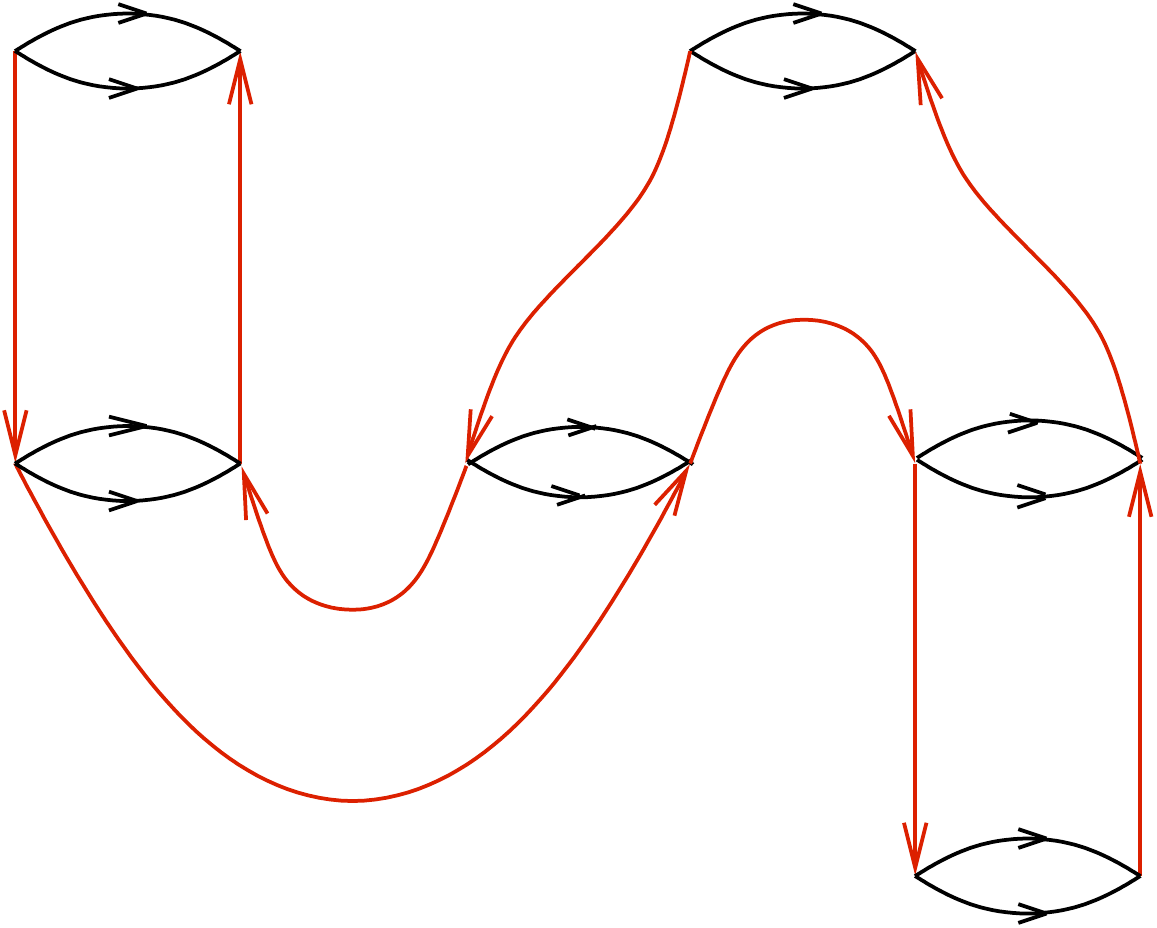}} \quad \cong \quad \raisebox{-18pt}{\includegraphics[height=0.45in]{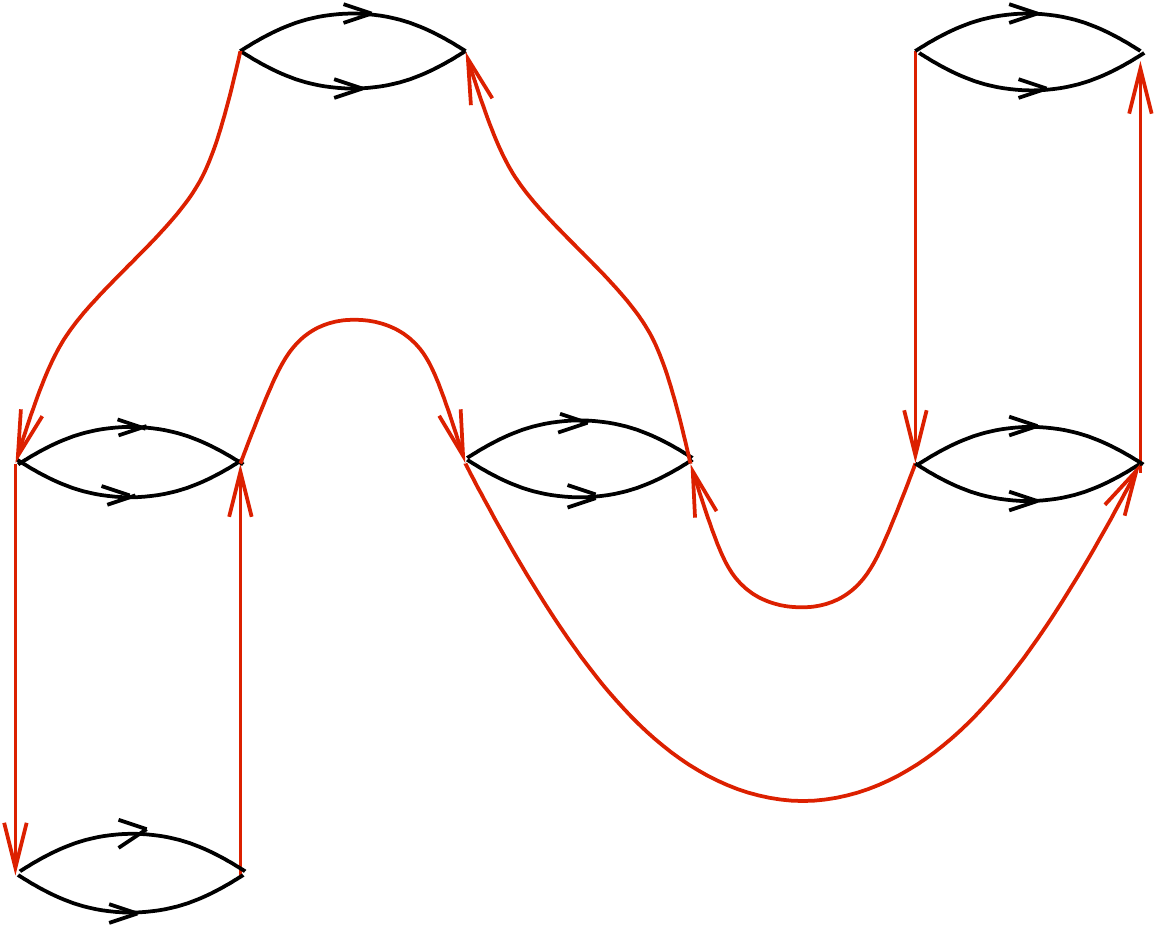}} \label{eq:singmult_equiv}
\end{equation}
\begin{equation}
\raisebox{-8pt}{\includegraphics[height=0.27in]{mult.pdf}}\quad \cong \quad  \raisebox{-18pt}{\includegraphics[height=0.45in]{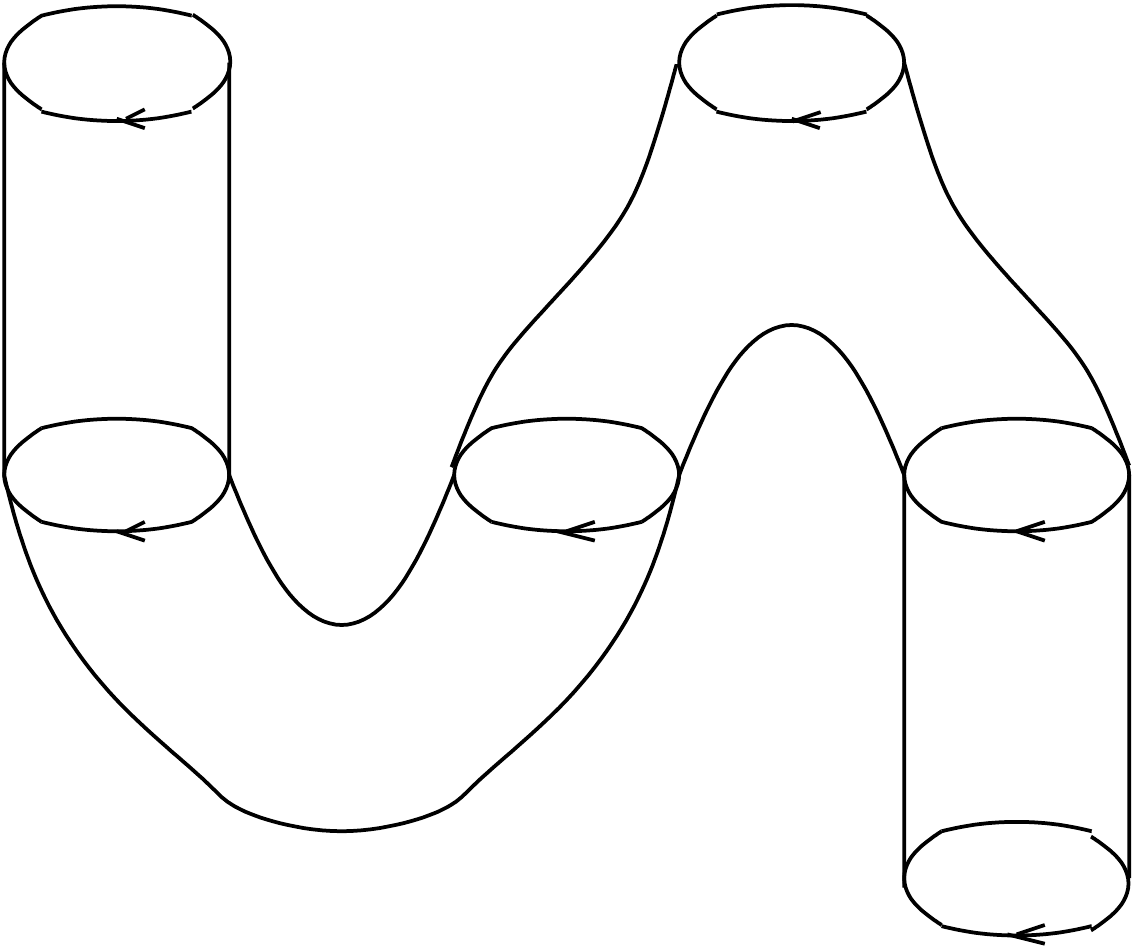}} \quad \cong \quad \raisebox{-18pt}{\includegraphics[height=0.45in]{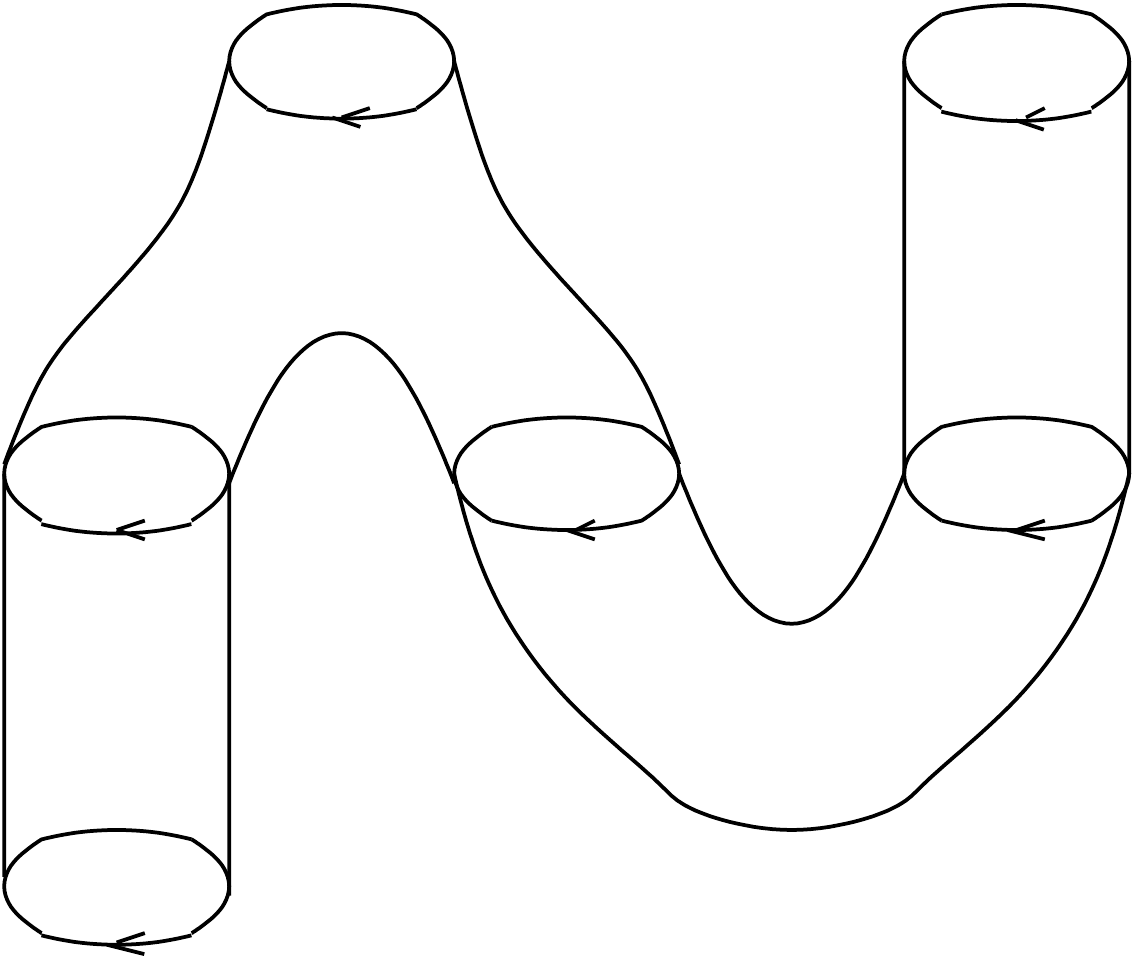}} \label{eq:mult_equiv}
\end{equation}
\end{proposition}
\begin{proof}
The proof of these equivalences is done in a similar manner as in the previous proposition.
\end{proof}

\begin{proposition}
The following singular cobordisms are equivalent:
\begin{equation}
\raisebox{-13pt}{\includegraphics[height=0.6in]{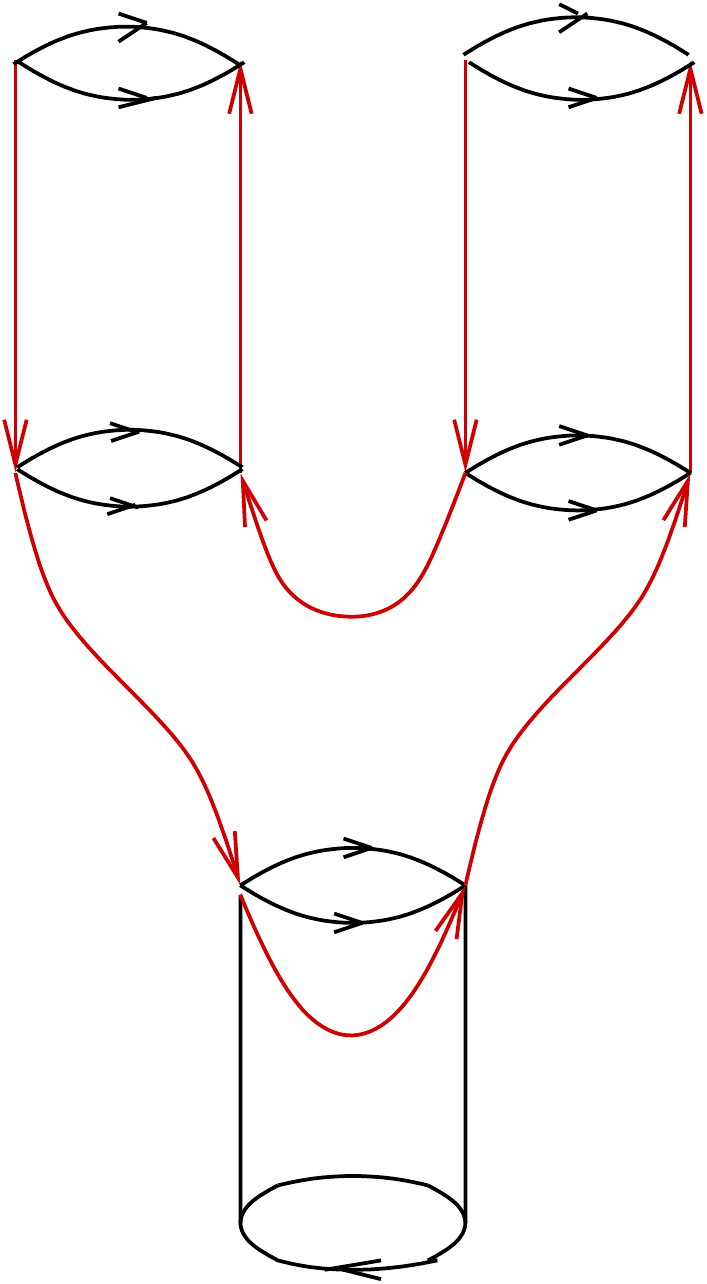}} \quad \cong \quad \raisebox{-13pt}{\includegraphics[height=0.6in]{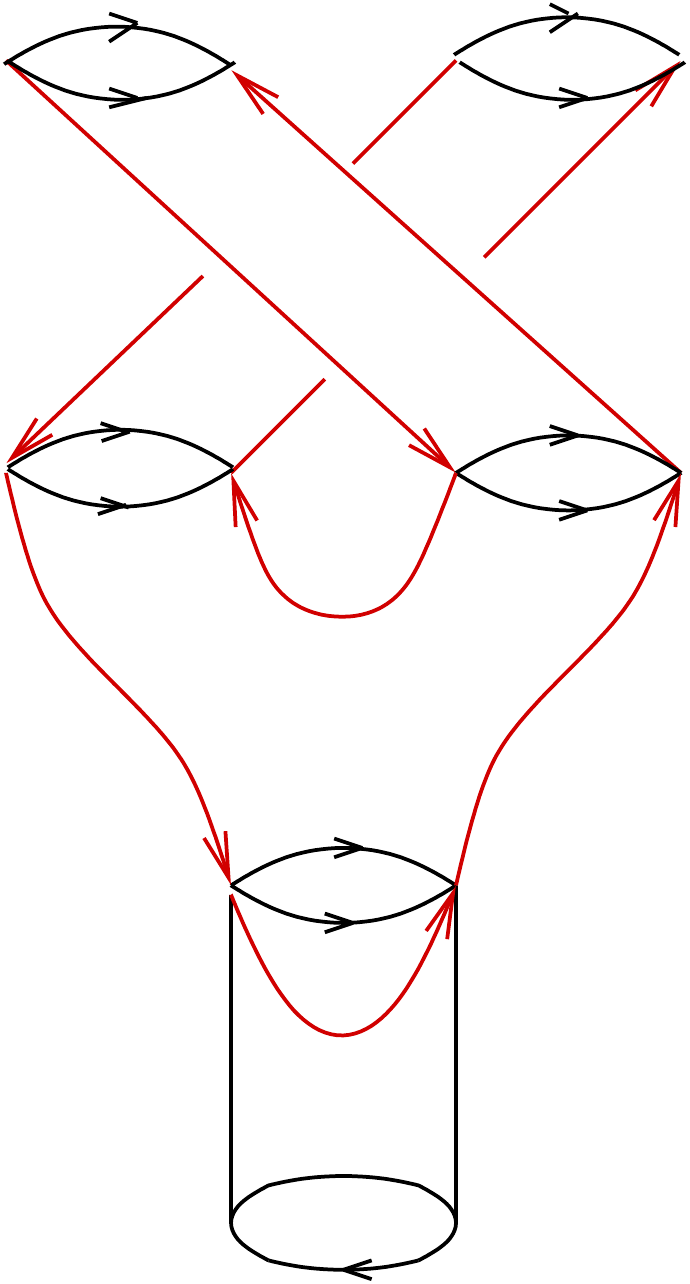}} \hspace{2cm} \raisebox{-13pt}{\includegraphics[height=0.6in]{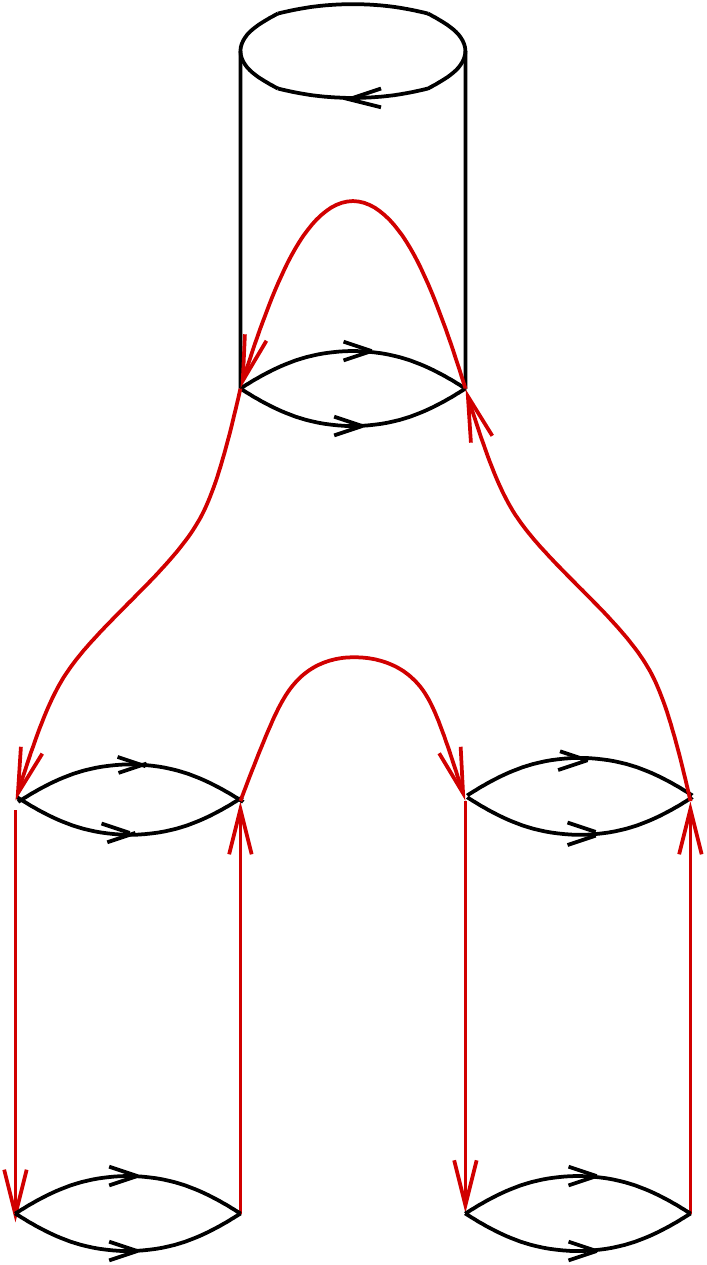}} \quad \cong \quad \raisebox{-13pt}{\includegraphics[height=0.6in]{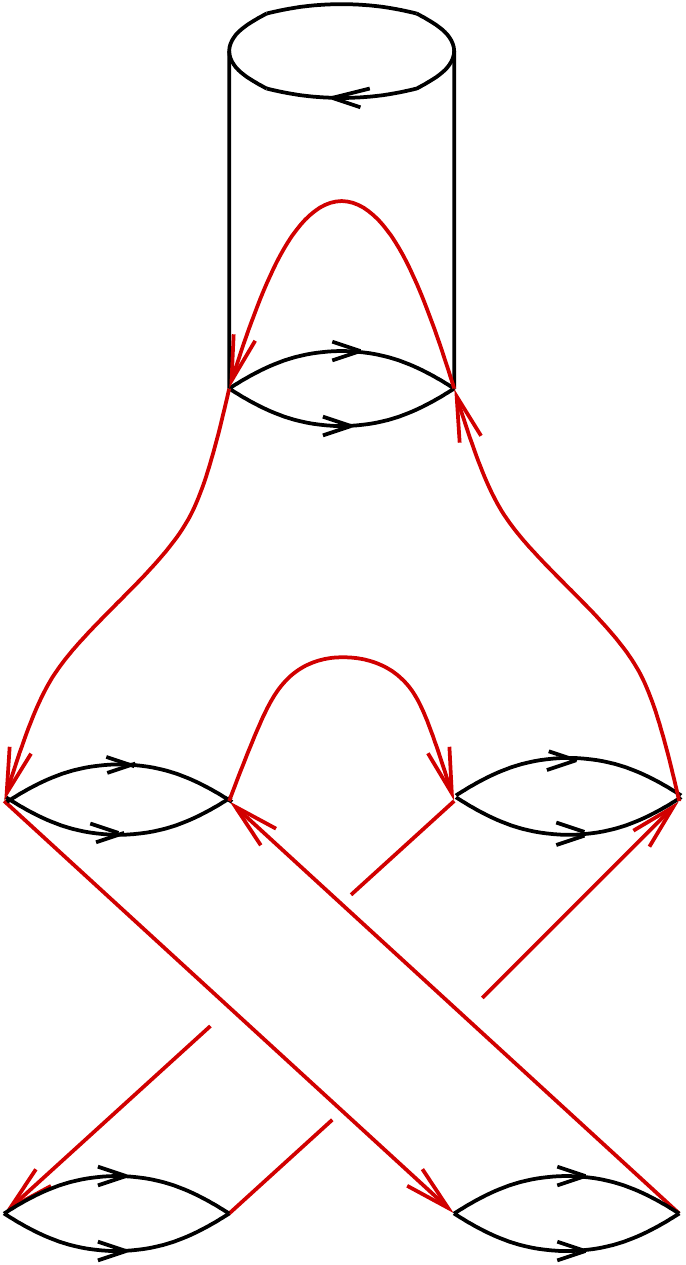}}\label{eq:weak_com}
\end{equation}

\begin{equation}
\raisebox{-15pt}{\includegraphics[height=0.6in]{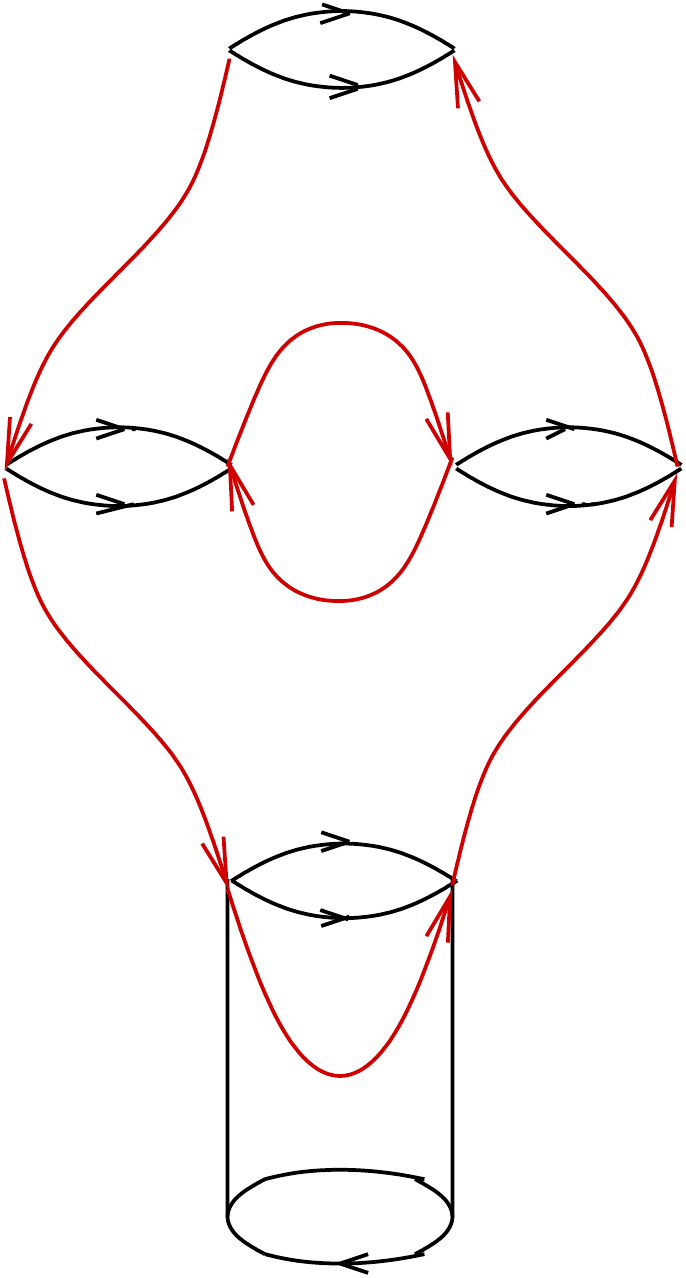}} \quad \cong \quad \raisebox{-15pt}{\includegraphics[width = 0.3in, height=0.75in]{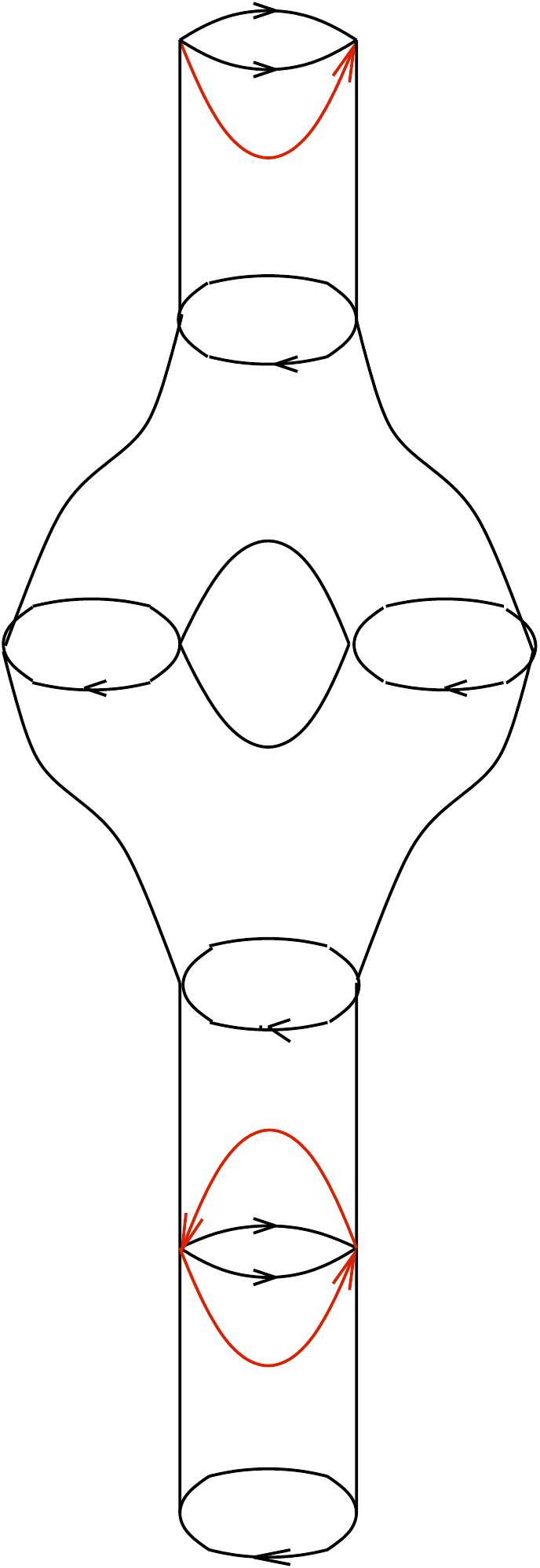}}   
 \hspace{2cm} \raisebox{-15pt}{\includegraphics[height=0.6in]{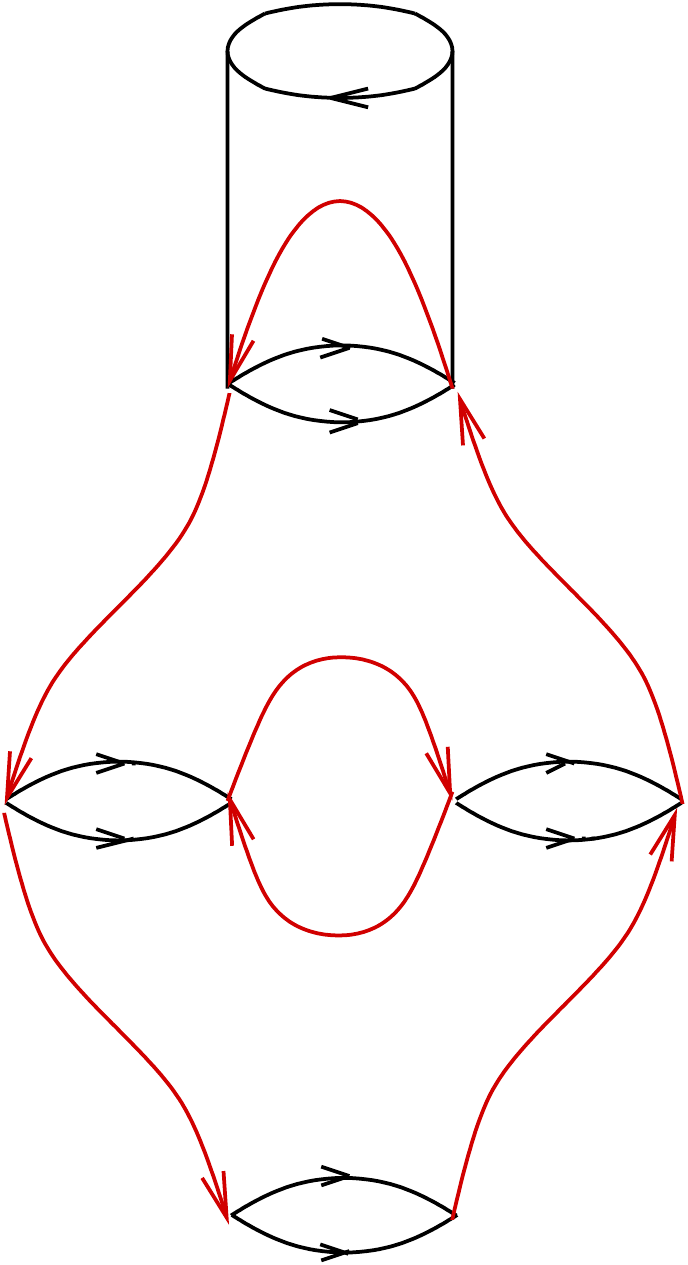}} \quad \cong \quad \raisebox{-15pt}{\includegraphics[width = 0.3in, height=0.75in]{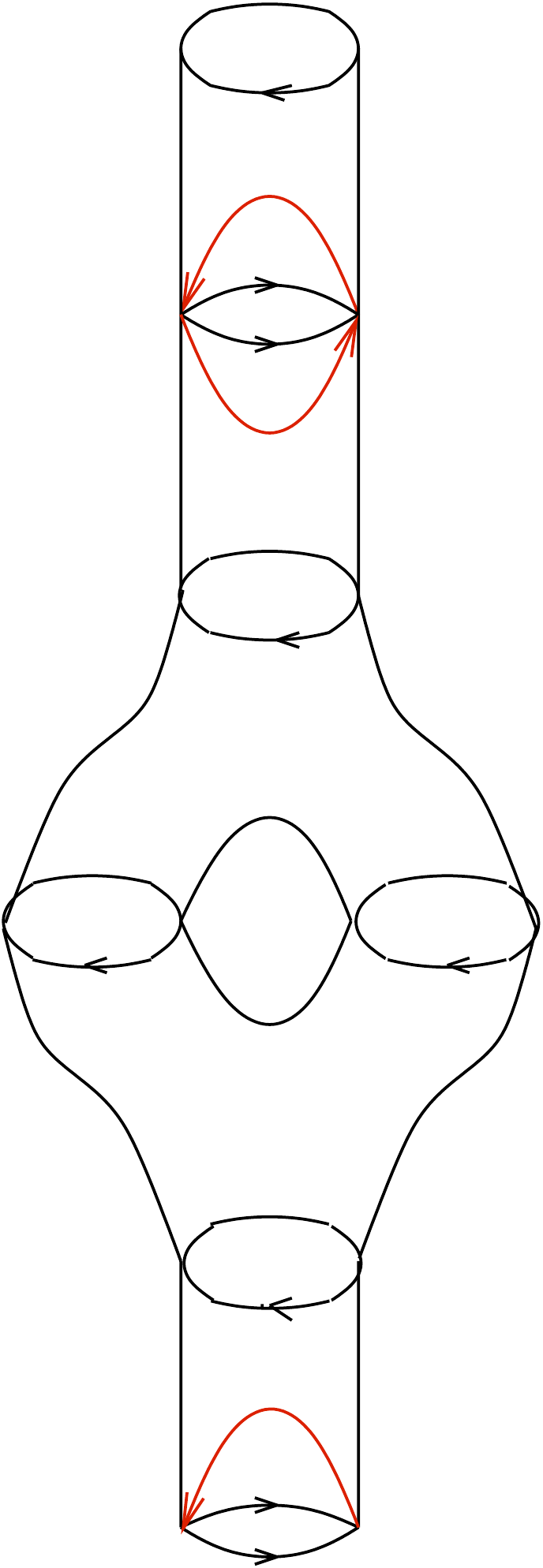}} \label{eq:remove_sing_genusop}
\end{equation}
\end{proposition}

\begin{proof}
The first diffeomorphism in Equation~\eqref{eq:weak_com} is obtained from the following sequence of diffeomorphisms:

\[
\raisebox{-13pt}{\includegraphics[height=0.6in]{web_frob15.pdf}} \stackrel{\cong}{\eqref{eq:zig_zag}} 
\psset{xunit=.22cm,yunit=.22cm}
\begin{pspicture}(7,5)
 \rput(1, 4){\includegraphics[height=0.55in]{web_frob15.pdf}}
 \rput(1,0.1){\includegraphics[height=0.23in]{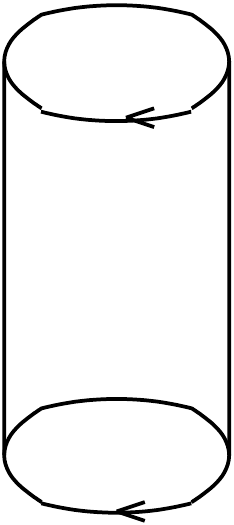}}
  \rput(2.1,-1.6){\includegraphics[height=0.18in]{pairing1.pdf}}
  \rput(3.2,0.1){\includegraphics[height=0.23in]{id_circle.pdf}}
  \rput(4.2,1.8){\includegraphics[height=0.18in]{copairing1.pdf}}
  \rput(5.3,0.1){\includegraphics[height=0.23in]{id_circle.pdf}}
 \end{pspicture}
 \stackrel{\cong}{\eqref{eq:cozipper_pairing}} 
 \psset{xunit=.22cm,yunit=.22cm}
\begin{pspicture}(7,5)
 \rput(1, 5){\includegraphics[height=0.23in]{singmult.pdf}}
  \rput(1,2.9){\includegraphics[height=0.23in]{identity_web.pdf}}
  \rput(2.1,0){\includegraphics[height=0.36in]{cozipper_pairing4.pdf}}
  \rput(4.3,2.5){\includegraphics[height=0.18in]{copairing1.pdf}}
  \rput(5.4,0.8){\includegraphics[height=0.23in]{id_circle.pdf}}
\end{pspicture}
\stackrel{\cong}{}
\psset{xunit=.22cm,yunit=.22cm}
\begin{pspicture}(8,5)
  \rput(3,1){\includegraphics[height=0.36in]{sing_pairing_inv2.pdf}}
  \rput(5.2,3.7){\includegraphics[height=0.2in]{zipper.pdf}}
  \rput(6.3,5.2){\includegraphics[height=0.18in]{copairing1.pdf}}
  \rput(7.4,3.5){\includegraphics[height=0.23in]{id_circle.pdf}}
 \end{pspicture}
 \stackrel{\cong}{\eqref{eq:singpairing_inv_sym}}
\psset{xunit=.22cm,yunit=.22cm}
\begin{pspicture}(8,5)
  \rput(3,1){\includegraphics[height=0.36in]{sing_pairing_inv1.pdf}}
  \rput(5.2,3.7){\includegraphics[height=0.2in]{zipper.pdf}}
  \rput(6.3,5.2){\includegraphics[height=0.18in]{copairing1.pdf}}
  \rput(7.4,3.5){\includegraphics[height=0.23in]{id_circle.pdf}}
 \end{pspicture}
  \]
  \vspace{0.15cm}
\[
\stackrel{\cong}{\eqref{eq:singpairing_inv_sym}}
 \psset{xunit=.22cm,yunit=.22cm}
\begin{pspicture}(8,8)
 \rput(2.9,-0.2){\includegraphics[height=0.18in]{sing_pairing1.pdf}}
 \rput(2.9,1.6){\includegraphics[height=0.23in]{braiding_WW.pdf}}
  \rput(2.9,3.7){\includegraphics[height=0.23in]{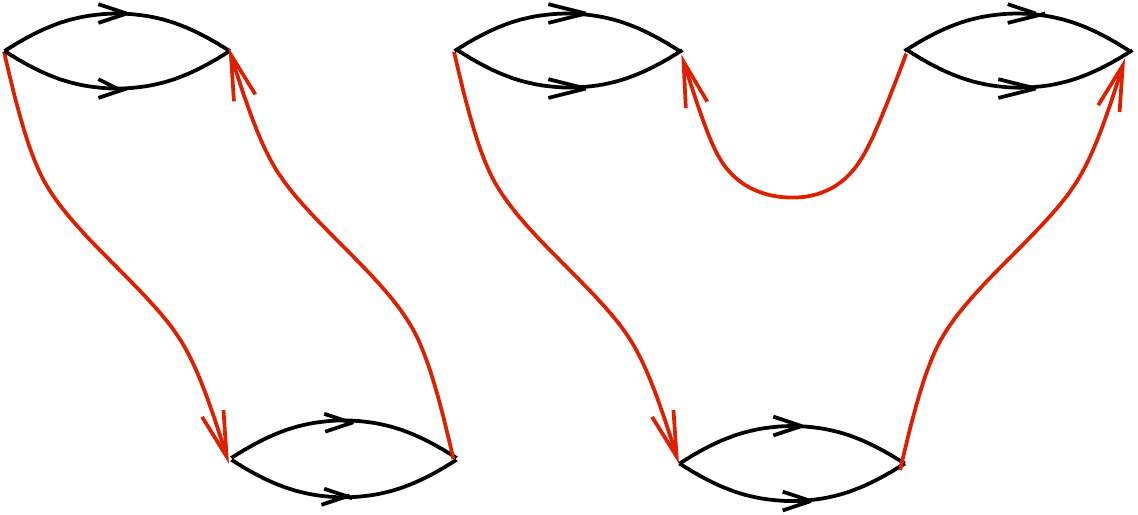}}
  \rput(5.2,5.7){\includegraphics[height=0.2in]{zipper.pdf}}
  \rput(6.3,7.2){\includegraphics[height=0.18in]{copairing1.pdf}}
  \rput(7.4,5.5){\includegraphics[height=0.23in]{id_circle.pdf}}
 \end{pspicture}
 \stackrel{\cong}{Nat}
  \psset{xunit=.22cm,yunit=.22cm}
\begin{pspicture}(8,8)
  \rput(3,2){\includegraphics[height=0.55in]{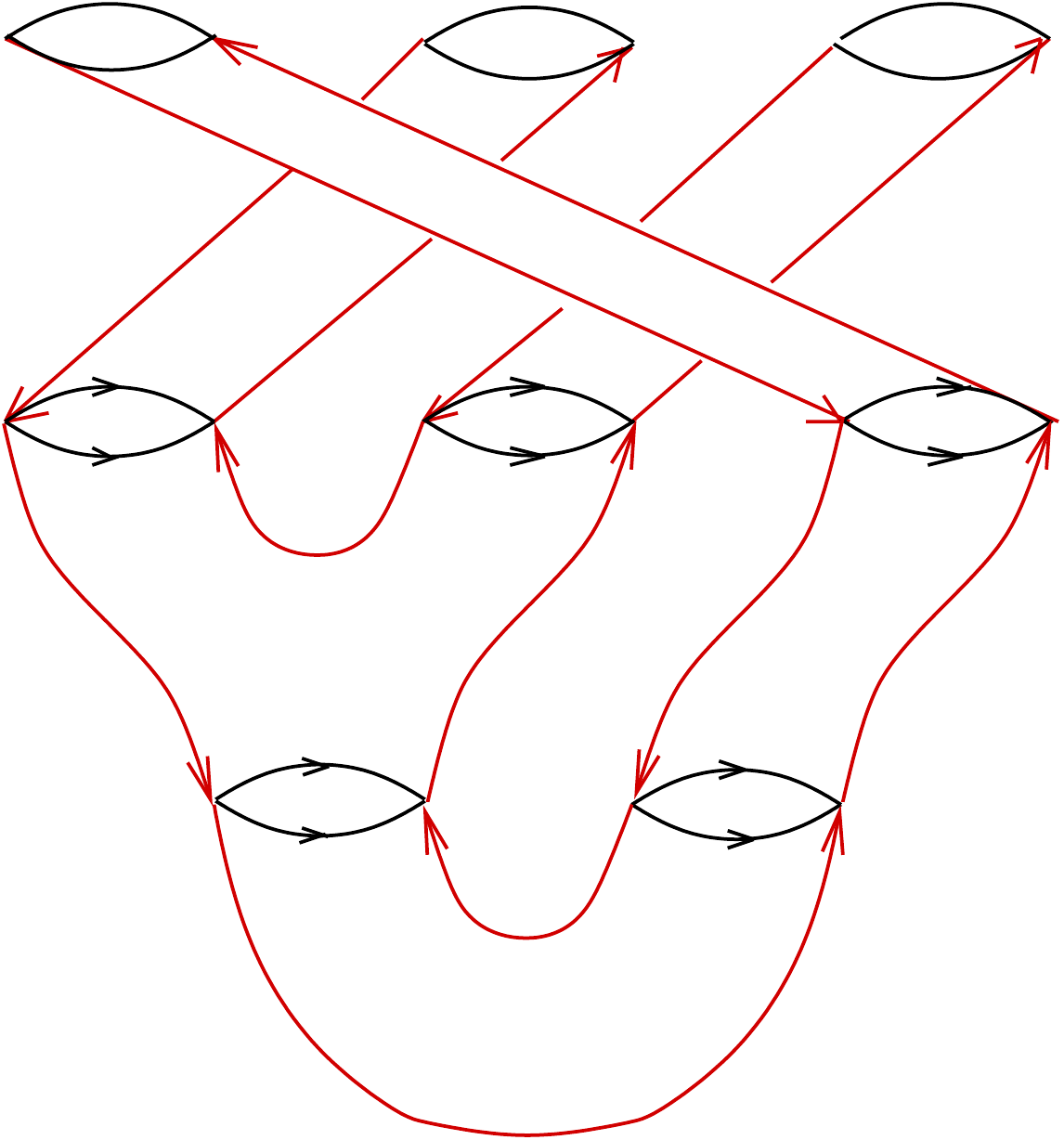}}
    \rput(5.2,5.8){\includegraphics[height=0.2in]{zipper.pdf}}
    \rput(6.3,7.3){\includegraphics[height=0.18in]{copairing1.pdf}}
  \rput(7.4,5.6){\includegraphics[height=0.23in]{id_circle.pdf}}
 \end{pspicture}
 \stackrel{\cong}{\eqref{eq:singpairing_inv_sym}}
   \psset{xunit=.22cm,yunit=.22cm}
\begin{pspicture}(8,8)
  \rput(3,2){\includegraphics[height=0.55in]{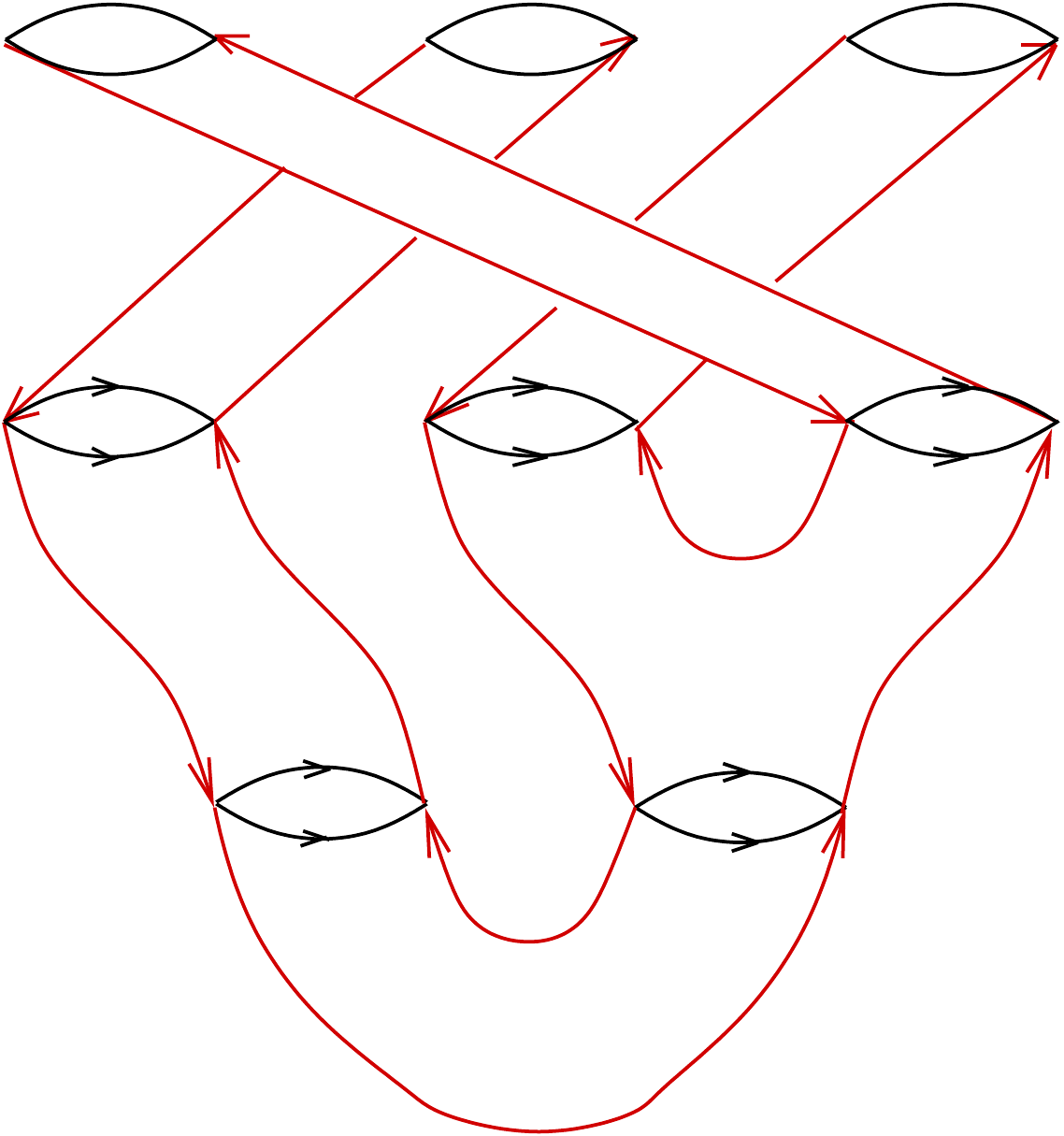}}
    \rput(5.2,5.8){\includegraphics[height=0.2in]{zipper.pdf}}
    \rput(6.3,7.3){\includegraphics[height=0.18in]{copairing1.pdf}}
  \rput(7.4,5.6){\includegraphics[height=0.23in]{id_circle.pdf}}
 \end{pspicture}
  \stackrel{\cong}{\eqref{eq:center}}
   \psset{xunit=.22cm,yunit=.22cm}
\begin{pspicture}(8,8)
 \rput(2.9,0.5){\includegraphics[height=0.36in]{sing_pairing_inv1.pdf}}
   \rput(5.2,3.2){\includegraphics[height=0.2in]{zipper.pdf}}
   \rput(5.2,5){\includegraphics[height=0.23in]{id_circle.pdf}}
  \rput(6.3,6.7){\includegraphics[height=0.18in]{copairing1.pdf}}
  \rput(7.4,5){\includegraphics[height=0.23in]{id_circle.pdf}}
  \rput(2.9,3.4){\includegraphics[height=0.23in]{identity_web.pdf}}
   \rput(0.6,3.4){\includegraphics[height=0.23in]{identity_web.pdf}}
     \rput(1.7,5.5){\includegraphics[height=0.23in]{braiding_WW.pdf}}
 \end{pspicture}
 \]
 \[
 \stackrel{\cong}{\eqref{eq:singpairing_inv_sym}}
 \psset{xunit=.22cm,yunit=.22cm}
\begin{pspicture}(8,8)
 \rput(2.9,0.5){\includegraphics[height=0.36in]{sing_pairing_inv2.pdf}}
   \rput(5.2,3.2){\includegraphics[height=0.2in]{zipper.pdf}}
   \rput(5.2,5){\includegraphics[height=0.23in]{id_circle.pdf}}
  \rput(6.3,6.7){\includegraphics[height=0.18in]{copairing1.pdf}}
  \rput(7.4,5){\includegraphics[height=0.23in]{id_circle.pdf}}
  \rput(2.9,3.4){\includegraphics[height=0.23in]{identity_web.pdf}}
   \rput(0.6,3.4){\includegraphics[height=0.23in]{identity_web.pdf}}
     \rput(1.7,5.5){\includegraphics[height=0.23in]{braiding_WW.pdf}}
\end{pspicture}
 \cong
  \psset{xunit=.22cm,yunit=.22cm}
\begin{pspicture}(8,8)
  \rput(1.5,5.3){\includegraphics[height=0.23in]{braiding_WW.pdf}}
 \rput(1.5, 3.2){\includegraphics[height=0.23in]{singmult.pdf}}
  \rput(2.6,0.3){\includegraphics[height=0.36in]{cozipper_pairing4.pdf}}
  \rput(4.8,2.8){\includegraphics[height=0.18in]{copairing1.pdf}}
  \rput(5.9,1.1){\includegraphics[height=0.23in]{id_circle.pdf}}
\end{pspicture}
 \stackrel{\cong}{\eqref{eq:cozipper_pairing}} 
   \psset{xunit=.22cm,yunit=.22cm}
\begin{pspicture}(8,8)
  \rput(1.5,5.3){\includegraphics[height=0.23in]{braiding_WW.pdf}}
 \rput(1.5, 3.2){\includegraphics[height=0.23in]{singmult.pdf}}
  \rput(1.5,1.3){\includegraphics[height=0.2in]{cozipper.pdf}}
  \rput(4.7,3.2){\includegraphics[height=0.18in]{copairing1.pdf}}
  \rput(5.8,1.5){\includegraphics[height=0.23in]{id_circle.pdf}}
    \rput(3.7,1.5){\includegraphics[height=0.23in]{id_circle.pdf}}
  \rput(2.6,-0.2){\includegraphics[height=0.18in]{pairing1.pdf}}
\end{pspicture}
\stackrel{\cong}{\eqref{eq:zig_zag}}  \raisebox{-7pt}{\includegraphics[height=0.6in]{web_frob16.pdf}} 
\]

The first diffeomrphism depicted in Equation~\eqref{eq:remove_sing_genusop} follows from the following sequence of diffeomorphisms:
\[
\raisebox{-15pt}{\includegraphics[height=0.6in]{remove_sgenusop1.pdf}} \stackrel {\cong}{\eqref{eq:weak_com}} 
  \psset{xunit=.22cm,yunit=.22cm}
\begin{pspicture}(5,5)
 \rput(2,1){\includegraphics[height=0.6in]{web_frob16.pdf}}
  \rput(2,5.2){\includegraphics[height=0.23in]{singcomult.pdf}}
 \end{pspicture}
 \stackrel {\cong}{\eqref{eq:genus_one}} 
\raisebox{-15pt}{\includegraphics[width = 0.3in, height=0.75in]{remove_sgenusop2}}  
\]
The second equivalences of singular cobordisms in Equations~\eqref{eq:weak_com} and \eqref{eq:remove_sing_genusop} are obtained in a similar manner.
\end{proof}

It is easy to see that the next three proposition hold.

\begin{proposition}
The \textit{singular genus-one operator} \raisebox{-10pt}{\includegraphics[height=0.35in]{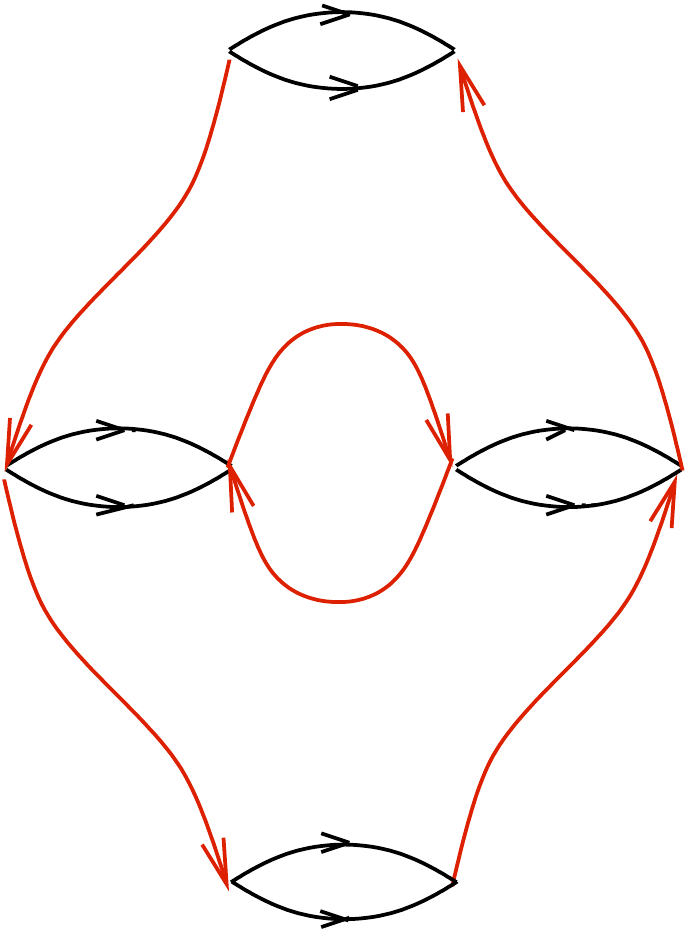}} can be moved around freely in any diagram. Specifically, the following cobordisms are equivalent:
\begin{equation}
\raisebox{-18pt}{\includegraphics[height=0.65in]{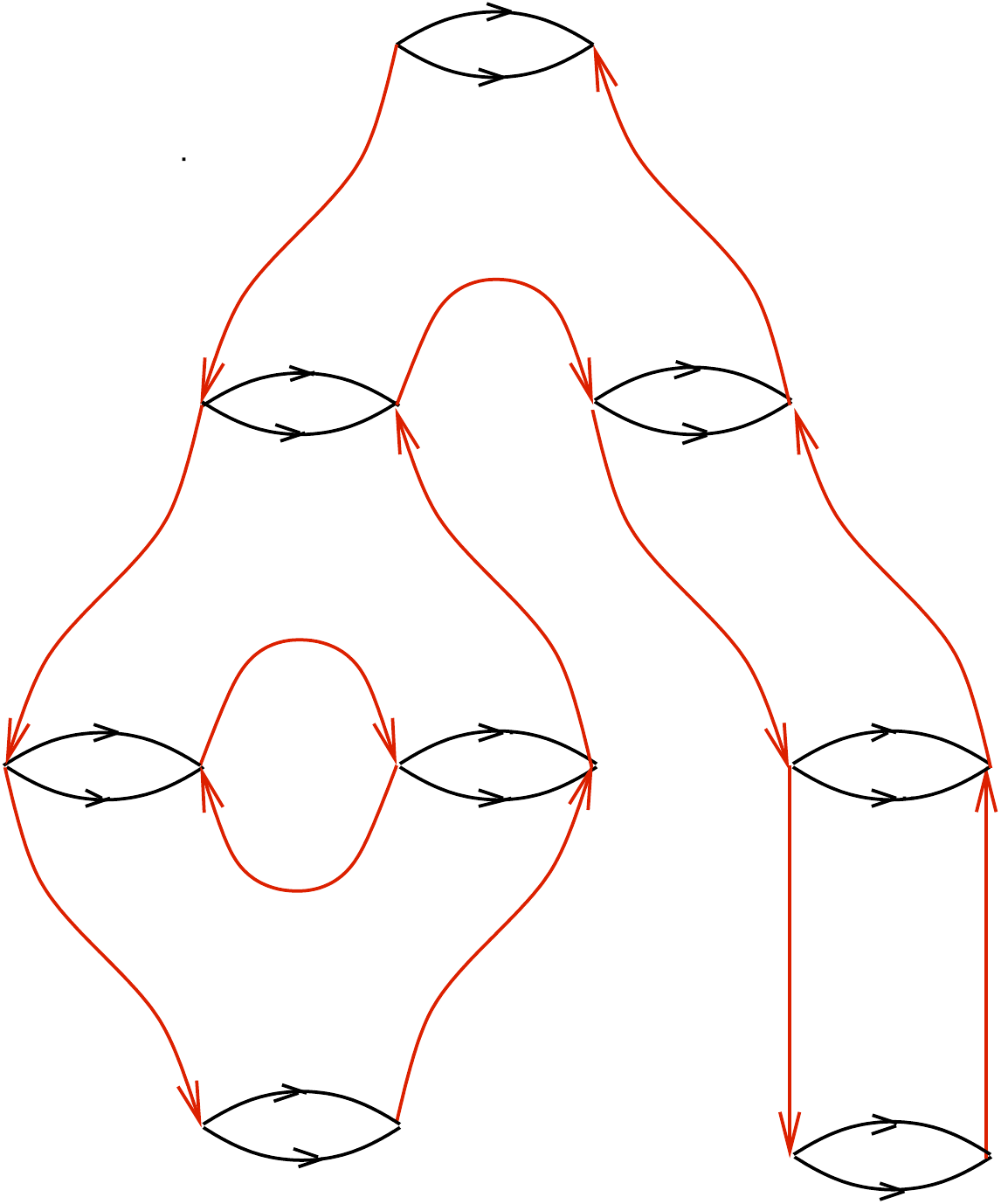}} \quad \cong \quad \raisebox{-18pt}{\includegraphics[height=0.65in]{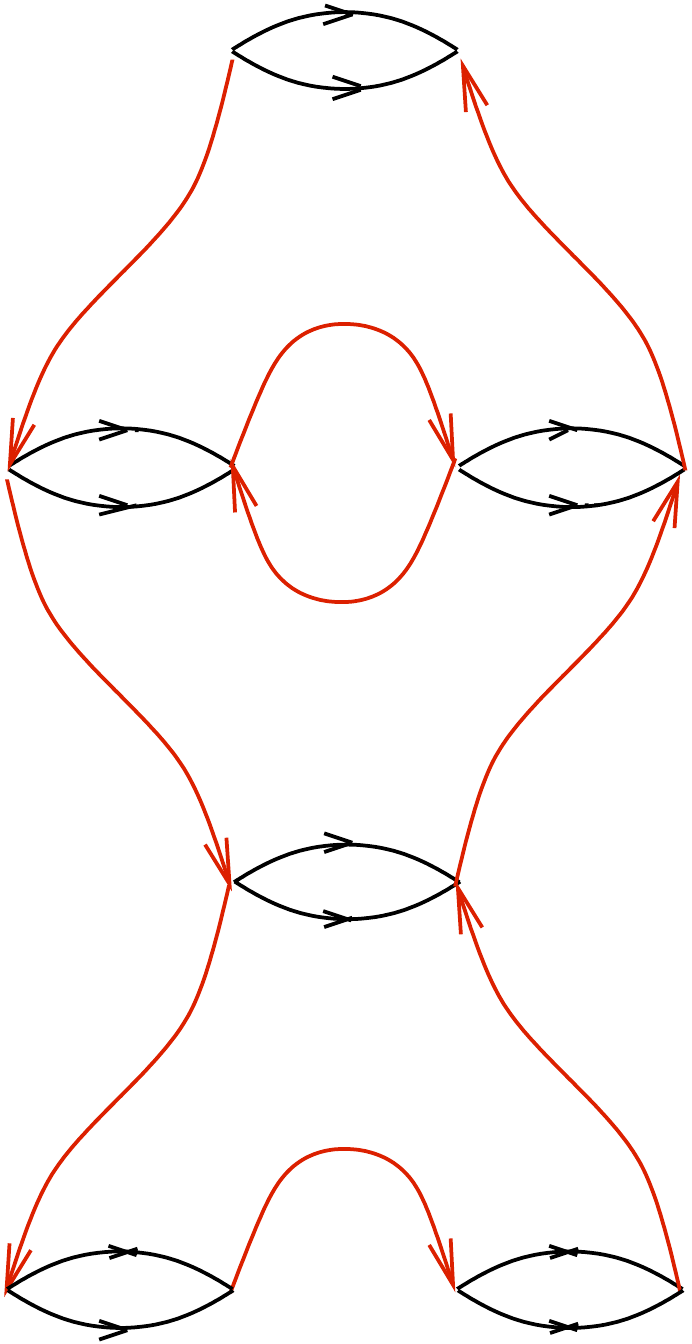}} \quad \cong \quad \raisebox{-18pt}{\includegraphics[height=0.65in]{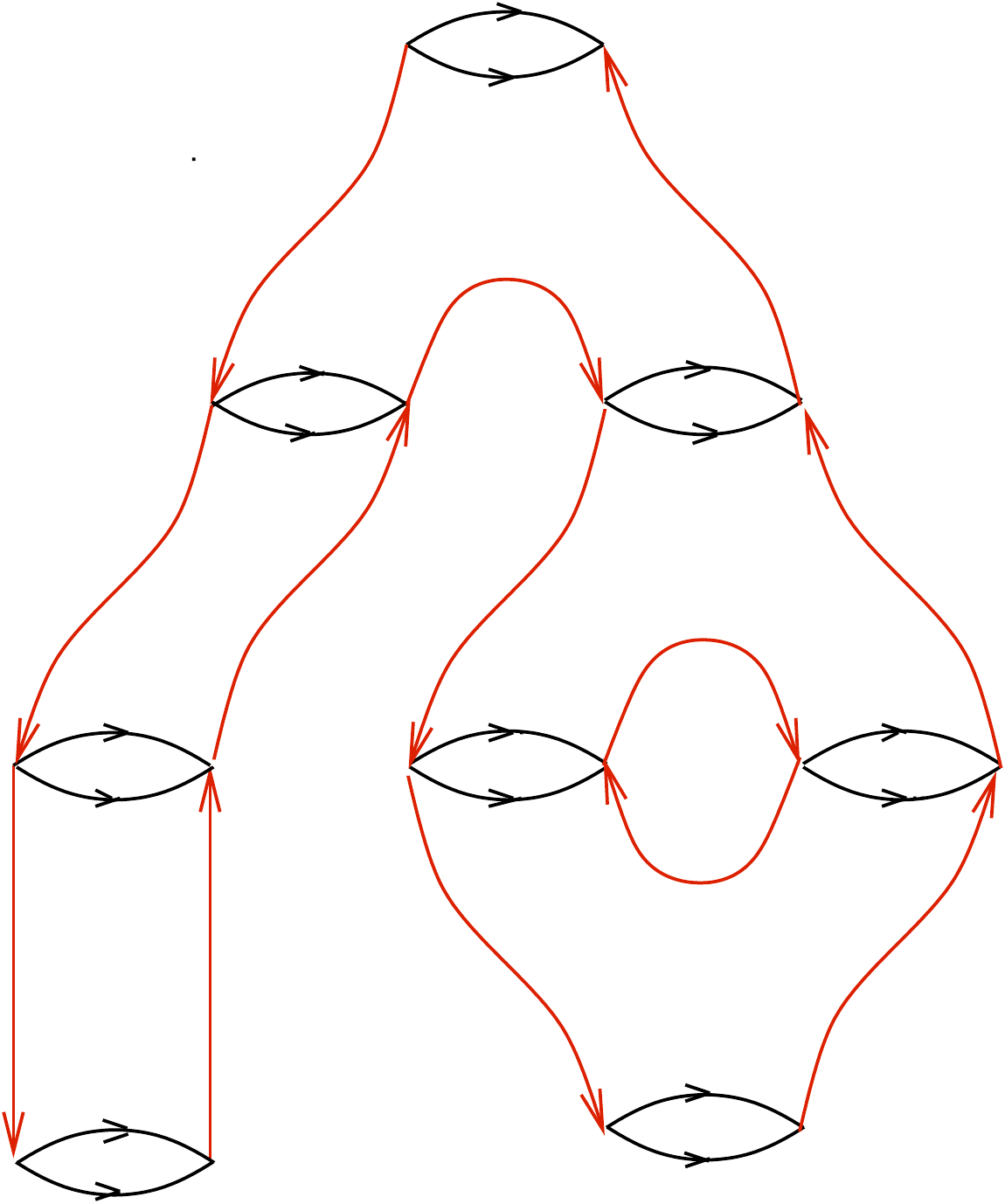}} \label{eq:sing_genus_one_comult}
\end{equation}
\begin{equation}
\raisebox{-18pt}{\includegraphics[height=0.65in]{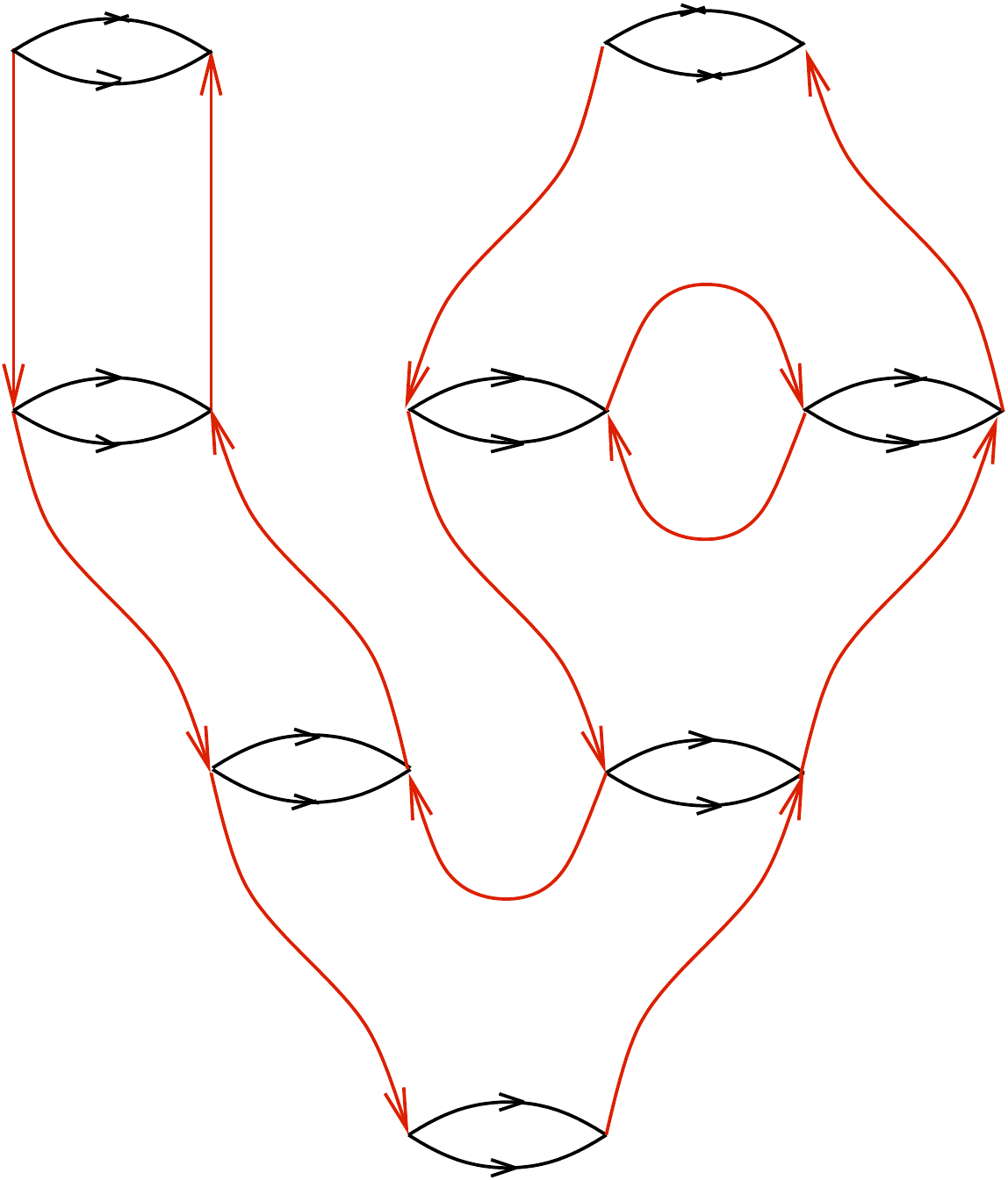}} \quad \cong \quad \raisebox{-18pt}{\includegraphics[height=0.65in]{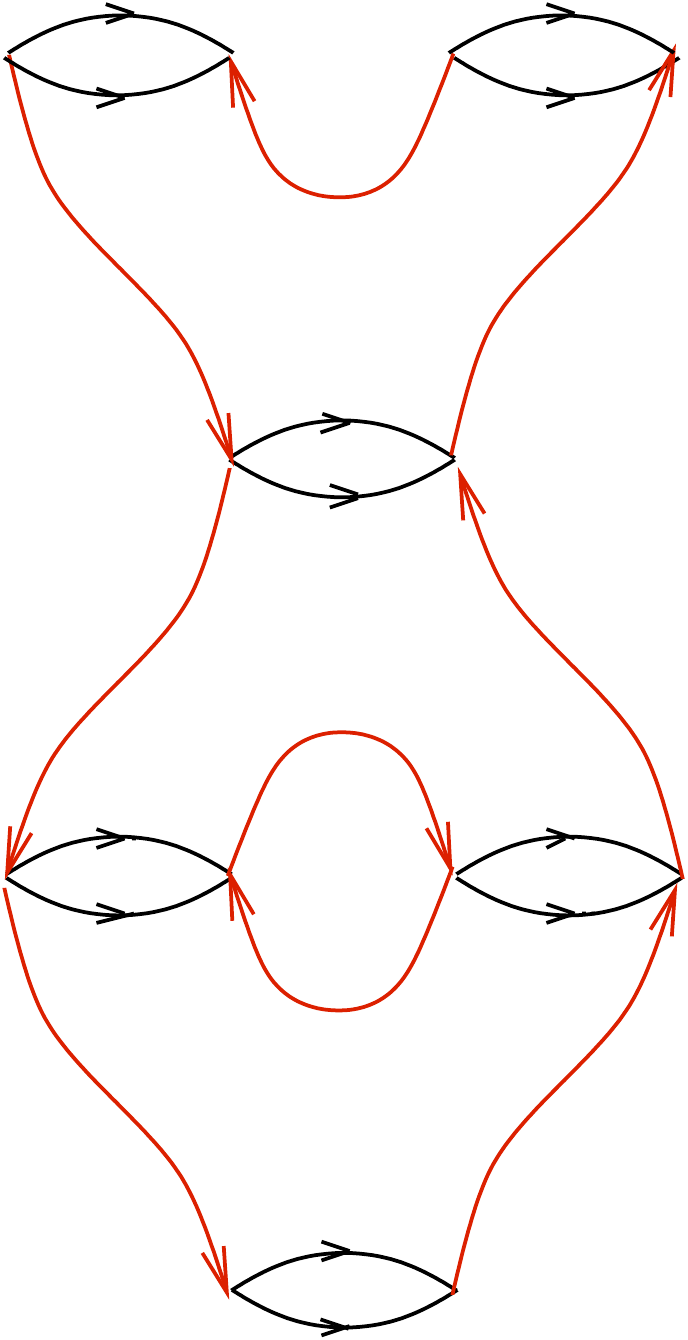}} \quad \cong \quad \raisebox{-18pt}{\includegraphics[height=0.65in]{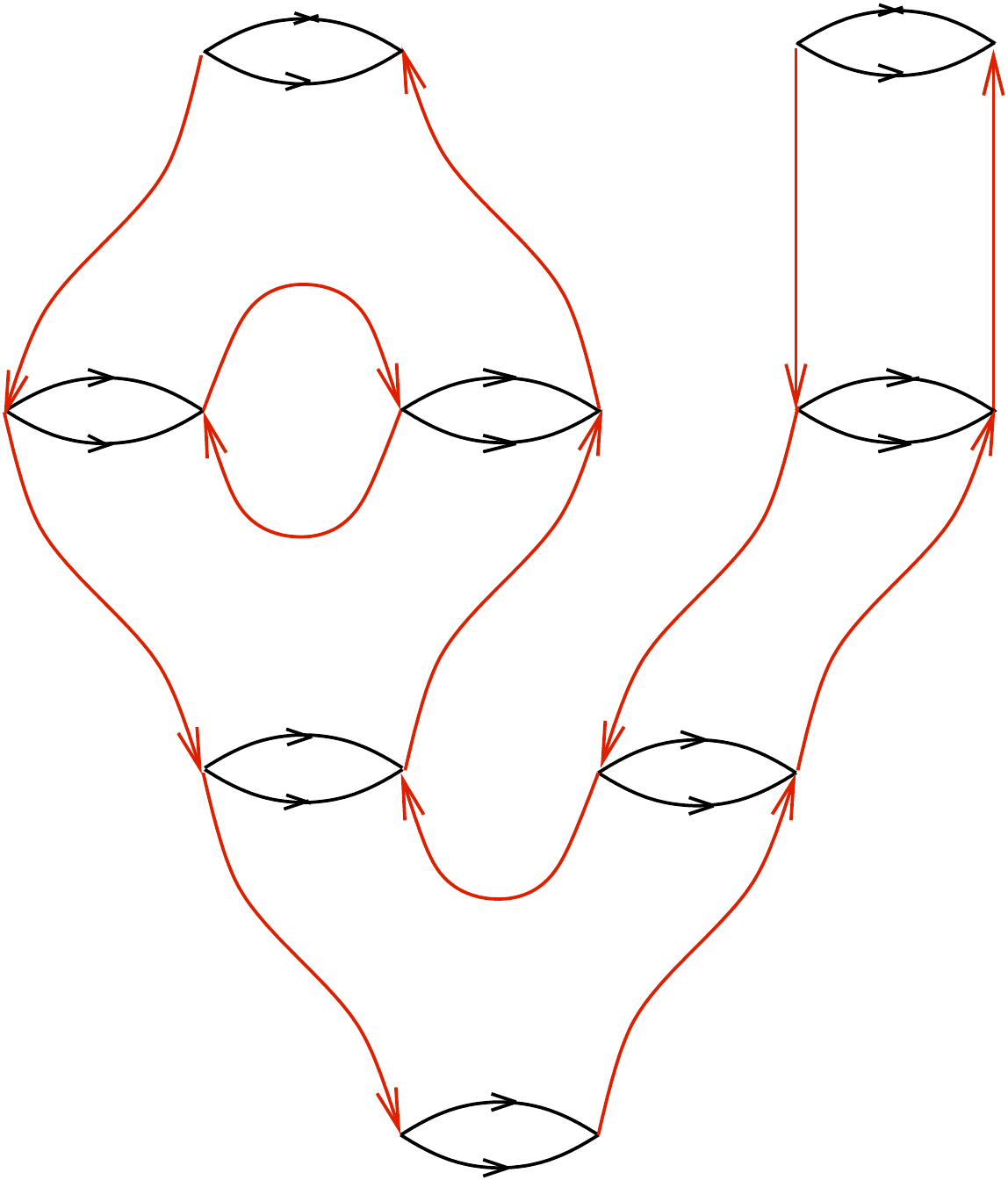}}  \label{eq:sing_genus_one_mult}
\end{equation}
\end{proposition}

\begin{proposition}
The \textit{genus-one operator} \raisebox{-10pt}{\includegraphics[height=0.35in]{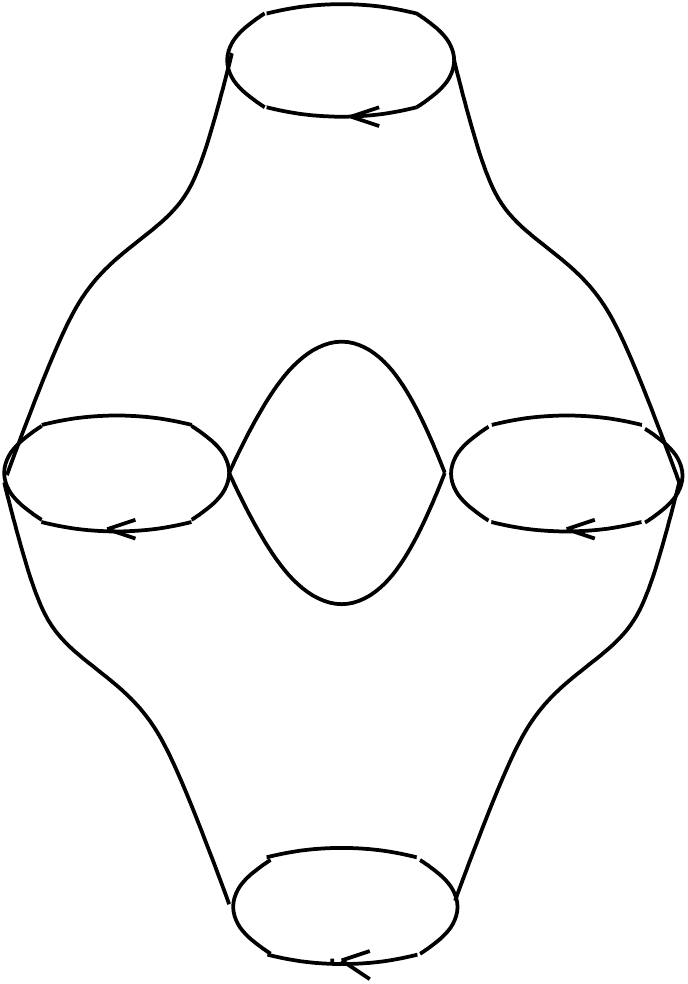}} can be moved around freely in any diagram. Specifically, the following cobordisms are equivalent:
\begin{equation}
\raisebox{-18pt}{\includegraphics[height=0.65in]{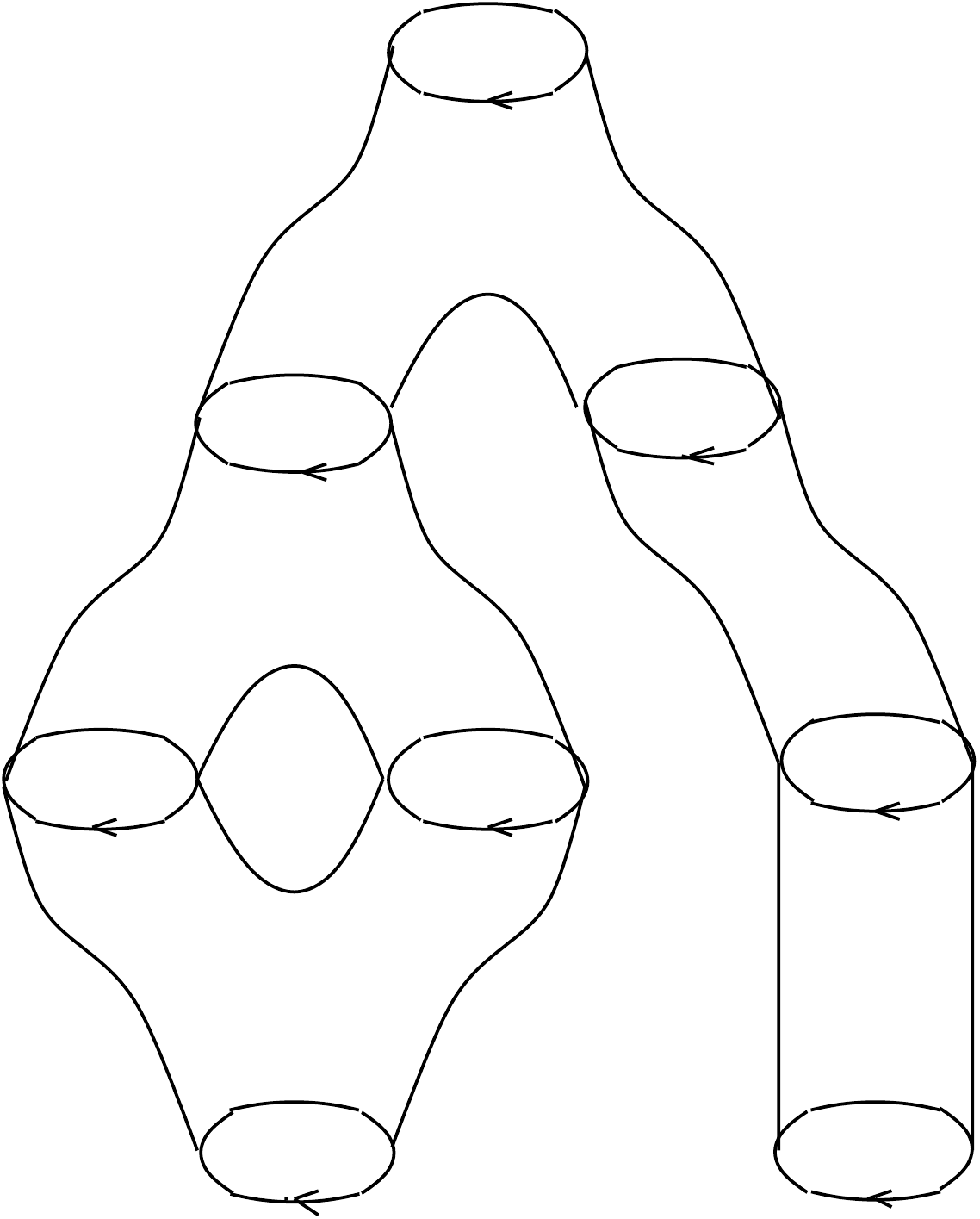}} \quad \cong \quad \raisebox{-18pt}{\includegraphics[height=0.65in]{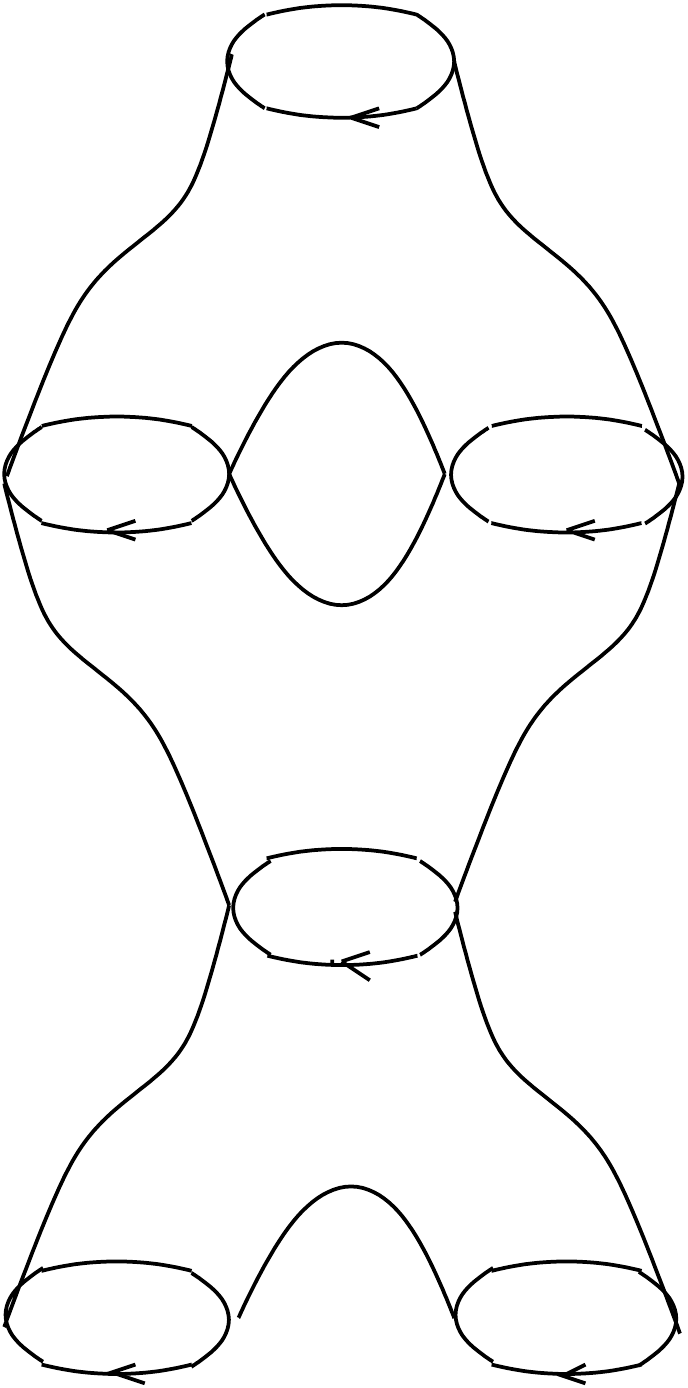}} \quad \cong \quad \raisebox{-18pt}{\includegraphics[height=0.65in]{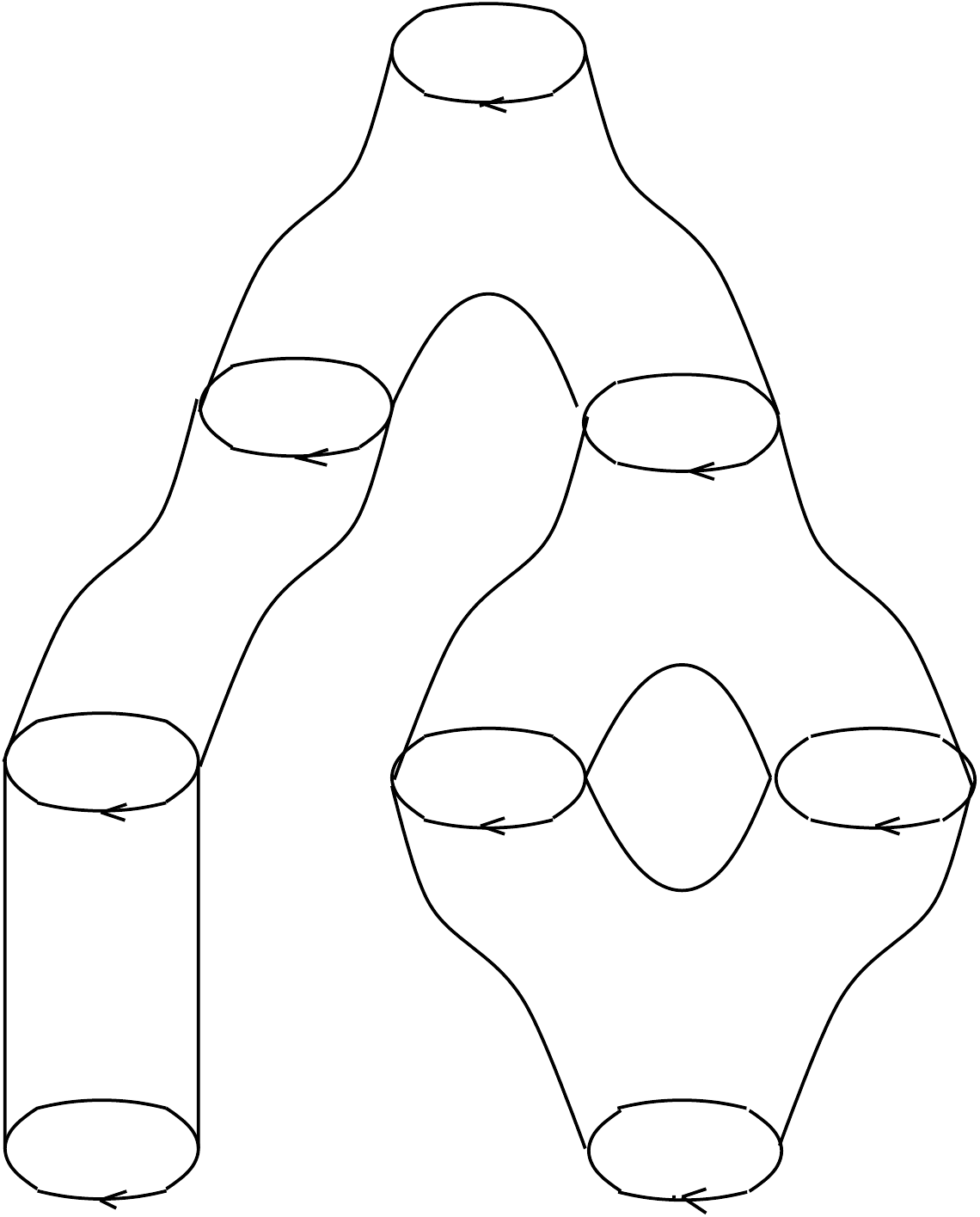}} \label{eq:genus_one1}
\end{equation}
\begin{equation}
\raisebox{-18pt}{\includegraphics[height=0.65in]{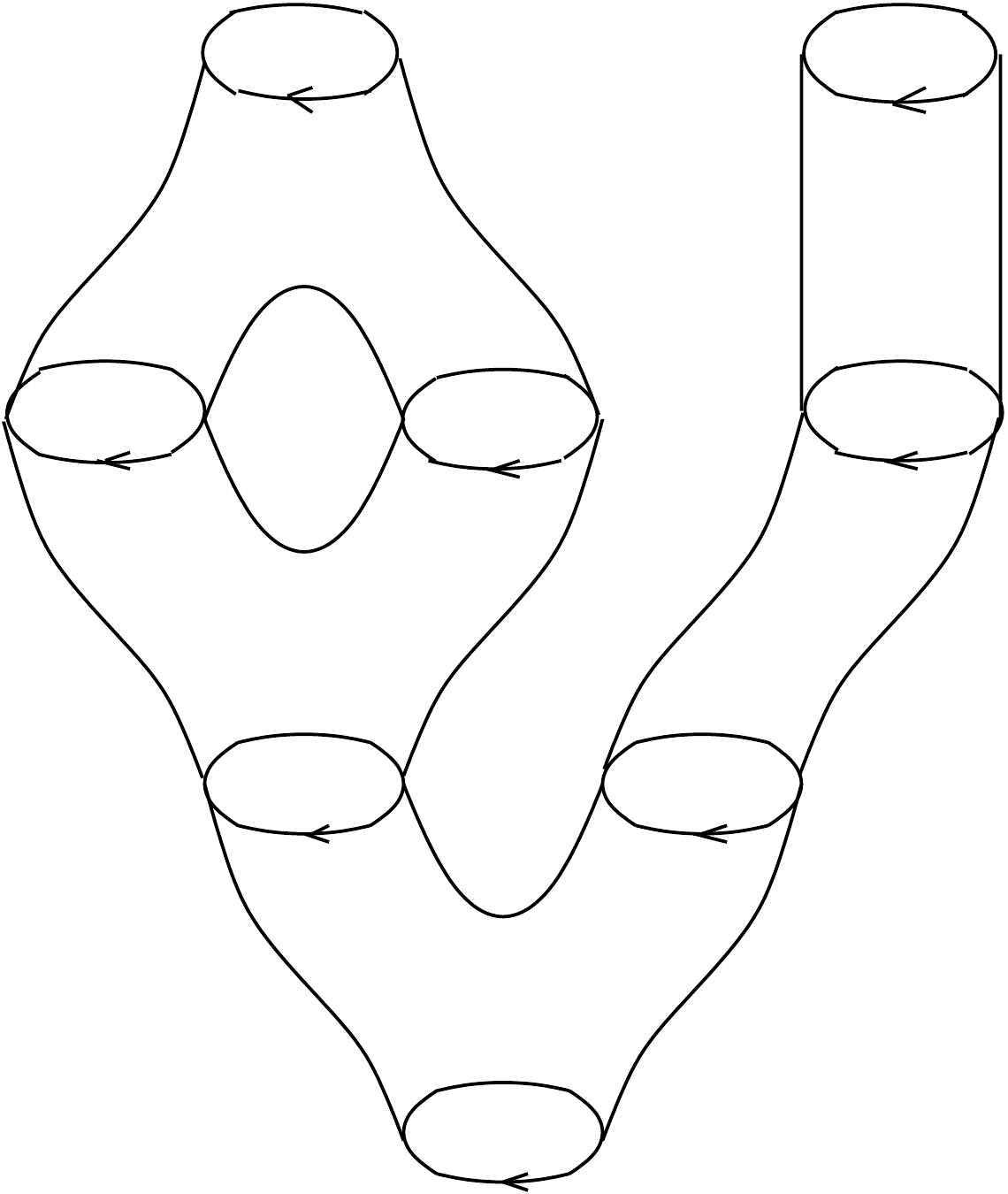}} \quad \cong \quad \raisebox{-18pt}{\includegraphics[height=0.65in]{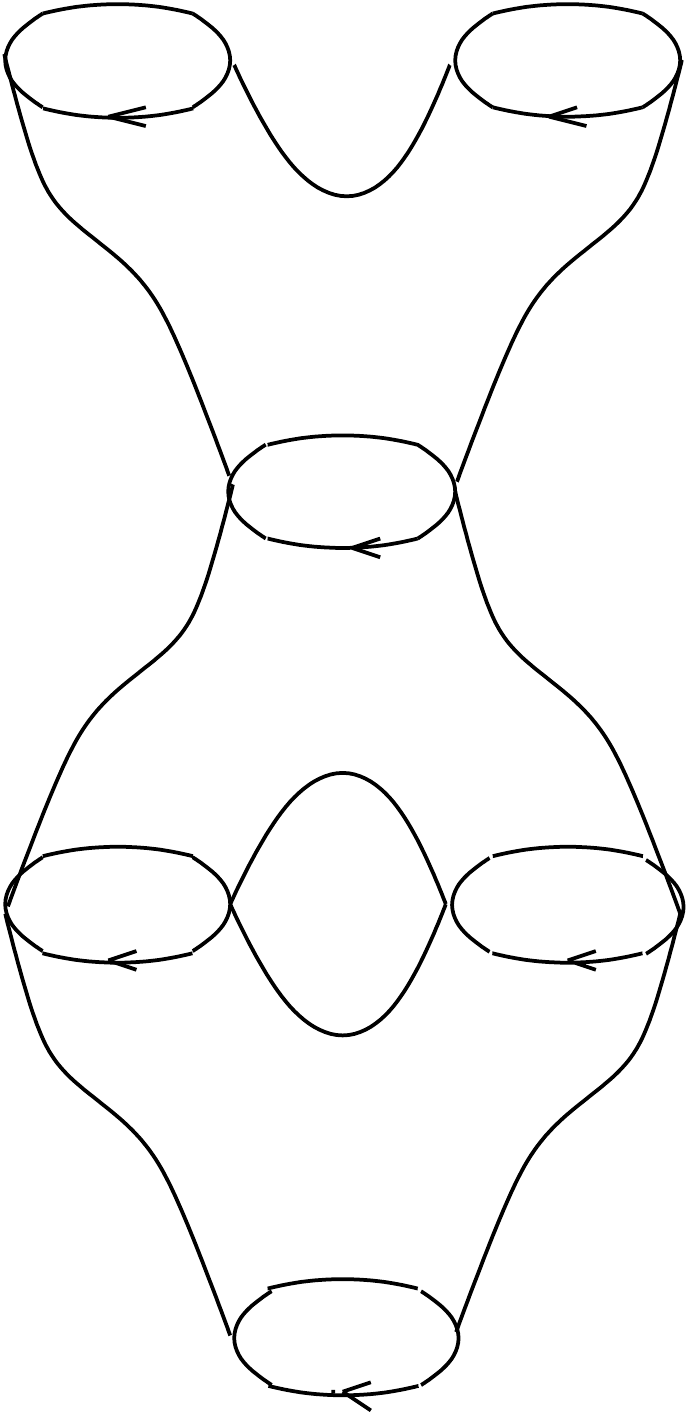}} \quad \cong \quad \raisebox{-18pt}{\includegraphics[height=0.65in]{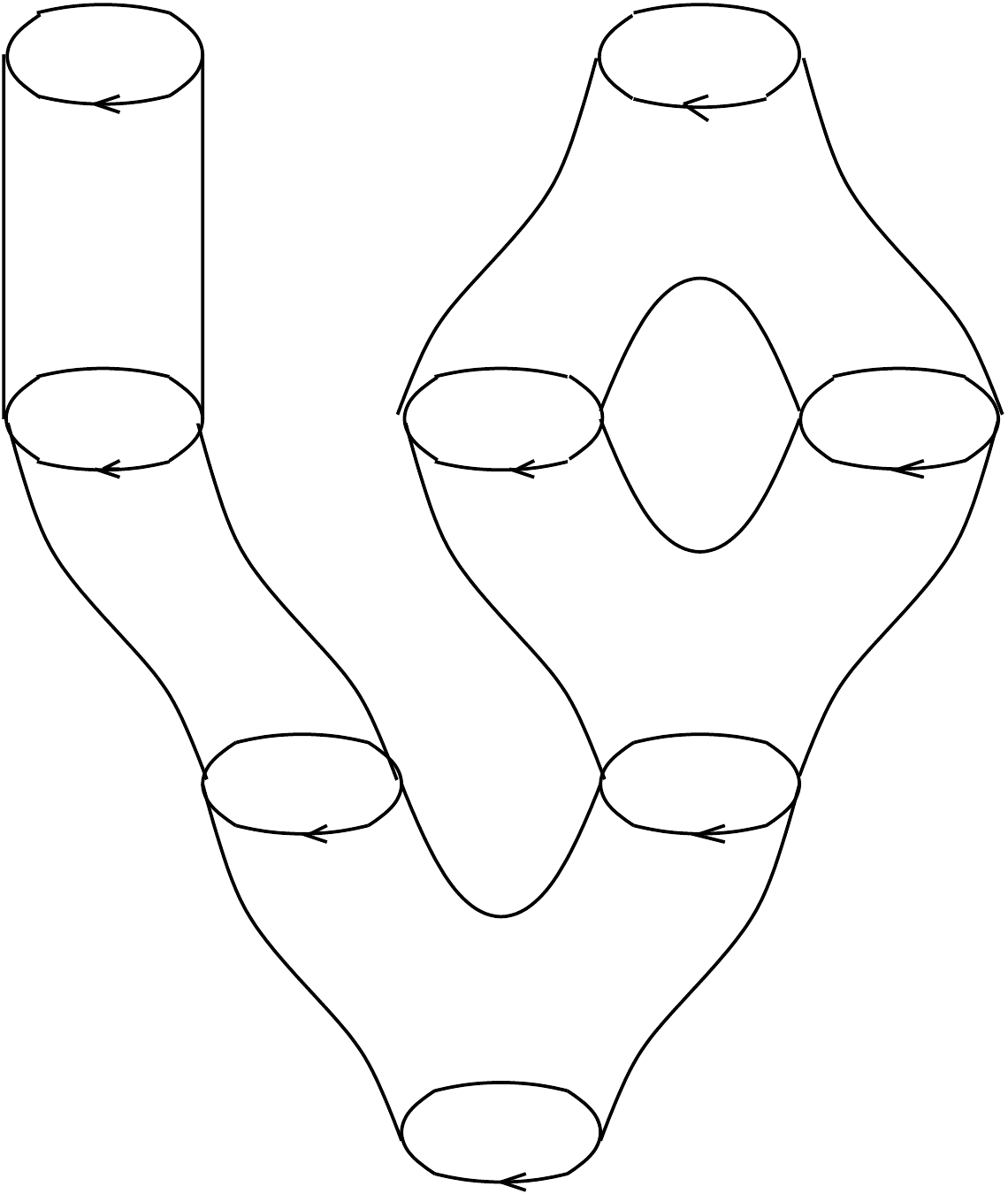}} \label{eq:genus_one2}
\end{equation}
\end{proposition}

\begin{proposition}
Singular cobordisms of the form $\raisebox{-8pt}{\includegraphics[height=0.32in]{zipper_cozipper.pdf}}$ can be moved freely in any diagram. That is, the following singular cobordisms are equivalent:

\begin{equation}
\raisebox{-13pt}{\includegraphics[height=0.52in]{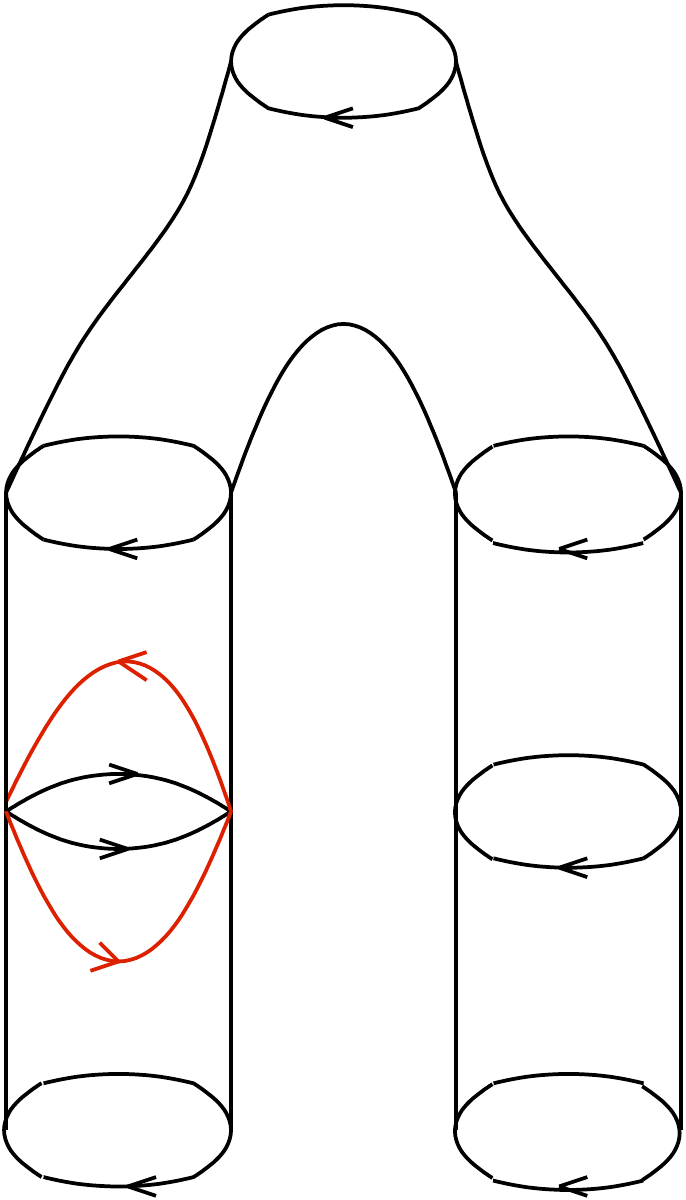}} \quad \cong \quad \raisebox{-13pt}{\includegraphics[height=0.52in]{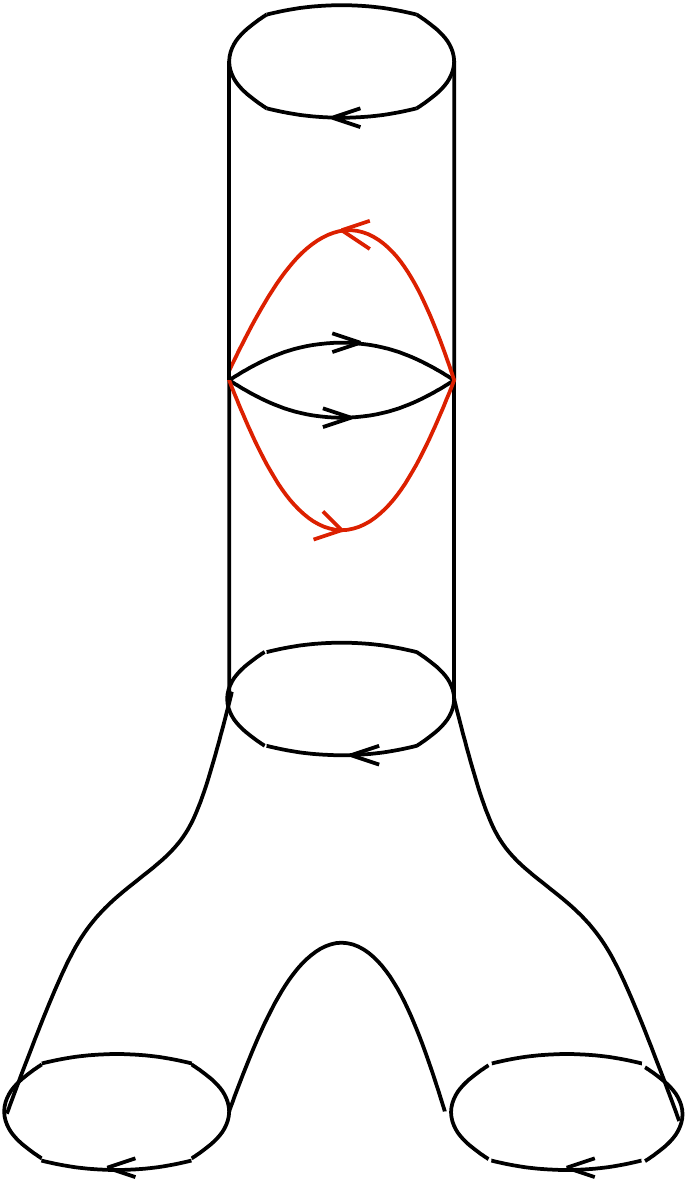}}\quad \cong \quad \raisebox{-13pt}{\includegraphics[height=0.52in]{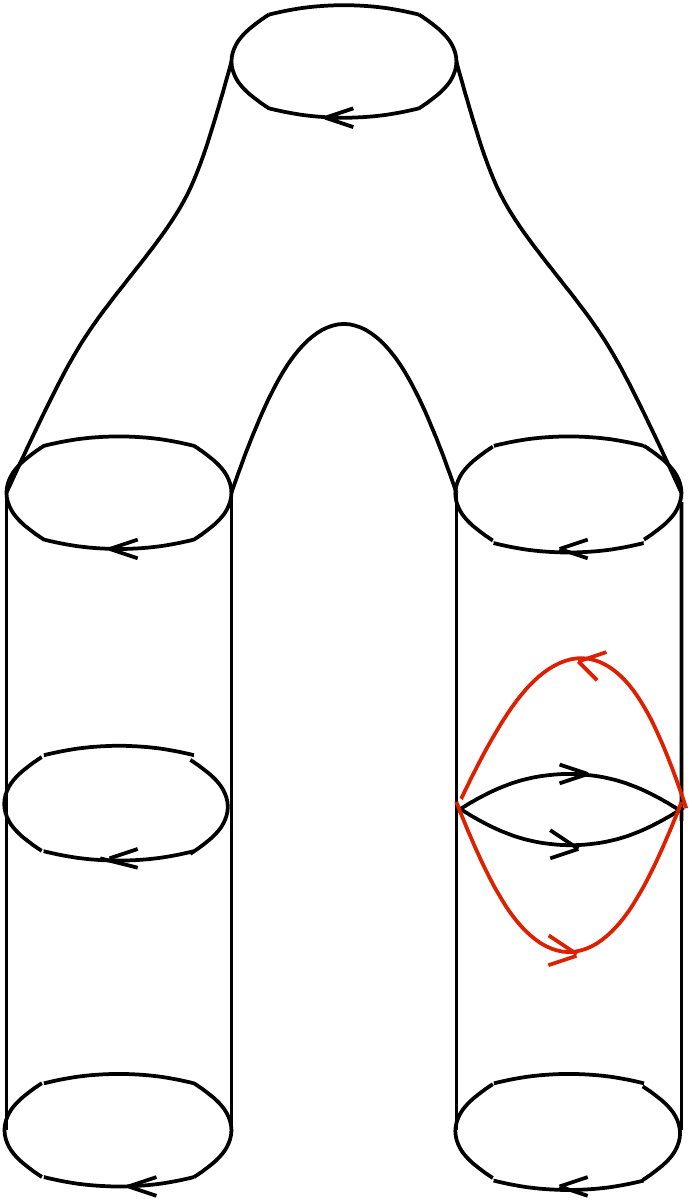}}\label{eq:zipper_cozipper1}
\end{equation}
\begin{equation}
\raisebox{-18pt}{\includegraphics[height=0.65in]{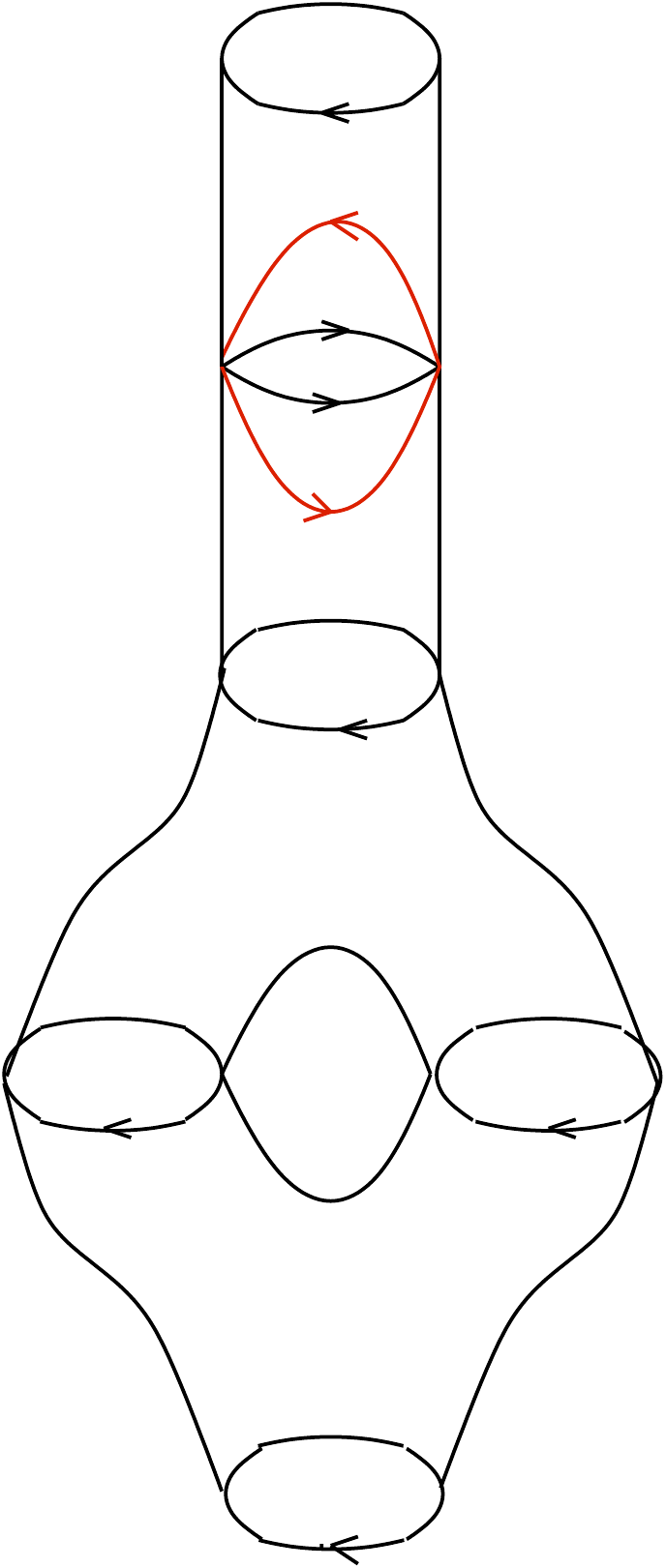}} \quad \cong \quad \raisebox{-18pt}{\includegraphics[height=0.65in]{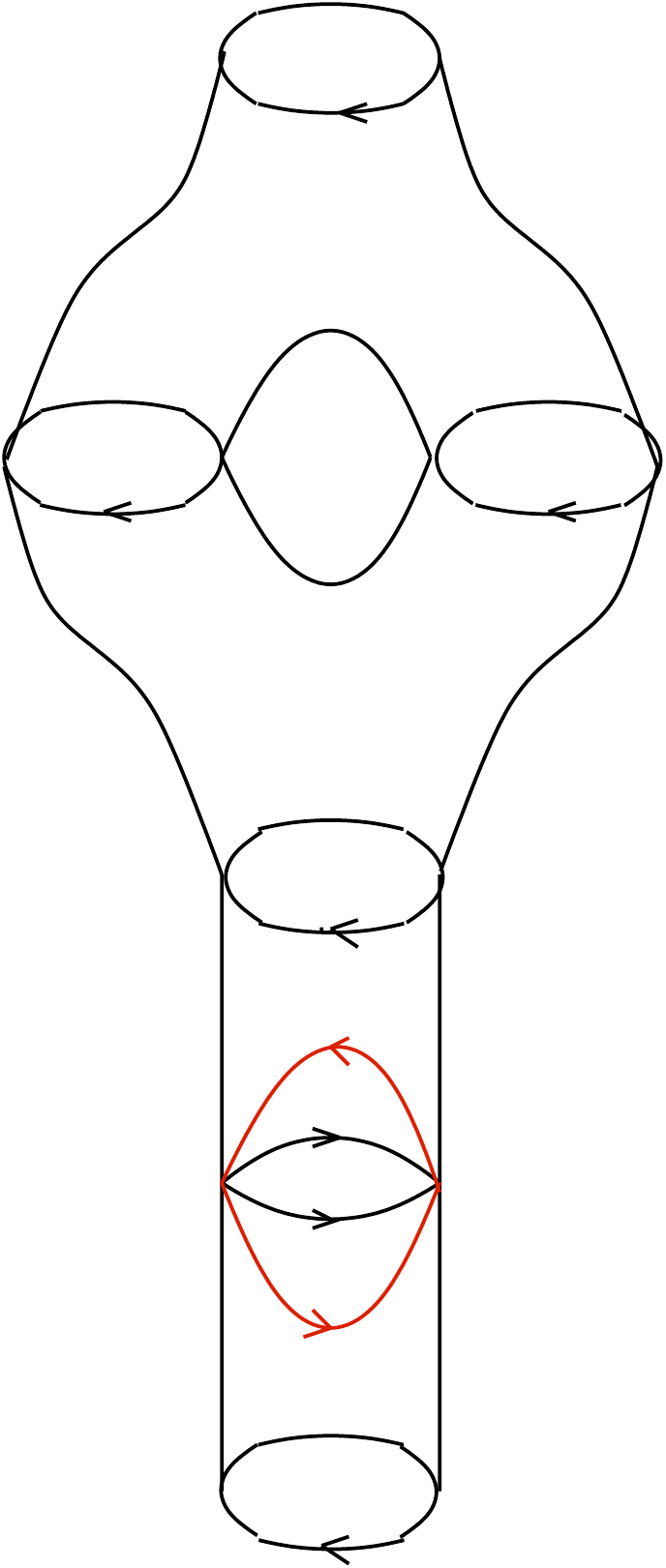}}\label{eq:zipper_cozipper2}
\end{equation}
\begin{equation}
\raisebox{-13pt}{\includegraphics[height=0.52in]{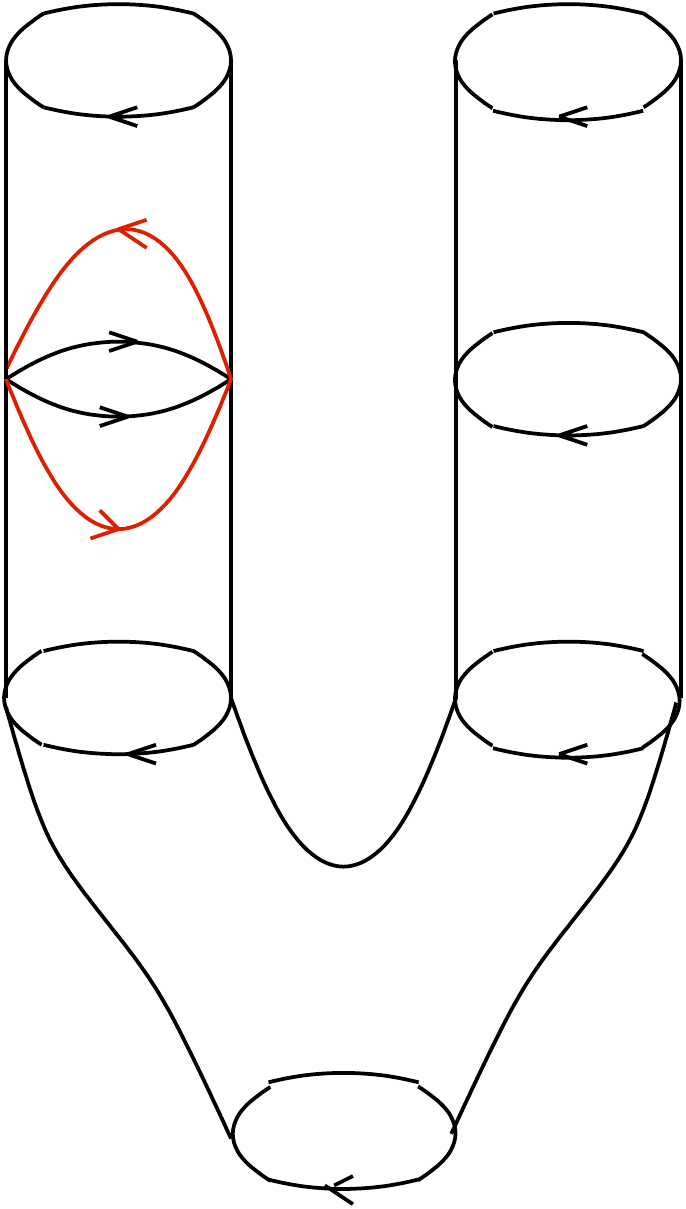}} \quad \cong \quad \raisebox{-13pt}{\includegraphics[height=0.52in]{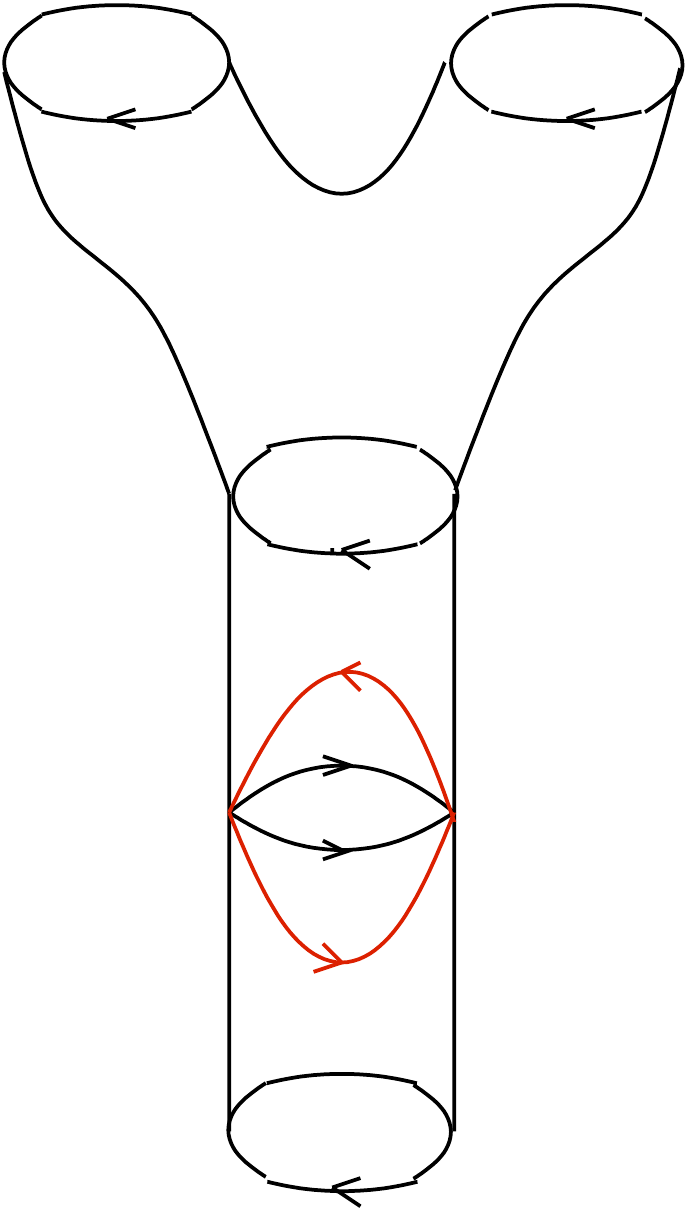}}\quad \cong \quad \raisebox{-13pt}{\includegraphics[height=0.52in]{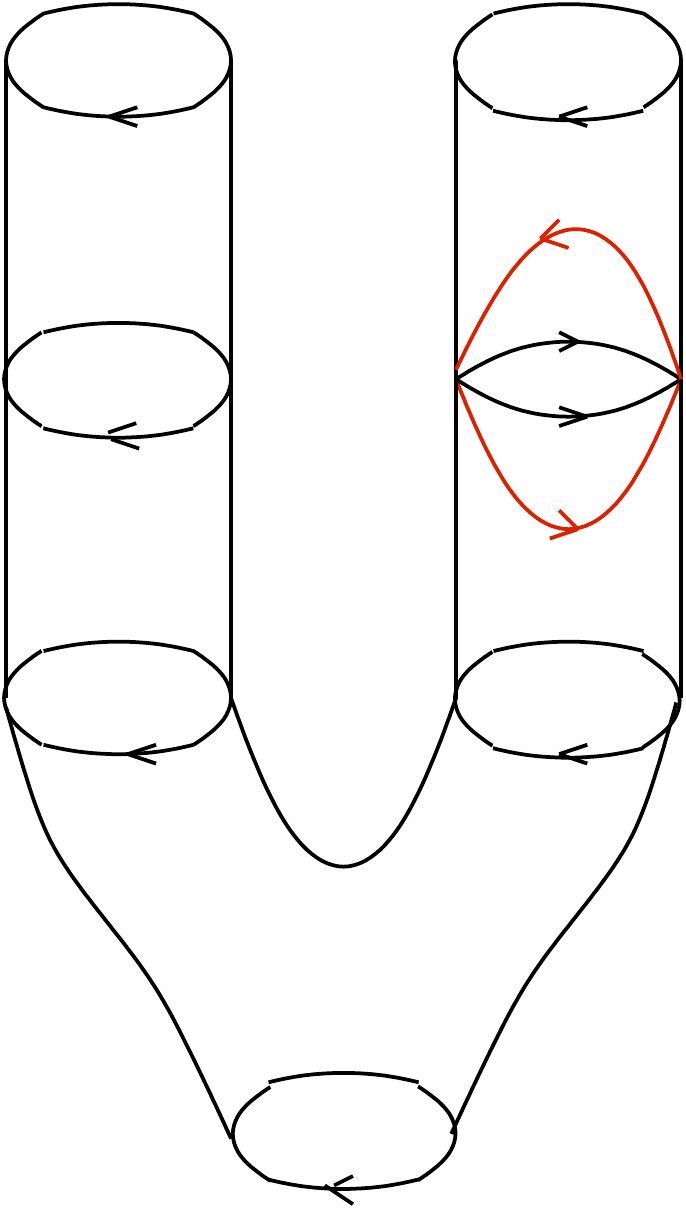}} \label{eq:zipper_cozipper3}
\end{equation}
\end{proposition}

%%%%%%%%%%%%%%%%%%%%%%%%%%%%%%%%%%%%%%%%%%%%%%%%%%%
\subsection{The normal form of a connected singular $2$-cobordisms} \label{sec:normal_form}
%%%%%%%%%%%%%%%%%%%%%%%%%%%%%%%%%%%%%%%%%%%%%%%%%%%

In this subsection we define the \textit{normal form} of a connected singular cobordism $\Sigma$ with a given topological structure, namely genus, singular number and singular boundary permutation. The reader will find many similarities with the description of the normal form of an open-closed cobordism given in~\cite{LP}. 

\subsubsection{Particular case} We define first the normal form of a connected singular cobordism whose source consists entirely of copies of the bi-web $1 = \raisebox{-3pt}{\includegraphics[height=0.15in]{singcircle.pdf}}$, and whose target consists entirely of copies of the circle $0 = \raisebox{-3pt}{\includegraphics[height=0.15in]{circle.pdf}}$. Specifically, we consider singular cobordisms $\Sigma \co \textbf{n} \to \textbf{m}$ for which $\textbf{n} = (1, 1, \dots, 1)$ and $\textbf{m} = (0, 0, \dots, 0),$ and denote the set of all such cobordisms by $\textbf{Sing-2Cob}_{W \to C}(\textbf{n}, \textbf{m}).$  Then we give the normal form for an arbitrary connected singular cobordism by using the zig-zag identities (\ref{eq:sing_zig_zag}) and (\ref{eq:zig_zag}).

Notice that relations of the form $(\Sigma' \coprod \id_{\textbf{m}}) \circ (\id_{\textbf{n}'} \coprod \Sigma) = \Sigma' \coprod \Sigma$ hold in $\textbf{Sing-2Cob}$ for any $\Sigma \co \textbf{n} \to \textbf{m}$ and $\Sigma' \co \textbf{n}' \to \textbf{m}',$ and we will make use of them in order to have small  heights for diagrams.

\begin{definition}\label{def:normalform}
Let $\Sigma \in \textbf{Sing-2Cob}_{W \to C}(\textbf{n}, \textbf{m})$ be a connected cobordism with singular boundary permutation $\sigma(\Sigma),$ genus $g(\Sigma)$ and singular number $s(\sigma),$ and write the singular boundary permutation as a product of disjoint cycles $\sigma(\Sigma) = \sigma_1 \sigma_2 \dots \sigma_r, r \in \mathbb{N} \cup \{0\},$ where $\sigma_k$ has length $q_k \in \mathbb{N}, 1 \leq k \leq r.$  The normal form of $\Sigma$ is the composition
\begin{equation}
 \mbox{NF}_{W \to C}(\Sigma) = E_{|\textbf{m}|} \circ D_{g(\Sigma)} \circ C_{s(\Sigma)} \circ B_r \circ  (\coprod _{k =1}^r A(q_k)) \circ \Sigma_{\overline{\sigma(\Sigma)}} \label{eq:normalform} 
 \end{equation}
of the following singular cobordisms:
\begin{enumerate}

\item [1.] For each cycle $\sigma_k,$ the singular cobordism $A(q_k)$ consists of $q_k-1$ singular multiplications followed by a cozipper, as depicted below:
\begin{equation*}
A(q_k): =  \raisebox{10pt}{\includegraphics[height=0.22in]{singmult.pdf}}\ddots \raisebox{-37pt}{\includegraphics[height=0.52in]{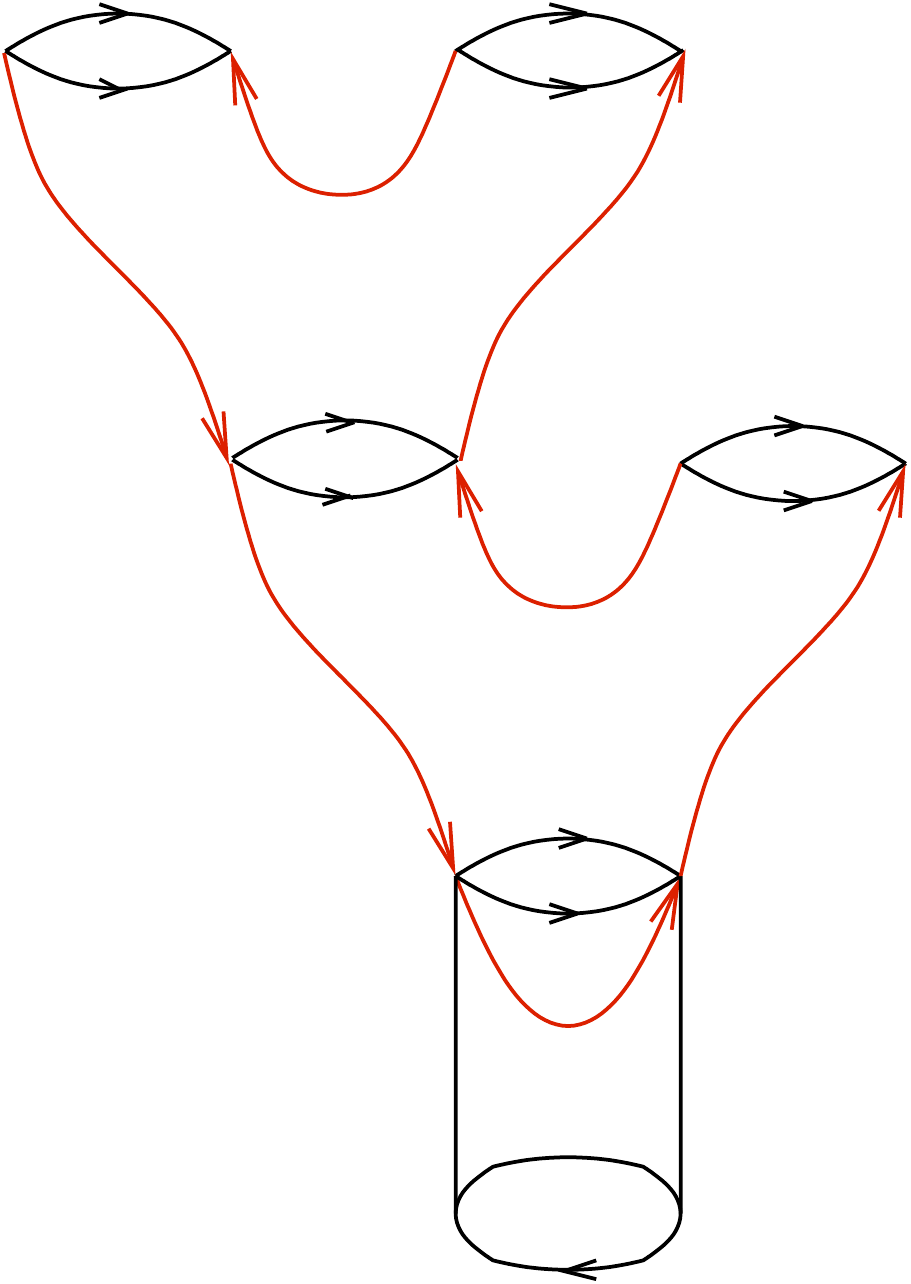}}
\end{equation*}
The normal form contains the free union of such cobordisms for each cycle $\sigma_k, 1 \leq k \leq r.$ If $q_k = 1$ then $A(q_k)$ is a cozipper, and if $|\textbf{n}| = 0$ then $r = 0,$ and the free union $\displaystyle \coprod _{k =1}^r A(q_k)$ is replaced by the empty set. 

\item [2.] If $r \geq 1,$ then the singular cobordism $B_r$ consists of $r-1$ multiplications
\begin{equation*}
B_r : =  \raisebox{10pt}{\includegraphics[height=0.19in]{mult.pdf}} \ddots  \raisebox{-25pt}{\includegraphics[height=0.33in]{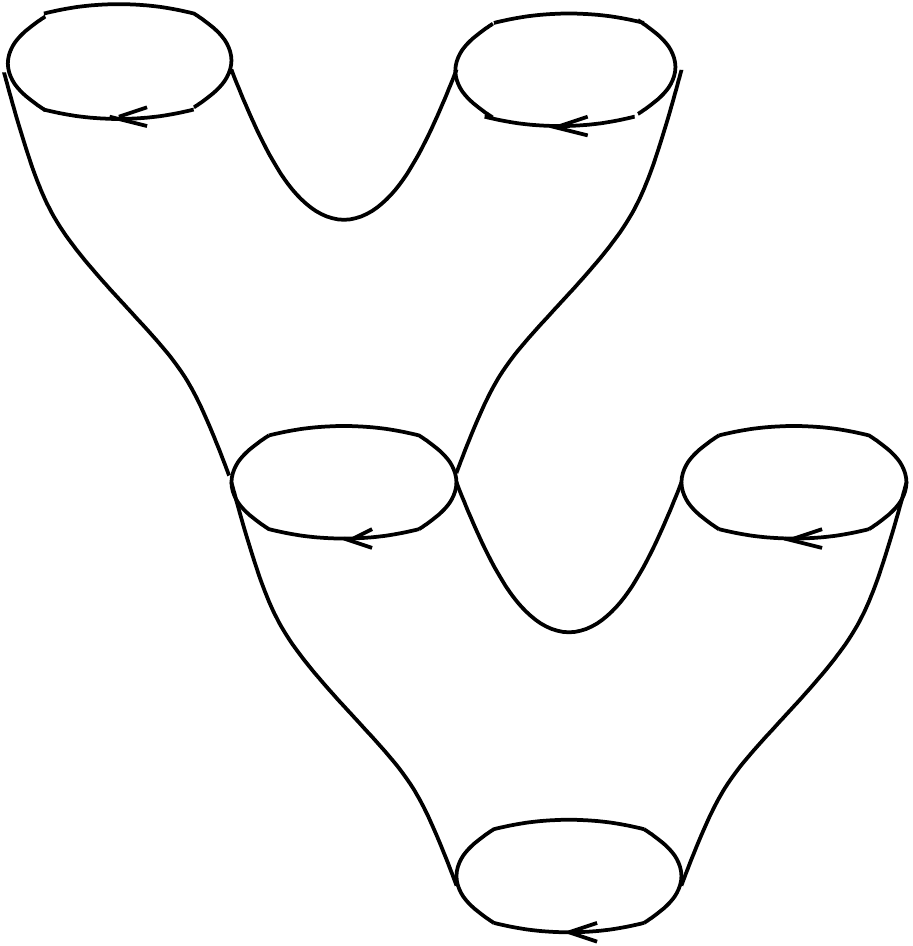}}
\end{equation*}
If $r = 0$ then $B_0 : =   \raisebox{-5pt}{\includegraphics[height=0.18in]{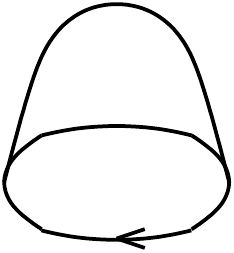}}.$

\item [3.] If the singular number $s(\sigma) \geq 1,$ the singular cobordism $C_{s(\sigma)}$ is the  composite
\begin{equation*}
C_{s(\Sigma)} : = \underbrace{S\circ S \circ \dots \circ S}_{s(\Sigma)} \quad \mbox{where} \quad S : = \raisebox{-11pt}{\includegraphics[height=0.38in]{zipper_cozipper.pdf}}\,.
\end{equation*}
If $s(\Sigma) = 0$ then $C_{s(\Sigma)} = \emptyset.$ 

\item [4.] If $g(\Sigma) \geq 1,$ the singular cobordism $D_{g(\Sigma)}$ is the composite
\begin{equation*}
D_{g(\Sigma)} : = \underbrace{G \circ G \circ \dots \circ G}_{g(\Sigma)} \quad \mbox{where} \quad G : = \raisebox{-11pt}{\includegraphics[height=0.35in]{genus_one.pdf}}\,.
\end{equation*}
If $g(\Sigma) = 0$ then $D_{s(\Sigma)} = \emptyset.$ 

\item[5.] If $| \textbf{m} | \geq 1,$ then the singular cobordism $E_{|\textbf{m}|}$ consists of $| \textbf{m} | -1$ comultiplications, as depicted below:
\begin{equation*}
E_{| \textbf{m} |} : = \raisebox{-31pt}{\includegraphics[height=0.33in]{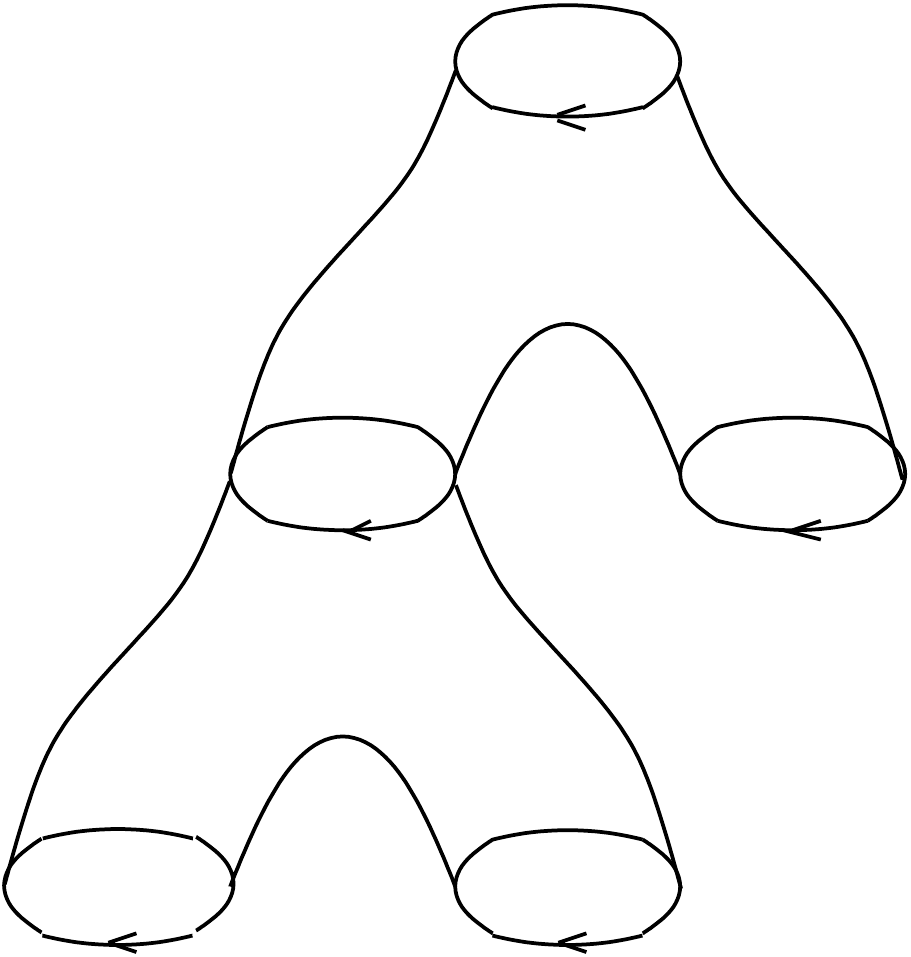}}
\vdots \raisebox{13pt}{\includegraphics[height=0.20in]{comult.pdf}} 
\end{equation*}
 If $| \textbf{m} | = 0$ then $E_0 : = \raisebox{-5pt}{\includegraphics[height=0.18in]{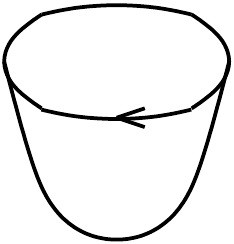}}.$
 
 \item [6.] $\Sigma_{\overline{\sigma(\Sigma)}}$ represents the (permutation) singular cobordisms induced by the permutation $\overline{\sigma(\Sigma)}$ given below. Let $\tau(\Sigma)$ be the singular boundary permutation of the cobordism  
 \begin{equation}
  E_{|\textbf{m}|} \circ D_{g(\Sigma)} \circ C_{s(\Sigma)} \circ B_r \circ (\coprod _{k =1}^r A(q_k)).\label{eq:simple_normalform}
  \end{equation}
 Then $\overline{\sigma(\Sigma)}$ is the permutation that satisfies
 \[\sigma(\Sigma) = \overline{\sigma(\Sigma)}\,^{-1} \cdot \tau(\Sigma) \cdot \overline{\sigma(\Sigma)}.  \]
 \end{enumerate}
 Note that precomposing  $E_{|\textbf{m}|} \circ D_{g(\Sigma)} \circ C_{s(\Sigma)} \circ B_r \circ (\coprod _{k =1}^r A(q_k))$ with $\Sigma_{\overline{\sigma(\Sigma)}}$ yields a singular cobordism whose singular boundary permutation is $\sigma(\Sigma)$.
\end{definition}

In Figure~\ref{fig:NormalForm} we show a cobordism of the form \eqref{eq:simple_normalform}, that is, the normal form of a cobordism in $\textbf{Sing-2Cob}_{W \to C}(\textbf{n}, \textbf{m}),$ without precomposition with $\Sigma_{\overline{\sigma(\Sigma)}}.$
\begin{figure}
\[
\psset{xunit=.22cm,yunit=.22cm}
\begin{pspicture}(5,60)
 \rput(1.6,3){\includegraphics[height=0.23in]{comult.pdf}}
 \rput(2.7,5){\includegraphics[height=0.23in]{comult.pdf}}
 \rput(3.8,7){\includegraphics[height=0.23in]{comult.pdf}}
 \rput(3.8,10.5){$\vdots$}
 \rput(4.9,13){\includegraphics[height=0.23in]{comult.pdf}}
 \rput(4.9,16){\includegraphics[height=0.41in]{genus_one.pdf}}
 \rput(5,20){$\vdots$}
 \rput(4.9,23.25){\includegraphics[height=0.41in]{genus_one.pdf}}
 \rput(4.9,27){\includegraphics[height=0.35in]{zipper_cozipper.pdf}}
  \rput(5,31.2){$\vdots$}
  \rput(4.9,34.05){\includegraphics[height=0.35in]{zipper_cozipper.pdf}}
   \rput(4.2,47.7){\includegraphics[height=0.23in]{singmult.pdf}}
    \rput(5, 45.7){$\ddots$}
  \rput(6.6,42.9){\includegraphics[height=0.23in]{singmult.pdf}}
  \rput(7.75,40.8){\includegraphics[height=0.23in]{singmult.pdf}}
   \rput(7.75,38.9){\includegraphics[height=0.20in]{cozipper.pdf}}
   \rput(4.9, 36.7) {\includegraphics[height=0.23in]{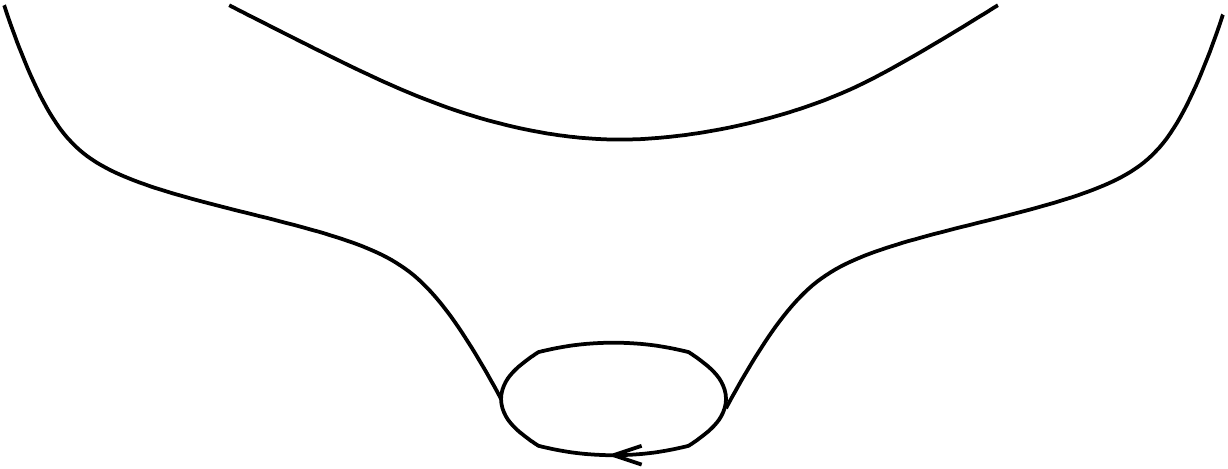}}
    \rput(2.1, 39){\includegraphics[height=0.23in]{mult.pdf}}
   \rput(1, 41){\includegraphics[height=0.23in]{mult.pdf}}
   \rput(-2, 43.5){$\ddots$}
   \rput(-3, 45.8) {\includegraphics[height=0.23in]{hugemult.pdf}}
  \rput(-0.12,47.9){\includegraphics[height=0.20in]{cozipper.pdf}}
  \rput(-5.81,47.9){\includegraphics[height=0.20in]{cozipper.pdf}}
  \rput(-0.12, 49.9){\includegraphics[height=0.23in]{singmult.pdf}}
  \rput(-5.85, 49.9){\includegraphics[height=0.23in]{singmult.pdf}}
  \rput(-1.3,52){\includegraphics[height=0.23in]{singmult.pdf}}
   \rput(-7,52){\includegraphics[height=0.23in]{singmult.pdf}}
   \rput(-4, 55){$\ddots$}
   \rput(-10, 55){$\ddots$} 
   \rput(-11,57){\includegraphics[height=0.23in]{singmult.pdf}}
   \rput(-5,57){\includegraphics[height=0.23in]{singmult.pdf}}
\end{pspicture}
\]
\caption{Normal Form of a cobordism in $\textbf{Sing-2Cob}_{W \to C}(\textbf{n}, \textbf{m}),$ without precomposition with a permutation}{ \label{fig:NormalForm}}
\end{figure}

The following two results say that a cobordism given in its normal form is invariant, up to equivalence, under composition with certain permutation morphisms.

\begin{proposition}
Let $[\Sigma] \in \textbf{Sing-2Cob}_{W \to C}(\textbf{n}, \textbf{m}).$ Then 
\[ [\sigma^{\textbf{m}} \circ \mbox{NF}_{W \to C}(\Sigma)] = [\mbox{NF}_{W \to C}(\Sigma)] = [\mbox {NF}_{W \to C} (\Sigma) \circ \sigma_k^{\textbf{n}} ] \] 
for any $\sigma \in S_{|\textbf{m}|}$ and for all cycles $\sigma_k \in S_{|\textbf{n} |}, 1 \leq k \leq r,$ that appear in the decomposition of $\sigma(\Sigma) = \sigma_1 \sigma_2 \dots \sigma_r$ into disjoint cycles.
\end{proposition}

%%%%%%%%%%%%%%%%%%%%%%%%%%%%%%%%%%%%%%%%%%%%%%%%%%%
\subsubsection{General case} We use the normal form for a connected singular cobordism in $\textbf{Sing-2Cob}_{W \to C}(\textbf{n}, \textbf{m})$ and the duality property for the bi-web and the circle  
 to obtain the normal form of a generic connected morphism  $ [\Sigma] \in \textbf{Sing-2Cob}(\textbf{n}, \textbf{m}).$ 
 
Let $\Sigma$ be a representative of the equivalence class $[\Sigma]$ and let $\textbf{n}_0 \coprod \textbf{n}_1$ be the permutation of $\textbf{n}$ such that $\textbf{n}_0 = (0, 0, \dots, 0)$ and $\textbf{n}_1 = (1, 1, \dots, 1).$ Similarly, let $\textbf{m}_0 \coprod \textbf{m}_1$ be the permutation of $\textbf{m}$ such that $\textbf{m}_0 = (0, 0, \dots, 0)$ and $\textbf{m}_1 = (1, 1, \dots, 1).$ In order to use the normal form described above, we need to associate to $[\Sigma]$ a singular cobordism whose source contains only copies of the bi-web and whose target contains only copies of the circle. We define the map 
\begin{align*}
 f \co \textbf{Sing-2Cob}(\textbf{n}, \textbf{m}) &\to \textbf{Sing-2Cob}_{W \to C}(\textbf{m}_1 \coprod \textbf{n}_1, \textbf{m}_0 \coprod \textbf {n}_0)\\ 
 [\Sigma] &\mapsto f([\Sigma])
 \end{align*}
where the singular cobordism $f([\Sigma])$ is defined as follows.
Let $\sigma_1$ be the permutation cobordism corresponding to $\sigma_1 \in S_{|\textbf{n}|}$ that sends $\textbf{n}$ to $\textbf{n}_1 \coprod \textbf{n}_0.$ Similarly, denote by$\sigma_2$ the cobordism corresponding to the permutation $\sigma_2 \in S_{|\textbf{m}|}$ that sends $\textbf{m}$ to $\textbf{m}_1 \coprod \textbf{m}_0.$ 

We define $f([\Sigma])$ as the singular cobordism obtained from $[\Sigma]$ by precomposing with $\sigma_1^{-1},$ postcomposing with $\sigma_2$ and gluing copairings on every circle that $\textbf{n}_0$ contains, and singular pairings on every bi-web that $\textbf{m}_1$ contains. 

For exemplification, consider $\textbf{n} = (0, 1, 1), \textbf{m} = (0, 1, 0)$ and $[\Sigma]$ an arbitrary singular cobordism from $\textbf{n}$ to $\textbf{m}$ (note that $|\textbf{n}|$ and $|\textbf{m}|$ do not  have to be equal).  The corresponding permutation cobordisms $\sigma_1, \sigma_2$ and the image of $[\Sigma]$ under $f$ are given in Figure~\ref{fig:f(sigma)}.

\begin{figure}
\begin{equation}
\begin{array}{c}
\sigma_1 \quad = \quad \raisebox{-15pt}{\includegraphics[height=0.5in]{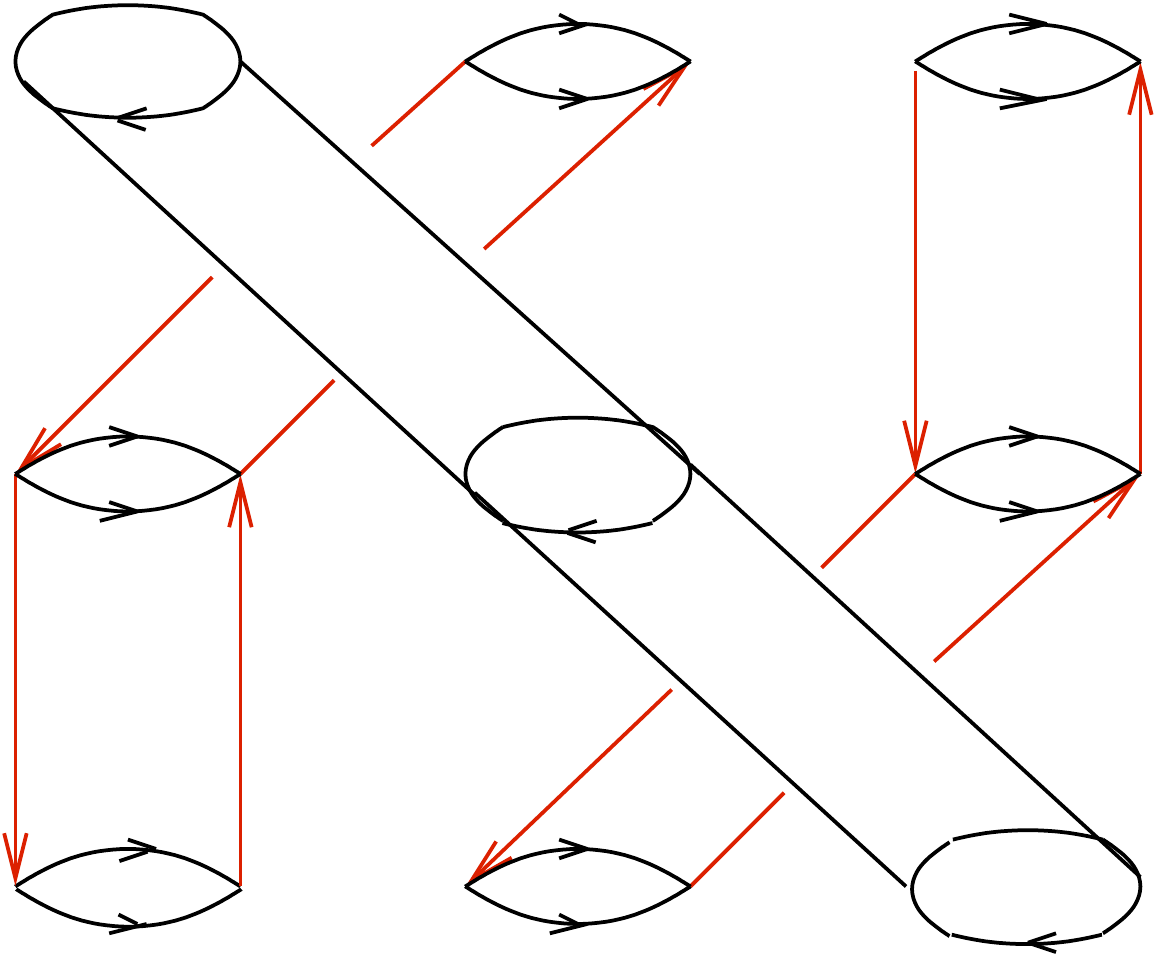}}\\
\vspace{0.3cm}\\
\sigma_2 \quad = \quad \raisebox{-10pt}{\includegraphics[height=0.35in]{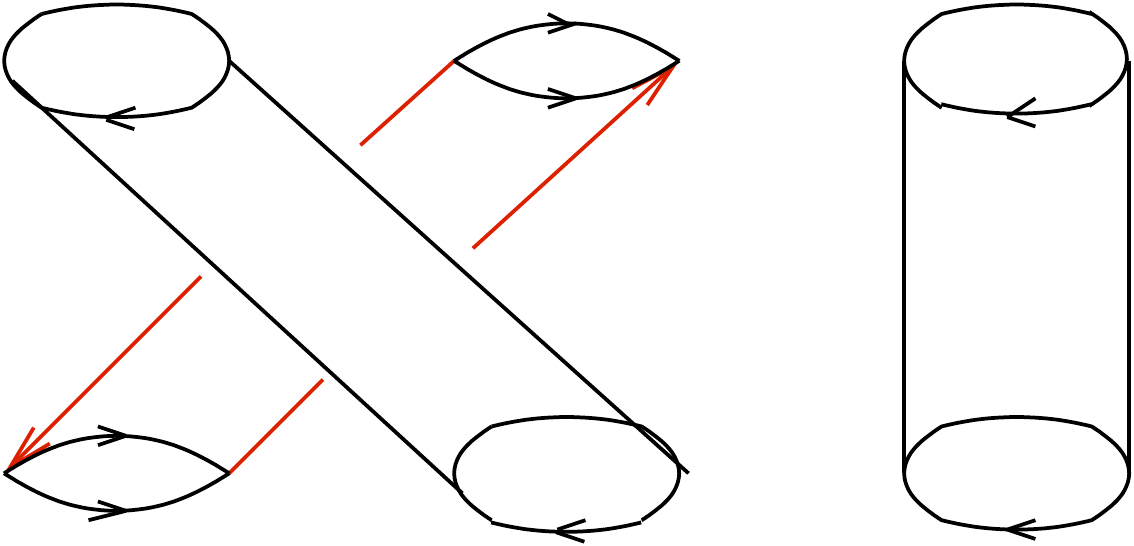}}
\end{array}
\qquad
f \co \raisebox{-10pt}{\includegraphics[height=0.37in]{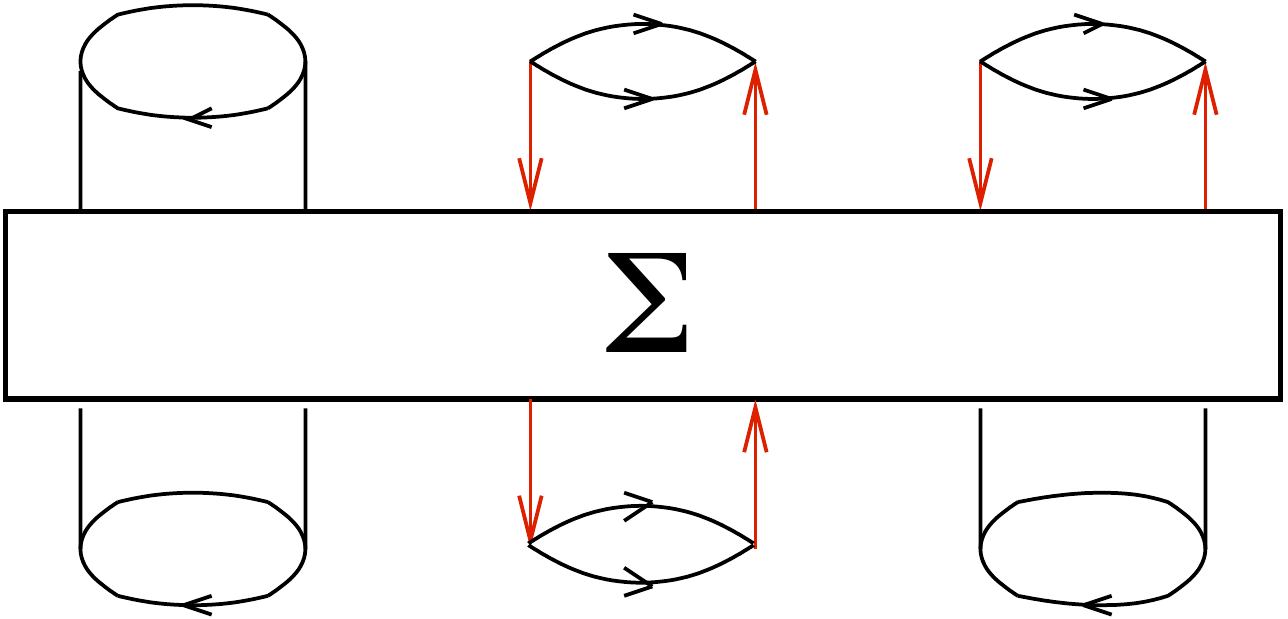}} \quad \mapsto \quad
\raisebox{-50pt}{\includegraphics[height=1.3in]{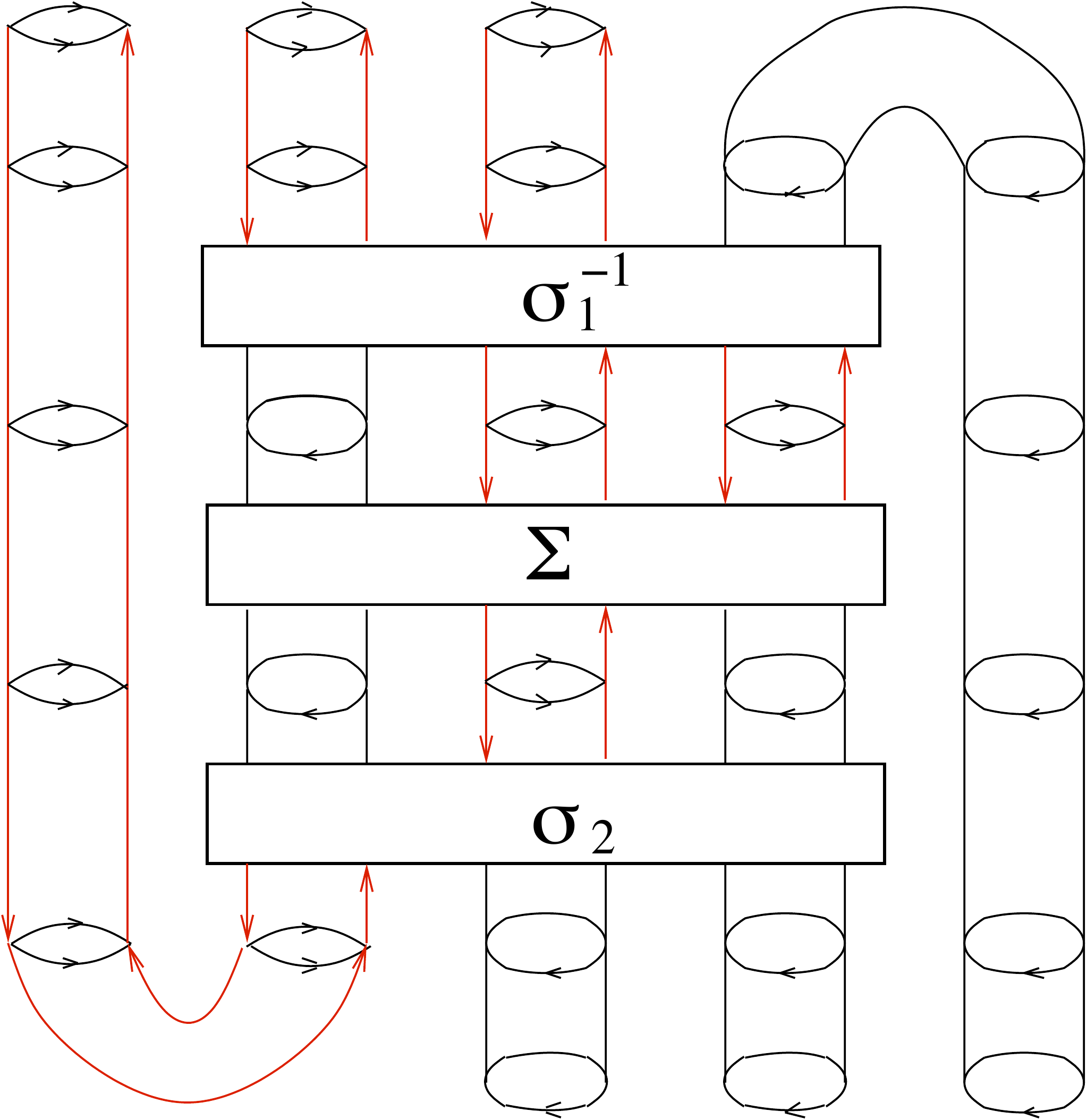}}
\end{equation} 
\caption{}{\label{fig:f(sigma)}}
\end{figure}

Notice that the mapping $f$ is well-defined, namely $f([\Sigma])$ is a morphism in the category $\textbf{Sing-2Cob}_{W \to C}(\textbf{m}_1 \coprod \textbf{n}_1, \textbf{m}_0 \coprod \textbf {n}_0),$ and $[f([\Sigma])] = [f([\Sigma'])]$ whenever $[\Sigma] = [\Sigma'].$ Therefore it makes sense to consider the normal form $\mbox{NF}_{W \to C}(f([\Sigma])).$

We also remark that $f([\Sigma])$ has a certain structure, in the sense that its source $\textbf{n}'$ and target $\textbf{m}'$  can be decomposed into free unions $\textbf{n}' = \textbf{n}'_t \coprod \textbf{n}'_s$ and $\textbf{m}' = \textbf{m}'_t \coprod \textbf{m}'_s,$ such that the copies of the bi-web in $\textbf{n}'_t$ (or $\textbf{n}'_s$) and the copies of the circle in $\textbf{m}'_t$ (or $\textbf{m}'_s$) correspond to the copies of the bi-web and of the circle coming from the target (or source) of $\sigma_2 \circ \Sigma \circ \sigma_1^{-1}.$ The permutation $\sigma_1$ is an element of $S_{|\textbf{n}'_s| + |\textbf{m}'_s|}$, while $\sigma_2$ is an element of $S_{|\textbf{n}'_t| + |\textbf{m}'_t|}.$

We define an inverse mapping $f^{-1}$ that associates to $[\Phi] \in \textbf{Sing-2Cob}_{W \to C}(\textbf{n}', \textbf{m}')$ the singular cobordism $f^{-1}([\Phi]) \in \textbf{Sing-2Cob}_{W \to C}(\sigma_2(\textbf{n}'_t \coprod \textbf{m}'_t), \sigma_1 (\textbf{n}'_s \coprod \textbf {m}'_s)).$ The cobordism $f^{-1}([\Phi])$ is obtained by gluing singular copairings to the bi-webs in $\textbf{n}'_t$ and ordinary pairings to the circles in $\textbf{m}'_s,$ and then by precomposing the resulting cobordism with the cobordism corresponding to $\sigma_1$ and by postcomposing it with the cobordism corresponding to $\sigma_2^{-1}.$   

This map is well-defined as well, and defines a bijection between the morphisms in $\textbf{Sing-2Cob}(\textbf{n}, \textbf{m})$ and those in $\textbf{Sing-2Cob}_{W \to C}(\textbf{m}_1 \coprod \textbf{n}_1, \textbf{m}_0 \coprod \textbf {n}_0).$ 

Going back to the example in Figure~\ref{fig:f(sigma)}, we give in~\eqref{eq:nform_equiv_Sigma} the singular cobordism $[f^{-1}([f([\Sigma])])] \cong [\Sigma].$
\begin{equation}
 \raisebox{-10pt}{\includegraphics[height=0.37in]{nform_Sigma.pdf}} \quad \cong \quad \raisebox{-90pt}{\includegraphics[height=2.5in]{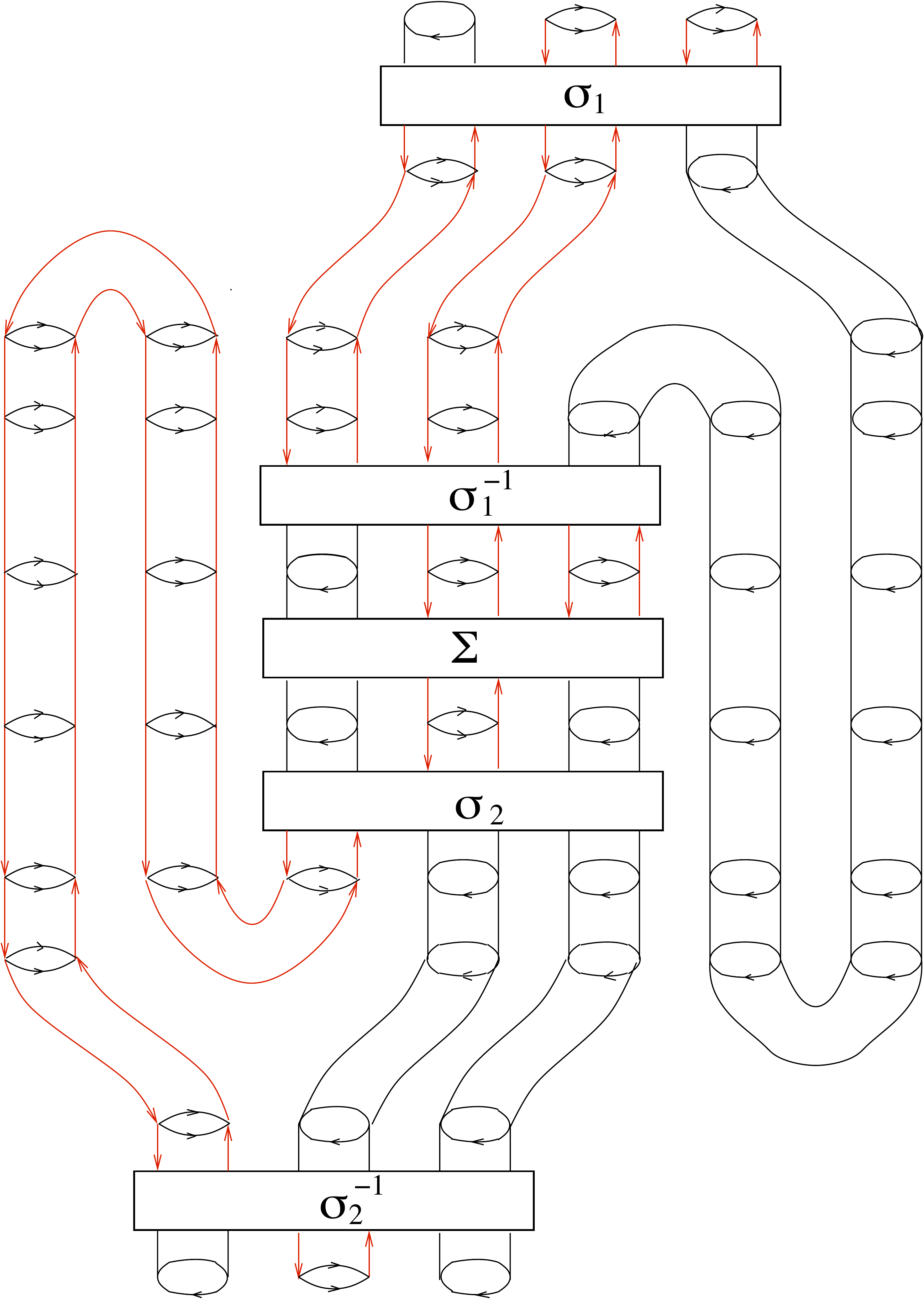}}
\label{eq:nform_equiv_Sigma}
\end{equation}

\begin{definition}
Let $[\Sigma] \in \textbf{Sing-2Cob}(\textbf{n}, \textbf{m})$ where $\Sigma$ is a connected cobordism. We define the normal form of $[\Sigma]$ by
\begin{equation}
[\mbox{NF}(\Sigma)] : = f^{-1}([\mbox{NF}_{W \to C}(f([\Sigma]))]).\label{eq:general_normalform}
\end{equation}
\end{definition}
\vspace{0.2cm}

%%%%%%%%%%%%%%%%%%%%%%%%%%%%%%%%%%%%%%%%%%%%%%%%%%%
\subsection{Non-connected singular 2-cobordisms}
%%%%%%%%%%%%%%%%%%%%%%%%%%%%%%%%%%%%%%%%%%%%%%%%%%%

We treat the case of non-connected cobordisms via disjoint unions and permutations of the factors of disjoint unions, following Kock's work~\cite{K} for the case of ordinary 2-cobordisms. Since every permutation can be written as a product of transpositions, the following singular cobordisms are sufficient to do this:
\begin{equation}\raisebox{-10pt}{\includegraphics[height=0.35in]{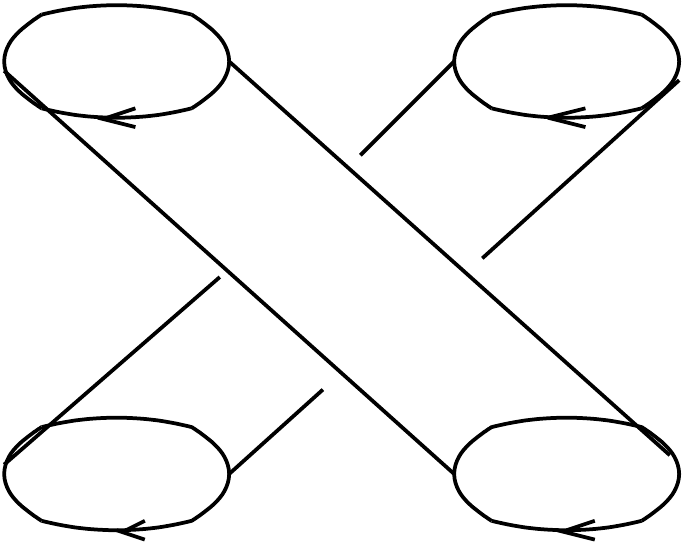}}\quad \raisebox{-10pt}{\includegraphics[height=0.35in]{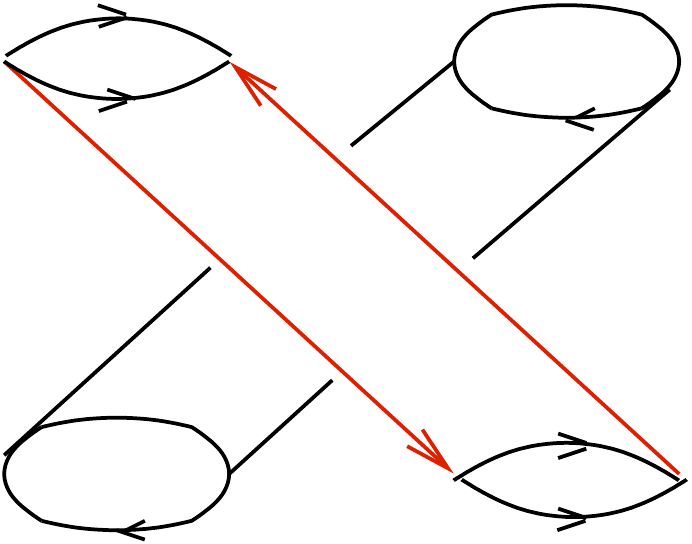}} \quad
\raisebox{-10pt}{\includegraphics[height=0.35in]{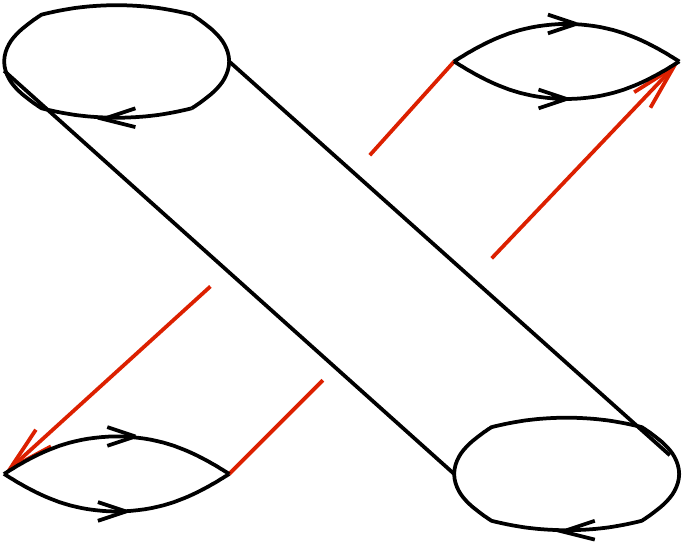}} \quad \raisebox{-10pt}{\includegraphics[height=0.35in]{braiding_WW.pdf}}\end{equation}

There is no need to talk about crossing over or under, since our cobordisms are abstract manifolds, thus not embedded anywhere. 

Without loss of generality, we assume that $\Sigma \co \textbf{n} \to \textbf{m}$ has two connected components, $\Sigma_1$ and $\Sigma_2,$ and that $\textbf{n} = (n_1, n_2, \cdots, n_{\vert n \vert}).$ The source boundary of $\Sigma_1$ is a tuple $\textbf{p}$ whose components form a subset of $\{n_1, n_2, \cdots, n_{\vert n \vert} \},$ and the source boundary of $\Sigma_2$ is the tuple $\textbf{q},$ which is the complement of $\textbf{p}$ in $\{n_1, n_2, \cdots, n_{\vert n \vert} \}.$   

We can permute the components of $\textbf{n}$ by applying a diffeomorphism $\textbf{n} \to \textbf{n},$ so that the components of $\textbf{p}$ come before those of $\textbf{q}$. This diffeomorphism induces a cobordism $S,$ and we can consider the singular cobordism $S \Sigma.$  Applying the same method to the target boundary of $\Sigma,$ which is also the target boundary of $S \Sigma,$ there is a permutation singular cobordism $T \co \textbf{m} \to \textbf{m}$ so that $ \Sigma' = S \Sigma T \co \textbf{n} \to \textbf{m}$ is a singular cobordism which is the disjoint union (as a cobordism) of $\Sigma_1$ and $\Sigma_2.$ Then $\Sigma \cong S^{-1} \Sigma' T^{-1},$ where $S^{-1}$ and $T^{-1}$ are the permutation cobordisms which are the inverses of $S$ and $T,$ respectively. For example, $S^{-1}$ is the diffeomorphism that permutes the components of $\textbf{n}$ such that the components of $\textbf{p}$ come after those of $\textbf{q}.$

As an example, we consider the following singular cobordism:

\begin{equation*} \Sigma \hspace{1cm} \raisebox{-5pt}{\includegraphics[height=0.45in]{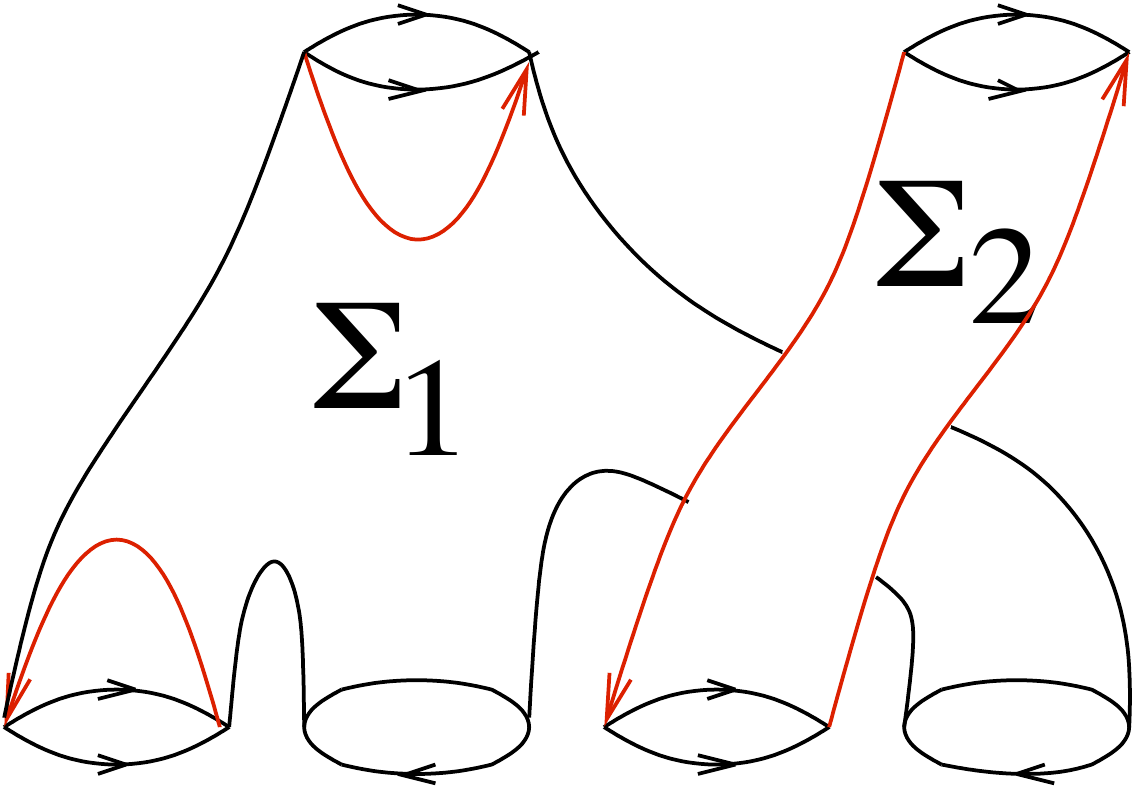}}\end{equation*}

For the given cobordism we don't need to permute the source boundary of $\Sigma,$ thus $S$ is the disjoint union (as a cobordism) of two cylinders, but we do permute the target boundary of $\Sigma$ by composing with a cobordism $T.$ The composed cobordism $S \Sigma T$ is the disjoint union of its connected components $(S \Sigma T)_1$ and $(S \Sigma T)_2$: 

\begin{equation*}\includegraphics[height=1in]{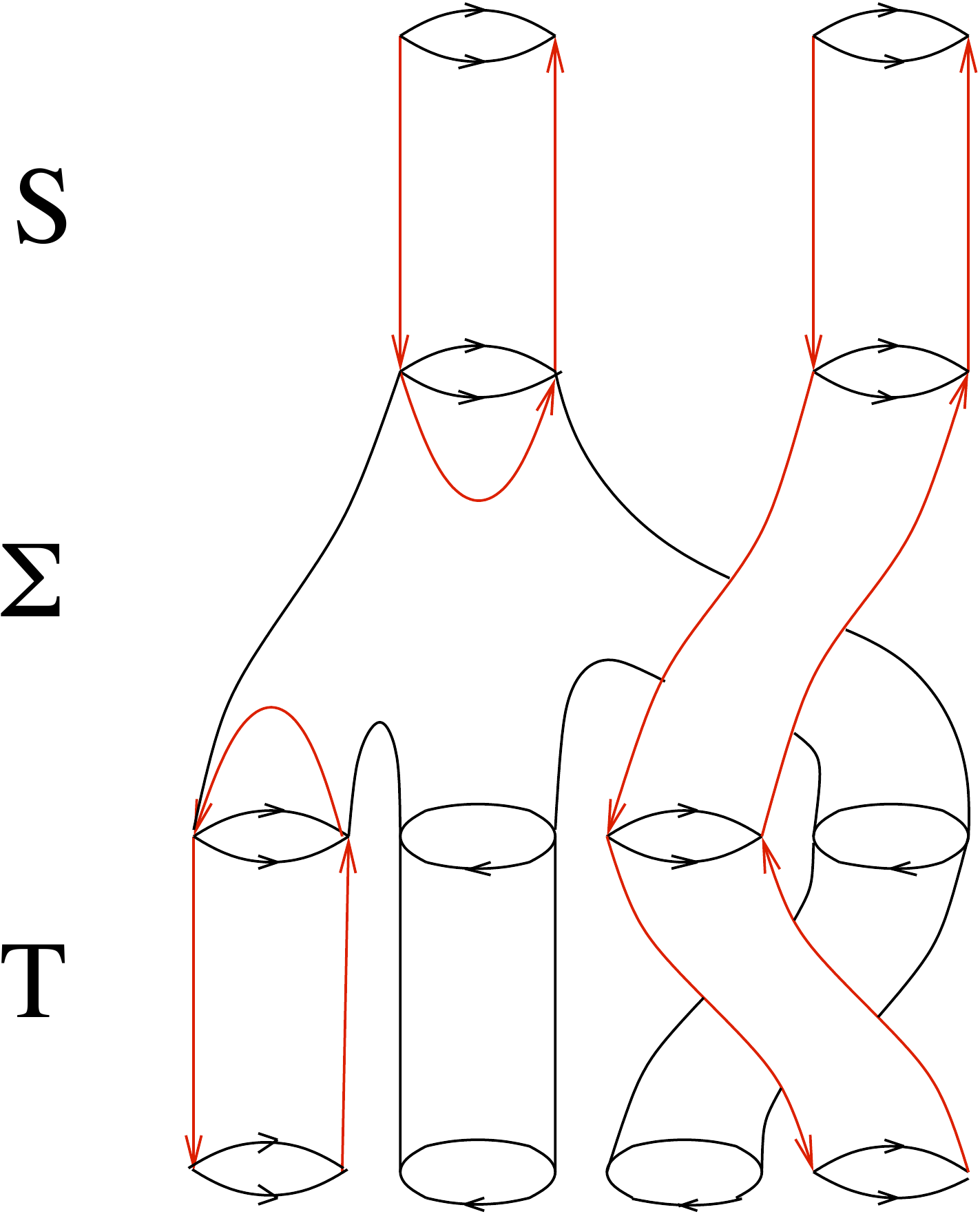} \hspace {1cm} \raisebox{20pt}{\includegraphics[height=.55in]{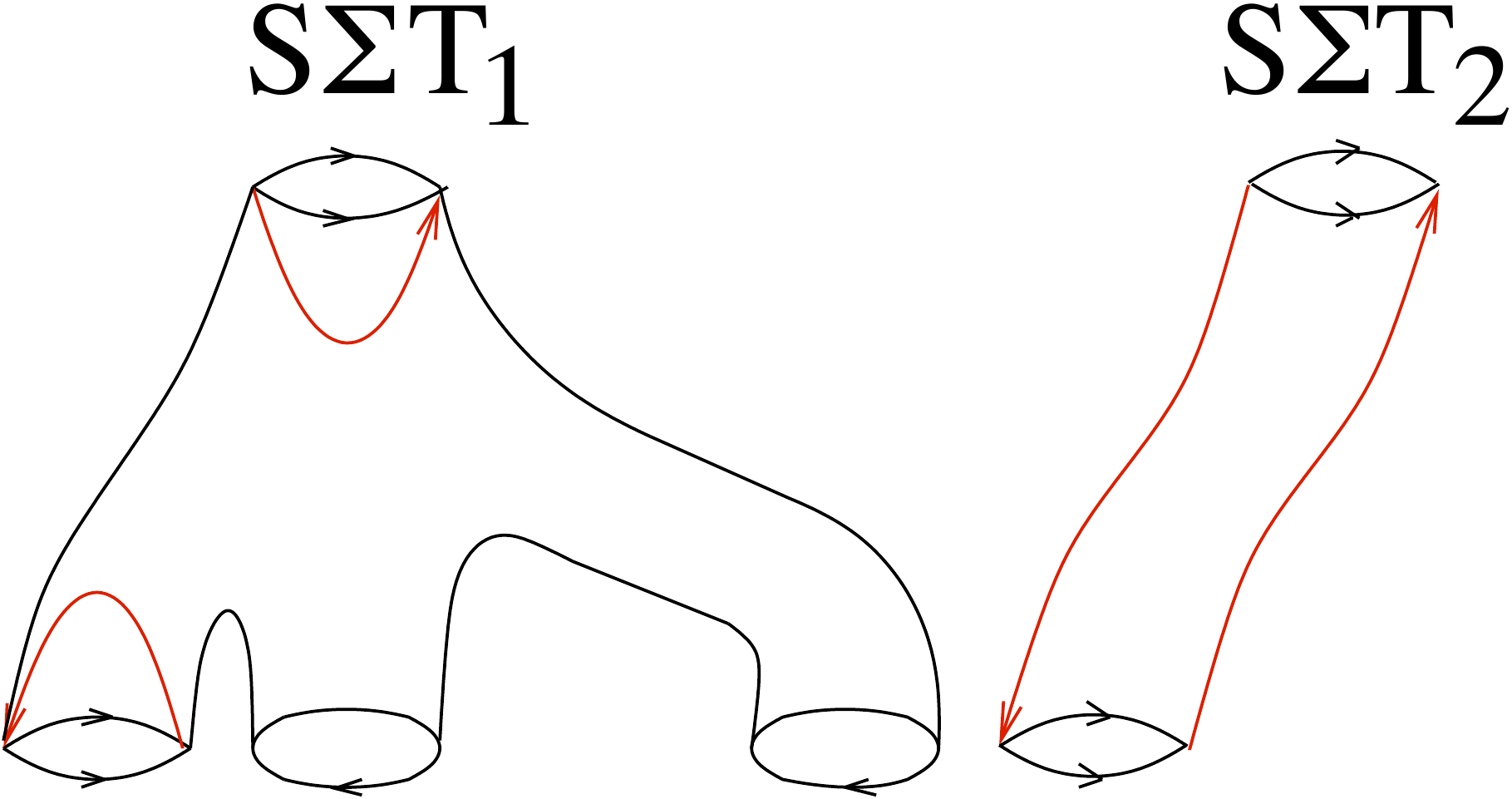}}\end{equation*}

 We have proved the following:
 
\begin{lemma}
Every singular cobordism is equivalent to a composition of a permutation cobordism with a disjoint union of connected cobordisms, followed by a permutation cobordism.
\end{lemma}

%%%%%%%%%%%%%%%%%%%%%%%%%%%%%%%%%%%%%%%%%%%%%%%%%%%
\subsection{Sufficiency of the relations}\label{sec:sufficiency}
%%%%%%%%%%%%%%%%%%%%%%%%%%%%%%%%%%%%%%%%%%%%%%%%%%%

In this subsection, we show that the relations described in Proposition~\ref{prop:relations} are sufficient in order to relate any connected singular cobordism $[\Sigma] \in \textbf{Sing-2Cob}_{W\to C}(\textbf{n}, \textbf{m})$ to its normal form $\mbox{NF}_{W \to C}(\Sigma).$

We use the notation \raisebox{-3pt}{\includegraphics[height=0.15in]{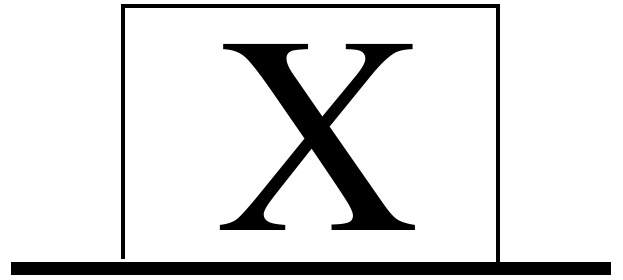}} (or \raisebox{-3pt}{\includegraphics[height=0.15in]{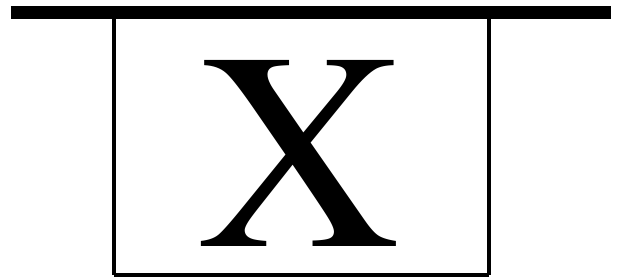}}) for an arbitrary singular cobordism $X$ whose target (or source) is not glued to any other cobordism in the decomposition of $\Sigma.$

The following terminology is borrowed from~\cite[Definition 3.21]{LP}.
\begin{definition}
Let $[\Sigma] \in \textbf{Sing-2Cob}(\textbf{n}, \textbf{m})$ be connected. The \textit{height} of a generator  $G$ in the decomposition of $\Sigma$ is the following number defined inductively:
\begin{align*}
&h(\raisebox{-3pt}{\includegraphics[height=0.15in]{counit.pdf}}) = h(\raisebox{-3pt}{\includegraphics[height=0.15in]{singcounit.pdf}}) = h(\raisebox{-3pt}{\includegraphics[height=0.15in]{height0.pdf}}): = 0,\\
& h(\,\raisebox{-10pt}{\includegraphics[height=0.25in]{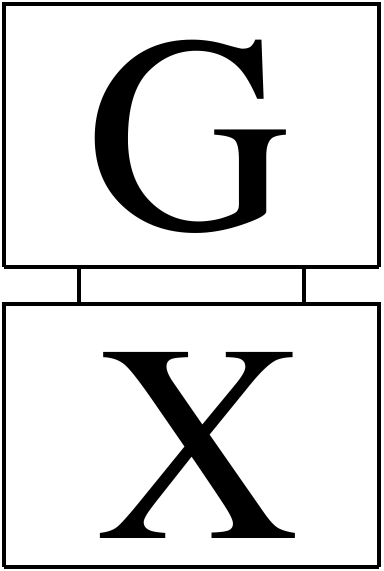}}\,) = 1+ h(X), \\
& h(\,\raisebox{-10pt}{\includegraphics[height=0.25in]{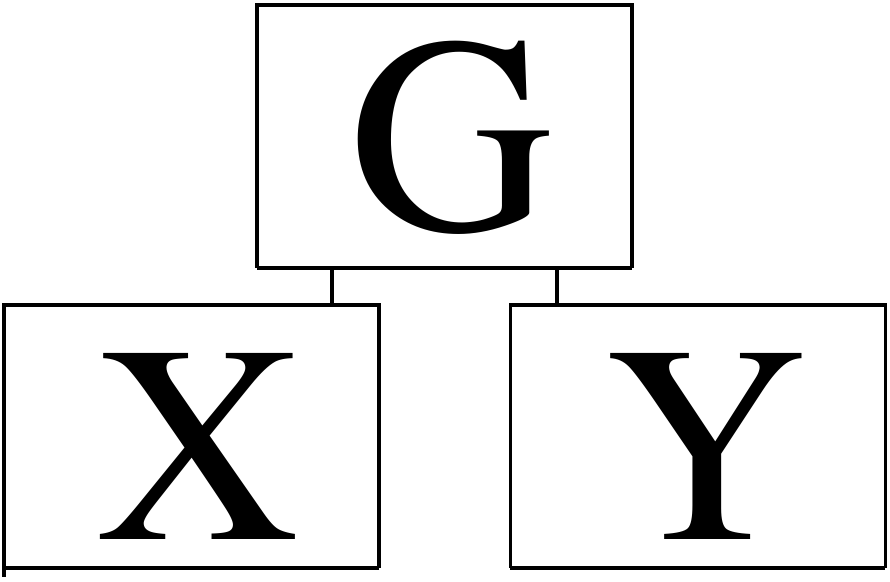}}\,) = 1 + h(X) + h(Y), 
 \end{align*}
where $X$ and $Y$ are arbitrary cobordisms in the decomposition of $\Sigma.$ 
\end{definition}

\begin{theorem}\label{thm:sufficiency of relations}
Let $[\Sigma] \in \textbf{Sing-2Cob}_{W\to C}(\textbf{n}, \textbf{m})$ be a connected singular $2$-cobordism. Then $\Sigma$ is equivalent to its normal form, namely we have
\[ [\Sigma] = [ \mbox{NF}_{W \to C}(\Sigma)].  \]
\end{theorem}

\begin{proof}The proof is similar in spirit to that of~\cite[Theorem 3.22]{LP}, with the difference that it uses our cobordisms and their topological structure.
We consider $\Sigma$ be given in an arbitrary decomposition and construct a step by step diffeomorphism (relative to boundary) from this decomposition to the normal form $\mbox{NF}_{W \to C}(\Sigma).$
\begin{enumerate}

\item [I.] The decomposition of $\Sigma$ is equivalent to one without \textit{singular cups} \raisebox{-3pt}{\includegraphics[height=0.15in]{singcounit.pdf}} and \textit{singular caps} \raisebox{-3pt}{\includegraphics[height=0.15in]{singunit.pdf}}, by applying the following diffeomorphism: \newline
\[\raisebox{-3pt}{\includegraphics[height=0.15in]{singcounit.pdf}} \stackrel{\eqref{eq:cozipper_coalghom}  }{\longrightarrow} \raisebox{-8pt}{\includegraphics[height=0.35in]{cozipper_hom2.pdf}} \hspace{.5cm} \text{and} \hspace{.5cm} 
\raisebox{-3pt}{\includegraphics[height=0.15in]{singunit.pdf}} \stackrel{\eqref{eq:zipper_alghom}  }{\longrightarrow} \raisebox{-8pt}{\includegraphics[height=0.35in]{zipper_hom2.pdf}}\]
\bigbreak

\item [II.] The decomposition of $\Sigma$ is equivalent to one in which every \textit{singular comultiplication} \raisebox{-8pt}{\includegraphics[height=0.23in]{singcomult.pdf}} has its target in one of the following situations:
\begin{equation}
\psset{xunit=.22cm,yunit=.22cm}
\begin{pspicture}(5,5)
 \rput(1.5, 1.85){\includegraphics[height=0.23in]{singcomult.pdf}}
 \rput(2.7,4){\includegraphics[height=0.23in]{singcomult.pdf}}
 \end{pspicture}
 \qquad
 \begin{pspicture}(5,5)
  \rput(2.5, -0.3){?}
 \rput(3.9, -2.3){\includegraphics[height=0.23in]{singmult.pdf}}
 \rput(5, -0.2){\includegraphics[height=0.23in]{identity_web.pdf}}
  \rput(5, 1.9){\includegraphics[height=0.23in]{identity_web.pdf}}
 \rput(1.5, 1.8){\includegraphics[height=0.23in]{singmult.pdf}}
 \rput(2.7,4){\includegraphics[height=0.23in]{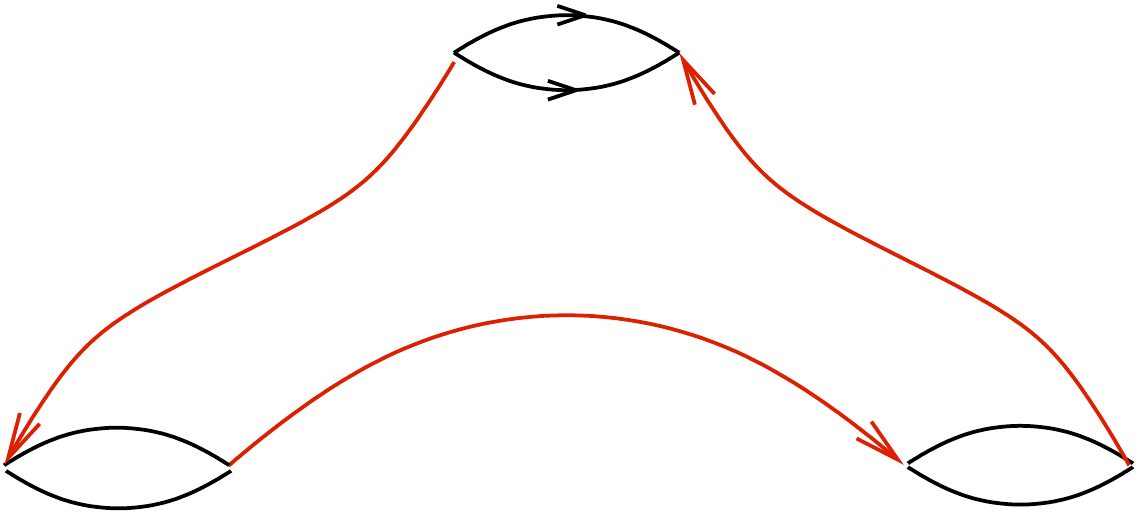}}
  \end{pspicture}
  \qquad
 \begin{pspicture}(5,5)
  \rput(2.9, -0.3){?}
  \rput(1.6, -2.4){\includegraphics[height=0.23in]{singmult.pdf}}
  \rput(0.42, -0.25){\includegraphics[height=0.23in]{identity_web.pdf}}
  \rput(0.42, 1.85){\includegraphics[height=0.23in]{identity_web.pdf}}
 \rput(3.9, 1.85){\includegraphics[height=0.23in]{singmult.pdf}}
 \rput(2.7,4){\includegraphics[height=0.23in]{huge_singcomult.pdf}}
   \end{pspicture} \label{eq:normalform_scomult}
 \end{equation}
  
  \noindent or 
  
  \begin{equation}
   \raisebox{-10pt}{\includegraphics[height=0.41in]{sing_genus_one.pdf}}
  \stackrel{\mbox{Def}}{=}
\raisebox{-10pt}{\includegraphics[height=0.41in]{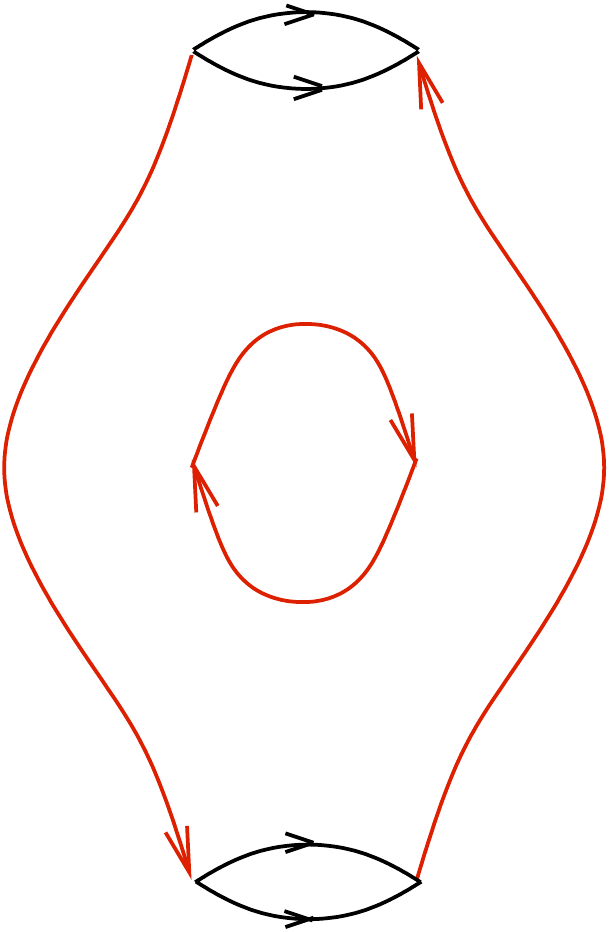}}\label{eq:sgenus_one}
\end{equation}
\noindent where the symbol $``?"$ is any singular cobordism which may or may not be connected to the singular multiplication at the bottom of the diagram. This is proved by considering every possible situation in which the singular comultiplication may appear.

\noindent a) We can exclude the cases $\psset{xunit=.22cm,yunit=.22cm}
\begin{pspicture}(5,5)
 \rput(2.5, 0.85){\includegraphics[height=0.23in]{singcomult.pdf}}
 \rput(1.35,-0.55){\includegraphics[height=0.1in]{singcounit.pdf}}
 \end{pspicture}$
 and 
$ 
\psset{xunit=.22cm,yunit=.22cm}\begin{pspicture}(5,5)
 \rput(2.5, 0.85){\includegraphics[height=0.23in]{singcomult.pdf}}
 \rput(3.66,-0.55){\includegraphics[height=0.1in]{singcounit.pdf}}
 \end{pspicture}$ by step (I).\newline
 b) Apply the diffeomorhism $ 
 \psset{xunit=.22cm,yunit=.22cm}
 \begin{pspicture}(5,5)
 \rput(3.85, -0.8){\includegraphics[height=0.23in]{singcomult.pdf}}
 \rput(2.7,1.3){\includegraphics[height=0.23in]{singcomult.pdf}}
 \end{pspicture}
\stackrel{\eqref{eq:web_frob2}}{\longrightarrow}\psset{xunit=.22cm,yunit=.22cm}\begin{pspicture}(5,5)
 \rput(1.55, -0.8){\includegraphics[height=0.23in]{singcomult.pdf}}
 \rput(2.7,1.3){\includegraphics[height=0.23in]{singcomult.pdf}}
 \end{pspicture}$ whenever possible.\vspace{.6cm}

\noindent c) The following diffeomorhisms reduce the height of the singular comultiplication:
 \begin{enumerate}
 \bigbreak
 \item [i)]
$\raisebox{-13pt}{\includegraphics[height=0.4in]{web_frob10.pdf}} \stackrel{\eqref{eq:web_frob3}}{\longrightarrow} \raisebox{-13pt}{\includegraphics[height=0.4in]{web_frob11.pdf}}  \stackrel{\eqref{eq:web_frob3}}{\longleftarrow} \raisebox{-13pt}{\includegraphics[height=0.4in]{web_frob12.pdf}}$
\bigskip

 \item [ii)] $ \psset{xunit=.2cm,yunit=.2cm}
 \begin{pspicture}(5,5)
 \rput(3.3, 3.45){\includegraphics[height=0.23in]{singcomult.pdf}}
 \rput(2,0){\includegraphics[height=0.4in]{sgenus_one.pdf}}
 \end{pspicture}
\stackrel{\eqref{eq:sing_genus_one_comult}}{\longrightarrow}\psset{xunit=.2cm,yunit=.2cm}\begin{pspicture}(5,5)
 \rput(2, 2.4){\includegraphics[height=0.4in]{sgenus_one.pdf}}
 \rput(2,-1){\includegraphics[height=0.23in]{singcomult.pdf}}
 \end{pspicture}
 \stackrel{\eqref{eq:sing_genus_one_comult}}{\longleftarrow}
  \begin{pspicture}(5,5)
 \rput(1.6, 3.45){\includegraphics[height=0.23in]{singcomult.pdf}}
 \rput(2.9,0){\includegraphics[height=0.4in]{sgenus_one.pdf}}
 \end{pspicture}$
 \end{enumerate}
\vspace{.6cm}

\noindent d) The following diffeomorphisms eliminate the singular comultiplication:

 \begin{enumerate}
 \bigbreak
 \item [i)] $\raisebox{-25pt}{\includegraphics[height=0.7in]{genus1_rel_sing.pdf}} \stackrel{\eqref{eq:genus_one}}{\longrightarrow} \raisebox{-25pt}{\includegraphics[width = 0.35in, height=0.7in]{genus1_rel.pdf}}$
 \item [ii)]  $ \psset{xunit=.2cm,yunit=.2cm}
 \begin{pspicture}(5,5)
 \rput(3.3, 2.15){\includegraphics[height=0.28in]{singcomult.pdf}}
 \rput(1.75,-0.53){\includegraphics[height=0.25in]{cozipper.pdf}}
 \end{pspicture}
\quad \stackrel{\eqref{eq:singcomult_equiv}}{\longrightarrow}
\psset{xunit=.2cm,yunit=.2cm}
 \begin{pspicture}(5,5)
 \rput(3, 1.75){\includegraphics[height=0.2in]{sing_copairing1}}
 \rput(6, -0.6){\includegraphics[height=0.28in]{singmult.pdf}}
 \rput(1.56,-0.35){\includegraphics[height=0.23in]{cozipper.pdf}}
 \end{pspicture}
 \qquad \stackrel{\eqref{eq:cozipper_copairing}}{\longrightarrow}
\psset{xunit=.2cm,yunit=.2cm}
 \begin{pspicture}(5,5)
 \rput(3, 2.9){\includegraphics[height=0.22in]{copairing1}}
 \rput(6, -1.6){\includegraphics[height=0.28in]{singmult.pdf}}
 \rput(4.45,0.9){\includegraphics[height=0.23in]{zipper.pdf}}
 \end{pspicture} \hspace{1cm} \mbox{and}$
\vspace{.6cm}

$ \psset{xunit=.2cm,yunit=.2cm}
 \begin{pspicture}(5,5)
 \rput(3.3, 2.15){\includegraphics[height=0.28in]{singcomult.pdf}}
 \rput(4.9,-0.53){\includegraphics[height=0.25in]{cozipper.pdf}}
 \end{pspicture}
\quad \stackrel{\eqref{eq:singcomult_equiv}}{\longrightarrow}
\psset{xunit=.2cm,yunit=.2cm}
 \begin{pspicture}(5,5)
 \rput(6, 1.75){\includegraphics[height=0.2in]{sing_copairing1}}
 \rput(3, -0.6){\includegraphics[height=0.28in]{singmult.pdf}}
 \rput(7.45,-0.35){\includegraphics[height=0.23in]{cozipper.pdf}}
 \end{pspicture}
 \qquad \stackrel{\eqref{eq:cozipper_copairing}}{\longrightarrow}
\psset{xunit=.2cm,yunit=.2cm}
 \begin{pspicture}(5,5)
 \rput(5.9, 2.9){\includegraphics[height=0.22in]{copairing1}}
 \rput(2.9, -1.6){\includegraphics[height=0.28in]{singmult.pdf}}
 \rput(4.45,0.9){\includegraphics[height=0.23in]{zipper.pdf}}
 \end{pspicture} $
 \end{enumerate}
 \vspace{.8cm}
 e) Iterate steps (IIa)-(IId). Since each step either removes the singular comultiplication or reduces its height, and since the target of $\Sigma$ does not contain bi-webs, this process terminates with every singular comultiplication in one of the situations described above.  
\bigbreak
 
 \item [III.] We look now at the possible cases in which the source and target of the \textit{singular multiplication} \raisebox{-8pt}{\includegraphics[height=0.23in]{singmult.pdf}} may appear. 
 
\noindent We remark first that we don't need to consider the case \raisebox{-8pt}{\includegraphics[height=0.4in]{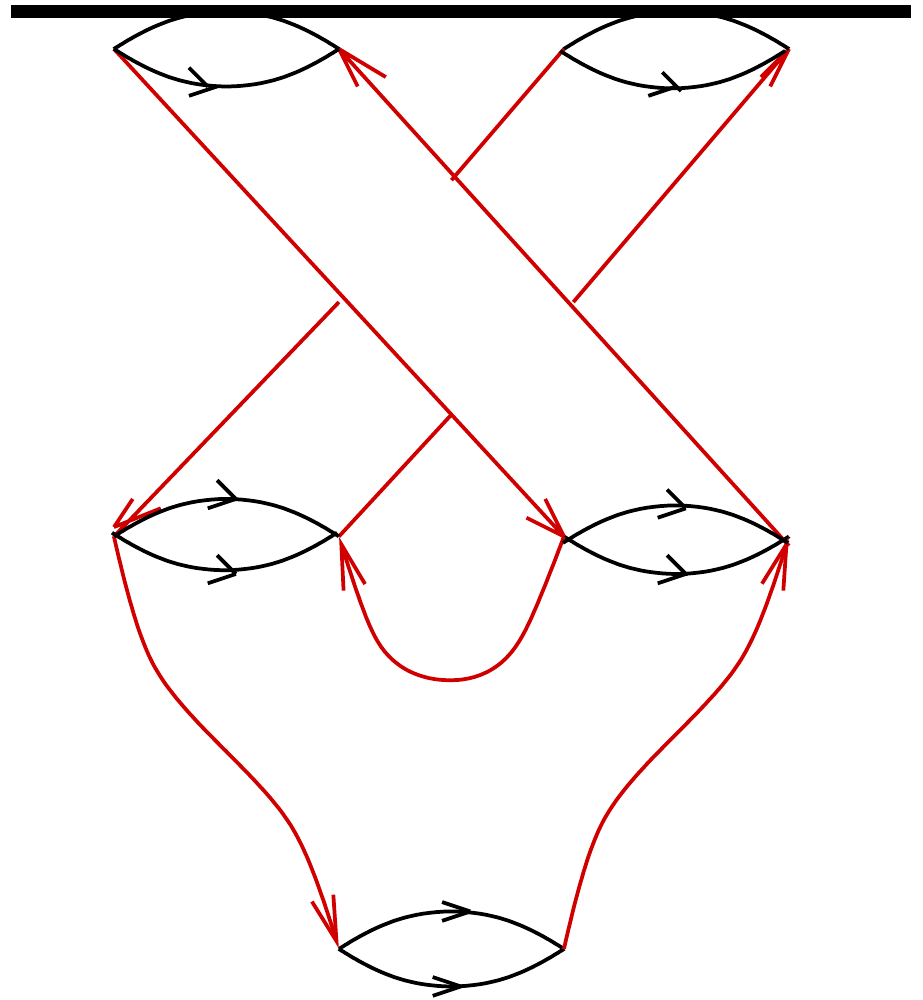}}, since the part $\Sigma_{\overline{\sigma(\Sigma)}}$ in the definition of the normal form of $\Sigma$ (see Equation \eqref{eq:normalform}) is taking care of the twist cobordism \raisebox{-8pt}{\includegraphics[height=0.25in]{braiding_WW.pdf}}. Similarly, the cases 

%%%%%%%%%%%%%%%%%%%%%%%%%%%%%%%%%%%%%%%
 \[ \psset{xunit=.22cm,yunit=.22cm}
  \begin{pspicture}(5,5)
 \rput(3.5, -0.8){\includegraphics[height=0.23in]{singmult.pdf}}
 \rput(3.5,1.3){\includegraphics[height=0.23in]{braiding_WW.pdf}}
 \rput(4,3.4){\includegraphics[height=0.23in]{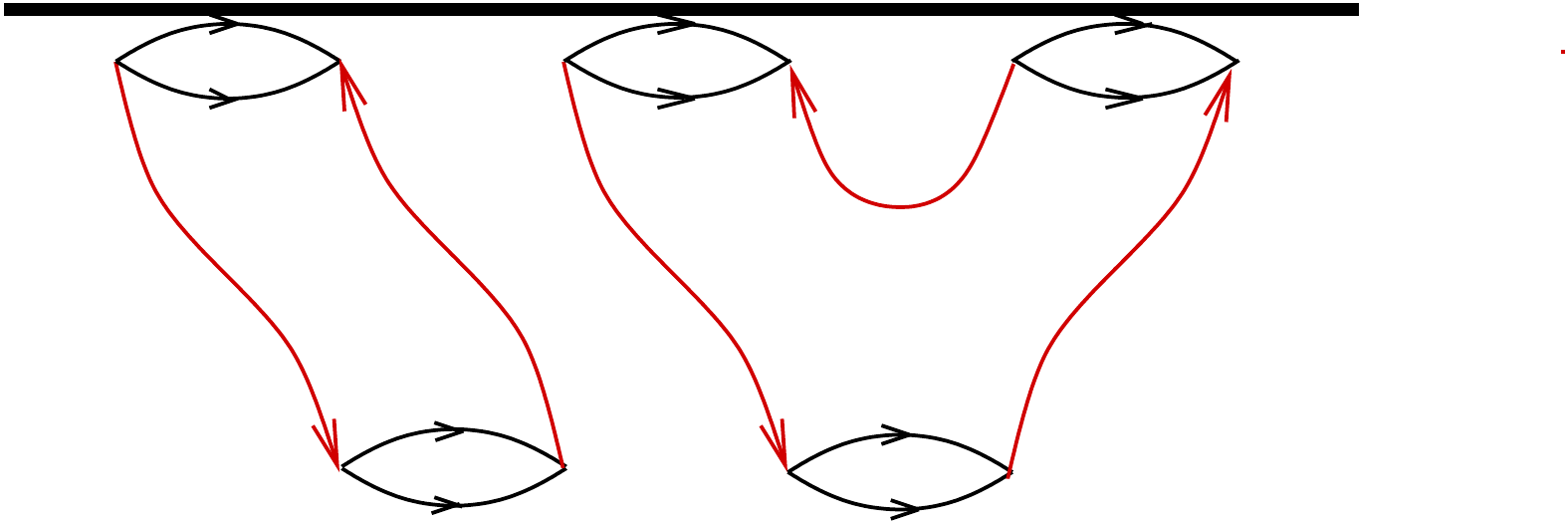}}
 \end{pspicture} \qquad
 \psset{xunit=.22cm,yunit=.22cm}
  \begin{pspicture}(5,5)
 \rput(3.5, -0.8){\includegraphics[height=0.23in]{singmult.pdf}}
 \rput(3.5,1.3){\includegraphics[height=0.23in]{braiding_WW.pdf}}
 \rput(3.5,3.4){\includegraphics[height=0.23in]{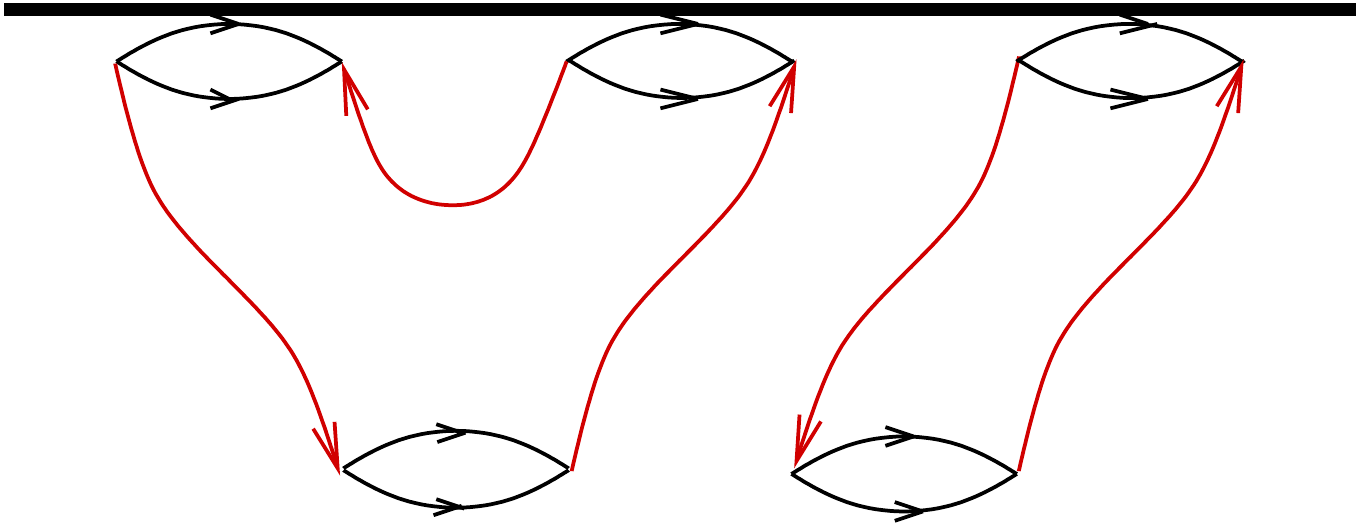}}
 \end{pspicture} \qquad
  \psset{xunit=.22cm,yunit=.22cm}
  \begin{pspicture}(5,5)
 \rput(3.5, -0.8){\includegraphics[height=0.23in]{singmult.pdf}}
 \rput(3.5,1.3){\includegraphics[height=0.23in]{braiding_WW.pdf}}
 \rput(1.8,3.4){\includegraphics[height=0.23in]{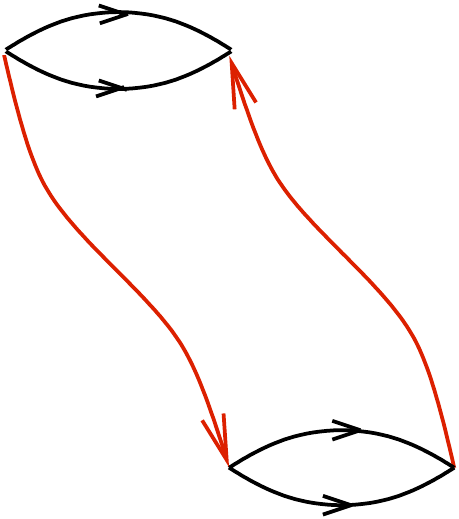}}
 \rput(5.2,3.4){\includegraphics[height=0.23in]{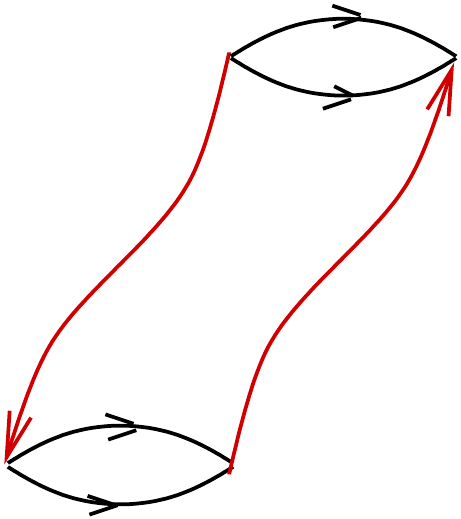}}
 \rput(1.3,5.5){\includegraphics[height=0.23in]{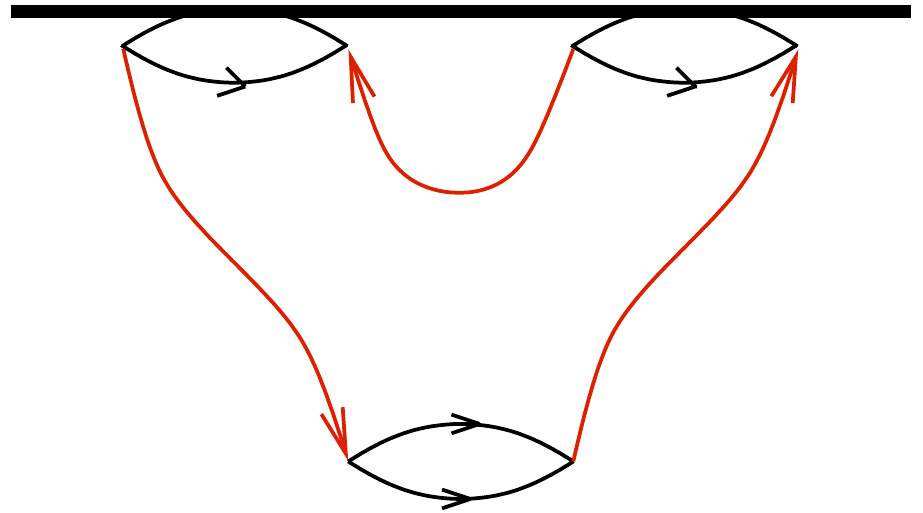}}
 \rput(5.7,5.5){\includegraphics[height=0.23in]{nform_smult1.pdf}}
  \end{pspicture}
 \]
 
\vspace{0.5cm}
 \noindent can be excluded by moving upwards the twist past the singular multiplication(s) (that is, by the naturality of the twist). In each case, the corresponding diffeomorphism yields a cobordism equivalent to $\Sigma$, whose decomposition starts with a permutation cobordism, which will be reflected in the  $\Sigma_{\overline{\sigma(\Sigma)}}$ part of the normal form. 
 
\noindent More generally, we use this method whenever the source of 
 $ \psset{xunit=.22cm,yunit=.22cm}
  \begin{pspicture}(3,4)
 \rput(1.5, 1.2){\includegraphics[height=0.23in]{singmult.pdf}}
 \rput(1.5,3.3){\includegraphics[height=0.23in]{braiding_WW.pdf}}
 \end{pspicture}$
has only singular multiplications and identities above it. If there are also singular comultiplications, we apply step IV below. Therefore, we have:
 %%%%%%%%%%%%%%%%%%%%%%%%%%%%%%%%%%%%%%%%
 
 \noindent a) The decomposition of $\Sigma$ is equivalent to one in which the source of the singular multiplication appears in one of the following situations:
 \[ \raisebox{-3pt}{\includegraphics[height=0.25in]{nform_smult1.pdf}} \quad \raisebox{-8pt}{\includegraphics[height=0.47in]{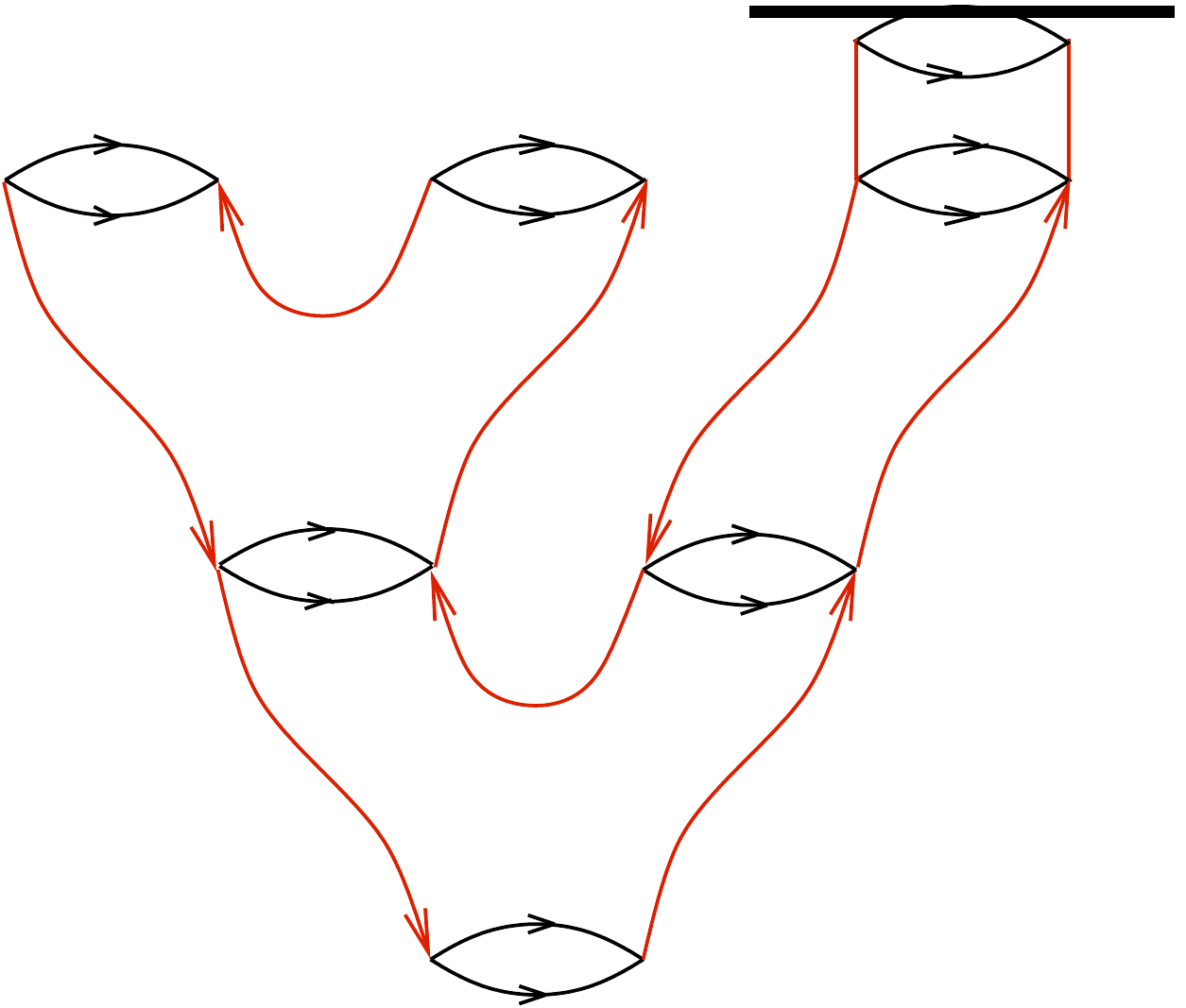}}\quad \raisebox{-10pt}{\includegraphics[height=0.55in]{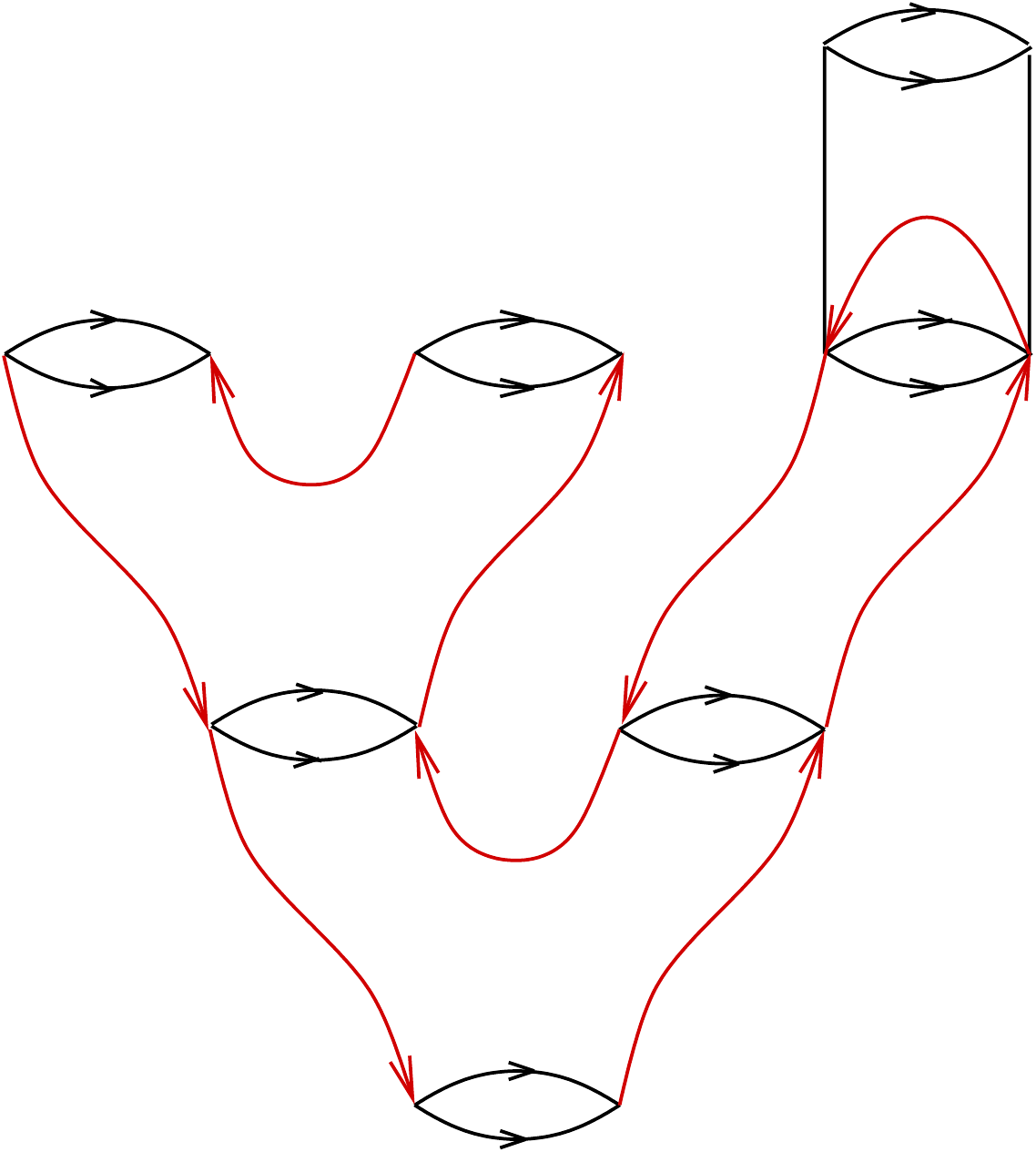}} \quad \raisebox{-8pt}{\includegraphics[height=0.46in]{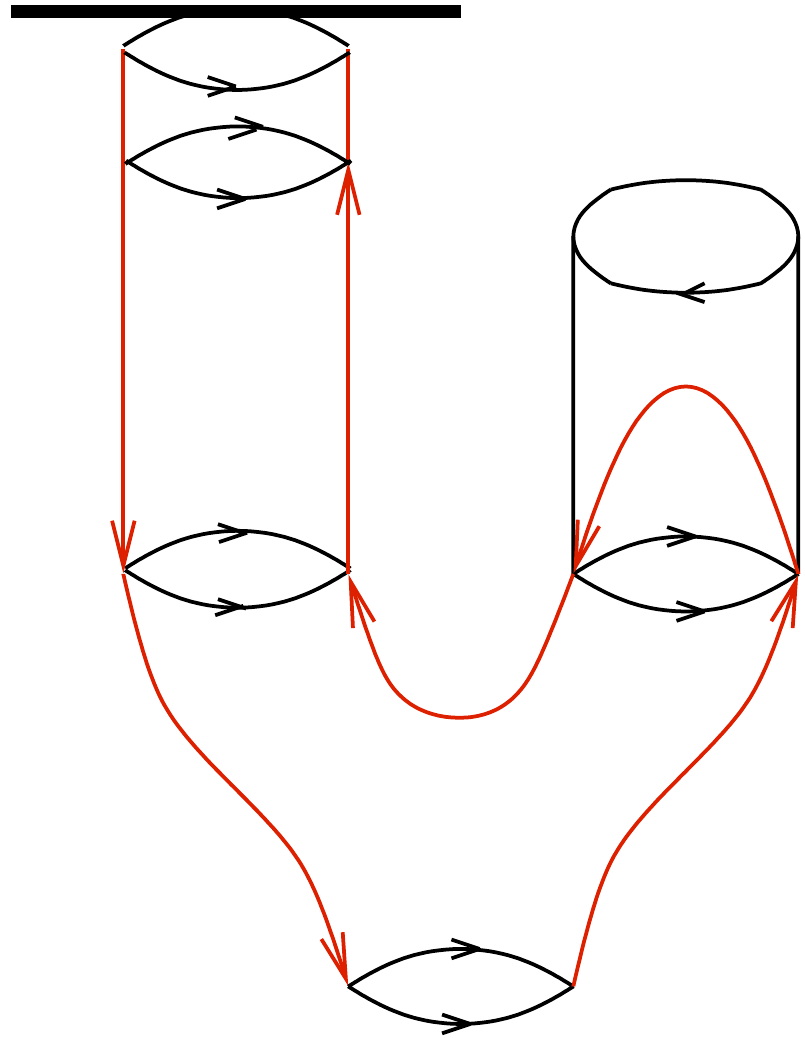}} \quad \raisebox{-8pt}{\includegraphics[height=0.46in]{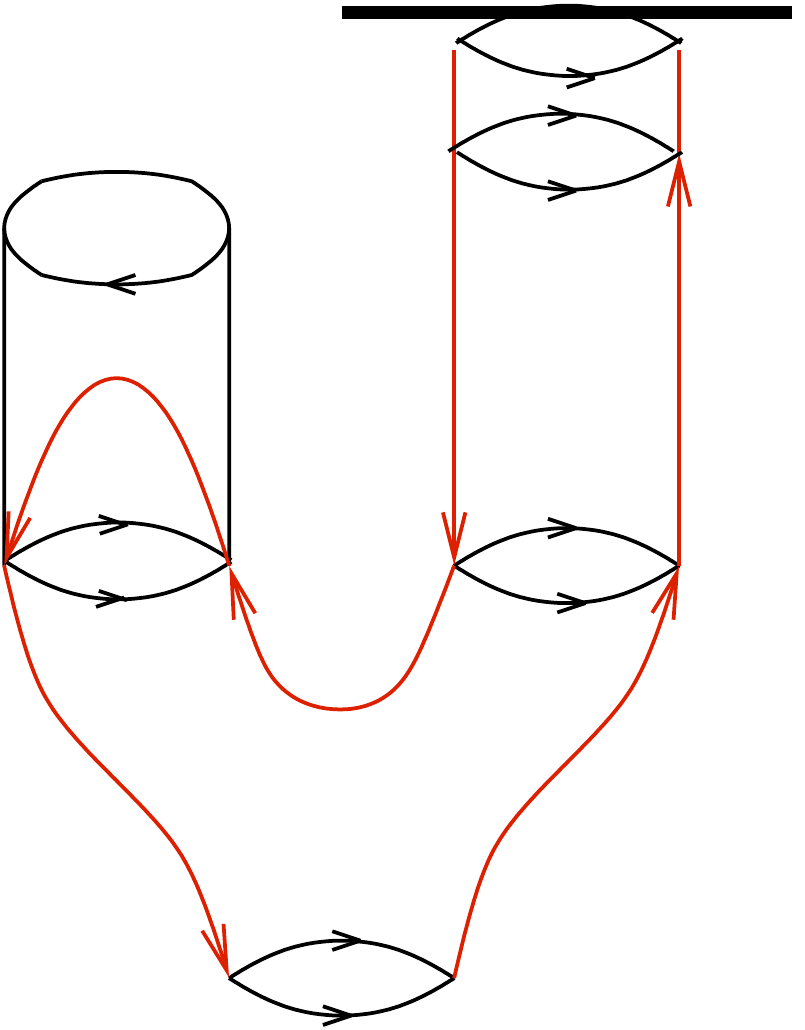}}
  \psset{xunit=.22cm,yunit=.22cm}\begin{pspicture}(5,5)
 \rput(2.5, 0){\includegraphics[height=0.23in]{singmult.pdf}}
  \rput(3.7, 2.1){\includegraphics[height=0.23in]{identity_web.pdf}}
    \rput(1, 2.1){?}
  \rput(2.5,4.2){\includegraphics[height=0.23in]{singcomult.pdf}}
 \end{pspicture}
  \psset{xunit=.22cm,yunit=.22cm}\begin{pspicture}(5,5)
 \rput(2.5, 0){\includegraphics[height=0.23in]{singmult.pdf}}
  \rput(3.9, 2.1){?}
    \rput(1.3, 2.1){\includegraphics[height=0.23in]{identity_web.pdf}}
  \rput(2.5,4.2){\includegraphics[height=0.23in]{singcomult.pdf}}
 \end{pspicture}
 \]
 or that depicted in~\eqref{eq:sgenus_one}. One can see that this claim holds by considering all possible situations of the singular comultiplication and then applying the following diffeomorphisms. Each iteration of the following steps either removes the singular multiplication or increases its height. As a result, each singular multiplication ends up into one of the situations above.
 
\begin{enumerate}
\item [i)]  $ 
 \psset{xunit=.22cm,yunit=.22cm}\begin{pspicture}(5,5)
 \rput(1.55, -0.8){\includegraphics[height=0.23in]{singmult.pdf}}
 \rput(2.7,1.3){\includegraphics[height=0.23in]{singmult.pdf}}
 \end{pspicture}
 \stackrel{\eqref{eq:web_frob1}}{\longrightarrow}
 \psset{xunit=.22cm,yunit=.22cm}
 \begin{pspicture}(5,5)
 \rput(3.85, -0.8){\includegraphics[height=0.23in]{singmult.pdf}}
 \rput(2.7,1.3){\includegraphics[height=0.23in]{singmult.pdf}}
 \end{pspicture}$
 \vspace{.5cm}

 \item [ii)] $ \raisebox{-13pt}{\includegraphics[height=0.4in]{zipper_hom1_left.pdf}}\stackrel{\eqref{eq:zipper_alghom}}{\longrightarrow}\raisebox{-13pt}{\includegraphics[height=0.4in]{zipper_hom1_right.pdf}} 
 $

  \item [iii)] $ \psset{xunit=.2cm,yunit=.2cm}
 \begin{pspicture}(5,5)
 \rput(3.27, -1.5){\includegraphics[height=0.23in]{singmult.pdf}}
 \rput(2,1.95){\includegraphics[height=0.4in]{sgenus_one.pdf}}
 \end{pspicture}
\stackrel{\eqref{eq:sing_genus_one_mult}}{\longrightarrow}\psset{xunit=.2cm,yunit=.2cm}\begin{pspicture}(5,5)
 \rput(2, -0.5){\includegraphics[height=0.4in]{sgenus_one.pdf}}
 \rput(2,2.95){\includegraphics[height=0.23in]{singmult.pdf}}
 \end{pspicture}
 \stackrel{\eqref{eq:sing_genus_one_mult}}{\longleftarrow}
  \begin{pspicture}(5,5)
 \rput(1.65, -1.5){\includegraphics[height=0.23in]{singmult.pdf}}
 \rput(2.9,1.95){\includegraphics[height=0.4in]{sgenus_one.pdf}}
 \end{pspicture}$
\vspace{.6cm}

\item [iv)] The diffeomorphisms of (IIc)i and (IId)i.
\end{enumerate}
\vspace{0.5cm}

\noindent b) The decomposition of $\Sigma$ is equivalent to one in which the target of the singular multiplication appears in one of the following situations:
\[ \psset{xunit=.22cm,yunit=.22cm}
 \begin{pspicture}(5,5)
 \rput(3.85, -0.8){\includegraphics[height=0.23in]{singmult.pdf}}
 \rput(2.7,1.3){\includegraphics[height=0.23in]{singmult.pdf}}
 \end{pspicture} \quad
 \psset{xunit=.22cm,yunit=.22cm}
 \begin{pspicture}(5,5)
 \rput(2.7, -0.62){\includegraphics[height=0.2in]{cozipper.pdf}}
 \rput(2.7,1.3){\includegraphics[height=0.23in]{singmult.pdf}}
 \end{pspicture} \quad
 \psset{xunit=.22cm,yunit=.22cm}
 \begin{pspicture}(5,5)
 \rput(2.7, -0.8){\includegraphics[height=0.23in]{singcomult.pdf}}
 \rput(2.7,1.3){\includegraphics[height=0.23in]{singmult.pdf}}
 \end{pspicture}  
 \]
 \bigbreak
 \begin{enumerate}
 \item [i)] We first show that the source of every cozipper \raisebox{-8pt}{\includegraphics[height=0.23in]{cozipper.pdf}} can be put so that it appears in one of the following situations:
 \[ \raisebox{-3pt}{\includegraphics[height=0.18in]{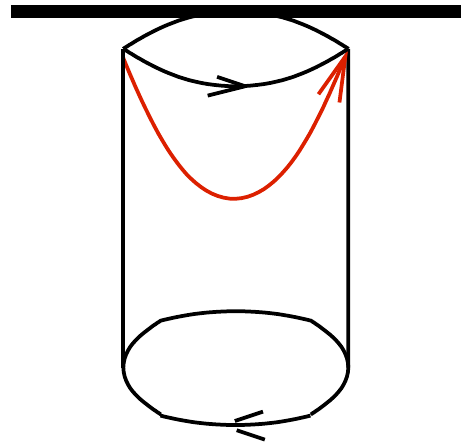}} \qquad \raisebox{-8pt}{\includegraphics[height=0.4in]{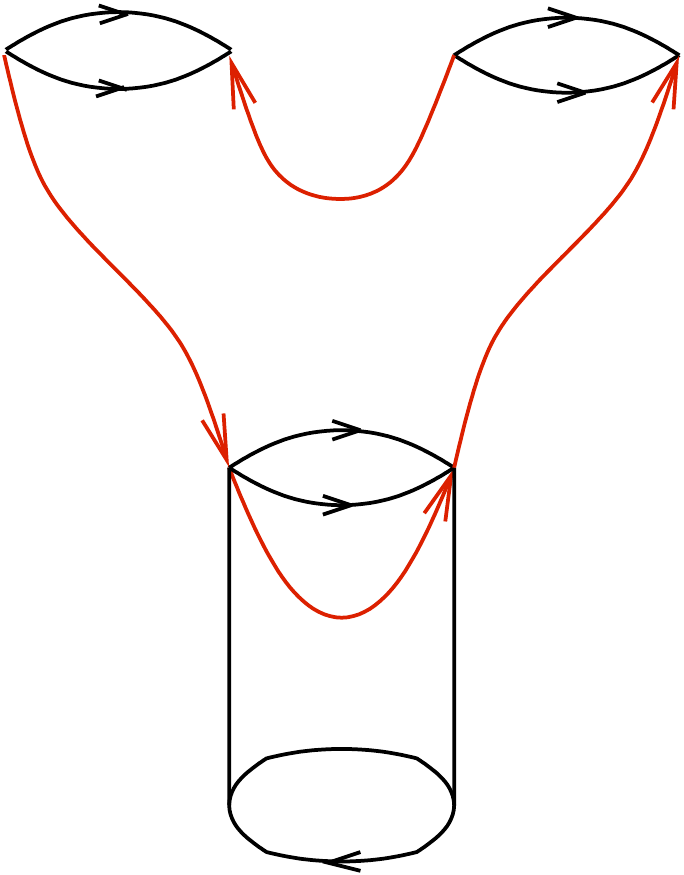}}\qquad \raisebox{-8pt}{\includegraphics[height=0.4in]{zipper_cozipper.pdf}} \stackrel{Def}{=} \raisebox{-8pt}{\includegraphics[height=0.4in]{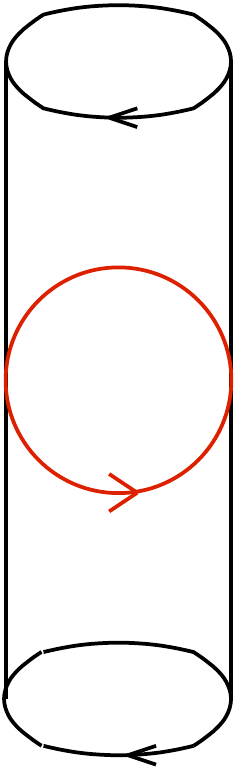}}\]
 
 This claim is proved by applying the steps (I), (IId)ii and the diffeomorphism: 
 \begin{equation}
 \psset{xunit=.2cm,yunit=.2cm}
 \begin{pspicture}(5,5)
 \rput(3.85, 2.2){\includegraphics[height=0.4in]{sgenus_one.pdf}}
 \rput(3.9,-1){\includegraphics[height=0.19in]{cozipper.pdf}}
 \end{pspicture} \,\, \stackrel{\eqref{eq:remove_sing_genusop}}{\longrightarrow}\,\,
\psset{xunit=.2cm,yunit=.2cm}
 \begin{pspicture}(5,5)
 \rput(1.85, 4.24){\includegraphics[height=0.19in]{cozipper.pdf}}
 \rput(1.85, 1.2){\includegraphics[height=0.4in]{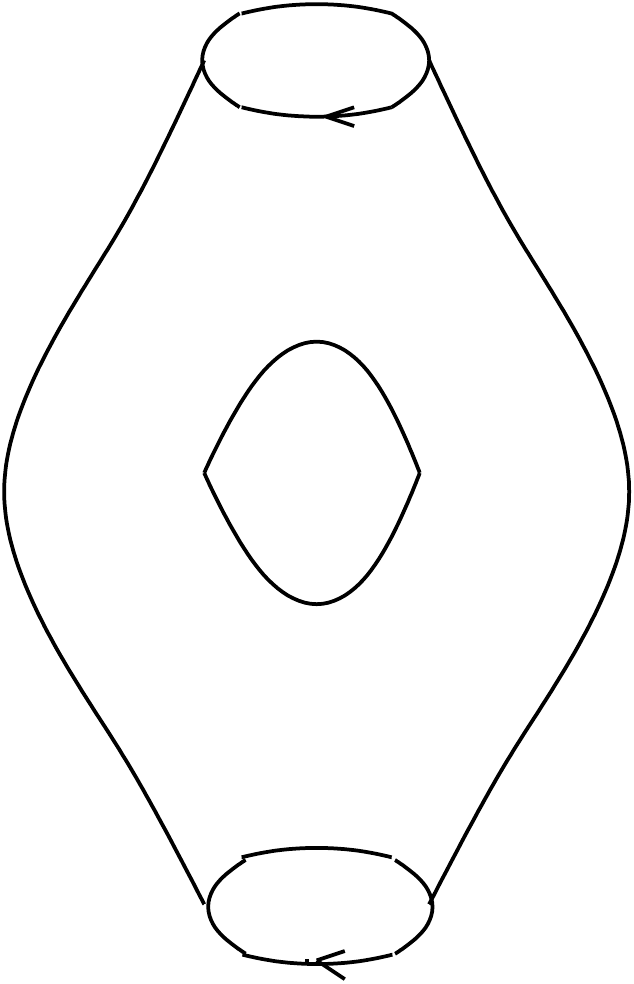}}
 \rput(1.9,-2.55){\includegraphics[height=0.3in]{nform_zipper_cozipper.pdf}}
 \end{pspicture} \label{eq:remove_sgenus_one}
 \end{equation}
\vspace{0.2cm}

 whenever it is possible.
 \item [ii)] The singular genus-one operator \raisebox{-8pt}{\includegraphics[height=0.3in]{sgenus_one.pdf}} can be removed by iterating the step (IIIa)iii, the diffeomorphism given in~\eqref{eq:remove_sgenus_one} and the following diffeomorphism: 
 \[ \psset{xunit=.2cm,yunit=.2cm}\begin{pspicture}(5,5)
 \rput(2, 2.4){\includegraphics[height=0.4in]{sgenus_one.pdf}}
 \rput(2,-1){\includegraphics[height=0.23in]{singcomult.pdf}}
 \end{pspicture}
 \stackrel{\eqref{eq:sing_genus_one_comult}}{\longrightarrow}
 \psset{xunit=.2cm,yunit=.2cm}
 \begin{pspicture}(5,5)
 \rput(3.3, 3.45){\includegraphics[height=0.23in]{singcomult.pdf}}
 \rput(2,0){\includegraphics[height=0.4in]{sgenus_one.pdf}}
 \end{pspicture}
\]
\vspace{0.2cm}

This process either reduces the height of the singular genus-one operator or removes it. Since the height of the operator cannot be zero, the process guarantees to remove the singular genus-one operator.
\end{enumerate}
 \bigbreak
 \item [IV.] In this step we show that there exists a sequence of diffeomorphisms that removes all singular comultiplications. Consider the set of all such comultiplications that appear in the decomposition of $\Sigma$ and choose one of minimal height.  

\noindent We can exclude the case $\psset{xunit=.22cm,yunit=.22cm}
\begin{pspicture}(5,5)
 \rput(1.55, -1.15){\includegraphics[height=0.23in]{singcomult.pdf}}
 \rput(2.7,1){\includegraphics[height=0.23in]{singcomult.pdf}}
 \end{pspicture}$ since the singular comultiplication is of mini-
 \vspace{0.1cm}
 
\noindent  since the singular comultiplication is of minimal height. 
 From steps (II) and (IIIb)ii we know that the remaining situations to consider are: 
\[
 \psset{xunit=.22cm,yunit=.22cm}
\begin{pspicture}(5,5)
 \rput(2.5, -0.3){?}
 \rput(3.9, -2.3){\includegraphics[height=0.23in]{singmult.pdf}}
 \rput(5, -0.2){\includegraphics[height=0.23in]{identity_web.pdf}}
  \rput(5, 1.9){\includegraphics[height=0.23in]{identity_web.pdf}}
 \rput(1.5, 1.8){\includegraphics[height=0.23in]{singmult.pdf}}
 \rput(2.7,4){\includegraphics[height=0.23in]{huge_singcomult.pdf}}
 \end{pspicture}
  \hspace{2cm}
 \begin{pspicture}(5,5)
  \rput(2.9, -0.3){?}
  \rput(1.6, -2.4){\includegraphics[height=0.23in]{singmult.pdf}}
  \rput(0.42, -0.25){\includegraphics[height=0.23in]{identity_web.pdf}}
  \rput(0.42, 1.85){\includegraphics[height=0.23in]{identity_web.pdf}}
 \rput(3.9, 1.85){\includegraphics[height=0.23in]{singmult.pdf}}
 \rput(2.7,4){\includegraphics[height=0.23in]{huge_singcomult.pdf}}
 \end{pspicture} 
  \]
 \vspace{0.3cm}
 
 \noindent where ``?" may be any singular cobordism that contains no singular comultiplication. Since the above cobordisms are symmetric, it is enough to consider only one case, say the first one.
Using step IIIb and the assumption that the singular comultiplication is of minimal height, there are exactly two possible situations for the first generator in the decomposition of ``?", namely:

\[
 \psset{xunit=.22cm,yunit=.22cm}
\begin{pspicture}(5,5)
 \rput(2.65, -0.34){\includegraphics[height=0.23in]{singmult.pdf}}
 \rput(3.3, -2.3){?}
 \rput(5.1, -4.5){\includegraphics[height=0.23in]{singmult.pdf}}
  \rput(6.2, -2.4){\includegraphics[height=0.23in]{identity_web.pdf}}
 \rput(6.2, -0.3){\includegraphics[height=0.23in]{identity_web.pdf}}
  \rput(5.6, 1.8){\includegraphics[height=0.23in]{web_id_t.pdf}}
 \rput(1.6, 1.8){\includegraphics[height=0.23in]{singmult.pdf}}
 \rput(2.7,4){\includegraphics[height=0.23in]{huge_singcomult.pdf}}
 \end{pspicture}
  \hspace{2cm}
 \begin{pspicture}(5,5) 
   \rput(1.5, -0.2){\includegraphics[height=0.21in]{cozipper.pdf}}
 \rput(2.3, -2.3){?}
 \rput(3.9, -4.4){\includegraphics[height=0.23in]{singmult.pdf}}
  \rput(5, -2.3){\includegraphics[height=0.23in]{identity_web.pdf}}
 \rput(5, -0.2){\includegraphics[height=0.23in]{identity_web.pdf}}
  \rput(5, 1.9){\includegraphics[height=0.23in]{identity_web.pdf}}
 \rput(1.5, 1.8){\includegraphics[height=0.23in]{singmult.pdf}}
 \rput(2.7,4){\includegraphics[height=0.23in]{huge_singcomult.pdf}}
 \end{pspicture}
\]
 \vspace{0.5cm}
 
\noindent  Iteratively applying the diffeomorphism 
  $ \psset{xunit=.22cm,yunit=.22cm}
 \begin{pspicture}(5,1)
 \rput(3.85, -0.8){\includegraphics[height=0.23in]{singmult.pdf}}
 \rput(2.7,1.3){\includegraphics[height=0.23in]{singmult.pdf}}
 \end{pspicture}
 \stackrel{\eqref{eq:web_frob1}}{\longrightarrow}
 \begin{pspicture}(5,5)
  \rput(1.55, -0.8){\includegraphics[height=0.23in]{singmult.pdf}}
 \rput(2.7,1.3){\includegraphics[height=0.23in]{singmult.pdf}}
 \end{pspicture}
$ and considering \vspace{0.35cm}

\noindent again the next two possible situations in the decomposition of ``?", we see that after all,  there are the following two possible cases:

\begin{equation*}
 \psset{xunit=.22cm,yunit=.22cm}
\begin{pspicture}(5,5)
 \rput(3.9, -2.4){\includegraphics[height=0.23in]{singmult.pdf}}
 \rput(2.15, -0.35){\includegraphics[height=0.23in]{web_id_t.pdf}}
  \rput(5.05, -0.25){\includegraphics[height=0.23in]{identity_web.pdf}}
 \rput(5.05, 1.9){\includegraphics[height=0.23in]{identity_web.pdf}}
 \rput(1.5, 1.8){\includegraphics[height=0.23in]{singmult.pdf}}
 \rput(2.7,4){\includegraphics[height=0.23in]{huge_singcomult.pdf}}
 \end{pspicture}
  \hspace{2cm}
 \begin{pspicture}(5,5) 
   \rput(1.5, -0.2){\includegraphics[height=0.21in]{cozipper.pdf}}
 \rput(2.3, -2.3){?}
 \rput(3.9, -4.4){\includegraphics[height=0.23in]{singmult.pdf}}
  \rput(5, -2.3){\includegraphics[height=0.23in]{identity_web.pdf}}
 \rput(5, -0.15){\includegraphics[height=0.23in]{identity_web.pdf}}
  \rput(5, 1.95){\includegraphics[height=0.23in]{identity_web.pdf}}
 \rput(1.5, 1.8){\includegraphics[height=0.23in]{singmult.pdf}}
 \rput(2.7,4){\includegraphics[height=0.23in]{huge_singcomult.pdf}}
 \end{pspicture}
\end{equation*}
 \vspace{0.7cm}
 
 \noindent a) In the first case, the singular comultiplication is eliminated by applying the following sequence of diffeomorphisms:
 \[ 
  \psset{xunit=.22cm,yunit=.22cm}
  \hspace{2cm}
\begin{pspicture}(8.5,6)
 \rput(3.9, -2.4){\includegraphics[height=0.23in]{singmult.pdf}}
 \rput(2.15, -0.35){\includegraphics[height=0.23in]{web_id_t.pdf}}
  \rput(5.05, -0.25){\includegraphics[height=0.23in]{identity_web.pdf}}
 \rput(5.05, 1.9){\includegraphics[height=0.23in]{identity_web.pdf}}
 \rput(1.5, 1.8){\includegraphics[height=0.23in]{singmult.pdf}}
 \rput(2.7,4){\includegraphics[height=0.23in]{huge_singcomult.pdf}}
 \end{pspicture}
\quad \rput(-1.5,0) {\stackrel{\eqref{eq:singcomult_equiv}}{\longrightarrow}}
\begin{pspicture}(8.5,6)
 \rput(5.1, -2.3){\includegraphics[height=0.23in]{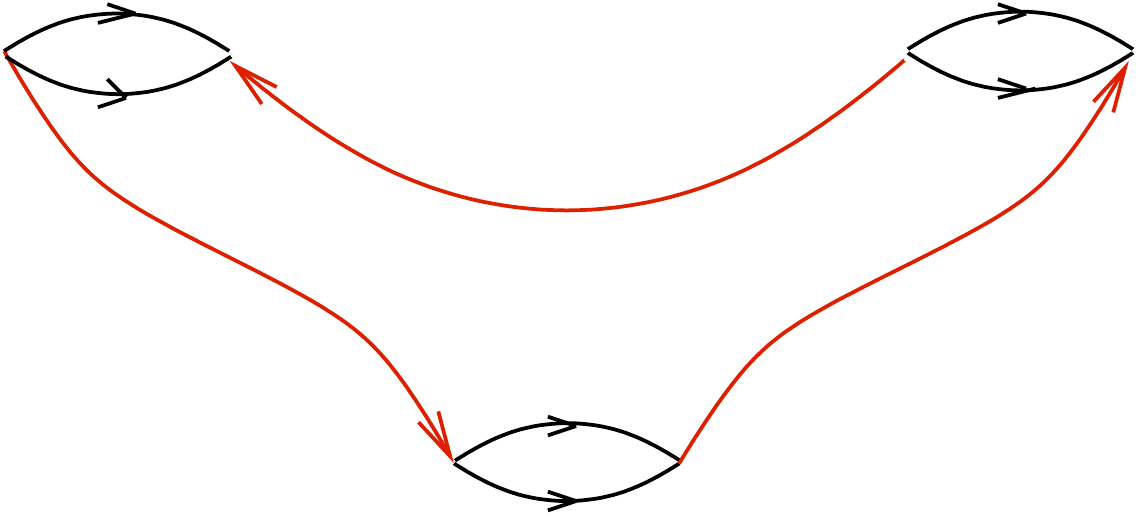}}
 \rput(2.75, -0.15){\includegraphics[height=0.23in]{singmult.pdf}}
  \rput(6.8, -0.15){\includegraphics[height=0.23in]{web_id_t.pdf}}
 \rput(5.7, 2){\includegraphics[height=0.23in]{web_id_t.pdf}}
 \rput(1.6, 2){\includegraphics[height=0.23in]{singmult.pdf}}
 \rput(4,3.8){\includegraphics[height=0.18in]{sing_copairing1.pdf}}
 \end{pspicture}
 \qquad
 \rput(-1,0) {\stackrel{\eqref{eq:web_frob1}}{\longrightarrow}}
 \begin{pspicture}(8,6)
 \rput(3.75, -2.3){\includegraphics[height=0.23in]{singmult.pdf}}
 \rput(4.9, -0.15){\includegraphics[height=0.23in]{huge_singmult.pdf}}
  \rput(6.75, 1.95){\includegraphics[height=0.23in]{web_id_t.pdf}}
 \rput(2.6, 2){\includegraphics[height=0.23in]{singmult.pdf}}
 \rput(3.86,3.8){\includegraphics[height=0.177in]{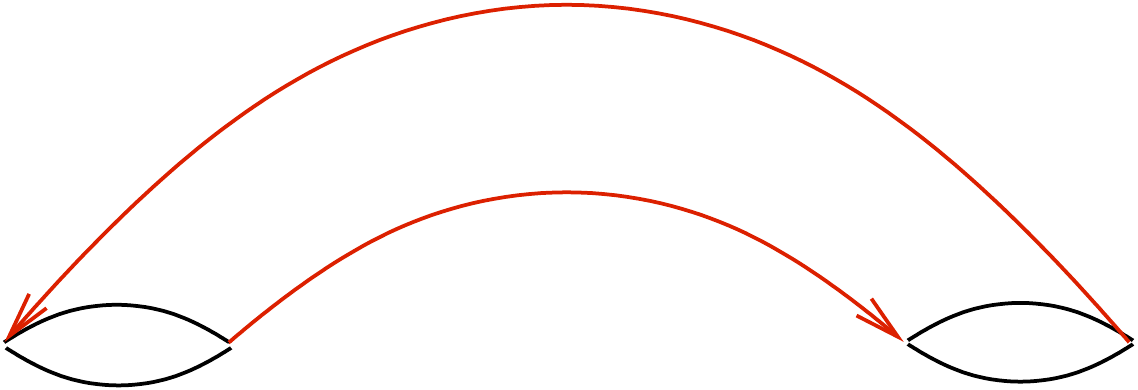}}
 \end{pspicture}
 \qquad
  \rput(-1,0) {\stackrel{\eqref{eq:singmult_equiv}}{\longrightarrow}}
  \hspace{0.8cm}
 \begin{pspicture}(13,6)
 \rput(3.75, -4.3){\includegraphics[height=0.23in]{singmult.pdf}}
 \rput(4.9, -2.15){\includegraphics[height=0.23in]{singmult.pdf}}
   \rput(6.05, -0.09){\includegraphics[height=0.23in]{identity_web.pdf}}
     \rput(3.7, -0.08){\includegraphics[height=0.23in]{identity_web.pdf}}
       \rput(0.25, 0.2){\includegraphics[height=0.18in]{sing_pairing1.pdf}}
  \rput(2.6, 2){\includegraphics[height=0.23in]{singcomult.pdf}}
    \rput(6.05, 2){\includegraphics[height=0.23in]{identity_web.pdf}}
      \rput(-0.9, 2){\includegraphics[height=0.23in]{identity_web.pdf}}
 \rput(2.63,3.9){\includegraphics[height=0.2in]{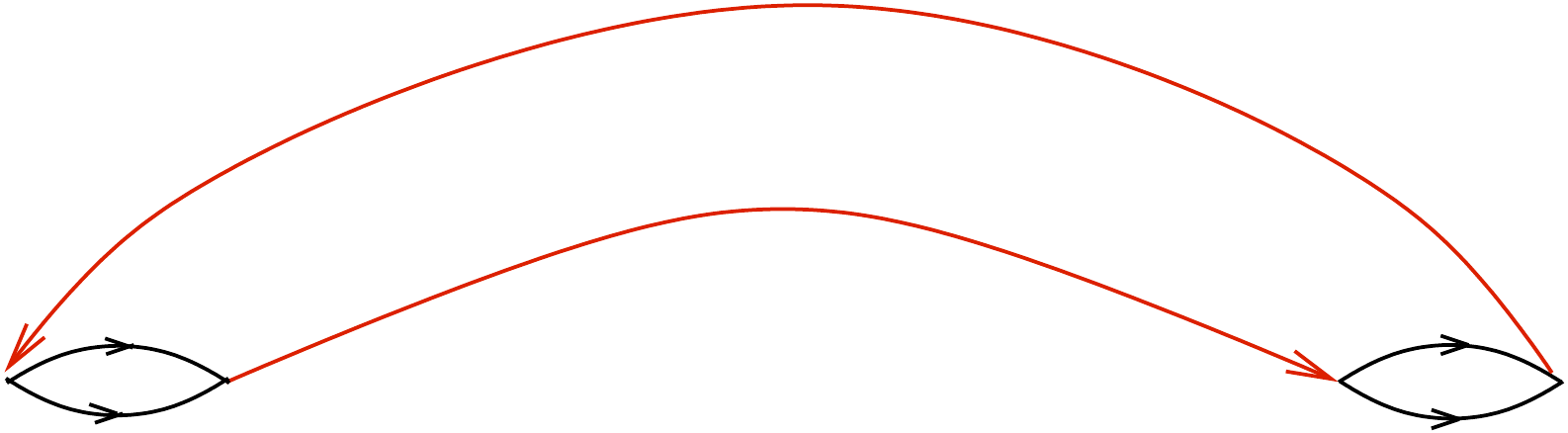}}
 \end{pspicture}
\]
\bigbreak
\[
 \psset{xunit=.22cm,yunit=.22cm}
 \hspace{1cm}
 \rput(0,-1) {\stackrel{\mbox{Nat}}{\longrightarrow}} \hspace{1cm}
 \begin{pspicture}(8,6)
   \rput(-0.45,3.17){\includegraphics[height=0.18in]{sing_copairing1.pdf}}
   \rput(4.4, 1.3){\includegraphics[height=0.24in]{singcomult.pdf}}
  \rput(-0.45,1.28){\includegraphics[height=0.24in]{braiding_WW.pdf}}
 \rput(2, -2){\includegraphics[height=0.43in]{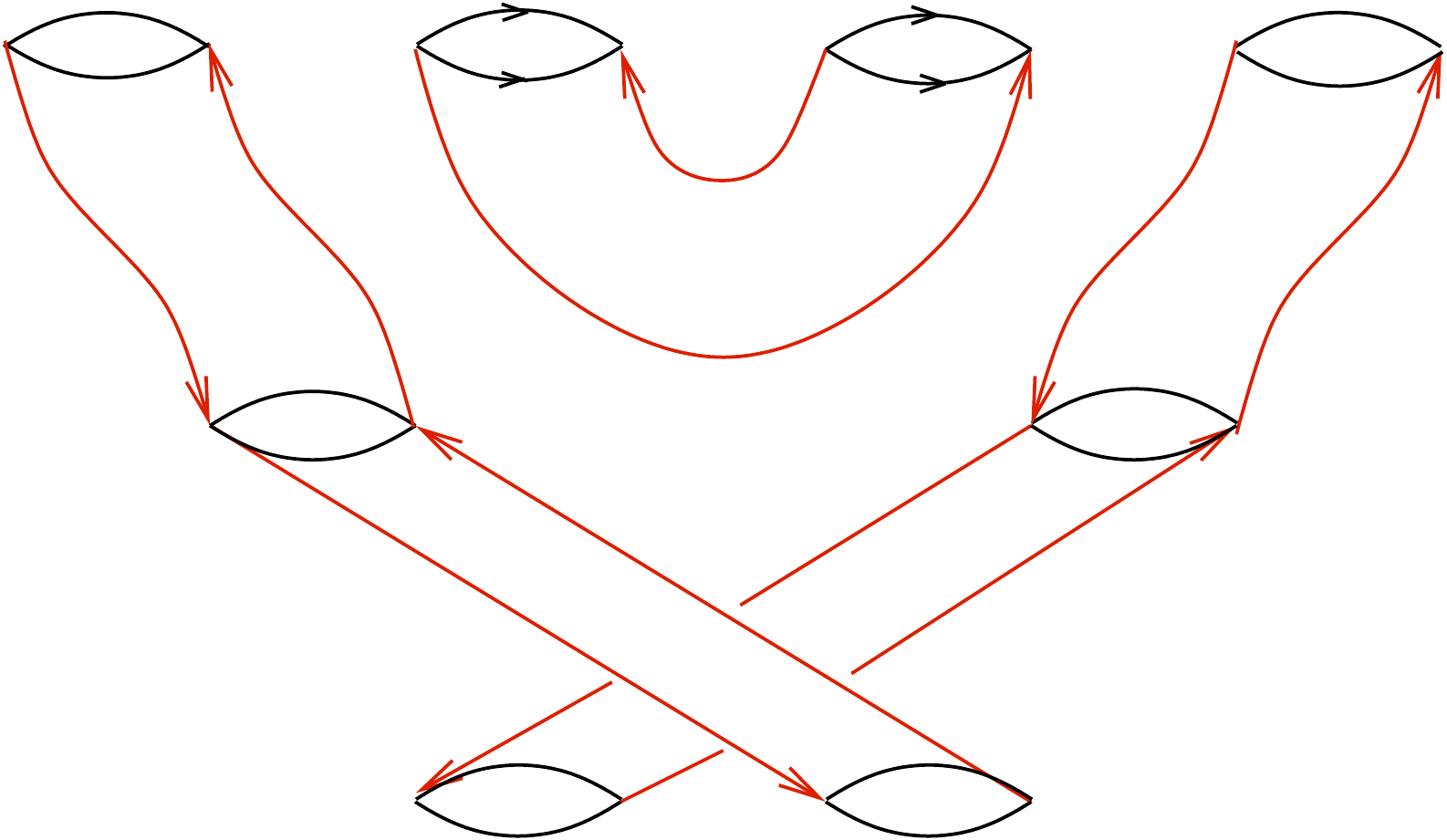}}
 \rput(2, -5.3){\includegraphics[height=0.24in]{singmult.pdf}}
  \rput(0.8, -7.5){\includegraphics[height=0.24in]{singmult.pdf}}
\end{pspicture}
 \rput(0,-2) {\longrightarrow} \hspace{1cm}
 \begin{pspicture}(8.5,6)
   \rput(4.4, 1.3){\includegraphics[height=0.24in]{singcomult.pdf}}
  \rput(-0.45,1){\includegraphics[height=0.18in]{sing_copairing1.pdf}}
 \rput(2, -2){\includegraphics[height=0.43in]{nform_sing_comult.pdf}}
 \rput(2, -5.3){\includegraphics[height=0.24in]{singmult.pdf}}
  \rput(0.8, -7.5){\includegraphics[height=0.24in]{singmult.pdf}}
\end{pspicture} 
 \rput(0,-1) {\stackrel{ \eqref{eq:sing_zig_zag}}{\longrightarrow}}
 \begin{pspicture}(8.5,6)
  \rput(4.6, -1){\includegraphics[height=0.6in]{genus1_rel_sing.pdf}}
   \rput(3.5, -5.3){\includegraphics[height=0.24in]{singmult.pdf}}
  \end{pspicture} 
  \rput(0.5,-1) {\stackrel{ \eqref{eq:genus_one}}{\longrightarrow}}
   \begin{pspicture}(8.5,6)
  \rput(4.6, -1){\includegraphics[height=0.75in]{genus1_rel.pdf}}
   \rput(3.45, -6.1){\includegraphics[height=0.24in]{singmult.pdf}}
  \end{pspicture} \vspace{2cm}
\]

\noindent b) To remove the singular comultiplication in the second case above, we apply the following sequence of diffeomorphisms:
\[
\hspace{1.5cm}
 \psset{xunit=.22cm,yunit=.22cm}
\begin{pspicture}(10,10)
 \rput(2.5, 9){\includegraphics[height=0.23in]{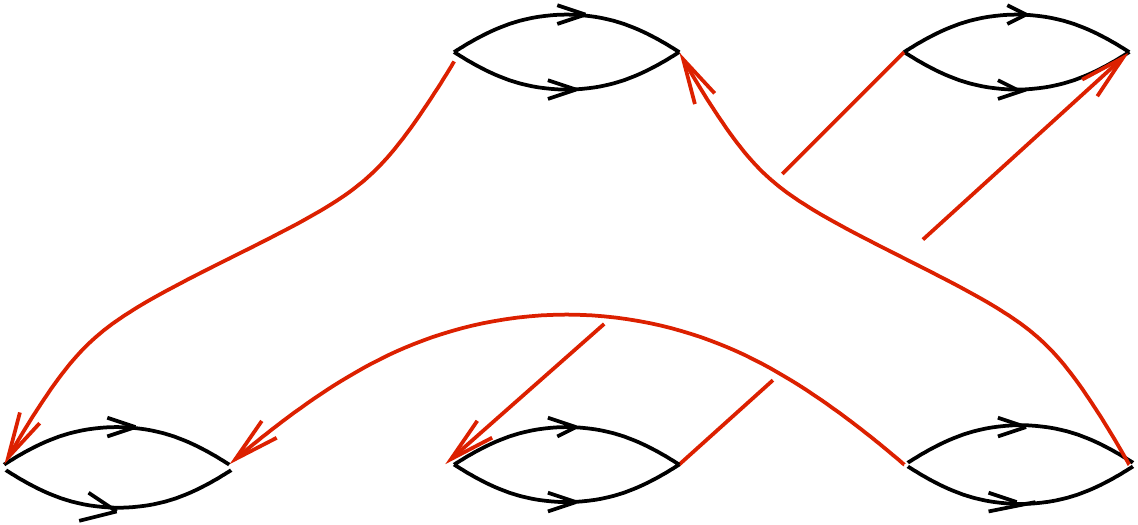}}
  \rput(1.3, 6.9){\includegraphics[height=0.23in]{singmult.pdf}}
 \rput(4.8, 6.9){\includegraphics[height=0.23in]{identity_web.pdf}}
  \rput(4.8, 4.8){\includegraphics[height=0.23in]{identity_web.pdf}}
    \rput(4.8, 2.7){\includegraphics[height=0.23in]{identity_web.pdf}}
  \rput(1.3, 4.9){\includegraphics[height=0.2in]{cozipper.pdf}}
 \rput(3.65, 0.6){\includegraphics[height=0.23in]{singmult.pdf}}
\end{pspicture}
\quad \rput(-3,6) {\stackrel{\eqref{eq:weak_com}}{\longrightarrow}}
\begin{pspicture}(10,10)
 \rput(2.5, 9){\includegraphics[height=0.23in]{huge_singcomult_twist.pdf}}
  \rput(1.3, 6.9){\includegraphics[height=0.23in]{braiding_WW.pdf}}
  \rput(1.3, 4.8){\includegraphics[height=0.23in]{singmult.pdf}}
 \rput(4.8, 6.9){\includegraphics[height=0.23in]{identity_web.pdf}}
  \rput(4.8, 4.8){\includegraphics[height=0.23in]{identity_web.pdf}}
    \rput(4.8, 2.7){\includegraphics[height=0.23in]{identity_web.pdf}}
     \rput(4.8, 0.6){\includegraphics[height=0.23in]{identity_web.pdf}}
  \rput(1.3, 2.85){\includegraphics[height=0.2in]{cozipper.pdf}}
 \rput(3.65, -1.5){\includegraphics[height=0.23in]{singmult.pdf}}
\end{pspicture}
\quad \rput(-3,6) {\stackrel{\mbox{Nat}}{\longrightarrow}}
\begin{pspicture}(10,10)
 \rput(3.6, 9){\includegraphics[height=0.23in]{singcomult.pdf}}
  \rput(2.45, 6.9){\includegraphics[height=0.23in]{identity_web.pdf}}
    \rput(0.15, 6.9){\includegraphics[height=0.23in]{identity_web.pdf}}
  \rput(1.3, 4.8){\includegraphics[height=0.23in]{singmult.pdf}}
 \rput(4.8, 6.9){\includegraphics[height=0.23in]{identity_web.pdf}}
  \rput(4.8, 4.8){\includegraphics[height=0.23in]{identity_web.pdf}}
    \rput(4.8, 2.7){\includegraphics[height=0.23in]{identity_web.pdf}}
     \rput(4.8, 0.6){\includegraphics[height=0.23in]{identity_web.pdf}}
  \rput(1.3, 2.85){\includegraphics[height=0.2in]{cozipper.pdf}}
 \rput(3.65, -1.5){\includegraphics[height=0.23in]{singmult.pdf}}
\end{pspicture}
\quad \rput(-3,6) {\stackrel{\eqref{eq:web_frob3}}{\longrightarrow}}
\begin{pspicture}(10,10)
 \rput(2.65, 6.9){\includegraphics[height=0.23in]{singmult.pdf}}
   \rput(2.65, 4.8){\includegraphics[height=0.23in]{singcomult.pdf}}
    \rput(3.8, 2.7){\includegraphics[height=0.23in]{identity_web.pdf}}
     \rput(3.8, 0.6){\includegraphics[height=0.23in]{identity_web.pdf}}
  \rput(1.5, 2.85){\includegraphics[height=0.2in]{cozipper.pdf}}
 \rput(2.65, -1.5){\includegraphics[height=0.23in]{singmult.pdf}}
\end{pspicture}
\quad \rput(-3,6) {\stackrel{\mbox{(IId)ii}}{\longrightarrow}}
\begin{pspicture}(10,10)
 \rput(4.9, 6.8){\includegraphics[height=0.23in]{singmult.pdf}}
  \rput(1.45, 6.15){\includegraphics[height=0.2in]{copairing1.pdf}}
   \rput(4.9, 4.7){\includegraphics[height=0.23in]{identity_web.pdf}}
  \rput(2.65, 4.55){\includegraphics[height=0.2in]{zipper.pdf}}
   \rput(3.8, 2.6){\includegraphics[height=0.23in]{singmult.pdf}}
 \rput(2.65, 0.5){\includegraphics[height=0.23in]{singmult.pdf}}
\end{pspicture}
\]
\vspace{0.12cm}

\noindent c) We reapply steps II and III if needed.

\noindent d) We iterate steps IVa-b and IVc until the last singular comultiplication has been eliminated. We remark that this process will terminate after a finite number of iterations, since steps II and III do not increase the number of singular comultiplications. 
\vspace{.5cm}

\item [V.] After the first four steps of the proof, all singular cups, caps and comultiplications have been eliminated from the decomposition of $\Sigma,$ and the resulting decomposition has the following properties: 

\noindent a) Every singular multiplication has its source in one of the following situations
 \[ \raisebox{-3pt}{\includegraphics[height=0.25in]{nform_smult1.pdf}} \quad \raisebox{-8pt}{\includegraphics[height=0.47in]{nform_smult2.pdf}}\quad \raisebox{-10pt}{\includegraphics[height=0.55in]{nform_smult3.pdf}} \quad \raisebox{-8pt}{\includegraphics[height=0.46in]{nform_smult4.pdf}} \quad \raisebox{-8pt}{\includegraphics[height=0.46in]{nform_smult5.pdf}}\]
 and its target in one of the following situations
 \[ \psset{xunit=.22cm,yunit=.22cm}
 \begin{pspicture}(5,5)
 \rput(3.85, 0.6){\includegraphics[height=0.23in]{singmult.pdf}}
 \rput(2.7,2.75){\includegraphics[height=0.23in]{singmult.pdf}}
 \end{pspicture} \quad
 \psset{xunit=.22cm,yunit=.22cm}
 \begin{pspicture}(5,5)
 \rput(2.7, 0.8){\includegraphics[height=0.2in]{cozipper.pdf}}
 \rput(2.7,2.75){\includegraphics[height=0.23in]{singmult.pdf}}
 \end{pspicture}  \]
 \bigbreak
 
 \noindent b) Every cozipper  \raisebox{-8pt}{\includegraphics[height=0.23in]{cozipper.pdf}} appears in one of the situations explained in (IIIb)i.
 
 \noindent c) Every singular genus-one operator \raisebox{-8pt}{\includegraphics[height=0.3in]{sgenus_one.pdf}} has been eliminated from the decomposition.
 \vspace{0.5cm}
 
 \item [VI.] We show now that the zipper \raisebox{-8pt}{\includegraphics[height=0.23in]{zipper.pdf}} can be eliminated from the decomposition of $\Sigma$ or that its target can be put in the situation \raisebox{-13pt}{\includegraphics[height=0.4in]{zipper_cozipper.pdf}}\,, thus eliminated completely by defining $\raisebox{-13pt}{\includegraphics[height=0.4in]{nform_zipper_cozipper.pdf}}: = \raisebox{-13pt}{\includegraphics[height=0.4in]{zipper_cozipper.pdf}}\,.$
 
\noindent The following situations:
 \begin{equation*}
\psset{xunit=.22cm,yunit=.22cm}
\begin{pspicture}(5,5)
 \rput(2.7, 2.6){\includegraphics[height=0.12in]{singcounit.pdf}}
 \rput(2.7,4){\includegraphics[height=0.23in]{zipper.pdf}}
 \end{pspicture}
 \qquad
 \begin{pspicture}(5,5)
 \rput(2.3, 2.7){\includegraphics[height=0.28in]{singcomult.pdf}}
 \rput(2.3,5){\includegraphics[height=0.23in]{zipper.pdf}}
 \end{pspicture}
 \quad
 \begin{pspicture}(5,5)
 \rput(3, 6){\includegraphics[height=0.23in]{zipper.pdf}}
 \rput(3,2.7){\includegraphics[height=0.44in]{sgenus_one.pdf}}
 \end{pspicture}
  \end{equation*}
are excluded by steps I, IV and (IIIb)ii, respectively. It remains to consider the cases:
 \begin{equation*}
\psset{xunit=.22cm,yunit=.22cm}
\begin{pspicture}(5,5)
 \rput(1, 4.6){\includegraphics[height=0.23in]{zipper.pdf}}
 \rput(2.4,2.3){\includegraphics[height=0.28in]{singmult.pdf}}
 \end{pspicture}
 \qquad
 \begin{pspicture}(5,5)
 \rput(2.3, 2.3){\includegraphics[height=0.28in]{singmult.pdf}}
 \rput(3.7,4.6){\includegraphics[height=0.23in]{zipper.pdf}}
 \end{pspicture}
\end{equation*}
The second case can be reduced to the first one by applying the following sequence of diffeomorphisms:

\begin{equation*}
\raisebox{-13pt}{\includegraphics[height=0.5in]{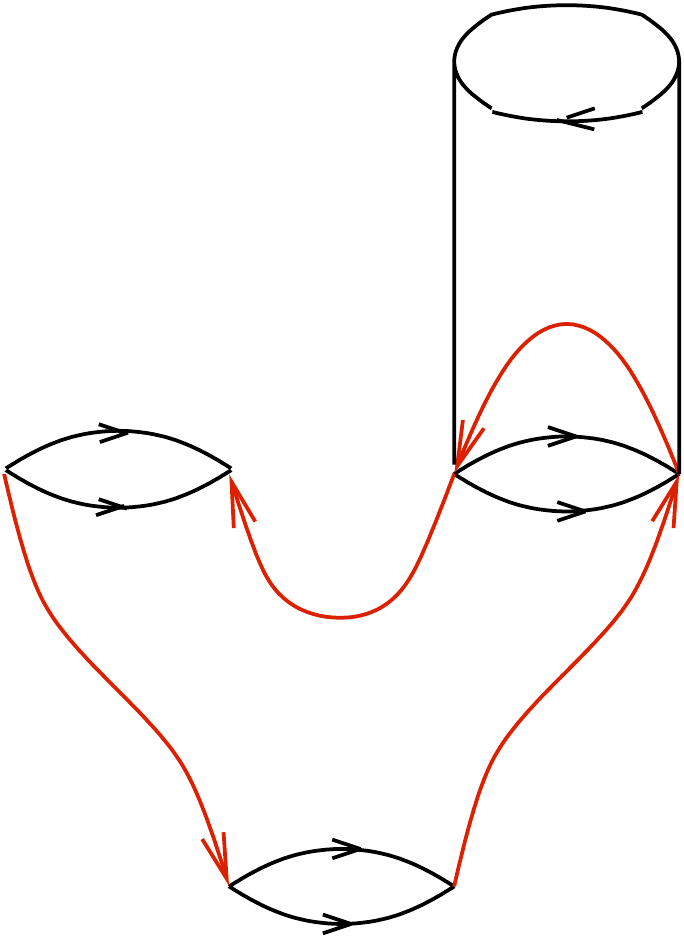}}\,\stackrel{\eqref{eq:center}}{\longrightarrow} \, \raisebox{-13pt}{\includegraphics[height=0.65in]{center2.pdf}}\, \longrightarrow \,\raisebox{-13pt}{\includegraphics[height=0.65in]{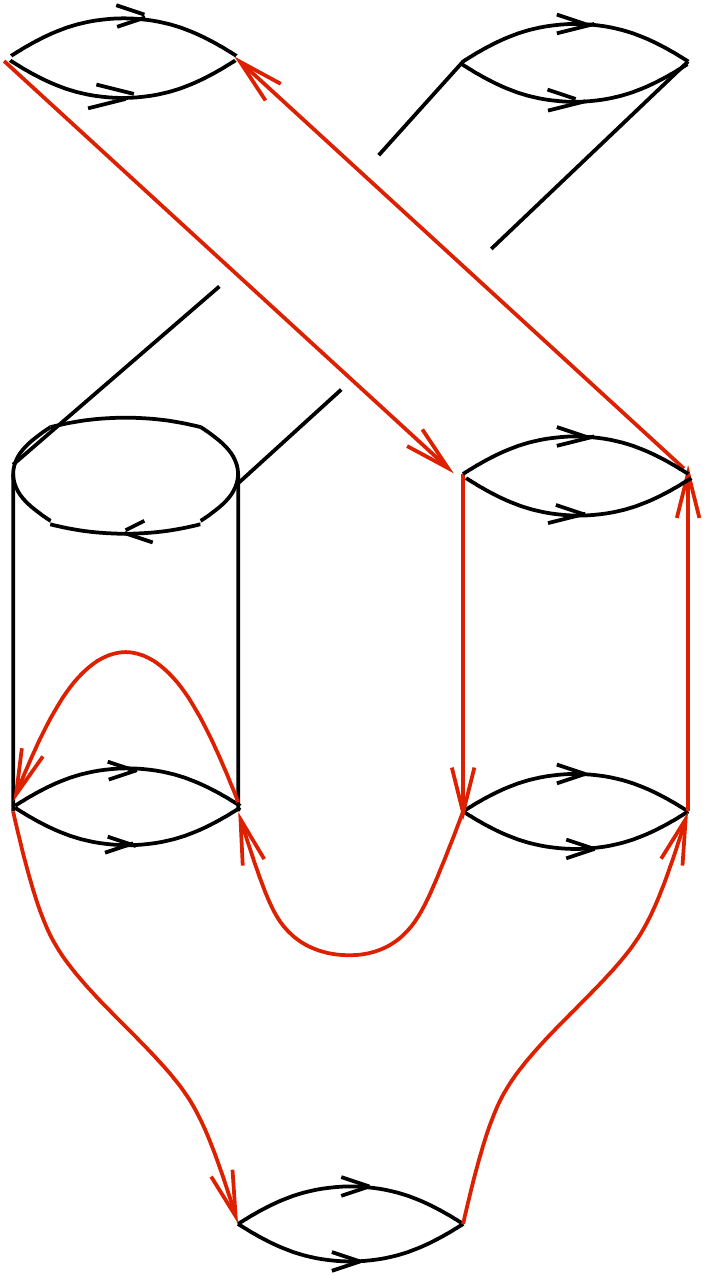}} 
\end{equation*}
  
 \noindent Then, by taking into account the second part of step Va, we need to consider either
 $ \psset{xunit=.22cm,yunit=.22cm}
  \begin{pspicture}(5,5)
 \rput(3.85, -0.8){\includegraphics[height=0.23in]{singmult.pdf}}
 \rput(2.7,1.3){\includegraphics[height=0.23in]{singmult.pdf}}
 \rput(1.55, 3.23){\includegraphics[height=0.2in]{zipper.pdf}}
 \end{pspicture} \quad \mbox{or} \quad
\begin{pspicture}(5,5)
 \rput(1.25, 3.15){\includegraphics[height=0.2in]{zipper.pdf}}
 \rput(2.4,1.2){\includegraphics[height=0.23in]{singmult.pdf}}
  \rput(2.4,-0.75){\includegraphics[height=0.2in]{cozipper.pdf}}
 \end{pspicture}.
 $

\noindent a) For the first possibility we apply the diffeomorphism 
 $ \psset{xunit=.22cm,yunit=.22cm}
 \begin{pspicture}(5,5)
 \rput(3.85, -0.8){\includegraphics[height=0.23in]{singmult.pdf}}
 \rput(2.7,1.3){\includegraphics[height=0.23in]{singmult.pdf}}
  \rput(1.55, 3.23){\includegraphics[height=0.2in]{zipper.pdf}}
 \end{pspicture}\quad
  \stackrel{\eqref{eq:web_frob1}}{\longrightarrow}\quad
  \begin{pspicture}(5,5)
     \rput(-0.7, 3.1){\includegraphics[height=0.19in]{zipper.pdf}}
 \rput(1.55, -0.8){\includegraphics[height=0.23in]{singmult.pdf}}
 \rput(2.7,1.3){\includegraphics[height=0.23in]{singmult.pdf}}
   \rput(-0.15, 1.26){\includegraphics[height=0.22in]{web_id_t.pdf}}
 \end{pspicture}$
which reduces the height of the zipper.
\vspace{0.3cm}

\noindent b) For the second possibility we employ a sequence of diffeomorphisms that removes the zipper:
 \[ \psset{xunit=.22cm,yunit=.22cm}
 \begin{pspicture}(5,5)
  \rput(1.25, 3.15){\includegraphics[height=0.2in]{zipper.pdf}}
 \rput(2.4,1.2){\includegraphics[height=0.23in]{singmult.pdf}}
  \rput(2.4,-0.75){\includegraphics[height=0.2in]{cozipper.pdf}}
 \end{pspicture} 
 \stackrel{\eqref{eq:singmult_equiv}}{\longrightarrow}
 \begin{pspicture}(5,5)
  \rput(1.15, 1){\includegraphics[height=0.2in]{zipper.pdf}}
 \rput(4.7,1.2){\includegraphics[height=0.23in]{singcomult.pdf}}
  \rput(5.9,-0.75){\includegraphics[height=0.2in]{cozipper.pdf}}
   \rput(2.35,-0.67){\includegraphics[height=0.18in]{sing_pairing1.pdf}}
 \end{pspicture} \quad
  \stackrel{\eqref{eq:cozipper_pairing}}{\longrightarrow}
   \begin{pspicture}(5,5)
  \rput(3.55, 1.25){\includegraphics[height=0.2in]{cozipper.pdf}}
 \rput(4.7,3.2){\includegraphics[height=0.23in]{singcomult.pdf}}
  \rput(5.9,1.25){\includegraphics[height=0.2in]{cozipper.pdf}}
   \rput(2.45,-0.26){\includegraphics[height=0.18in]{pairing1.pdf}}
 \end{pspicture} \quad \stackrel{\eqref{eq:cozipper_coalghom}}{\longrightarrow}  
  \begin{pspicture}(5,5)
  \rput(4.55, 1.35){\includegraphics[height=0.23in]{comult.pdf}}
 \rput(4.55,3.12){\includegraphics[height=0.19in]{cozipper.pdf}}
  \rput(2.35,-0.38){\includegraphics[height=0.18in]{pairing1.pdf}}
 \end{pspicture}\quad  \stackrel{\eqref{eq:mult_equiv}}{\longrightarrow}  
  \begin{pspicture}(5,5)
 \rput(2.3, 0.3){\includegraphics[height=0.23in]{mult.pdf}}
 \rput(3.4,2.1){\includegraphics[height=0.2in]{cozipper.pdf}}
 \end{pspicture}
\]
 By repeating these steps if necessary, and applying $\raisebox{-13pt}{\includegraphics[height=0.4in]{zipper_cozipper.pdf}} \longrightarrow \raisebox{-13pt}{\includegraphics[height=0.4in]{nform_zipper_cozipper.pdf}}\,,$ we guarantee that the zipper \raisebox{-8pt}{\includegraphics[height=0.23in]{zipper.pdf}} has been eliminated from the decomposition of $\Sigma.$
 \vspace{0.3cm}
 
 \item [VII.] The (resulting) decomposition of $\Sigma$ is equivalent to one in which the \textit{ordinary multiplication}  \raisebox{-8pt}{\includegraphics[height=0.23in]{mult.pdf}} has its source in one of the following situations:
 
 \[  \psset{xunit=.22cm,yunit=.22cm}
 \begin{pspicture}(5,5)
 \rput(0.25, 1.05){\includegraphics[height=0.2in]{cozipper.pdf}}
 \rput(2.56,1.05){\includegraphics[height=0.2in]{cozipper.pdf}}
  \rput(1.4,-0.75){\includegraphics[height=0.23in]{mult.pdf}}
 \end{pspicture} \quad
 \begin{pspicture}(5,5)
 \rput(0, 1.2){\includegraphics[height=0.2in]{mult.pdf}}
 \rput(4.85,1.2){\includegraphics[height=0.2in]{cozipper.pdf}}
  \rput(2.4,-0.75){\includegraphics[height=0.2in]{hugemult.pdf}}
 \end{pspicture} \quad
  \begin{pspicture}(5,5)
 \rput(2.4,0){\includegraphics[height=0.35in]{genus_one.pdf}}
 \end{pspicture}
 \stackrel{Def}{=}  \begin{pspicture}(5,5)
 \rput(2.4,0){\includegraphics[height=0.35in]{nform_genus_one.pdf}}
 \end{pspicture}\]
 \vspace{0.3cm}
 
 \noindent a) We can exclude the cases  
 \[ \includegraphics[height=0.2in]{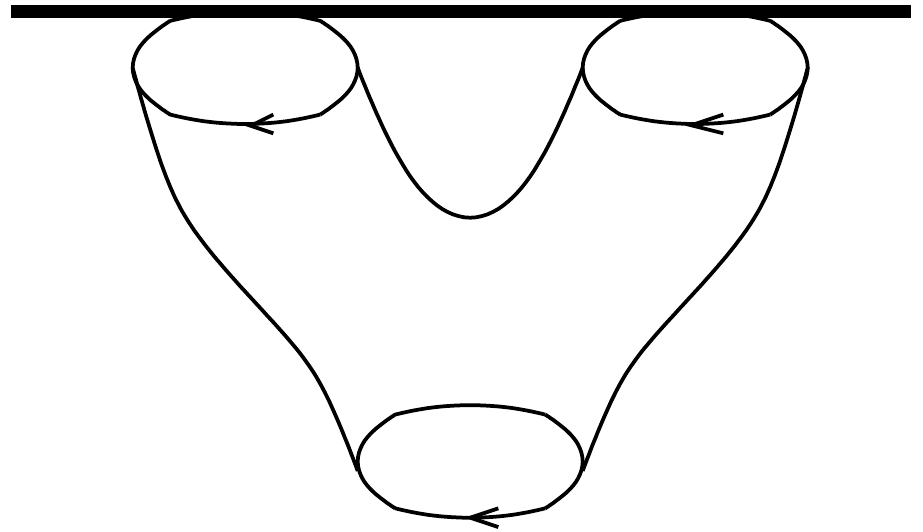}\quad \includegraphics[height=0.4in]{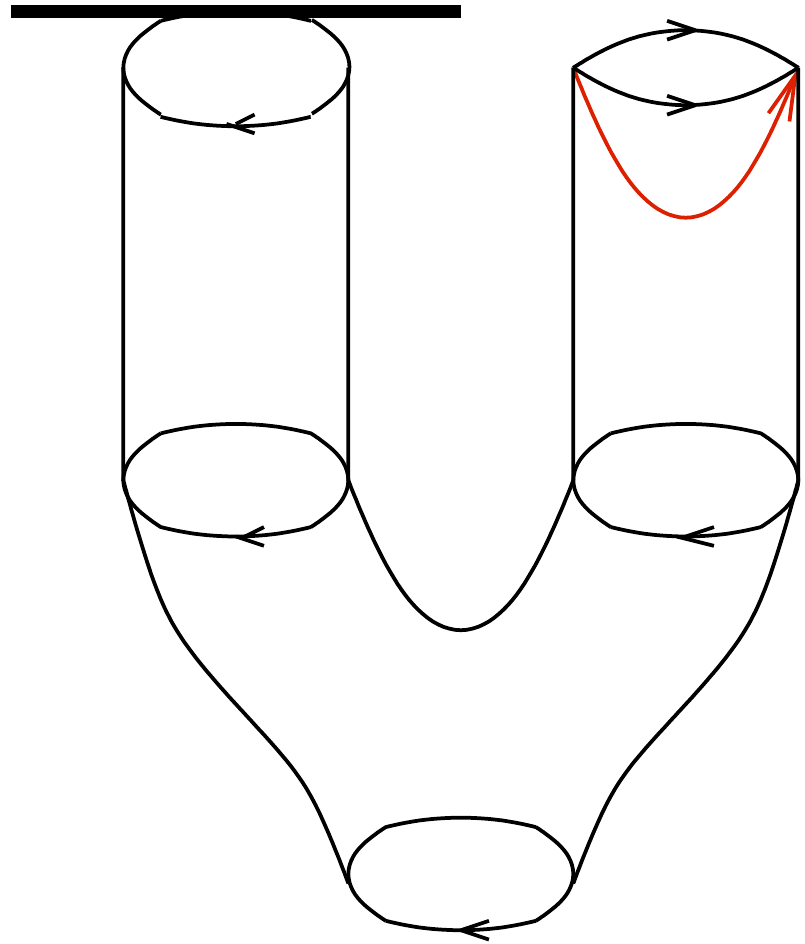} \quad \includegraphics[height=0.4in]{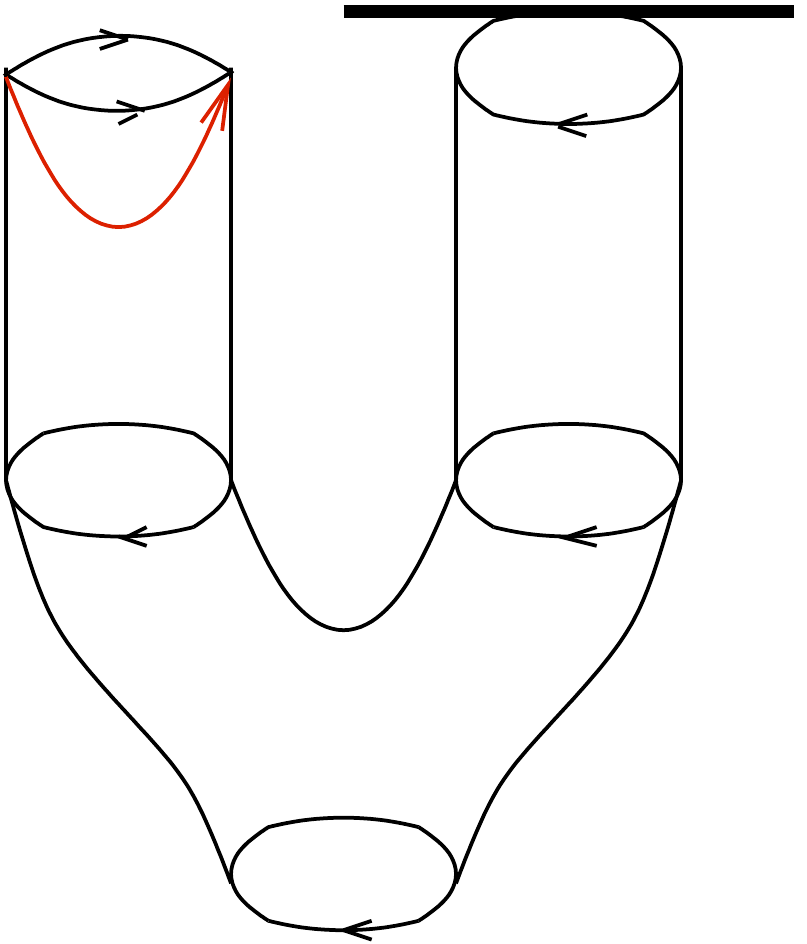}\]
 since we assume that the source of $\Sigma$ is a free union of bi-webs. 

\noindent To prove the claim we iterate the following diffeomorphisms whenever possible:
 
 \noindent b)
 $\psset{xunit=.22cm,yunit=.22cm}
 \begin{pspicture}(5,5)
 \rput(1.45, 1.35){\includegraphics[height=0.11in]{unit.pdf}}
 \rput(2.56,0){\includegraphics[height=0.23in]{mult.pdf}}
  \end{pspicture} \stackrel{\eqref{eq:circle_frob1}}{\longrightarrow}
   \begin{pspicture}(5,5)
 \rput(2.45, 0){\includegraphics[height=0.3in]{id_circle.pdf}}
 \end{pspicture}
 \stackrel{\eqref{eq:circle_frob1}}{\longleftarrow}
  \begin{pspicture}(5,5)
 \rput(3.7, 1.35){\includegraphics[height=0.11in]{unit.pdf}}
 \rput(2.56,0){\includegraphics[height=0.23in]{mult.pdf}}
  \end{pspicture}
\quad \mbox{and} \quad
 \psset{xunit=.22cm,yunit=.22cm}\begin{pspicture}(5,5)
 \rput(1.55, -0.8){\includegraphics[height=0.23in]{mult.pdf}}
 \rput(2.7,1.2){\includegraphics[height=0.23in]{mult.pdf}}
 \end{pspicture}
 \stackrel{\eqref{eq:circle_frob1}}{\longrightarrow}
 \psset{xunit=.22cm,yunit=.22cm}
 \begin{pspicture}(5,5)
 \rput(3.85, -0.8){\includegraphics[height=0.23in]{mult.pdf}}
 \rput(2.7,1.2){\includegraphics[height=0.23in]{mult.pdf}}
 \end{pspicture}$
 \vspace{0.5cm}
 
\noindent c)  $\psset{xunit=.22cm,yunit=.22cm}
 \begin{pspicture}(5,5)
 \rput(2.45, 2){\includegraphics[height=0.23in]{comult.pdf}}
 \rput(2.45,-1){\includegraphics[height=0.4in]{circle_frob13.pdf}}
  \end{pspicture} \stackrel{\eqref{eq:circle_frob4}}{\longrightarrow}
   \begin{pspicture}(5,5)
   \rput(2.4,1){\includegraphics[height=0.2in]{comult.pdf}}
  \rput(2.4,-0.75){\includegraphics[height=0.2in]{mult.pdf}}
 \end{pspicture}
$ and, more generally, $\raisebox{-13pt}{\includegraphics[height=0.4in]{circle_frob13.pdf}} \stackrel{\eqref{eq:circle_frob4}}{\longrightarrow}\raisebox{-13pt}{\includegraphics[height=0.4in]{circle_frob14.pdf}}$ 
\vspace{0.7cm}
 
 \noindent d) $\raisebox{-13pt}{\includegraphics[height=0.4in]{circle_frob10.pdf}} \stackrel{\eqref{eq:circle_frob3}}{\longrightarrow} \raisebox{-13pt}{\includegraphics[height=0.4in]{circle_frob11.pdf}}  \stackrel{\eqref{eq:circle_frob3}}{\longleftarrow} \raisebox{-13pt}{\includegraphics[height=0.4in]{circle_frob12.pdf}}$
 \vspace{0.2cm}
 
 \noindent e)
$
\raisebox{-13pt}{\includegraphics[height=0.52in]{zipper_cozipper6.pdf}} \stackrel{\eqref{eq:zipper_cozipper3}}{\longrightarrow} \raisebox{-13pt}{\includegraphics[height=0.52in]{zipper_cozipper7.pdf}} \stackrel{\eqref{eq:zipper_cozipper3}}{\longleftarrow}\raisebox{-13pt}{\includegraphics[height=0.52in]{zipper_cozipper8.pdf}}
$
\vspace{0.2cm}

\noindent f)
$\raisebox{-18pt}{\includegraphics[height=0.65in]{genus_one_4.pdf}} \stackrel{\eqref{eq:genus_one2}}{\longrightarrow} \raisebox{-18pt}{\includegraphics[height=0.65in]{genus_one_5.pdf}}  \stackrel{\eqref{eq:genus_one2}}{\longleftarrow} \raisebox{-18pt}{\includegraphics[height=0.65in]{genus_one_6.pdf}}$
\vspace{0.2cm}

\noindent Each of the above diffeomorphisms either removes the ordinary multiplication or increases its height, therefore applying these moves whenever possible assures that the process ends with each ordinary multiplication in the decomposition of $\Sigma$ in one of the claimed situations.

\vspace{0.5cm}
\item [VIII.] The resulting decomposition of $\Sigma$ is equivalent to one in which each \textit{ordinary comultiplication} \raisebox{-8pt}{\includegraphics[height=0.23in]{comult.pdf}} is in one of the following situations:
\[\raisebox{-10pt}{\includegraphics[height=0.25in]{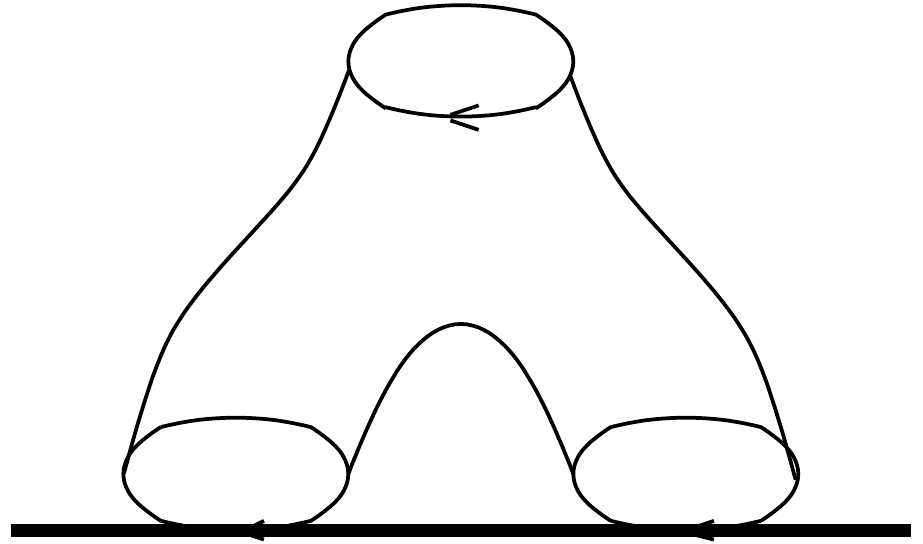}} \quad \raisebox{-10pt}{\includegraphics[height=0.4in]{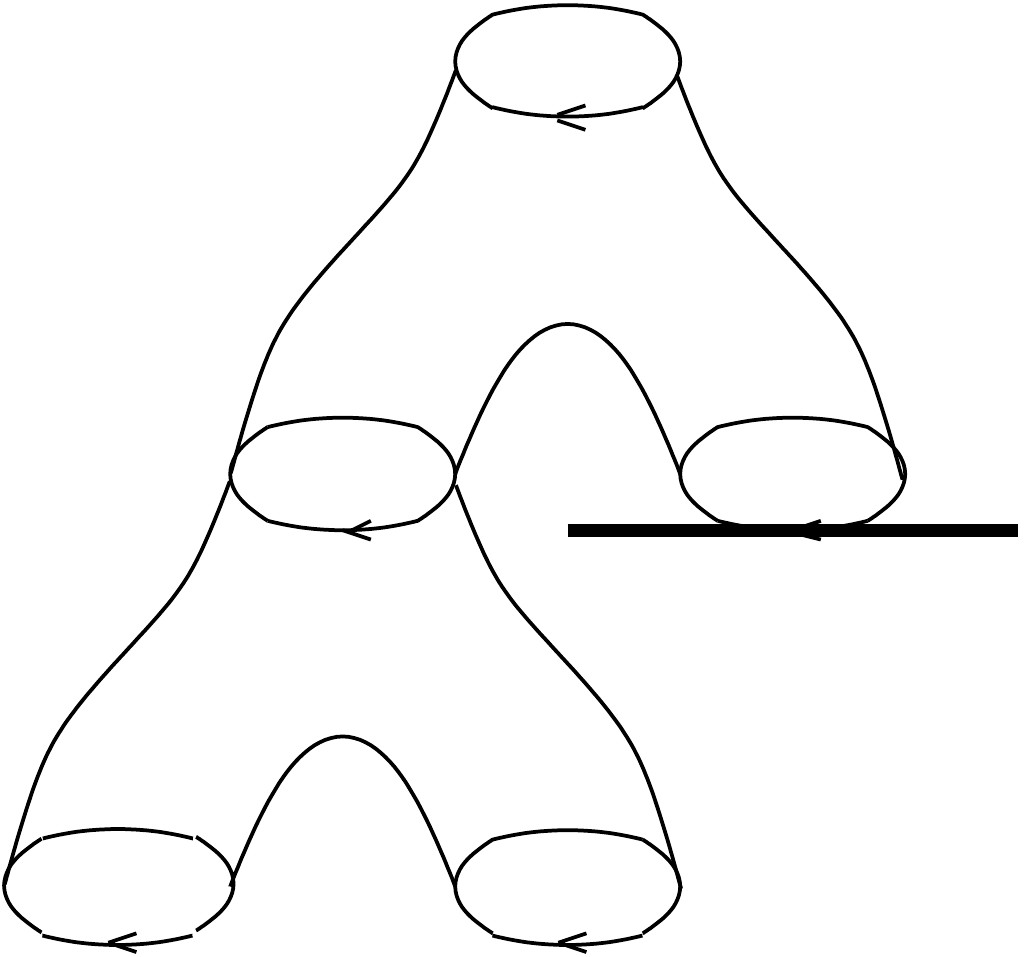}} \quad \raisebox{-10pt}{\includegraphics[height=0.4in]{genus_one.pdf}} \stackrel{Def}{=} \raisebox{-10pt}{\includegraphics[height=0.4in]{nform_genus_one.pdf}} \]

\noindent a) Employing step VII, we can exclude the cases:
\[ \raisebox{-5pt}{\includegraphics[height=0.4in]{circle_frob10.pdf}} \quad  \raisebox{-5pt}{\includegraphics[height=0.4in]{circle_frob12.pdf}}\quad \psset{xunit=.22cm,yunit=.22cm}
 \begin{pspicture}(5,5)
 \rput(2.45, 4){\includegraphics[height=0.23in]{comult.pdf}}
 \rput(2.45,1){\includegraphics[height=0.4in]{circle_frob13.pdf}}
  \end{pspicture}
\]
  Moreover, every zipper has been eliminated at step VI, thus we can also exclude:
\[\psset{xunit=.22cm,yunit=.22cm}
 \begin{pspicture}(5,5)
 \rput(2, 2){\includegraphics[height=0.26in]{comult.pdf}}
 \rput(3.2,0.1){\includegraphics[height=0.2in]{zipper.pdf}}
  \end{pspicture}  \quad
   \begin{pspicture}(5,5)
 \rput(2, 2){\includegraphics[height=0.26in]{comult.pdf}}
 \rput(0.8,0.1){\includegraphics[height=0.2in]{zipper.pdf}}
  \end{pspicture} \quad
   \begin{pspicture}(5,5)
 \rput(2, 2){\includegraphics[height=0.26in]{comult.pdf}}
 \rput(3.2,0.1){\includegraphics[height=0.2in]{zipper.pdf}}
 \rput(0.8,0.1){\includegraphics[height=0.2in]{zipper.pdf}}
  \end{pspicture}
\] 

\noindent The claim follows by iterating whenever possible the following diffeomorphisms:

\noindent b)
$\psset{xunit=.22cm,yunit=.22cm}
 \begin{pspicture}(5,5)
 \rput(2.45, 2){\includegraphics[height=0.23in]{comult.pdf}}
 \rput(2.45,0){\includegraphics[height=0.23in]{braiding_CC.pdf}}  \end{pspicture}
 \longrightarrow \raisebox{-5pt}{\includegraphics[height=0.23in]{comult.pdf}}
$

\noindent c) 
 $\psset{xunit=.22cm,yunit=.22cm}
 \begin{pspicture}(5,5)
 \rput(1.45, 0){\includegraphics[height=0.11in]{counit.pdf}}
 \rput(2.56,1.35){\includegraphics[height=0.23in]{comult.pdf}}
  \end{pspicture} \stackrel{\eqref{eq:circle_frob2}}{\longrightarrow}
   \begin{pspicture}(5,5)
 \rput(2.45, 0){\includegraphics[height=0.3in]{id_circle.pdf}}
 \end{pspicture}
 \stackrel{\eqref{eq:circle_frob2}}{\longleftarrow}
  \begin{pspicture}(5,5)
 \rput(3.7, 0){\includegraphics[height=0.11in]{counit.pdf}}
 \rput(2.56,1.35){\includegraphics[height=0.23in]{comult.pdf}}
  \end{pspicture}
\quad \mbox{and} \quad
 \psset{xunit=.22cm,yunit=.22cm}
 \begin{pspicture}(5,5)
 \rput(3.85, -0.8){\includegraphics[height=0.23in]{comult.pdf}}
 \rput(2.7,1.2){\includegraphics[height=0.23in]{comult.pdf}}
 \end{pspicture} 
  \stackrel{\eqref{eq:circle_frob2}}{\longrightarrow}
 \psset{xunit=.22cm,yunit=.22cm}\begin{pspicture}(5,5)
 \rput(1.55, -0.8){\includegraphics[height=0.23in]{comult.pdf}}
 \rput(2.7,1.2){\includegraphics[height=0.23in]{comult.pdf}}
 \end{pspicture}
$
 \vspace{0.5cm}

  \noindent d)
$
\raisebox{-13pt}{\includegraphics[height=0.52in]{zipper_cozipper1.pdf}} \stackrel{\eqref{eq:zipper_cozipper1}}{\longrightarrow} \raisebox{-13pt}{\includegraphics[height=0.52in]{zipper_cozipper2.pdf}} \stackrel{\eqref{eq:zipper_cozipper1}}{\longleftarrow}\raisebox{-13pt}{\includegraphics[height=0.52in]{zipper_cozipper3.pdf}}
$
\vspace{0.2cm}

\noindent e)
$\raisebox{-18pt}{\includegraphics[height=0.65in]{genus_one_1.pdf}} \stackrel{\eqref{eq:genus_one1}}{\longrightarrow} \raisebox{-18pt}{\includegraphics[height=0.65in]{genus_one_2.pdf}}  \stackrel{\eqref{eq:genus_one1}}{\longleftarrow} \raisebox{-18pt}{\includegraphics[height=0.65in]{genus_one_3.pdf}}$
\vspace{0.2cm}

\noindent Notice that each of the above diffeomorphisms either decreases the height of the ordinary comultiplication or removes it, thus the process must end after a finite number of iterations. 
\vspace{0.3cm}

\item [IX.] The cobordism with a singular circle \raisebox{-8pt}{\includegraphics[height=0.3in]{nform_zipper_cozipper.pdf}} can be put in an equivalent decomposition of $\Sigma$, so that it has above it one of the following cobordisms:
\[ \raisebox{-8pt}{\includegraphics[height=0.15in]{unit.pdf}} \quad \raisebox{-8pt}{\includegraphics[height=0.23in]{cozipper.pdf}} \quad \raisebox{-8pt}{\includegraphics[height=0.37in]{nform_zipper_cozipper.pdf}} \quad \raisebox{-8pt}{\includegraphics[height=0.25in]{mult.pdf}}   \]
The case in which \raisebox{-8pt}{\includegraphics[height=0.3in]{nform_zipper_cozipper.pdf}} has above it the ordinary comultiplication is excluded by step VIIId, while the case in which it has
\raisebox{-8pt}{\includegraphics[height=0.3in]{nform_genus_one.pdf}} above it is eliminated by iterating the diffeomorphism
\[\psset{xunit=.22cm,yunit=.22cm}
 \begin{pspicture}(5,5)
 \rput(2.5, -1.4){\includegraphics[height=0.3in]{nform_zipper_cozipper.pdf}}
 \rput(2.5,2){\includegraphics[height=0.4in]{nform_genus_one.pdf}}
  \end{pspicture} \stackrel{\eqref{eq:zipper_cozipper2}}{\longrightarrow}
   \begin{pspicture}(5,5)
 \rput(2.5, -0.82){\includegraphics[height=0.4in]{nform_genus_one.pdf}}
 \rput(2.5, 2.6){\includegraphics[height=0.3in]{nform_zipper_cozipper.pdf}}
 \end{pspicture}.
\]
\vspace{0.5cm}

\item [X.] We claim that the resulting decomposition of $\Sigma$ is now in the normal form. This follows   from steps Va, VI, VII, VIII and IX, and the following two remarks.
\bigbreak

\noindent a) Whenever an \textit{ordinary cap}  \raisebox{-3pt}{\includegraphics[height=0.14in]{unit.pdf}} appears in the resulting decomposition of $\Sigma,$ then \raisebox{-3pt}{\includegraphics[height=0.14in]{unit.pdf}} has its target in one of the following situations:
\[\psset{xunit=.22cm,yunit=.22cm}
 \begin{pspicture}(5,5)
 \rput(2.5, 0){\includegraphics[height=0.14in]{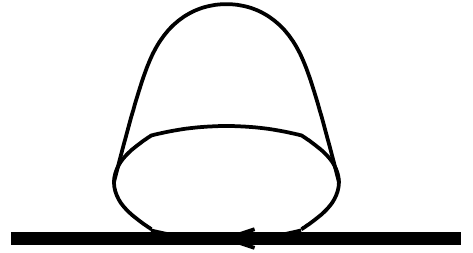}}
  \end{pspicture} 
 \begin{pspicture}(5,5)
 \rput(2.5, 0){\includegraphics[height=0.3in]{comult.pdf}}
 \rput(2.5,1.75){\includegraphics[height=0.14in]{unit.pdf}}
  \end{pspicture} 
  \begin{pspicture}(5,5)
 \rput(2.5, 0.8){\includegraphics[height=0.14in]{unit.pdf}}
 \rput(2.5,0){\includegraphics[height=0.14in]{counit.pdf}}
  \end{pspicture}
   \begin{pspicture}(5,5)
 \rput(2.5, 1.35){\includegraphics[height=0.14in]{unit.pdf}}
 \rput(2.5,-1){\includegraphics[height=0.41in]{nform_zipper_cozipper.pdf}}
  \end{pspicture}  
   \begin{pspicture}(5,5)
 \rput(2.5, 1.85){\includegraphics[height=0.14in]{unit.pdf}}
 \rput(2.5,-1){\includegraphics[height=0.5in]{nform_genus_one.pdf}}
  \end{pspicture} 
  \]
  \vspace{0.2cm}
  
\noindent The other situations are excluded by steps VI and VIIb.
\bigbreak

\noindent b) Whenever an \textit{ordinary cup}  \raisebox{-3pt}{\includegraphics[height=0.14in]{counit.pdf}} appears in the resulting decomposition of $\Sigma,$ then the source of \raisebox{-3pt}{\includegraphics[height=0.14in]{counit.pdf}} is in one of the following situations:
 \[\psset{xunit=.22cm,yunit=.22cm}
 \begin{pspicture}(5,5)
 \rput(2.5, 1){\includegraphics[height=0.26in]{cozipper.pdf}}
  \rput(2.5,-0.4){\includegraphics[height=0.14in]{counit.pdf}}
  \end{pspicture} 
 \begin{pspicture}(5,5)
 \rput(2.5, 1){\includegraphics[height=0.3in]{mult.pdf}}
 \rput(2.5,-0.7){\includegraphics[height=0.14in]{counit.pdf}}
  \end{pspicture} 
  \begin{pspicture}(5,5)
 \rput(2.5, 0.8){\includegraphics[height=0.14in]{unit.pdf}}
 \rput(2.5,0){\includegraphics[height=0.14in]{counit.pdf}}
  \end{pspicture}
   \begin{pspicture}(5,5)
 \rput(2.5,-0.8 ){\includegraphics[height=0.14in]{counit.pdf}}
 \rput(2.5,1.5){\includegraphics[height=0.41in]{nform_zipper_cozipper.pdf}}
  \end{pspicture}  
   \begin{pspicture}(5,5)
 \rput(2.5, -1.3){\includegraphics[height=0.14in]{counit.pdf}}
 \rput(2.5,1.5){\includegraphics[height=0.5in]{nform_genus_one.pdf}}
  \end{pspicture} 
  \]
  \vspace{0.2cm}
  
\noindent The other situations are excluded by step VIIIc. This completes the proof.
\end{enumerate}
\end{proof}

\begin{corollary}
Let $[\Sigma] \in \textbf{Sing-2Cob}(\textbf{n}, \textbf{m})$ be connected. Then $[\Sigma] = [\mbox{NF}(\Sigma)].$
\end{corollary}
\begin{proof}
$ [\mbox{NF}(\Sigma)] = [f^{-1}([\mbox{NF}_{W \to C}(f([\Sigma]))])] =  [f^{-1}([f([\Sigma])])]=[\Sigma].$
\end{proof}

\begin{corollary}
If $[\Sigma], [\Sigma'] \in \textbf{Sing-2Cob}(\textbf{n}, \textbf{m})$ are connected cobordisms with the same singular boundary permutation, genus  and singular number, then $[\Sigma] = [\Sigma'].$
\end{corollary}
\begin{proof}
This follows at once from the fact that the normal form of a singular cobordism is characterized by the singular boundary permutation, genus and singular number of the cobordism.
\end{proof}

Putting together the results of this section, we obtain:
\begin{theorem}
The symmetric monoidal category $\textbf{Sing-2Cob}$ is generated (under composition and disjoint union) by the following singular $2$-cobordisms:

\[\raisebox{-13pt}{\includegraphics[height=0.4in]{comult.pdf}}  \quad  \raisebox{-13pt}{\includegraphics[height=0.4in]{mult.pdf}}  \quad  \raisebox{-5pt}{\includegraphics[height=.2in]{unit.pdf}}  \quad \raisebox{-5pt}{\includegraphics[height=0.2in]{counit.pdf}} \quad \raisebox{-13pt}{\includegraphics[height=0.4in]{zipper.pdf}}  \quad \raisebox{-13pt}{\includegraphics[height=0.4in]{cozipper.pdf}}\quad \raisebox{-13pt}{\includegraphics[height=0.4in]{singcomult.pdf}}  \quad \raisebox{-13pt}{\includegraphics[height=0.4in]{singmult.pdf}} \]
  \[\raisebox{-10pt}{\includegraphics[height=0.35in]{braiding_CC.pdf}}\quad \raisebox{-10pt}{\includegraphics[height=0.35in]{braiding_WC.pdf}} \quad
\raisebox{-10pt}{\includegraphics[height=0.35in]{braiding_CW.pdf}} \quad \raisebox{-10pt}{\includegraphics[height=0.35in]{braiding_WW.pdf}}\]

with relations given in Proposition~\ref{prop:relations}.
\end{theorem}

%%%%%%%%%%%%%%%%%%%%%%%%%%%%%%%%%%%%%%%%%%%%%%%%%%%
\section{Twin TQFTs}\label{sec:TQFTs}
%%%%%%%%%%%%%%%%%%%%%%%%%%%%%%%%%%%%%%%%%%%%%%%%%%%

In this section we define the notion of twin TQFTs and show that the category of twin TQFTs is equivalent to the category of twin Frobenius algebras.

\begin{definition}
Let $\mathcal{C}$ be a symmetric monoidal category. A \textit{twin Topological Quantum Field Theory} (TQFT) in $\mathcal{C}$ is a symmetric monoidal functor $\textbf{Sing-2Cob} \to \mathcal{C}.$ A \textit{homomorphism of twin TQFTs} is a monoidal natural transformation of such functors. We denote by \textbf{T-TQFT}(\textit{C}) the category of twin TQFTs in $\mathcal{C}.$ 
\end{definition}

\begin{theorem}
The category $\textbf{Sing-2Cob}$ is equivalent as a symmetric monoidal category to the category \textbf{Th}(\textbf{T-Frob}).
\end{theorem}

\begin{proof}
We need to construct a functor $\Lambda: \textbf{Sing-2Cob} \to \textbf{Th}(\textbf{T-Frob}).$ In general, a monoidal functor is completely determined, up to equivalence, by its values on the generators of the source category. On the generating objects of \textbf{Sing-2Cob}, $\Lambda$ is defined as follows:
\[\Lambda \co \begin{cases}
\emptyset  \quad \quad\,\, \mapsto \textbf{1}  \\
\raisebox{-3pt}{\includegraphics[height=0.15in]{circle.pdf}} \,\,\,\,\, \mapsto C \\
  \raisebox{-3pt}{\includegraphics[height=0.15in]{singcircle.pdf}}\, \, \mapsto W \end{cases}\]
  
 Given a general object $\textbf{n} = (n_1, n_2,n_3, \cdots, n_k)$ in \textbf{Sing-2Cob}, the functor $\Lambda$ associates the tensor product in \textbf{Th}(\textbf{T-Frob}) of copies of $C$ and $W,$ with all parenthesis to the left. That is, $\Lambda(\textbf{n}) = (((\Lambda(n_1) \otimes \Lambda(n_2)) \otimes \Lambda(n_3)) \cdots \Lambda(n_k))$ with $n_i \in \{0,1\}$ and $\Lambda(0) : = C$ and $\Lambda(1) : =W.$

On the generating morphisms in \textbf{Sing-2Cob}, $\Lambda$ is defined as explained in Figure~\ref{fig:TQFT}.

\begin{figure}
\begin{align*}
\raisebox{-10pt}{\includegraphics[height=0.35in]{id_circle.pdf}}\quad  &\mapsto \quad [\,1_C \co C \to C\,]\hspace{2.8cm} \raisebox{-5pt}{\includegraphics[height=0.2in]{unit.pdf}}  \quad\mapsto \quad [\,\iota_C \co \textbf{1} \to C\,]\\
\raisebox{-10pt}{\includegraphics[height=0.35in]{mult.pdf}}\quad  &\mapsto \quad [\,m_C \co C\otimes C 
\to C\,] \hspace{1.95cm} \raisebox{-5pt}{\includegraphics[height=0.2in]{counit.pdf}} \quad \mapsto \quad [\,\epsilon_C \co C \to \textbf{1}\,]\\
\raisebox{-10pt}{\includegraphics[height=0.35in]{comult.pdf}}\quad  &\mapsto \quad [\,\Delta_C \co C \to C\otimes C\,]\hspace{2cm}\raisebox{-10pt}{\includegraphics[height=0.35in]{zipper.pdf}}\quad  \mapsto \quad [\,z \co C \to W\,] \\
\raisebox{-10pt}{\includegraphics[height=0.35in]{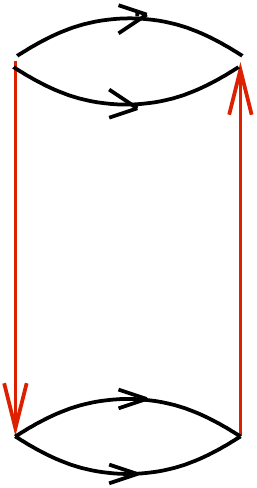}}\quad  &\mapsto \quad [\,1_W \co W \to W\,] \hspace{2.5cm} \raisebox{-5pt}{\includegraphics[height=0.2in]{singunit.pdf}} \quad \mapsto \quad [\,\iota_W \co \textbf{1} \to W\,]\\
\raisebox{-10pt}{\includegraphics[height=0.35in]{singmult.pdf}}\quad  &\mapsto \quad [\,m_W \co W\otimes  W\to W \,]\hspace{1.5cm}\raisebox{-5pt}{\includegraphics[height=0.2in]{singcounit.pdf}}\quad  \mapsto \quad [\,\epsilon_W \co W \to \textbf{1}\,]\\
\raisebox{-10pt}{\includegraphics[height=0.35in]{singcomult.pdf}}\quad  &\mapsto \quad [\,\Delta_W \co W  \to W \otimes W\,]\hspace{1.7cm} \raisebox{-10pt}{\includegraphics[height=0.35in]{cozipper.pdf}}\quad \mapsto \quad [\,z^* \co W \to C\,]\\
\raisebox{-10pt}{\includegraphics[height=0.35in]{braiding_CC.pdf}}\quad  &\mapsto \quad[\, \tau_{C,C} \co C \otimes C \to C\otimes C\,] \hspace{.5cm}
\raisebox{-10pt}{\includegraphics[height=0.35in]{braiding_WW.pdf}}\quad  \mapsto \quad [\,\tau_{W,W} \co W \otimes W \to W\otimes W \,]\\
\raisebox{-10pt}{\includegraphics[height=0.35in]{braiding_WC.pdf}}\quad  &\mapsto \quad [\,\tau_{W,C} \co W \otimes C \to C\otimes W\,] \hspace{.5cm}
\raisebox{-10pt}{\includegraphics[height=0.35in]{braiding_CW.pdf}}\quad  \mapsto \quad [\,\tau_{C,W} \co C \otimes W \to W\otimes C\,] 
\end{align*}
\caption{Assignments of $\Lambda$ on the generating morphisms}
\label{fig:TQFT}\end{figure}

There is an obvious way to extend inductively $\Lambda$ to a map defined on all morphisms of  \textbf{Sing-2Cob}. From the coherence theorems for symmetric monoidal categories, it follows that this assignment is well defined and extends to all general morphisms in \textbf{Sing-2Cob}. Moreover, the relations given in Proposition~\ref{prop:relations} and the proof that these are all the required relations in \textbf{Sing-2Cob} imply that the image of $\Lambda$ is a twin Frobenius algebra in  \textbf{Th}(\textbf{T-Frob}) and, in particular, that $\Lambda$ defines a functor $\textbf{Sing-2Cob}  \to \textbf{Th}(\textbf{T-Frob}).$
  
Given $\textbf{n} = (n_1, n_2, n_3,\dots, n_k), \,\textbf{m} = (m_1, m_2, m_3, \dots, m_l) \in \textbf{Sing-2Cob}$ we construct a natural isomorphism $\Lambda_2 \co \Lambda(\textbf{n}) \otimes \Lambda(\textbf{m}) \to \Lambda(\textbf{n} \amalg \textbf{m})$ as follows:
  \begin{align*} 
  \Lambda(\textbf{n}) &= (((\Lambda(n_1) \otimes \Lambda(n_2)) \otimes \Lambda(n_3)) \cdots \Lambda(n_k)),\\
  \Lambda(\textbf{m}) &= (((\Lambda(m_1) \otimes \Lambda(m_2)) \otimes \Lambda(m_3)) \cdots \Lambda(m_l)),\\
  \Lambda(\textbf{n} \amalg \textbf{m}) &= ((((((\Lambda(n_1) \otimes \Lambda(n_2)) \otimes \Lambda(n_3)) \cdots \Lambda(n_k)) \otimes \Lambda(m_1)) \otimes \Lambda(m_2)) \cdots \Lambda(m_l)).
  \end{align*}
 Define $\Lambda_0: = 1_{\textbf{1}}.$ It can be easily verified that the triple $(\Lambda, \Lambda_2, \Lambda_0)$ defines a symmetric monoidal functor.
 
 On the other hand, by reversing the arrows in the assignments of $\Lambda$ on the generating objects and morphisms in \textbf{Sing-2Cob}, we see that the singular cobordisms define a twin Frobenius algebra structure on the oriented circle and bi-web (thus we obtain the relations provided in Proposition~\ref{prop:relations}). Therefore, from the results given in Section~\ref{sec:theory}, we obtain a strict symmetric monoidal functor $\overline{\Lambda} \co \textbf{Th}(\textbf{T-Frob}) \to \textbf{Sing-2Cob}.$ If two objects in $ \textbf{Th}(\textbf{T-Frob})$ are related by a sequence of associators and unit constraints, then they are mapped to the same object in \textbf{Sing-2Cob}. 
 
 Given a general object $\textbf{n} \in \textbf{Sing-2Cob}$ we have that $\overline{\Lambda} \Lambda (\textbf{n}) = \textbf{n}.$ Thus $\overline{\Lambda} \Lambda = 1_{\textbf{Sing-2Cob}}.$ If $X$ is an object of \textbf{Th}(\textbf{T-Frob}), then $X$ is a parenthesized word made up of symbols $\textbf{1}, C, W$ and $\otimes.$ It is not hard to see that $\Lambda \overline{\Lambda}(X)$ is isomorphic to $X$ by a sequence of associators and unit constraints. In conclusion, $\Lambda$ and $\overline{\Lambda}$ define an equivalence of categories. 
 \end{proof}

\begin{corollary}
The category \textbf{T-Frob}($\mathcal{C}$) of twin Frobenius algebras in $\mathcal{C}$ is equivalent, as a symmetric monoidal category, to the category \textbf{T-TQFT}($\mathcal{C}$) of twin TQFTs in $\mathcal{C}$. 
\end{corollary}

%EXAMPLE
%%%%%%%%%%%%%%%%%%%%%%%%%%%%%%%%%%%%%%%%%%%%%%%%%%%
\section{Examples of twin Frobenius algebras} \label{sec:example}
%%%%%%%%%%%%%%%%%%%%%%%%%%%%%%%%%%%%%%%%%%%%%%%%%%%

\begin{example}\label{example}
Let $i$ be the primitive fourth root of unity and let $R = \mathbb{Z}[i][a, h]$ be the ring of polynomials in indeterminates $a$ and $h$ and with Gaussian integer coefficients. Consider also the ring $\mathcal{A} = R [X]/(X^2 - hX - a) = \brak{1,X}_{R}$ with inclusion map $\iota \co R \to \mathcal{A}, \,\,\iota(1) = 1.$ We remark that we consider these two rings for our first example, because they play an important role in~\cite{CC0, CC1, CC2}. 

It is well-known (see~\cite[Section 4]{Ka}) that $\iota$ is a Frobenius extension if and only if there exists an $\mathcal{A}$-bimodule map $\Delta \co \mathcal{A} \to \mathcal{A} \otimes_R \mathcal{A}$ and an $R$-module map $\epsilon \co \mathcal{A} \to R$ such that $\Delta$ is coassociative and cocommutative, and $(\epsilon \otimes \id) \Delta = \id.$

A \textit{Frobenius system} is a Frobenius extension together with a choice of the maps $\epsilon$ and $\Delta.$ We denote a Frobenius system by $\mathcal{F} = (R, \mathcal{A}, \epsilon, \Delta)$ (following~\cite{Ka} and~\cite{Kh2}).

We consider two Frobenius systems $\mathcal{F}_C = (R, \mathcal{A}, \epsilon_C, \Delta_C)$ and $\mathcal{F}_W = (R, \mathcal{A}, \epsilon_W, \Delta_W),$
with  \[\begin{cases}\epsilon_C(1) = 0 \\ \epsilon_C(X) = 1, \end{cases} \quad \,\begin{cases}\epsilon_W(1) = 0 \\ \epsilon_W(X) = -i,\end{cases} \]
and \[ \begin{cases}
 \Delta_C(1) = 1 \otimes X + X \otimes 1-h 1 \otimes 1\\ 
 \Delta_C(X) = X \otimes X + a 1 \otimes 1,\end{cases} \quad  \begin{cases}
 \Delta_W(1) = i(1 \otimes X + X \otimes 1 - h 1\otimes 1)\\ 
 \Delta_W(X) = i(X \otimes X + a 1 \otimes 1).\end{cases} \]
 
 $\mathcal{F}_W$ is a \textit{twisting} of $\mathcal{F}_C;$ that is, the comultiplication $\Delta_W$ and counit $\epsilon_W$ are obtained from $\Delta_C$ and $\epsilon_C$ by `twisting' them with invertible element $-i \in \mathcal{A}$: 
  \[ \epsilon_W (x) = \epsilon_C(-ix), \quad \Delta_W(x) = \Delta_C((-i)^{-1}x) = \Delta_C(ix), \,\,\text {for all}\, \,x \in \mathcal{A}.  \]
The fact that the above Frobenius systems differ by a twist is not surprising. Kadison showed that twisting by invertible elements of $\mathcal{A}$ is the only way to modify the counit and comultiplication in Frobenius systems (see~\cite[Theorem 1.6]{Ka}).
 
 We obtain two commutative Frobenius structures on $\mathcal{A}$:
 \[ \mathcal{A}_C = (\mathcal{A}, m_C, \iota_C, \Delta_C, \epsilon_C), \quad \mathcal{A}_W = (\mathcal{A}, m_W, \iota_W, \Delta_W, \epsilon_W),\]
 where $\iota_C = \iota_W = \iota.$ Multiplication maps $m_{C,W}\co \mathcal{A} \otimes \mathcal{A} \rightarrow \mathcal{A}$ are defined by the same rules:
$$ \begin{cases}
 m_{C,W}(1 \otimes X) = m_{C,W}(X \otimes 1) = X\\ 
m_{C,W}(1 \otimes 1) =1,  m_{C,W}(X \otimes X) = hX + a.
 \end{cases}$$
 
We define the following homomorphisms:
\[z \co \mathcal{A}_C \to \mathcal{A}_W, \begin{cases} z(1) = 1 \\z(X) = X, \end{cases} \quad z^* \co\mathcal{A}_W \to \mathcal{A}_C, \begin{cases} z^*(1) = -i \\z^*(X) = -iX. \end{cases} \]

A straightforward computation shows that $(\mathcal{A}_C, \mathcal{A}_W, z, z^*)$ is `almost' twin Frobenius, in the sense that all properties of a twin Frobenius algebra are satisfied  except for the ``genus-one condition" which holds up to a minus sign. 

Let \textbf{R-Mod} be the category of $R$-modules and module homomorphisms. We denote by $\mathcal{T} \co \textbf{Sing-2Cob} \to \textbf{R-Mod}$ the TQFT corresponding to $(\mathcal{A}_C, \mathcal{A}_W, z, z^*)$  (c.f. Section~\ref{sec:TQFTs}), which assigns the ground ring $R$ to the empty 1-manifold and assigns $\mathcal{A}^{\otimes k}$ to a generic object $\textbf{n} = (n_1, n_2, \dots, n_k)$ in \textbf{Sing-2Cob} with $\vert \textbf{n} \vert = k.$ The $j$-th factor of $\mathcal{A}^{\otimes k}$  is endowed with the structure $\mathcal{A}_C$ if $n_j = 0 = \raisebox{-3pt}{\includegraphics[height=0.15in]{circle.pdf}}, $ or with the structure $\mathcal{A}_W$ if $n_j = 1 = \raisebox{-3pt}{\includegraphics[height=0.15in]{singcircle.pdf}}.$ 

On the generating morphisms of the category \textbf{Sing-2Cob}, the functor $\mathcal{T}$ is defined as follows:
\begin{equation*} \mathcal{T} \co\raisebox{-13pt}{\includegraphics[height=0.4in]{comult.pdf}} \to \Delta_C \quad \mathcal{T} \co \raisebox{-13pt}{\includegraphics[height=0.4in]{mult.pdf}} \to m_C \quad \mathcal{T} \co \raisebox{-5pt}{\includegraphics[height=.2in]{unit.pdf}} \to \iota_C  \quad \mathcal{T} \co \raisebox{-5pt}{\includegraphics[height=0.2in]{counit.pdf}} \to \epsilon_C \label{eq:generators_circle}
\end{equation*}
\begin{equation*} \mathcal{T} \co \raisebox{-13pt}{\includegraphics[height=0.4in]{singcomult.pdf}} \to \Delta_W \quad  \mathcal{T} \co \raisebox{-13pt}{\includegraphics[height=0.4in]{singmult.pdf}} \to m_W \quad \mathcal{T} \co \raisebox{-5pt}{\includegraphics[height=.2in]{singunit.pdf}} \to \iota_W  \quad \mathcal{T} \co \raisebox{-5pt}{\includegraphics[height=0.2in]{singcounit.pdf}} \to \epsilon_W \label{eq:generators_web}
\end{equation*}
\begin{equation*}
\mathcal{T} \co\raisebox{-13pt}{\includegraphics[height=0.4in]{zipper.pdf}} \to z \quad \mathcal{T} \co\raisebox{-13pt}{\includegraphics[height=0.4in]{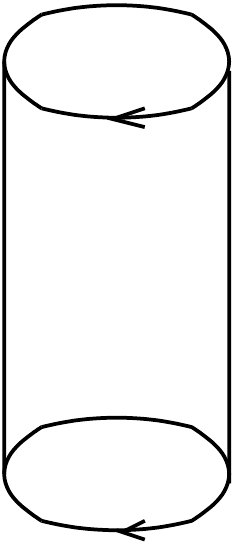}} \to \id_{\mathcal{A}_C} \quad \mathcal{T} \co\raisebox{-13pt}{\includegraphics[height=0.4in]{cozipper.pdf}} \to z ^* \quad \mathcal{T} \co \raisebox{-13pt}{\includegraphics[height=0.4in]{identity_web.pdf}} \to \id_{\mathcal{A}_W}.
\end{equation*} 
\end{example}

 It is worth noting that the TQFT defined in this example satisfies the local relations for the `dot free'  version of the universal $\mf{sl}(2)$ foam cohomology for links (see \cite[Section 4]{CC2}). To be precise, the following identities hold:
 
 \[ 2\,\mathcal{T}(\,\raisebox{-15pt}{\includegraphics[height=.45in]{identity_circle.pdf}}\,) = \mathcal{T}(\,\raisebox{-15pt}{\includegraphics[height=.45in]{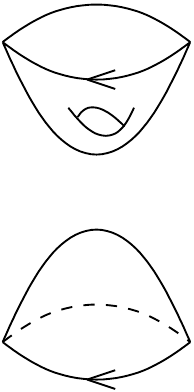}}\,) + \mathcal{T}(\,\raisebox{-15pt}{\includegraphics[height=.45in]{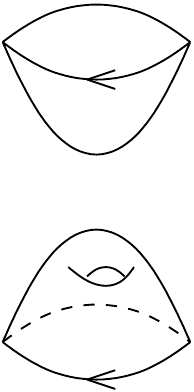}}\,),\hspace{1cm}\mathcal{T}(\, \raisebox{-8pt}{\includegraphics[width=0.3in]{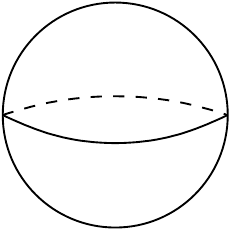}}\,)=0   \]
 \[ \mathcal{T}(\,\raisebox{-5pt}{\includegraphics[height=.25in]{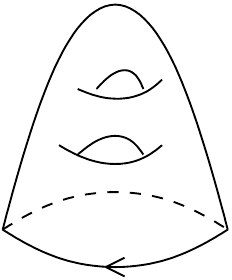}}\,) = (h^2 + 4a)\,\mathcal{T}(\, \raisebox{-5pt}{\includegraphics[height=.2in]{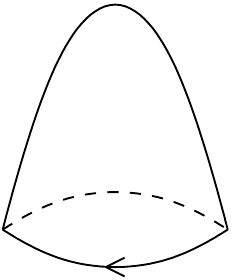}}\,), \hspace{1cm} \mathcal{T}(\,\raisebox{-4pt}{\includegraphics[width=0.4in]{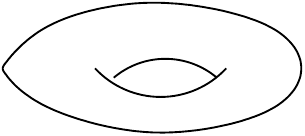}}\,)=2 \]
  \[\psset{xunit=.22cm,yunit=.22cm}
\mathcal{T}( \begin{pspicture}(3,3)
 \rput(1.5, 1){\includegraphics[height=0.2in]{singunit.pdf}}
  \rput(1.5,-0.2){\includegraphics[height=0.2in]{singcounit.pdf}}
  \end{pspicture}) = 0, \hspace{1cm}   
 \mathcal{T}(  \begin{pspicture}(5,5)
 \rput(2.5, 3.85){\includegraphics[height=0.13in]{unit.pdf}}
 \rput(2.5,1){\includegraphics[height=0.5in]{nform_genus_one.pdf}}
  \rput(2.5,-3.35){\includegraphics[height=0.4in]{nform_zipper_cozipper.pdf}}
   \rput(2.5, -5.6){\includegraphics[height=0.13in]{counit.pdf}}
 \end{pspicture} ) =-2i.
 \]
 \vspace{0.5cm}
 
  The last two identities are the `UFO' local relations used in~\cite{CC2} and depicted below:
 \[ \raisebox{-8pt}{\includegraphics[height=.3in]{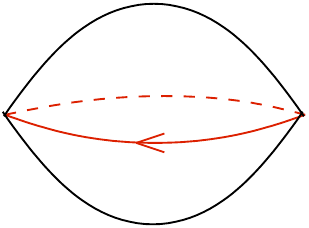}} = 0, \quad \raisebox{-8pt}{\includegraphics[height=.3in]{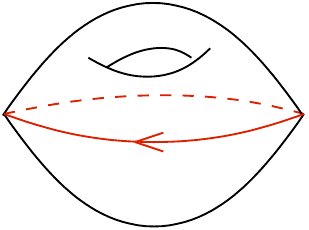}} = -2i.\]
 Notice that we also have $z \circ z^* = -i \id (\raisebox{-1pt}{\includegraphics[height=.1in]{singcircle.pdf}})$ and $ z^* \circ z = -i \id(\raisebox{-1pt}{\includegraphics[height=.1in]{circle.pdf}}),$ which are equivalent to:
 \[\psset{xunit=.22cm,yunit=.22cm}
\mathcal{T}( \begin{pspicture}(3,3)
 \rput(1.5, 1.85){\includegraphics[height=0.3in]{cozipper.pdf}}
 \rput(1.5, -0.5){\includegraphics[height=0.3in]{zipper.pdf}}
 \end{pspicture}) = -i\, \mathcal{T}(\begin{pspicture}(3,3)
 \rput(1.5, 0.5){\includegraphics[height=0.3in]{id_singcircle.pdf}}
 \end{pspicture}) \quad \mbox{and} \quad 
\mathcal{T}( \begin{pspicture}(3,3)
 \rput(1.5, 0.4){\includegraphics[height=0.45in]{nform_zipper_cozipper.pdf}}
  \end{pspicture}) = -i\, \mathcal{T}(\begin{pspicture}(3,3)
 \rput(1.5, 0.5){\includegraphics[height=0.3in]{id_circle.pdf}}
 \end{pspicture}).
 \]

The later are versions of the ``curtain identities" used in~\cite{CC0, CC1, CC2}.

Even though we have here only an `almost' twin Frobenius algebra, motivated by the above remarks, we believe that this example will play the key role in describing the universal $\mf{sl}(2)$ foam link cohomology using twin TQFTs. We will consider this problem in a subsequent paper.  
 
 \begin{example}
 Let $k$ be a field. Consider the commutative Frobenius algebra structure $(C, m_C,\iota_C, \Delta_C, \epsilon_C)$ on the truncated polynomial algebra $C = k[x]/(x^2),$ where 
 \[m_C(1 \otimes 1) =1, \,m_C(1 \otimes x) = m_C(x \otimes 1) = x, \,  m_C(x \otimes x) = 0,\]
  \[\Delta_C(1) = 1\otimes x + x \otimes 1, \, \Delta_C(x) = x \otimes x \quad \text{and} \quad \epsilon_C(1) = 0,\, \epsilon_C(x) = 1.\] 
 
Consider also the truncated polynomial algebra $W = k[y]/(y^n), n \geq  2,$ which admits a commutative and therefore symmetric Frobenius algebra structure $(W, m_W,\iota_W, \Delta_W, \epsilon_W)$ with $\epsilon_W(y^{n-1}) = 1$ and $\epsilon_W(y^a) = 0$ for all $0 \leq k \leq n-2.$ The rules for multiplication $m_W$ are obvious from the definition of $W.$ The comultiplication $\Delta_W \co W \to W \otimes_k W $ is dual to the multiplication via the counit map $\epsilon_W,$ and it is defined by 
 \[\Delta_W(y^a) = \sum_{i =0}^{n-1-a} y^{i+a} \otimes y^{n-1-i} \quad \mbox{for all} \quad 0 \leq a \leq n-1.\] 
  If char $k = n,$ then $(C, W, z, z^*)$ forms a twin Frobenius algebra with $ z(1) = 1, \, z(x) =  0$ and $z^*(y^{n-1}) = x, \, z^*(y^{a}) = 0$, for all $0 \leq a \leq n-2.$

For a field $k$ with char $k \neq n,$ the above example does not satisfy the genus-one condition, since $z \circ m_C \circ \Delta_C \circ z^* (y^a) = 0$ for all $0 \leq a \leq n-1,$ while $m_W \circ \tau_{W, W} \circ \Delta_W(1) = n y^{n-1}$ and $m_W \circ \tau_{W, W} \circ \Delta_W(y^a) = 0$, for all $1 \leq a \leq n-1.$
 \end{example}
 
 \begin{example}\label{example_twin}
 Consider the polynomial ring $R = \mathbb{Z}[a, h]$, and the truncated polynomial algebras 
 $ C = R[x]/(x^2 -hx -a)\, \, \mbox{and} \,\, W = R[y]/(y^2 -hy -a).$ Both $C$ and $W$ are commutative Frobenius with structure maps
 \[ \Delta_C(1) = 1 \otimes x + x \otimes 1 -h 1 \otimes 1, \,\, \Delta_C(x) = x \otimes x + a 1 \otimes 1, \mbox{and} \,\, \epsilon_C(1) = 0, \, \epsilon_C(x) = 1, \]
   \[ \Delta_W(1) = 1 \otimes y + y \otimes 1 -h 1 \otimes 1, \,\, \Delta_W(y) = y \otimes y + a 1 \otimes 1, \mbox{and} \,\, \epsilon_W(1) = 0, \, \epsilon_W(y) = 1. \]
   
Then $(C, W ,z, z^*)$ is twin Frobenius, with $z(1) = 1, z(x) = y$ and $z^*(1) = 1, z^*(y) =x.$
\end{example}

\begin{example}
 Let $R$ be a commutative ring, and let $\alpha \in R$ such that $\alpha^2 = 1$. Consider the algebras $ C = R[x]/(x^2 -hx -a)\, \, \mbox{and} \,\, W = R[y]/(y^2 -hy -a),$ where $a, h \in R$, with Frobenius algebra structures given by
 \[ \Delta_C(1) = 1 \otimes x + x \otimes 1 -h 1 \otimes 1, \,\, \Delta_C(x) = x \otimes x + a 1 \otimes 1, \,\,\epsilon_C(1) = 0, \, \epsilon_C(x) = 1, \]
   \[ \Delta_W(1) = \alpha(1 \otimes y + y \otimes 1 -h 1 \otimes 1), \,\, \Delta_W(y) = \alpha( y \otimes y + a 1 \otimes 1), \,\, \epsilon_W(1) = 0, \, \epsilon_W(y) = \alpha. \]
   
  We have that $(C, W ,z, z^*)$ is twin Frobenius, with $z(1) = 1, z(x) = y$ and $z^*(1) = \alpha, z^*(y) = \alpha x.$
  \end{example}

\textbf{Other examples.}
Examples of knowledgeable Frobenius algebras given by Lauda and Pfeiffer in~\cite[Examples 3.6--3.8]{LP2} are also twin Frobenius algebras, with the same restrictions on char $k$. Example 3.13 from~\cite{LP2} is a twin Frobenius algebra as well, but we need no restriction on char $k;$ actually, this is a particular case of our Example~\ref{example_twin} over an arbitrary field $k$ and for $h = 1.$  

In all of the above examples, the algebra $W$ is commutative. We provide now a couple of examples where $W$ is non commutative, but symmetric (we remark that these are particular cases of the algebras given in Examples 3.1 and 3.2 in ~\cite{LP2}). 

\begin{example}
Let $k$ be a field, $n \in \mathbb{N}$, and let $W$ be the matrix algebra of $n \times n$ matrices over $k$; that is $W = M_n(k)$. Denote by $\{e_{ij}\}$ the standard basis for $W$. There is a symmetric Frobenius algebra structure $(W, m_W,\iota_W, \Delta_C, \epsilon_W)$, with 
\[\Delta_W(e_{ij}) = \alpha \sum_{k=1}^n e_{ik} \otimes e_{kj}, \quad \epsilon_W(e_{ij}) = \alpha^{-1}\delta_{ij},\]
\[m_W(e_{ij} \otimes e_{kl}) = \delta_{jk}e_{il}, \quad \iota_W(1) = \sum_{i=1}^n e_{ii},\]
 where $\alpha = \pm 1 \in k$. 
 
Consider the (trivial) Frobenius algebra structure $(C, m_C,\iota_C, \Delta_C, \epsilon_C)$ on $C: = k$, with $\Delta_C(1) = 1 \otimes 1, \epsilon_C(1) = 1, m_C(1 \otimes 1) = 1$ and $\iota(1) = 1$. Then $(C, W, z, z^*)$ is a twin Frobenius algebra, with $z(1) = \sum_{i = 1}^n e_{ii}$ and $z^*(e_{ij}) = \alpha \delta_{ij}$.
\end{example}

\begin{example}
Let $k$ be a field containing $2$. Consider the free associative unital $k$-algebra $W: = \mathbb{H}_k$ of quaternions, generated by $I, J, K$ subject to the relations
\[I^2 = J^2 = K^2 = -1, IJ = -JI = K, JK = -KJ = I, KI = -IK = J.\]
Then $(W, m_W,\iota_W, \Delta_C, \epsilon_W)$ is symmetric Frobenius, with
\begin{eqnarray*}
\Delta_W(1) = \alpha(1 \otimes 1 - I\otimes I - J \otimes J - K \otimes K), \Delta_W(I) = \alpha(1 \otimes I + I\otimes 1 +J \otimes K - K \otimes J), \\
\Delta_W(J) = \alpha(1 \otimes J + J\otimes 1 + K \otimes I - I \otimes K), \Delta_W(K) = \alpha(1 \otimes K + K\otimes 1 +I \otimes J - J \otimes I),
\end{eqnarray*}
\begin{eqnarray*}
\epsilon_W(1) = \alpha^{-1}, \,\, \epsilon_W(I) = \epsilon_W(J) = \epsilon_W(K) =  0,
\end{eqnarray*}
where $\alpha = \pm 1/2$.

There is a twin Frobenius algebra $(C, W, z, z^*)$, where $C$ is the underlying field $k$ equipped with the trivial Frobenius structure (as in the previous example), and where
\[z(1) = 1,\,\, \text{and} \,\, z^*(1) = 4\alpha, \,\,z^*(I) =  z^*(J) =  z^*(K) = 0.  \]
\end{example}

\begin{remark}
Considering any of these (last two) examples $(C_1, W_1, z_1, z_1^*)$, where $C_1 = k$ with the trivial Frobenius algebra structure, we can tensor it with one of the twin Frobenius algebras $(C_2, W_2, z_2, z_2^{*})$ given in the first set of examples (with $C_2$ commutative). We obtain a twin Frobenius algebra $(C_2, W_1 \otimes W_2, z_1 \otimes z_2, z_1^* \otimes z_2^{*})$ with a bigger and possible non commutative $W_1 \otimes W_2$.
\end{remark}
 
 \textbf{Acknowledgements.} The author is grateful to the NSF for supporting her through grant DMS-0906401. She would also like to thank Aaron Lauda for helpful discussions that motivated her to return to this project.


\begin{thebibliography}{999}
\bibitem{A} L. Abrams, {\em Two-dimensional topological quantum field theories and Frobenius algebras}, J. Knot Theory Ramifications \textbf{5}, no. 5 (1996), 569-587.
\bibitem{BN1} D. Bar-Natan, {\em On Khovanov's categorification of the Jones polynomial}, Algebr. Geom. Topol. \textbf{2-16} (2002), 337-370; arXiv:math.QA/0201043.
\bibitem{BN2} D. Bar-Natan, {\em Khovanov's homology for tangles and cobordisms}, 
 Geom.Topol. \textbf{9} (2005), 1443-1499; arXiv:math.GT/0410495.
 \bibitem{CC0} C. Caprau, {\em $\mf{sl}(2)$ tangle homology with a parameter and singular cobordisms}, Algebr. Geom. Topol. \textbf{8} (2008) 729-756.
\bibitem{CC1} C. Caprau, {\em The universal $\mf{sl}(2)$ cohomology via webs and foams}, Topology Appl. \textbf{156} (2009); arXiv:math.GT/0802.2848.
\bibitem{CC2} C. Caprau, {\em On the $\mf{sl}(2)$ foam cohomology computations}, J. Knot Theory Ramifications \textbf{18}, Issue 9 (2009) 1313-1328; arXiv:math.GT/0805.4651.
\bibitem{CMW} D. Clark, S. Morrison, K. Walker, {\em Fixing the functoriality of Khovanov homology}, Geom.Topol. \textbf{13} (2009) 1499-1582; arXiv:math.GT/0701339.
\bibitem{D} R. Dijkgraaf, {\em A geometric approach to two dimensional conformal field theory}, Ph.D. thesis, University of Utrecht, Utrecht (1989).
\bibitem{Ka} L. Kadison, {\em New examples of Frobenius extensions}, University Lecture Series 14, American Mathematical Society, Providence, RI (1999).  
\bibitem{Kh1} M. Khovanov, {\em A categorification of the Jones polynomial}, Duke Math. J. \textbf{101} (2000) no. 3, 359-426; arXiv:math.QA/9908171.
\bibitem{Kh3} M. Khovanov, {$\mf{sl}(3)$ link homology},  Algebr. Geom. Topol. \textbf{4} (2004), 1056-1081; arXiv:math.QA/0304375.
\bibitem{Kh2} M. Khovanov, {Link homology and Frobenius Extensions}, Fundamenta Mathematicae, \textbf{190} (2006), 179-190; arXiv:math.QA/0411447.
\bibitem{K} J. Kock, {\em Frobenius Algebras and 2D Topological Quantum Field Theories}, London Mathematical Society Student Texts 59, Cambridge University Press, Cambridge, 2004. 
\bibitem{L} M. L. Laplaza, {\em Coherence for categories with group structure: an alternative approach}, J. Algebra \textbf {84} (1983) no. 2, 305-323.
\bibitem{LP} A. Lauda, H. Pfeiffer, {\em Open-Closed strings: Two-dimensional extended TQFTs and Frobenius algebras}, Topology Appl. \textbf{155} (2008) no. 7, 623--666; arXiv:math.AT/0510664.
\bibitem{LP2} A. Lauda, H. Pfeiffer, {\em Open-closed TQFTs extend Khovanov homology from links to tangles}, J. Knot Theory Ramifications \textbf{18} (2009), 87-150; arXiv:math.GT/0606331.
\bibitem{Lee} E. S. Lee, {\em An endomorphism of the Khovanov invariant}, Adv. Math. \textbf{197} (2005), no 2, 554-586; arXiv:math. GT/0210213.
\bibitem{S} S. F. Sawin, {\em Direct sum decomposition and indecomposable TQFTs}, J. Math Phys. \textbf{36} (1995) no.12,  6673-6680; arXiv:q-alg/9505926.
\end{thebibliography}
\end{document}